\documentclass[a4paper,10pt]{article}
\usepackage[utf8]{inputenc}
\usepackage{amsmath, amsfonts, amsthm,amssymb,  bbm, dsfont, lmodern,color, mathrsfs}
\usepackage{math}
\usepackage{hyperref,mathtools} 
\usepackage[left=2cm,right=2cm,top=3cm,bottom=3cm]{geometry}

\usepackage[font=small,labelfont=bf]{caption}
\usepackage{mathrsfs}
\usepackage{verbatim}
\usepackage{amsthm}
\usepackage{thmtools}

\newtheorem{lemma}{Lemma}[section]
\newtheorem{theorem}[lemma]{Theorem}
\newtheorem{proposition}[lemma]{Proposition}

\newtheorem{definition}[lemma]{Definition}

\theoremstyle{remark}
\newtheorem{remark}[lemma]{Remark}

\usepackage{pst-node}
\usepackage{pst-plot}
\usepackage{tikz-cd}
\usetikzlibrary{arrows.meta, positioning}
\usepackage{scalerel}
\usepackage{mathabx}
\usepackage{float}

\usepackage{cleveref}
\usepackage{mathrsfs}

\usepackage{subcaption}
\usepackage[titles,subfigure]{tocloft}  

\def\be{\begin{equation}}   \def\ee{\end{equation}}
\def\bea{\begin{eqnarray}}  \def\eea{\end{eqnarray}}

\newcommand\restr[2]{{
  \left.\kern-\nulldelimiterspace 
  #1 
  \vphantom{\big|} 
  \right|_{#2} 
  }}

\newcommand{\de}{{\rm d}}

\newcommand{\rvline}{\hspace*{-\arraycolsep}\vline\hspace*{-\arraycolsep}}

\makeatletter
\newcommand{\subalign}[1]{%
  \vcenter{%
    \Let@ \restore@math@cr \default@tag
    \baselineskip\fontdimen10 \scriptfont\tw@
    \advance\baselineskip\fontdimen12 \scriptfont\tw@
    \lineskip\thr@@\fontdimen8 \scriptfont\thr@@
    \lineskiplimit\lineskip
    \ialign{\hfil$\m@th\scriptstyle##$&$\m@th\scriptstyle{}##$\hfil\crcr
      #1\crcr
    }%
  }%
}
\makeatother

  \newcommand{\BVe}[4]{\left(\sB(\al,\mu,\e)\, \tf^{#1}_{{#2}} \ , \tf^{#3}_{#4} \right) }

\newcommand{\sL}{\mathscr{L}}
\newcommand{\sB}{\mathscr{B}}
\newcommand{\molt}[2]{\left(#1\,,\,#2\right)}

\newcommand{\kB}{\mathfrak{B}}
\renewcommand{\white}[1]{{\textcolor{white}{#1}}}
\newcommand{\ent}[6]{\begingroup 
\setlength\arraycolsep{-2pt}
\begin{matrix}{\tB_{#1}^{[#2]}} &\white{{|}^{|}}_{#3,#4}^{#5,#6}\end{matrix}
\endgroup}

\usepackage{todonotes}

\title{\bf 
McLean resonances and \\
$3d$ spectral instability of Stokes waves }

\numberwithin{equation}{section}

 \author{Massimiliano Berti, Alberto Maspero, Antonio Milosh Radakovic\footnote{
International School for Advanced Studies (SISSA), Via Bonomea 265, 34136, Trieste, Italy.
 \textit{Emails: }
 \texttt{berti@sissa.it}, \texttt{amaspero@sissa.it}, \texttt{aradakov@sissa.it}
 }}

\date{}

\begin{document}
\maketitle

\noindent 
{\bf Abstract.}  
The spectral instability of  traveling periodic water waves has been investigated for more than sixty years, since the seminal discovery of  Benjamin and Feir. Despite an extensive literature, no rigorous theory has been available for arbitrary three-dimensional 
--longitudinal and transverse-- perturbations. 
We establish the first rigorous description of the $3d $ unstable spectrum
of small-amplitude gravity  Stokes waves in deep water in a full neighborhood of the McLean resonant curves.
Our results reveal that the Benjamin–Feir instability and the first longitudinal high-frequency isola originate from the same resonant interaction, hidden in the purely longitudinal setting.
The dominant instabilities emerge for Fourier-Bloch  parameters 
  near the origin, corresponding to the $3d $ 
  Benjamin–Feir modulational instability. 
Our approach provides 
quantitative  bounds for the real parts of the unstable eigenvalues and establishes a computable necessary and sufficient criterion for the onset of instability near arbitrary high-frequency McLean curves.
These results are enabled by three key 
innovations:  ($i$) a Kato perturbative analysis allowing  Lipschitz-type singularities of the linearized operator with respect to the Fourier–Bloch parameters; ($ii$) a polar-analytic KAM-type decoupling  isolating the unstable eigenvalue pairs near the origin; and ($iii$) an analytic continuation argument in full neighborhoods of the 
McLean  curves.
A primary 
challenge  is to establish fine regularity  properties
for  the Dirichlet-Neumann operator conjugated via the Fourier-Bloch transform.

 {\footnotesize
\tableofcontents }

\section{Introduction}

Discovered in the pioneering work \cite{stokes} in 1847, Stokes waves are spatially periodic traveling solutions, constant along the $y $-direction, propagating along the  $ x $-axis at constant speed. 
The 
rigorous existence 
of small-amplitude Stokes waves was proved only several decades later in \cite{LC,Struik,Nek}.

A question  of fundamental physical importance concerns their stability properties. 
In the 1960s, Benjamin and Feir \cite{BF,Benjamin}, see also \cite{Whitham,Li,Zak1},  made the unexpected discovery --through experiments and formal analysis-- that small-amplitude Stokes waves in sufficiently deep water are unstable under long-wave  longitudinal perturbations.
This  phenomenon
was first established in rigorous mathematical terms in \cite{BrM,NS,ChenSu}.   
The complete description of the  
unstable  eigenvalues 
near zero --which form a ``figure-8"-- 
has been obtained only recently in  \cite{BMV1,BMV3,BMV_ed}.
In addition to the Benjamin–Feir mechanism, a distinct class of longitudinal ``high-frequency" unstable eigenvalues forming ``isolas" away from the origin 
was numerically detected by \cite{DO}. 
The unstable isola closest 
to zero has been rigorously established in \cite{HY} in finite depth  and   \cite{BMV4} in deep water. The existence of arbitrarily many such isolas is a recent result 
\cite{BMV5} in finite depth.

 \smallskip

From a physical perspective, however, 
perturbations in real water-wave environments are never purely one-dimensional
(experimental evidence of transverse instability is reported in \cite{Melville,TSD}) and 
a comprehensive stability theory must account 
for waves that depend 
on both spatial variables $ (x,y)$.
Formal results are  given e.g. in \cite{BR,D,TD, Ak}.

The rigorous mathematical analysis of such $ 3d$-instabilities 
for
small-amplitude gravity Stokes waves in deep water, 
is  the main challenge 
of this paper. 
We provide a  mathematical description that unifies and extend all the previous results restricted to longitudinal perturbations  with the recent progresses
\cite{CNS,CNS1,JTSY} on the local formation of isolas under transverse disturbances.   
Surprisingly, we show that the Benjamin–Feir figure 8 and the first longitudinal high-frequency isola are not independent phenomena, but originate from the same pair of resonant eigenvalue collisions within the full three-dimensional Bloch-Floquet parameter space.


\smallskip

The problem can be mathematically formulated as follows.  Let $\cL_\e$ denote the linearized 
operator at a $ 2 \pi $-periodic 
pure gravity Stokes wave of small amplitude $\e$, written in the reference frame moving alongside the wave (the operator $\cL_\e$ in \eqref{cLvero} is actually obtained after applying the 
``good unknown of Alinhac"
and a Levi-Civita transformation).  Since the coefficients of 
$\cL_\e$ depend only on the  $2\pi$-periodic 
variable $x$, it is natural to 
seek Fourier-Bloch wave
solutions of 
the  
linear system 
$\pa_t h = \cL_\e h$
of the form   
\begin{equation}\label{h.intro}
h(t,x,y) = \mathrm{Re} \big( e^{\lambda t} \, e^{\im (\alpha y + \mu x )} \, v(x) \big) \, ,
\end{equation}
where $v(x)$ is a $2\pi$-periodic function, $\mu \in \R$ is the \emph{longitudinal} Bloch-Floquet exponent, $\alpha \in \R$ is the \emph{transverse} Fourier wave number, and $\lambda \in \C$. If  
$\alpha = 0$, resp.
$\mu = 0 $, we refer to $ h $ as a  purely longitudinal,
resp. transverse, wave.
Equivalently,   $ v(x) $ 
in \eqref{h.intro} is an eigenvector 
of the 
operator  
\be\label{calLmua}
\cL(\al, \mu, \e) := e^{- \im (\al y + \mu x)} \circ  \cL_\e \circ e^{\im (\al y + \mu x)} \ , 
\ee
acting on $2\pi$-periodic functions of $x$,  
with eigenvalue $ \lambda $.
The Stokes wave is  \emph{spectrally unstable} if there exist  
$(\alpha,\mu)$ such that $\cL(\alpha,\mu,\e)$ has an eigenvalue $\lambda$ with $\Re \lambda \neq 0$.

Eigenvalues with non zero real part  may only bifurcate, for $\e \neq 0$,  from multiple purely imaginary eigenvalues of the unperturbed Hamiltonian operator
\begin{equation}\label{cLvero0}
    \cL (\alpha,\mu, 0)    
   = \begin{pmatrix}
        \pa_x +\im\mu 
        &  |D|_{\al,\mu} \\
        -1
         & \pa_x +\im\mu
    \end{pmatrix}  
   \, , \quad 
   |D|_{\al,\mu}:= ((D+\mu)^2+\al^2)^\frac12  \, , \quad D := -\im\pa_x \, , 
 \end{equation}  
whose spectrum is purely imaginary.
McLean 
\cite{Mc1,Mc2}
classified such {\it resonances}: the set of $(\alpha,\mu)$ for which $\cL(\alpha,\mu,0)$ admits multiple eigenvalues forms  a sequence of closed 
analytic varieties  --the \emph{McLean curves} $\cM^{(\tp)}$ in  \Cref{fig:Mc16} up to integer  translations in $\mu $. 
It turns out that for any  $ (\alpha,\mu) \in \cM^{(\tp)}$, two purely imaginary eigenvalues of $\cL(\alpha,\mu,0)$ collide, 
giving rise to a  ``McLean-resonance". The intersection  
points  $ \cM^{(\tp)} \cap  \{ \al = 0 \} $ correspond to 
the unbounded sequence of  purely longitudinal isolas in \cite{DO}.
At   
$(\alpha,\mu) = (0,0)$,  
the operator $\cL(0,0,0)$ possesses  the zero eigenvalue 
with  algebraic multiplicity four
(this happens because Stokes waves are not isolated, but form a four-dimensional family). 
The limit 
$ (\al,\mu) \to (0,0) $ corresponds to the 
longitudinal-transverse 
long waves  regime
where Benjamin–Feir instability is expected to  occur.   

\begin{figure}[h!]
    \centering
    \includegraphics[width=0.40\linewidth]{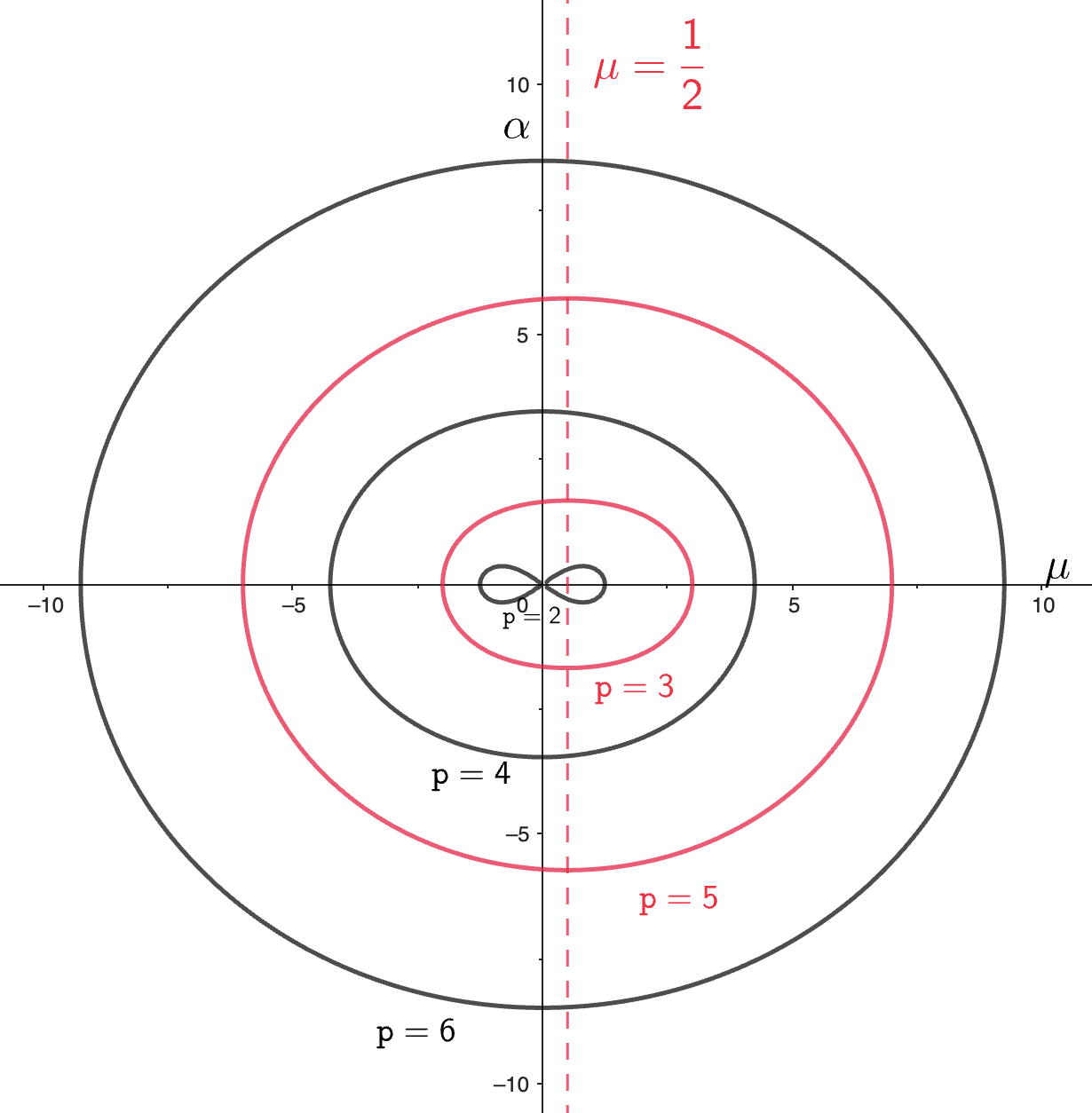}
    \hfill \begin{tikzpicture}

\node[anchor=south west, inner sep=0] (img) at (0,0)
    {\includegraphics[width=0.52\textwidth]{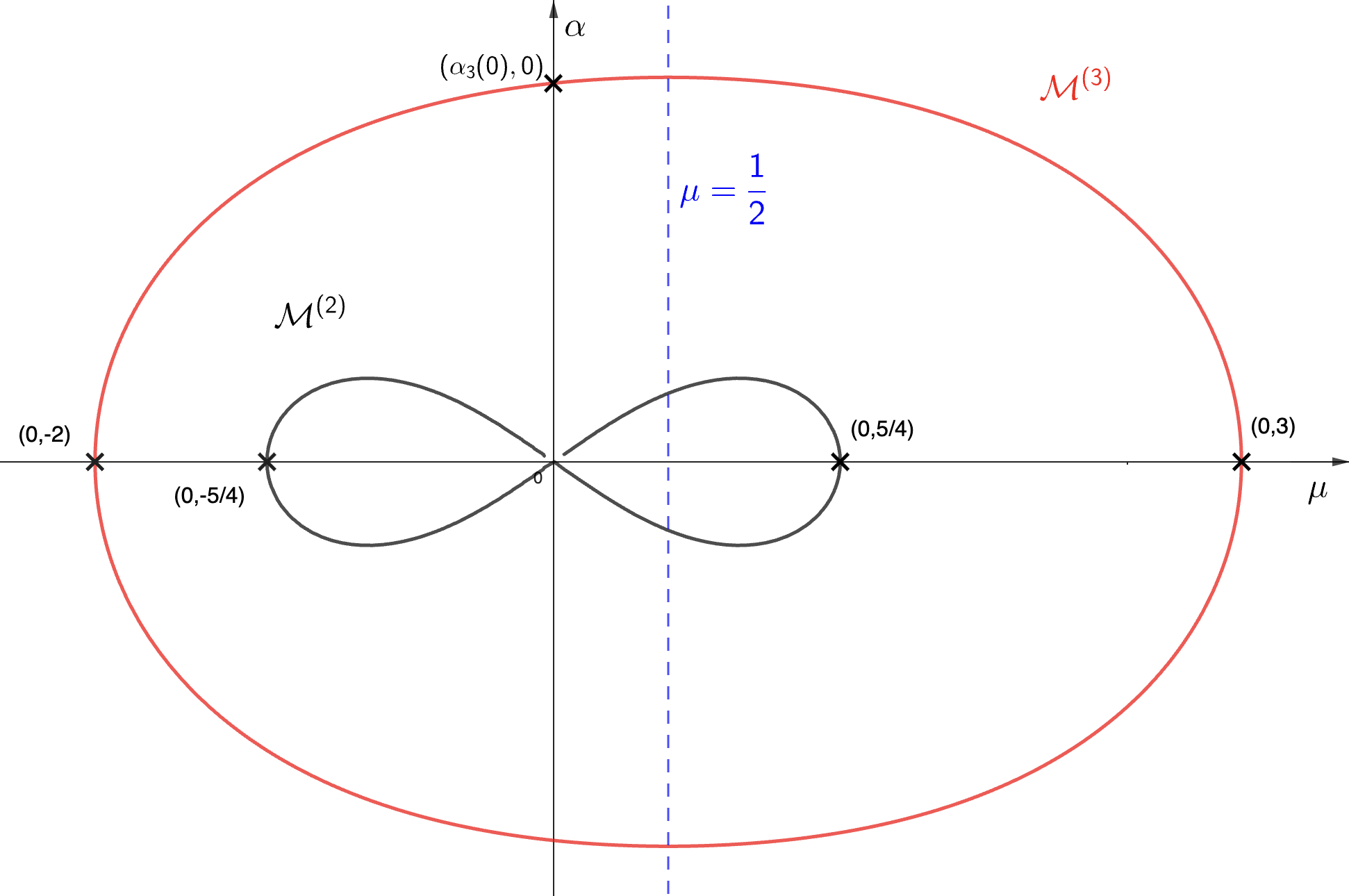}};

\begin{scope}[x={(img.south east)},y={(img.north west)}]

\node[green] at (0.44,0.58) {\cite{JTSY}};
\node[blue] at (0.44,0.85) { \cite{CNS}};
\node[violet] at (0.65,0.45) {\cite{BMV3}};
\node[orange] at (0.35,0.485) {\cite{NS,BMV1}};

\draw[green, thick]
(0.41,0.485) .. controls (0.44,0.51) .. (0.47,0.53);
\draw[blue, thick]
(0.411,0.85) -- (0.411,0.95);
\draw[violet, thick]
(0.58,0.485) -- (0.65,0.485);
\draw[orange, thick]
(0.38,0.485) -- (0.46,0.485);
\end{scope}

\end{tikzpicture}
    \caption{At the left, the unperturbed McLean  curves $ {\cal M}^{(\tp)}$, $ \tp \geq 2 $. On the right a synthesis of previous results in deep water: the longitudinal Benjamin–Feir  instability \cite{NS, BMV1} occurs 
for  $  \alpha = 0 $
near the origin 
(orange line), while the  first purely longitudinal unstable isola 
investigated in \cite{BMV4} appears near 
the intersection points 
$(0, \pm 5/4) \in \cM^{(2)}$
(purple line).
The purely transverse instability result  \cite{CNS} is located  
near  the intersection $\cM^{(3)} \cap \{ \mu = 0 \}$ (blue line). \Cref{TeoremoneIntro,TeoremoneFinale} extend  
the results of  \cite{JTSY}
(given in finite depth) for $(\al, \mu)$ near $ (0,0)$ (green curve),   as commented in \Cref{fig:spectrum}.
    }
    \label{fig:Mc16}
\end{figure}
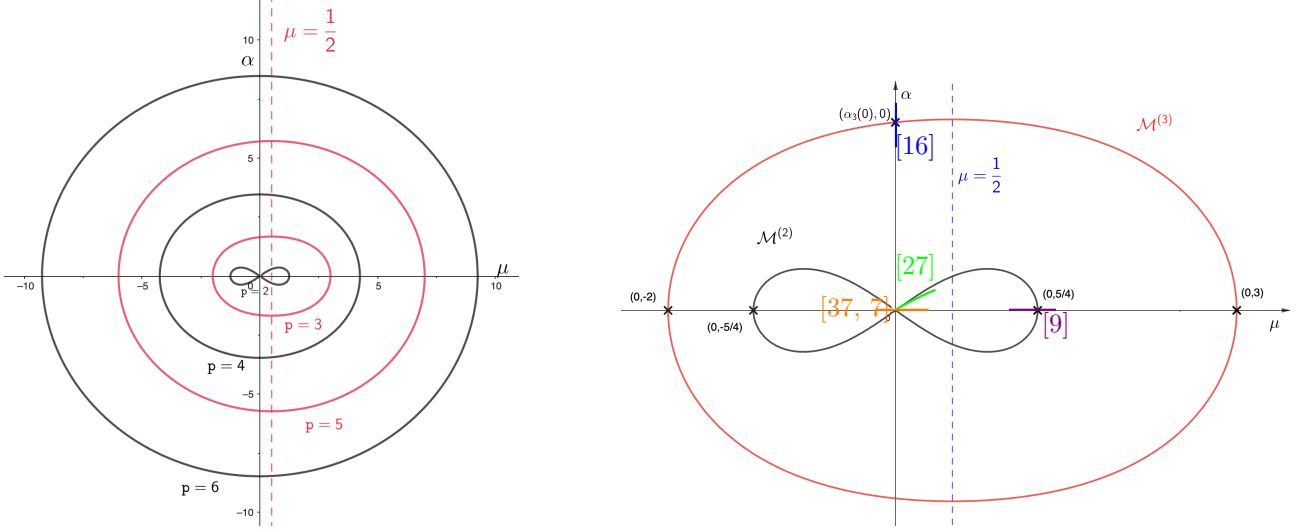

\begin{itemize}
\item {\sc Questions: } 
{\it What happens for 
$\e \neq 0$ small? 
Do multiple eigenvalues 
of 
$ \cL (\alpha, \mu, 0) $  leave the imaginary axis
under perturbation? 
For which Fourier-Bloch  parameters $(\alpha, \mu) $ ? How large is their real part ?  } 
\end{itemize}


Numerical results
\cite{Mc0,Mc1,Mc2,DO2} suggest 
that unstable eigenvalues appear all around the McLean curve 
$ \cM^{(2)} $.
Achieving rigorous mathematical results, however, presents several major difficulties:

\begin{itemize}
\item{i)} 
{\sc  Multiple eigenvalues}:
For 
$(\alpha,\mu) \sim (0,0) $ 
there are $ 4 $ interacting eigenvalues near zero, not only $ 2 $, and 
the eigenvalue zero
of  $ \cL (0,0,0)$ 
is defective. 
\item{ii)} 
{\sc  Lipschitz  Regularity}:
the operator $\cL(\alpha,\mu,\e)$
is expected 
to be merely Lipschitz continuous
with respect to  
$(\alpha,\mu)$, as the unperturbed operator   
$\cL(\alpha,\mu,0) $ in \eqref{cLvero0} contains the Lipschitz symbol  $ \sqrt{\alpha^2+\mu^2} $.  
This contrasts sharply with previous works where the dependence on parameters was analytic. 
\item{iii)} 
{\sc  Global  Parametrization:}
the parameters $(\alpha,\mu)$ vary globally in a neighborhood of the entire unperturbed McLean curves
$ \cM^{(\tp)}$, not only locally near specific resonant points, and thus we have to determine the possible emergence of unstable spectrum all around these curves.   
\end{itemize}

Rigorous  results emerged only very recently, cf. \Cref{fig:Mc16}.  The works \cite{CNS,CNS1} address  purely transverse disturbances  
locally near the 2 points
of 
intersection of
the McLean curve $ \cM^{(3)} $ with the axis $ \{ \mu = 0 \} $,  while  
\cite{JTSY} 
proves (in finite depth) the branching of a pair of unstable eigenvalues near zero
 for small $(\alpha,\mu) \in \cM^{(2)}$.
However, the latter result requires the Stokes wave amplitude  $ \epsilon (\alpha, \mu) $ to vanish  
as $ (\alpha, \mu) \to (0,0) $.
In these papers the difficulties 
i)-iii) do not arise because only $ 2 $
eigenvalues interact,  
and $ (\al,\mu)$ vary locally around resonant points rather than curves,  without approaching zero (ensuring that the symbol
$ \sqrt{\al^2+\mu^2}$ remains analytic). 
We also mention 
the transverse instability result 
\cite{HTW} in presence of surface tension.

\smallskip

The following result details the most dominant instabilities occuring near the McLean curve $\cM^{(2)}$,  
  as well as  the subsequent prominent instabilities emerging near 
  $\cM^{(3)}$. 
  
\begin{theorem}\label{TeoremoneIntro}
   For any $ 0 < |\e| \leq \e_0 $ small enough, for any $ \tp = 2, 3 $,  
   the unstable regions 
\begin{equation}\label{instaregionintrointro}
        \begin{aligned}
  \cU^{(\tp)}_\e :=  &       \big\{ (\al,\mu)\text{ near }\cM^{(\tp)}\, 
         :  \   \cL(\al,\mu,\e) \text{ possesses unstable eigenvalues }  \lambda_\pm^{(\tp)}(\al,\mu,\e)  \\  
        &   \text{satisfying} \quad   
\Re \, 
\lambda_\pm^{(\tp)}(\al,\mu,\e) \neq 0 
\quad  {\rm and} \quad        
  \lambda_+^{(\tp)}(\al,\mu,\e) = 
         - \overline{\lambda_-^{(\tp)} (\al,\mu,\e)}  \big\}  \neq \emptyset 
         \end{aligned}
    \end{equation}   
 are {\rm not} empty, shrink to $\cM^{(\tp)}$ as $\e\to 0$,
and its boundary 
$$ 
\pa \cU^{(\tp)}_\e = 
\cM^{(\tp)}_+ (\e)
\cup \cM^{(\tp)}_- (\e)
$$
is formed 
by two   closed analytic varieties $\cM^{(\tp)}_\pm (\e)$, referred to as  perturbed McLean curves:
\\[1mm]
    $\bullet$ The curve $\cM_+^{(2)}(\e)$ is a connected    variety, with a cross-type singularity at the origin. It encloses two bounded regions 
    containing  two simple closed curves that comprise $\cM^{(2)}_-(\e)$,  as shown in  \Cref{fig:pertmctotal} (left). Furthermore    
\begin{equation}\label{intersectionmcleans}
    \cM^{(2)}_+(\e)\cap \cM^{(2)}_-(\e) 
    = \emptyset  
   \, . 
\end{equation}    
    $\bullet$ The curves $\cM^{(3)}_\pm(\e)$ are closed and simple, and 
    their intersection $\cM^{(3)}_+(\e) \cap \cM^{(3)}_-(\e) $ consists of at most finitely many points,
as shown in \Cref{fig:pertmctotal} (right). 
    \\
    $\bullet$  
  There are  closed analytic  curves 
  $\cT^{(\tp)}(\e)$
  near $\cM^{(\tp)} $       
   along which the real part of the unstable eigenvalues satisfy 
   $ \Re\lambda_+^{(\tp)}(\al,\mu,\e) \simeq \e^\tp  
   $.
\end{theorem}

\begin{figure}[h!]
    \centering
    \includegraphics[width=0.54\linewidth]{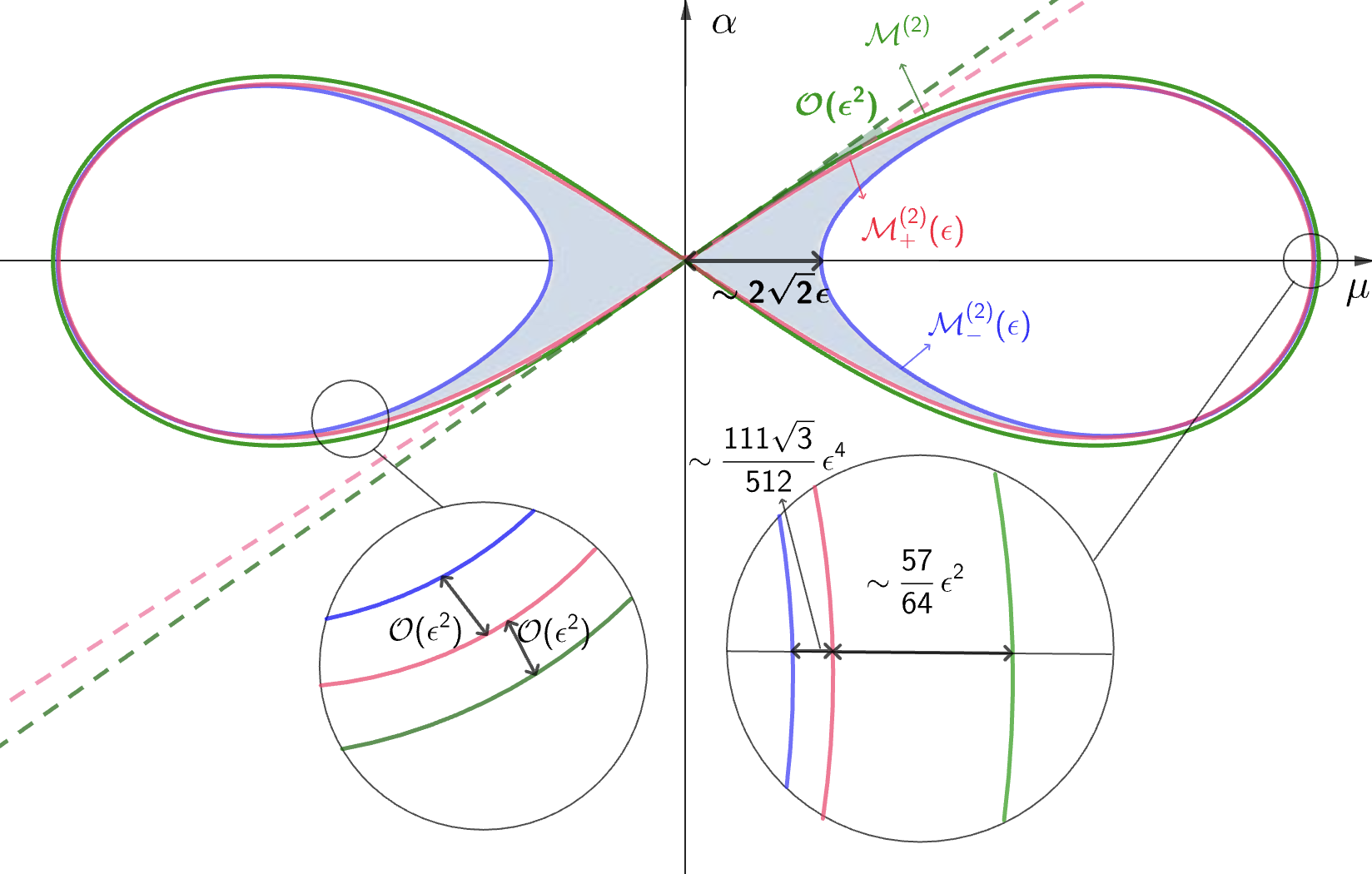}
    \hfill \includegraphics[width=0.42\linewidth]{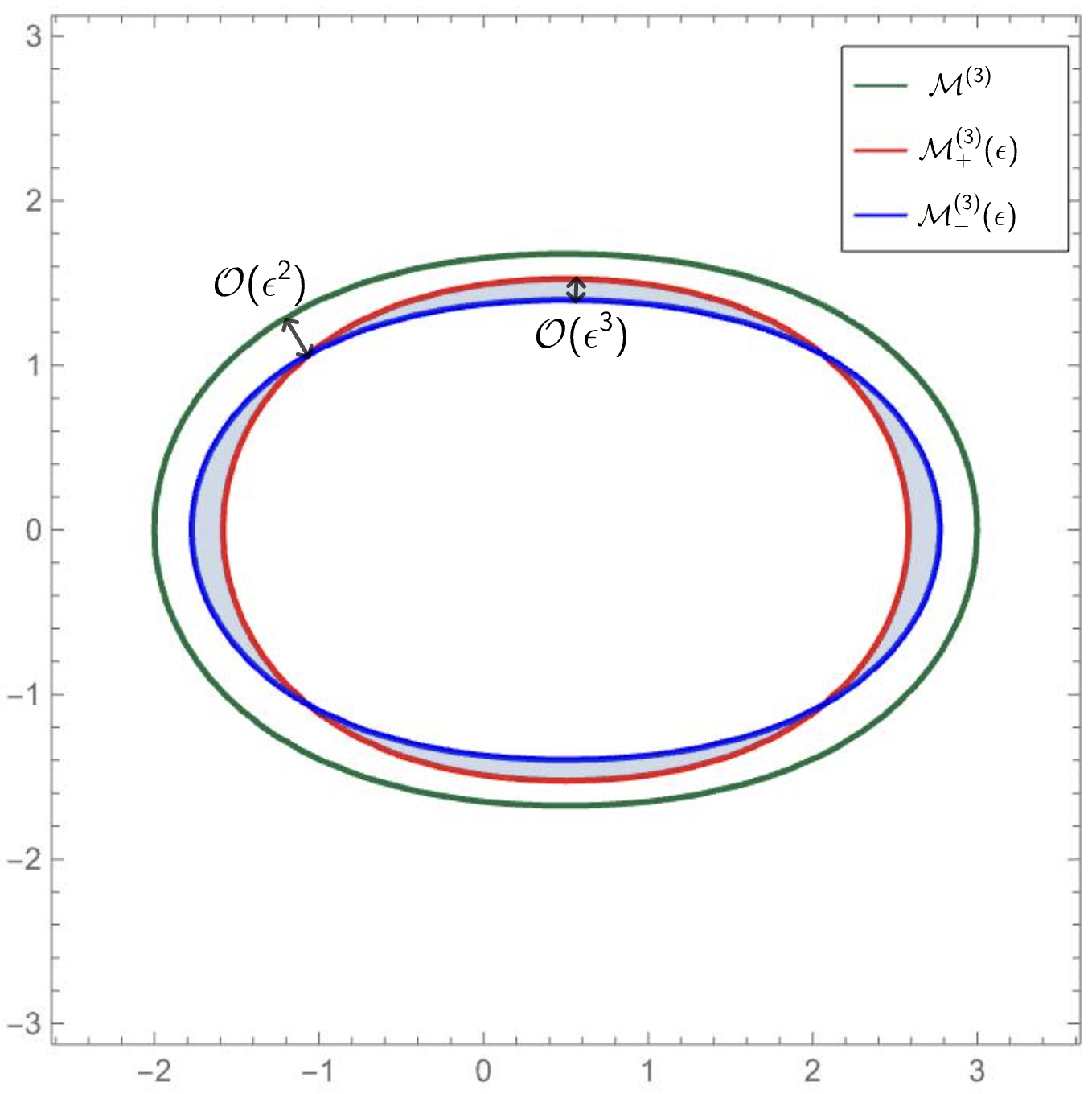}
    \caption{On the left, the  unperturbed McLean curve
    $ \cM^{(2)} $ (in  green)  bifurcates for $ \epsilon \neq 0 $  into the perturbed McLean curves $\cM_+^{(2)}(\e)$ (in red) and $\cM_-^{(2)}(\e)$  (in blue) which delimit the shaded instability region $\cU^{(2)}_\epsilon $. By \eqref{intersectionmcleans}  the  curves $\cM^{(2)}_\pm(\e)$ do not   intersect. 
    Near zero $\cM_-^{(2)}(\e)$ is approximated by the hyperbola $\mu^2 - 2\al^2 = 8\e^2$.
    On the right, a cartoon of  $\cM^{(3)}_+ (\e)$
    (in red) and $\cM^{(3)}_-(\e)$ (in blue),   that  could intersect  finitely many times.
    }
    \label{fig:pertmctotal}
\end{figure}

Complete  results
are given  in \Cref{TeoremoneFinale,TeoremoneMcLean}. 
Let us make some comments.
\\[1mm]
\noindent{\sc 1. Unification of longitudinal Benjamin-Feir and first high-frequency instability.}
\Cref{TeoremoneIntro,TeoremoneFinale,TeoremoneMcLean}  establish a unified origin for 
two instability phenomena previously viewed as distinct: the longitudinal Benjamin–Feir spectrum and the first unstable isola.
Although they appear to emerge from unrelated unperturbed eigenvalue pairs in a purely longitudinal setting, introducing the transverse parameter $ \alpha $   reveals their common genesis:  in the   extended  $(\alpha,\mu)$ space, 
both instabilities stem from the same pair of colliding unperturbed eigenvalues evaluated at different points of  $\cM^{(2)} \cap \{\alpha=0\}$.
Consequently, the two longitudinal unstable bands are merely   one-dimensional traces of the single, connected instability region 
$\cU^{(2)}_\e$.
\\[1mm]
\noindent{\sc 2. The global second spectral band.}
 \Cref{fig:spectrum} 
illustrates 
the curves traced by the eigenvalues 
$\lambda^{(2)}_\pm(\alpha,\mu,\e)$ in the complex plane as
$\mu$ varies,  for different  fixed values of 
 the transverse Fourier parameter $\alpha $. 
It shows how the  Benjamin–Feir figure 8 is deformed into the first longitudinal high-frequency isola as  
$ \alpha $ varies,  appearing  
as different slices of a single connected instability surface in the full 
Fourier-Bloch  parameter space.
\\[1mm]
{\sc 3. Splitting criterion.}
In view of 
\Cref{TeoremoneMcLean} 
the emergence of  unstable eigenvalues  is due to the {\em splitting} of the unperturbed McLean curve $\cM^{(\tp)}$  into separate curves $\cM^{(\tp)}_\pm(\e)$ as  $\e \neq 0$. 
The occurrence of this phenomenon is characterized by the non-vanishing of  an analytic function  
defined in a whole neighborhood of  $ \cM^{(\tp) }$. 
A sufficient condition is that one of its Taylor coefficient -- algorithmically computable --  is non-zero. 
This allows to prove \eqref{intersectionmcleans} 
providing a rigorous confirmation of McLean's numerical observations from the 1980s.

\begin{figure}[h!]
    \centering
    \begin{minipage}{0.25\linewidth}
        \centering
        \includegraphics[height=4cm, keepaspectratio]{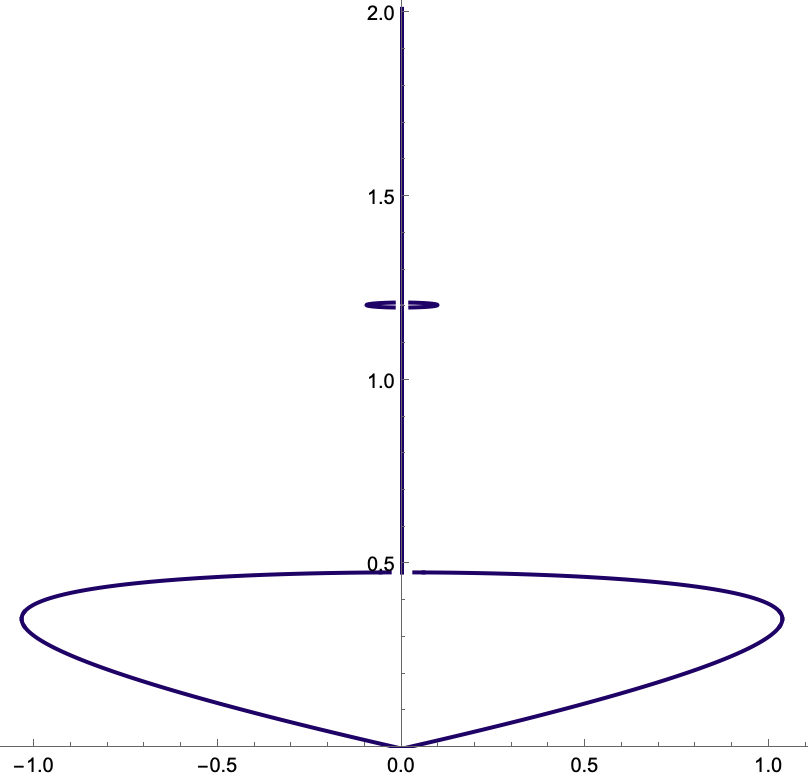}
        \captionof*{figure}{(A)}
    \end{minipage}
\hfill
    \begin{minipage}{0.25\linewidth}
        \centering
        \includegraphics[height=4cm, keepaspectratio]{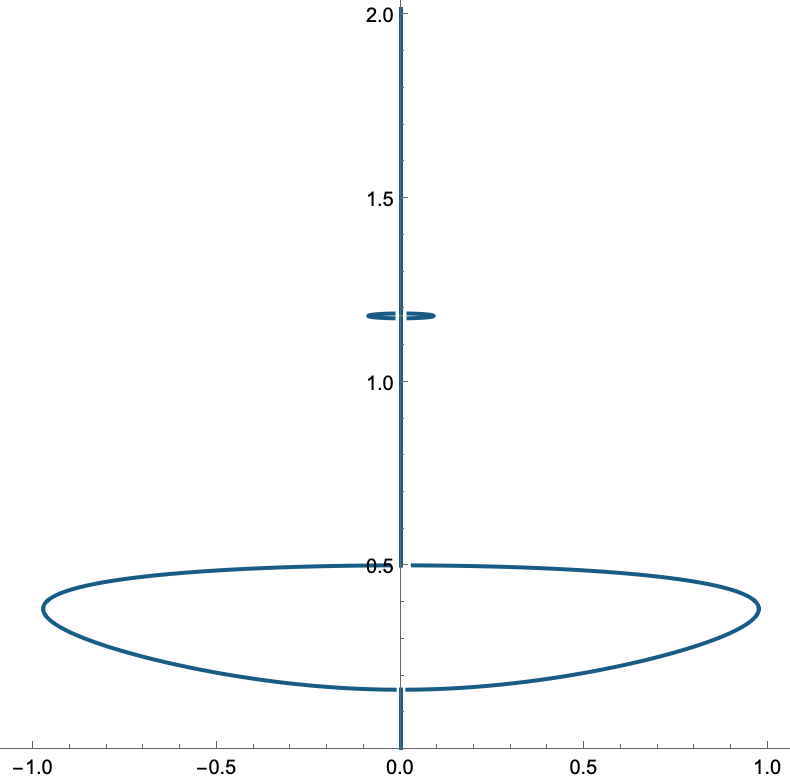}
        \captionof*{figure}{(B)}
    \end{minipage}
\hfill
    \begin{minipage}{0.11\linewidth}
        \centering
        \includegraphics[height=4cm, keepaspectratio]{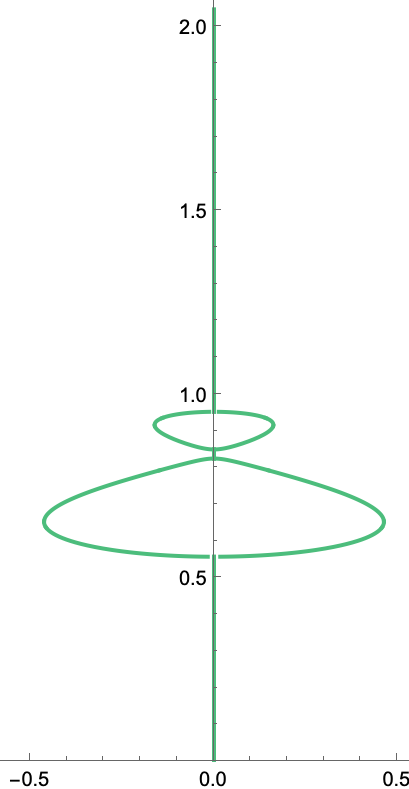}
        \captionof*{figure}{(C)}
    \end{minipage}
\hfill
    \begin{minipage}{0.11\linewidth}
        \centering
        \includegraphics[height=4cm, keepaspectratio]{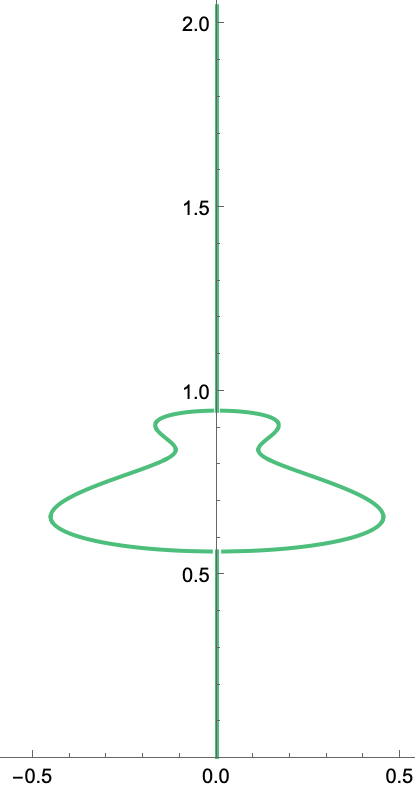}
        \captionof*{figure}{(D)}
    \end{minipage}
\hfill
    \begin{minipage}{0.11\linewidth}
        \centering
        \includegraphics[height=4cm, keepaspectratio]{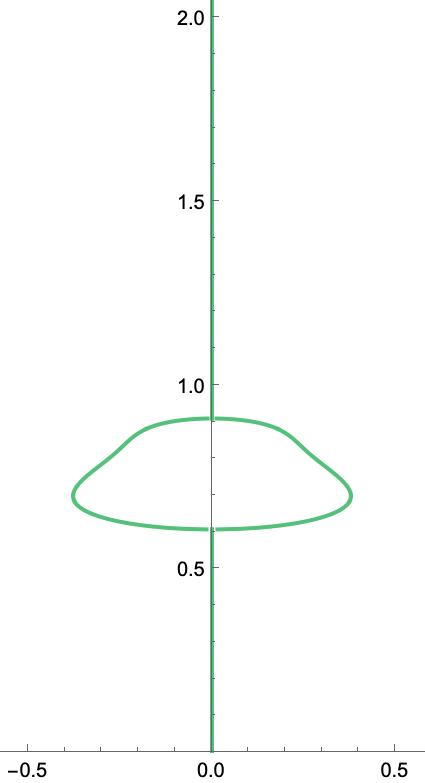}
        \captionof*{figure}{(E)}
    \end{minipage}
\hfill
    \begin{minipage}{0.11\linewidth}
        \centering
        \includegraphics[height=4cm, keepaspectratio]{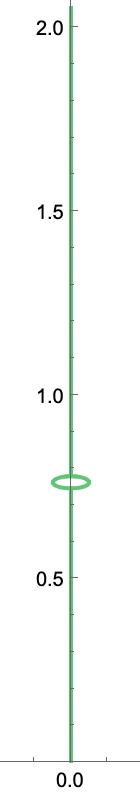}
        \captionof*{figure}{(F)}
    \end{minipage}
    \caption{
    In (A),  at $\alpha=0$, the eigenvalues trace the celebrated Benjamin–Feir ``figure-eight'' described in  \cite{BMV1} and the first high-frequency isola in \cite{BMV4}. 
As soon as $\alpha>0$, the ``figure-eight''
splits into two unstable isolas as shown in panel (B). These correspond  to the unstable spectral bands observed in finite depth in \cite{JTSY}, whose existence was established  only for amplitudes $\epsilon (\alpha, \mu) \to 0 $ 
as 
$(\alpha, \mu) \to 0 $.
As $\alpha$ increases to the value where the horizontal line in the $(\al, \mu)$ plane is tangent to $\cM^{(2)}_-(\e)$, see   \Cref{fig:pertmctotal},
 these components undergo a reconnection as shown in panel (C), 
 and subsequently merge into two  unstable  isolas as illustrated in panel (D).
 For even larger values of $\alpha$, these isolas shrink in size as shown in panel (E), until they 
 collapse onto the imaginary axis 
 when $\alpha$ reaches the 
maximal height of   $\cM^{(2)}_+(\epsilon)$. 
In  this  case the instability disappears, panel (F).
The topological transitions depicted in these graphs are rigorously proved by \eqref{lambda2.ae}.   
   }
    \label{fig:spectrum}
\end{figure}

\smallskip 

To conclude the introduction we describe the structure of the paper and some ideas of proof. 
In \Cref{part:I} we provide a comprehensive description of the  four eigenvalues of
$ \cL(\al, \mu, \epsilon)$
near zero, for  $(\alpha,\mu)$ small,  focusing on  the two unstable ones. The results of \Cref{part:I} are presented in \Cref{sec:22}. 
\\[1mm]
$ \bullet $
{\sc Fiber-Dirichlet-Neumann operator}.
A primary challenge is to establish 
the analytic properties
of the 
perturbed 
fiber Dirichlet-Neumann operator $ \cG (\alpha, \mu, \epsilon) $
obtained by conjugating the Dirichlet-Neumann operator via Bloch-Floquet transform as in \eqref{calLmua} (see \eqref{ellprobtransf}-\eqref{def:Galmu} for its precise definition).
\Cref{DNProp1} proves that  
$$
   \cG({\alpha, \mu, \e})    = |D|_{\al,\mu}  
   + \cG^\sharp({\alpha,\mu, \e}) \, , 
$$
where $ |D|_{\al,\mu} $
is the Fourier multiplier 
in \eqref{cLvero0} and 
$ \cG^\sharp({\alpha,\mu, \e})  $
is  a $ 1 $-smoothing operator of order $ \cO(\al^2 \epsilon )$.
Note that, in the purely longitudinal case, it simplifies to the Fourier multiplier  $ \cG (0,\mu, \epsilon ) = |D + \mu| $. 

Instead, in the $3d $ general case the operator 
$ {\cal L}(\alpha, \mu, \epsilon)$ 
exhibits  
Lipschitz singularities $(\al^2+\mu^2)^\frac12 $
as the unperturbed operator 
$ {\cal L}(\alpha, \mu, 0)$
in \eqref{cLvero0}.  
A key result 
is to recognize the decomposition 
\begin{equation}\label{cGame}
\cG(\al,\mu,\e) = \cG^{[\mathrm{I}]}(\al^2,\mu,\e) + (\al^2+\mu^2)^\frac12 \cG^{[\mathrm{II}]}(\al^2,\mu,\e)
\end{equation}
where the operators $\cG^{[\mathrm{I}]},\cG^{[\mathrm{II}]}$ depend analytically on their entries.

The  structure \eqref{cGame} is closed under composition, functional calculus and Cauchy integrals, and thus it is 
inherited by all the operators, matrices,  vectors, etc... throughout the Kato's reduction scheme.
This requires the analysis in appendices \ref{ap_analiticity}, \ref{ap_cAF}, \ref{ApC}.
The key step 
involves solving the  elliptic boundary value problem \eqref{ellprobtransf}, which depends on  $ (\alpha,\mu)$.  A major difficulty is that, for  $ (\alpha, \mu ) $ near $ (0,0)$, the solution of \eqref{ellprobtransf} lacks
fast decay 
on the zero Fourier mode,  
 as 
$ z \to  - \infty $ (cf.  \eqref{propg}). We overcome this obstacle,  by implementing a contraction mapping argument within a function space that permits a slightly divergent zero Fourier mode, cf. \Cref{rem:decay}.

\smallskip 

Then we perform a Kato reduction along a 
 $ 4  $ dimensional generalized space for any $ (\al,\mu,\epsilon)$ small enough. 
\\[1mm]
$ \bullet $ {\sc Block-decoupling.} A next key step is to perform a Taylor expansion 
of the $ 4 \times 4 $ 
reduced Kato matrix in the long wave approximation regime $(\alpha,\mu) \to (0,0) $, for any $ \epsilon $ fixed. In order to analyze the interaction of the four small eigenvalues, 
 we perform, following
\cite{BMV1},   a 
KAM-inspired block-diagonalization,  in the limit $(\alpha,\mu) \to (0,0) $, of  the reduced Kato matrix along its
$ 2 \times 2 $ stable and unstable subspaces. In the $ 3d $ case 
this procedure is very singular and 
breaks the structure \eqref{cGame}.
However we recognize that it  preserves 
the larger class of polar-analytic functions, cf. \Cref{def:tildeM}, namely,  which are analytic with respect to 
$$ 
\rho = \sqrt{\alpha^2+\mu^2} \, , 
\quad \theta = \arctan ( \alpha / \mu ) \, . 
$$
(note that  $ \cG({\alpha, \mu, \e})  $ in \eqref{cGame} is  polar-analytic near $(0,0)$).  
This is compatible with the cross at $(0,0) $ of $ \cM^{(2)}_+ (\e) $ in the $ (\alpha,\mu)$ plane, cf.  \Cref{fig:pertmctotal}.

\smallskip

This allows us to resolve difficulties $i$) and $ii$) mentioned above.

\smallskip

In \Cref{part:II}  we focus on $(\al, \mu)$   
near the unperturbed McLean curves $ \cM^{(\tp)}$,  away from the origin.  
\\[1mm]
$ \bullet $ {\sc Global Kato reduction.}
After establishing that the pair of 
colliding eigenvalues on 
the McLean curves  $ \cM^{(\tp)} $ remains uniformly separated from the rest of the spectrum in a neighborhood of
$ \cM^{(\tp)} $ (\Cref{lem:description,separation_eigenvalues}), we perform a Kato reduction. This reduction is  
perturbative in $ \e $, not in
$(\alpha, \mu) $,  and holds in a full neighborhood of the McLean curves. 
A key difficulty is that the spectral projectors associated with the resonant eigenvalues,
$$
            P^{(\tp)}_{\al,\mu,\e} := \frac{1}{2\pi\im}\oint_{\Gamma^{(\tp)}
(\al,\mu)} (\lambda - 
\cL(\al,\mu,\e))^{-1} \de\lambda 
$$
are, \emph{a priori}, defined only locally, as Cauchy integrals over contours $ \Gamma^{(\tp)}
(\al,\mu) $ enclosing the resonant eigenvalues, which vary 
along the  McLean curve.
However, by  analytic continuation,  the  spectral projector $ P^{(\tp)}_{\al,\mu,\e} $ extends analytically all around $ \cM^{(\tp)} $.
For $ \tp = 2 $, as $(\al,\mu)
\to (0,0) $, 
the two-dimensional Kato projector becomes ill-defined, due to interaction with the other two Benjamin-Feir  eigenvalues, see \Cref{rem:4aut}. However, by exploiting the block-decoupling procedure in \Cref{part:I} and the uniqueness of skew-Hamiltonian projectors, we  manage to analytically extend the projector  to a full neighborhood of the origin, for uniformly small $|\e|$.

\smallskip

In this way, we resolve difficulty $iii$) mentioned above. 
This global construction has  important consequences. 
\\[1mm]
$ \bullet $ 
{\sc Global splitting criterion. }
The entries of the 
 reduced $2\times2$ Kato matrix, which  represents the action of $\cL(\alpha,\mu,\e)$ on the range of 
$ P^{(\tp)}_{\al,\mu,\e} $, 
are analytic functions globally defined in a neighborhood of $\cM^{(\tp)}$, in contrast to   previous works 
where they were constructed only near specific  points.
Consequently, the splitting of the perturbed McLean curves is characterized by the non-vanishing of an analytic function $\fb^{(\tp)}(\alpha,\mu, \e)$,   defined on this entire neighborhood.
A sufficient condition for  the splitting is the non-vanishing of its first Taylor coefficient $ \fb_\tp(\al,\mu) $  
(cf. \Cref{lem:inst}). 
Crucially, by analyticity, verifying the global non-vanishing   
of $ \fb_\tp(\al,\mu) $ along $ \cM^{(\tp)}$ reduces to a local computation at a {\it single} point of the curve, implying splitting everywhere except possibly at finitely many points. 
In this paper we implement this novel conceptual approach  
for $ \tp = 2, 3 $.  
Using the expansions of Appendix \ref{sec:5} we verify that this  condition holds for $ \tp = 2 $  
(\Cref{lem:coeff.b2}) along the entire McLean curve 
$ \cM^{(2)} $,  while for $ \tp = 3 $ we rely on \cite{CNS}.

The complete results  of \Cref{part:II} are presented in \Cref{sec:23}. 


\section{Main results}\label{sec:intro}

The water waves equations describe the 
evolution of a 
$ 3 $-dimensional  
incompressible, inviscid and irrotational fluid
under the action of gravity, occupying the time dependent region 
with infinite depth 
$$
    \cD_\eta := \{(x,y,z)\in \R^3 \ :  \ 
z
\leq \eta(t,x,y)\}  \, .
$$
The irrotational velocity field is the gradient  
of a harmonic scalar potential $\Phi := \Phi(t,x,y,z) $  
determined by its trace $ \psi(t,x,y) =\Phi(t,x,y,\eta(t,x,y)) $ at the free surface
$ z = \eta (t, x,y ) $.
Actually $\Phi$ solves 
(see e.g. \cite[corollaries 2.46,2.49]{Lan}) 
the elliptic equation  
\begin{equation}\label{Phi.ell}
    \Delta_{x,y,z} \Phi = 0 \  \text{ in }\, {\mathcal D}_\eta, \quad
     \Phi(t, x,y,\eta(t,x,y)) = \psi(t,x,y) \, ,  \quad
\pa_z\Phi(t,x,y,z) \to 0 \ \text{ as }
   z \to - \infty \, . 
\end{equation}
\noindent 
{\bf Water waves equations.} 
The time evolution of the fluid is determined by two boundary conditions at the free surface. 
The first is that the fluid particles  remain, along the evolution, on the free surface   (kinematic
boundary condition), and the second one is that the pressure of the fluid  
is equal, at the free surface, to the constant atmospheric pressure  (dynamic boundary condition). 

As shown by Zakharov \cite{Zak1} and Craig-Sulem \cite{CS}, 
the time evolution of the fluid is determined by the 
following equations for the unknowns $ (\eta (t,x,y), \psi (t,x,y)) $,  
\begin{equation}\label{WWeq}
 \eta_t  = G(\eta)\psi \, , \quad 
  \psi_t  =  
- g \eta - \dfrac{|\grad \psi |^2}{2} + \dfrac{\big( G(\eta) \psi + \grad \eta \cdot \grad  \psi \big)^2}{2(1+|\grad \eta|^2)}  \, , 
\end{equation}
where  $G(\eta)$ is 
the Dirichlet-Neumann operator 
\begin{equation}
\label{DN1}
[G(\eta) \psi](x,y)  := \sqrt{1+|\nabla \eta|^2} \,  \pa_{\vec{n}} \Phi_{|z = \eta (x,y) }
= (\pa_z \Phi) (x,y, \eta(x,y)) - \nabla \eta(x,y) \cdot (\nabla \Phi)(x, y,  \eta(x,y))
\end{equation}
and  $\Phi(x,y,z)$ is the unique solution to \eqref{Phi.ell}. We denote by
$ \grad:= (\pa_x, \pa_y) $ 
and by
 $ \vec n $ 
 the exterior normal
 to $ \cD_\eta $.
\\[1mm]
{\bf Hamiltonian 
structure.} 
Equations \eqref{WWeq} admit the Hamiltonian formulation 
\begin{equation}\label{PoissonTensor}
 \pa_t \vet{\eta}{\psi} = \cJ \vet{\nabla_\eta \mathcal{H}}{\nabla_\psi \mathcal{H}}, \quad \quad \cJ:=\begin{bmatrix} 0 & \uno \\ -\uno & 0 \end{bmatrix},
 \end{equation}
where $ \nabla_{\eta}, \nabla_{\psi} $ 
 denote the variational gradients of the Hamiltonian
$ \mathcal{H}(\eta,\psi) := 
 \frac12 \int_{\R^2} \left( \psi \,G(\eta)\psi +g \eta^2 \right) \de x \de y $, 
which  is the sum of the kinetic  energy 
 and potential gravitational energy  of the fluid.  A simple scaling allows us to set the gravity field $ g = 1 $. 
The associated symplectic $ 2$-form is 
$$
{\cal W}  
\left( \begin{pmatrix}
\eta_1 \\
\psi_1
\end{pmatrix}, 
\begin{pmatrix}
\eta_2 \\
\psi_2
\end{pmatrix}
\right) 
= (  -  \psi_1 , \eta_2 )_{L^2} + (\eta_1  , \psi_2 )_{L^2}  \, . 
$$
{\bf Reversible  and translation invariance symmetries.} 
In addition to being Hamiltonian, the water waves 
system \eqref{WWeq} is reversible with respect to the involution 
$$
\rho\vet{\eta(x,y)}{\psi(x,y)} := \vet{\eta(-x,-y)}{-\psi(-x,-y)}, \quad \text{i.e. }
\mathcal{H} \circ \rho = \mathcal{H} \, ,
$$
or equivalently the water waves vector field $ X(\eta, \psi) $ 
anticommutes with $ \rho $,  i.e. $ X \circ \rho = - \rho \circ X $.
This property follows noting that the Dirichlet-Neumann operator satisfies 
(see e.g. \cite{BMV2}) 
$ G(  \eta^\vee ) [ \psi^\vee ] = \left( G(\eta) [\psi ] \right)^\vee  $ 
where 
$ f^\vee (x,y) := f (-  x,-y)  $. 
Furthermore the Dirichlet-to-Neumann operator is invariant under space translations
$$
 \tau_{ \theta} G(\eta)\psi = G( \tau_{\theta} \eta)[ \tau_{ \theta} \psi] \, , \quad 
 \tau_{ \theta} u (x,y) := u 
 ((x,y) +  \theta ) \, , \ \  \forall  \theta \in \R^2 \, ,
$$ 
and therefore system  \eqref{WWeq} is invariant 
under space translations. 
\\[1mm]
{\bf Stokes waves.} 
Noteworthy solutions of \eqref{WWeq} are  Stokes traveling  waves 
solutions of the form   
$$
\eta(t,x,y)=\breve \eta(x-ct) \, ,  \quad \psi(t,x,y)=\breve \psi(x-ct) \, , 
$$
for some real  $c$   and  $2\pi$-periodic functions  $(\breve \eta (x), \breve \psi (x)) $ depending on the space variable $ x $ only.
In a reference frame in translational motion with constant speed $\vec c 
=(c, 0) $,  the water waves equations \eqref{WWeq}  become
\begin{equation}\label{travelingWW}
\eta_t  = \vec{c} \cdot \grad \eta+G(\eta)\psi \, , \quad 
 \psi_t  = \vec{c} \cdot \grad \psi - \eta - \dfrac{|\grad \psi|^2}{2} + \dfrac{1}{2(1+|\grad \eta|^2)} \big( G(\eta) \psi + \grad \eta \cdot \grad \psi \big)^2   
\end{equation}
and Stokes waves  are  equilibrium 
steady solutions 
of \eqref{travelingWW}. 

Rigorous bifurcation of Stokes waves goes back to 
\cite{LC,Nek,Struik}.
 The following statement 
 is taken from  \cite{BMV2}. 
 
 \begin{theorem}\label{PTW}
{\bf (Stokes waves)} There exist $\e_0 >0$ and a unique family  of real analytic 
 solutions $(\eta_\e(x), \psi_\e(x), c_\e)$, parameterized by the amplitude $|\e| \leq \e_0$, of 
$$
c \, \eta_x+G(\eta)\psi = 0 \, , \quad 
c \, \psi_x -  \eta - \dfrac{\psi_x^2}{2} + 
\dfrac{1}{2(1+\eta_x^2)} \big( G(\eta) \psi + \eta_x \psi_x \big)^2  = 0 \, , 
$$
  such that
 $ \eta_\e (x), \psi_\e (x) $ are $2\pi$-periodic;  $\eta_\e (x) $ is even
and $\psi_\e (x) $ is odd. They have  the expansion 
$$
\eta_\e (x) =  \e \cos (x) + \tfrac{\e^2}{2} \cos (2x) + 
  \cO(\e^3)\, , \ 
  \psi_\e (x) = \e \sin (x) + \tfrac{\e^2}{2} \sin (2x)   
  +\cO(\e^3) \, ,  \ c_\e = 1 + \tfrac{\e^2}{2} +\cO(\e^3) \,.
$$
More precisely for any  $ \sigma \geq  0 $ and $ s > \frac52 $, there exists $ \e_0>0 $ such that
the map $\e \mapsto (\eta_\e, \psi_\e, c_\e)$ is analytic from $B_{\e_0}(0) \to H^{\sigma,s} (\T)\times H^{\sigma,s}(\T)\times \R$, where 
$ H^{\sigma,s}(\T) $ is the space of $ 2 \pi $-periodic analytic functions 
$$ 
u(x) = \sum_{k \in \mathbb{Z}} u_k e^{\im k x} \qquad \text{with} \qquad
\| u \|_{\sigma,s}^2 := \sum_{k \in \mathbb{Z}} |u_k|^2 \langle k \rangle^{2s} 
e^{2 \sigma |k|} < + \infty \, . 
$$
\end{theorem}
{\small {\noindent\bf Notation:} For any $ s \in \R $ we denote 
$H^s(\T) $ the Sobolev space $ H^{0,s}(\T)$ with norm $\norm{\cdot}_s:= \norm{\cdot}_{0,s} = \| \ \|_{H^s(\T)}$. As above we denote  
closed balls in $\R$ and $\R^2$ as
$    B_r(x_0):= \{ x \in \R\, : \ |x-x_0|\leq r\} $ and $  B_r(x_0,y_0):= \{ (x,y)\in\R^2 \, : \ |(x,y) - (x_0,y_0)| \leq r\} $. 
We may use the same notation for balls respectively in $\C$ and $\C^2$. 
}
 \normalfont

\smallskip

\noindent{\bf Linearization at Stokes wave.}
 In order to determine the stability/instability of a Stokes wave, we linearize  \eqref{travelingWW}  with $ \vec c= (c_\e,0) $, at the stationary solution $(\eta_\e(x),\psi_\e(x))$. 
Introducing  the celebrated  “good unknown” of Alinhac 
see, for instance \cite{NS,BMV1,BBHM}, 
the linearized system  
reads
\begin{equation}\label{OpL2}
   \pa_t \begin{bmatrix}
        u \\
        v
    \end{bmatrix} = \wt \cL_\e
    \begin{bmatrix}
        u \\
        v
    \end{bmatrix} \ , \quad 
    \wt\cL_\e = 
    \begin{pmatrix}
        - \pa_x\circ(V_\e -c_\e) &\vline & G(\eta_\e) \\
           \cline{2-2}
        -1 -(V_\e -c_\e) (B_\e)_x &\vline &  -(V_\e -c_\e)\pa_x \end{pmatrix}
\end{equation}
where $ (V_\e, B_\e) $ are respectively the horizontal and vertical velocities at the free surface.
The operator $\wt \cL_\e$ has the Hamiltonian and reversible form
\begin{equation*}
    \wt\cL_\e = \cJ \wt \cB_\e\ , \qquad \wt\cB_\e = \begin{pmatrix}
        1 +(V_\e-c_\e)(B_\e)_x  &\vline & (V_\e-c_\e)\pa_x \\
           \cline{2-2}
        - \pa_x\circ(V_\e-c_\e) &\vline & G(\eta_\e) \\
        \end{pmatrix} \, , \qquad \rho\circ \wt\cB_\e =\wt\cB_\e \circ \rho\, , \qquad \cJ\text{ in }\eqref{PoissonTensor}.
\end{equation*}
{\noindent\bf Conformal flattening of the domain.} 
In order to study  the action of the Dirichlet-Neumann operator under longitudinal 
and transverse   perturbations 
it  is convenient to use a classical  conformal flattening 
diffeomorphism  (see  Lemma \ref{LeviCivita})  to map  the fluid domain
\be\label{flatdom}
\cD_{\eta_\e} =  \{(x,y,z)\in \R^3\, :\ z \leq  \eta_\e(x)\} \quad  
\text{to the half space}
\quad  
\{(X,Y,Z)\in \R^3\, :\ Z \leq  0 \}  \, .
\ee
Such   diffeomorphism 
transforms  the free surface 
$ \{ z = \eta_\e (x) \} $ to 
$ \{ Z =  0\} $
and, 
restricted 
to the horizontal variables $x,y$,  undergoes 
a transformation of the form 
\begin{equation}\label{cocatsurf}
    x = X + \fp(X) 
\quad 
\, , \qquad y = Y \, , 
\end{equation}
where 
 $\fp(X)$   is a  small $2\pi$-periodic odd  function. We define the associated composition operator
\begin{equation}\label{diffeo}
(\mathfrak{P} u)(X,Y) := u(X+\fp(X),Y) \, . 
\end{equation}
The Levi-Civita map transforms the elliptic problem \eqref{Phi.ell} into 
\begin{equation}\label{transf_laplace}
    \begin{cases}
        \Delta_{X,Z} \vartheta (X,Y,Z) + (1+d_\e(X,Z))  \pa_Y^2 \vartheta (X,Y,Z) = 0 \, ,\qquad Z< 0\, ,\\
        \vartheta  (X,Y,Z)|_{Z=0} = \varsigma (X,Y) \, , \quad 
        \lim\limits_{Z\to -\infty}\pa_Z \vartheta  (X,Y, Z) = 0 \, , 
    \end{cases}
\end{equation}
where $d_\e(X,Z)$ is  the analytic function defined in \eqref{def:d} and $ \varsigma := \mathfrak{P} \psi $ is the new Dirichlet datum. 
Denoting by $ \vartheta := \vartheta_\varsigma $ the unique solution to \eqref{transf_laplace} in an appropriate sobolev space --see again \cite[Corollary 2.49]{Lan}-- we define the modified
 Dirichlet-to-Neumann  operator
\begin{equation}\label{def:cG}
    \cG_{\e} [\varsigma ](X,Y) := \pa_Z\vartheta_\varsigma  (X,Y,Z)_{|Z=0} \, .
\end{equation} 
The operator $ \cG_{\e}$ is related to the Dirichlet-Neumann operator by 
\begin{equation}\label{G.id}
            G(\eta_\e) =  \mathfrak{P}^{-1}\circ \frac{1}{1+\fp'}  \circ \cG_{ \e}\circ\mathfrak{P} 
            =  \mathfrak{P}^{\top}  \circ \cG_{ \e}\circ\mathfrak{P} 
        \end{equation}
where $\mathfrak{P}^{\top}$ denotes the  real transposed operator, as we report in Appendix \ref{ApC}.

Hereafter we drop for convenience the capital letters $X,Y,Z$ and use $x,y,z$ also for the new set of coordinates. 
Conjugating the operator $\wt \cL_\e$ 
in \eqref{OpL2} 
via the real, 
symplectic and reversibility preserving change of variables 
   $$
\cP :=
\begin{pmatrix}
    (1+\fp')\mathfrak{P} & 0 \\
    0 & \mathfrak{P} 
    \end{pmatrix} 
$$
we get the  real Hamiltonian and reversible operator
\begin{equation}\label{cLvero}
    \cL_{\e} 
    := 
   \cP \, \wt \cL_\e \, \cP^{-1}   
   = \begin{pmatrix}
        \pa_x \circ (1+p_\e(x))  
        &  \cG_{\e} \\
        -(1+a_\e (x))
         & (1+p_\e (x))\pa_x 
    \end{pmatrix} = 
    \cJ \cB_{\e}
 \end{equation}
where  
\begin{equation}\label{calL.ham}
 \cB_{\e} 
     = \begin{pmatrix}
        1+a_\e (x) & -(1+p_\e (x))\pa_x \\
        \pa_x \circ (1+p_\e(x)) & \cG_{\e}
    \end{pmatrix}\, ,
\end{equation}
is symmetric on $L^2(\R^2;\C^2)$,
 and  
$ p_\e(x), a_\e(x) $
in \eqref{calL.ham} are the  $ 2 \pi $-periodic, real, even functions 
\begin{equation}\label{def:pa}
 p_\e(x) :=  \displaystyle{\frac{ c_\e-V_\e (x+\mathfrak{p}(x))}{ 1+\mathfrak{p}_x(x)}} - 1 \, , \quad a_\e(x):=   \displaystyle{\frac{1+ (V_\e (x + \mathfrak{p}(x)) - c_\e)
 (B_\e)_x(x + \mathfrak{p}(x))  }{1+\mathfrak{p}_x(x)}} - 1 \, .
\end{equation}
These functions are  real analytic in $ \e $ and admit a Taylor expansion, cf. \cite{NS,BMV1},  
\be\label{apexp}
\begin{aligned}
p_\e (x) & = 
\sum_{\ell \geq 1} \e^\ell p_\ell (x) = 
-2\e\cos(x) + \e^2 \Big( \frac32-2\cos(2x) \Big)   
+\cO(\e^3) \, ,  \\
a_\e (x) &  =  \sum_{\ell \geq 1} \e^\ell a_\ell (x) 
= -2\e\cos(x)+2\e^2\big(1-\cos(2x)\big)  
+\cO(\e^3) \, , 
\end{aligned}  
\ee
where  each $ a_\ell (x), p_\ell (x) $ has the form, cf. \cite[equation (5.22)]{BMV5}, 
\be\label{pafparity}
    f(x) = \sum_{k=0}^\ell f^{[k]} \cos (kx)\,  \quad \text{with } \ f^{[k]} = 0 \ \ \forall k \not\equiv \ell \ \textup{mod }2\, .
\ee

\paragraph{Transverse and longitudinal perturbations.}\label{FB.trans}
Since the operator $\cL_\e$ in 
\eqref{cLvero}  
is  $2\pi $-periodic in the $x$ variable, and independent of the $y$ variable,  the natural framework for analyzing its linear instability is provided by the {\it Fourier-Floquet-Bloch theory}. 
The operator 
$ \cL(\al,\mu,\e) $ 
in  \eqref{calLmua}
turns out to be 
    \begin{equation}
    \label{cLame1}
\cL(\al,\mu,\e) :=\begin{pmatrix}
        (\pa_x+\im\mu) \circ (1 +p_\e(x)) & \cG (\al,\mu, \e)  \\[2mm]
        -(1+a_\e (x)) & (1+p_\e (x))\circ (\pa_x + \im \mu)
        \end{pmatrix}  \, 
\end{equation}
where   $\cG (\al,\mu, \e) $  is 
the {\em fiber Dirichlet-to-Neumann operator}, obtained by choosing in  \eqref{transf_laplace} the Dirichlet  datum  
$$ 
\varsigma(x,y) = e^{\im(\al y + \mu x)} g (x) \, , \quad g (x) \in H^1(\T;\C) \, , \quad (\alpha, \mu) \in \R^2 \, , 
$$
 and looking for an ansatz solution of 
 the  elliptic problem 
\eqref{transf_laplace} of the form 
$$
\vartheta_\varsigma (x,y,z)= e^{\im(\al y + \mu x)} \Theta_{g}(x,z)\, . 
$$
Then $ \Theta (x,z) :=\Theta_{g} (x, z) $ solves the following elliptic problem 
on the half cylinder $(x,z)\in \T\times\R_- $,  
\begin{equation}
\label{ellprobtransf}
    \begin{cases}
    \pa_z^2 \Theta (x,z) + (\pa_x+\im\mu)^2 \Theta (x,z) - \al^2 (1+d_\e(x,z)) \Theta (x,z) = 0\\
    \Theta(x,z)_{|z=0} =  g(x) \, , \quad 
    \lim_{z \to-\infty} \pa_z \Theta(x,z) = 0  \,  ,
\end{cases}
\end{equation}
where $d_\e(x,z)$ is the function in \eqref{def:d}. In
\Cref{sub:ell.f} we prove the existence of a unique solution of 
\eqref{ellprobtransf}.
Then, the  fiber Dirichlet-to-Neumann operator  is 
\begin{equation}\label{def:Galmu}
\cG(\al,\mu,\e)[g] (x):= \pa_z \Theta_{g} (\al,\mu,\e;x,z)\vert_{z=0} \, . 
\end{equation}
An eigenvalue $ \lambda $ of $ \cL (\alpha, \mu, \epsilon )$ with eigenvector $ v(x) $ provides a Bloch wave as \eqref{h.intro}.

\begin{remark}
The spectrum of $\cL(\al,\mu,\e)$ 
satisfies, arguing as  in  
\cite[equation (2.21)]{JTSY},   
$$
\sigma_{L^2(\R^2)\times H^{1/2}(\R^2)}(\cL_\e)  \supseteq \bigcup_{(\al, \mu) \in \R \times [-\frac12, \frac12)} 
\sigma_{L^2(\T)\times H^{1/2}(\T)}(\cL(\al, \mu, \e)) \, .
$$
\end{remark}

We regard
$\cL(\al,\mu,\e)$ as an operator $H^1(\T,\C^2)\to L^2(\T,\C^2)$  
equipped with the 
complex scalar product
\begin{equation}\label{scalarcomplex}
(f,g) := \frac{1}{2\pi} \int_{\T} \left( f_1 \overline{g_1} + f_2 \overline{g_2} \right) \, \text{d} x  \, , 
\quad
\forall f= \vet{f_1}{f_2}, \ \  g= \vet{g_1}{g_2} \in  L^2(\T, \C^2) \, .
\end{equation}

\begin{proposition}\label{DNonR3}
  There exists $\e_0 >0$ such that for any $|\e| \leq \e_0$, any $(\alpha, \mu) \in \R^2$, the operators 
  \be\label{dominiGL}
  \cG(\alpha, \mu, \e) :
  H^1(\T) \subset L^2 (\T) \to L^2 (\T) \, , 
  \quad 
  \cL(\al, \mu,\e) : 
  H^1(\T;\C^2) \subset L^2 (\T^2, \C^2) \to L^2 (\T;\C^2) \, , 
 \ee
 are well defined 
and satisfy  the following properties:
  \begin{itemize}
    \item[(i)]{\sc Self-adjointness and Hamiltonianity:} 
    $\cG(\al,\mu,\e)$ is self-adjoint on $L^2(\T)$, and $\cL(\al, \mu, \e)$ is  complex Hamiltonian, i.e.
    \be\label{cLame0}
    \cL(\al,\mu,\e)  = 
\cJ \cB (\al,\mu,\e)  \ , \quad \cJ \mbox{  in } \eqref{PoissonTensor}  \ , 
\ee
  where   
    \begin{equation}\label{operator Bmualep}
\cB(\al,\mu,\e)  := \begin{pmatrix}
        1+a_\e (x) & -(1+p_\e (x))\circ (\pa_x + \im \mu)\\
        (\pa_x+\im\mu) \circ (1 +p_\e(x)) & \cG (\al,\mu, \e) 
        \end{pmatrix}\,  
\end{equation}
 is self-adjoint  
with respect to the 
  $L^2(\T;\C^2)$  
complex 
scalar product \eqref{scalarcomplex}. 
\item[(ii)] {\sc Reversibility:}  For any $g \in H^1(\T)$, 
\begin{equation}\label{DN-rev2}
\cG(\al,\mu,  \e)  [ \bar{g}^\vee ] = \bar{\left(\cG(\al,\mu,  \e) [g] \right)}^\vee  \quad
\text{where} \quad  g^\vee (x) := g (-  x)  \, .
\end{equation} 
Consequently
$ \cL(\al,\mu,\e)$ 
is reversible, respectively $\cB(\al,\mu,\e)$ is reversibility-preserving, with respect to  the complex involution
\begin{equation}
\label{def:cinvolution}
   \varrho_c \vet{u(x)}{v(x)} := \vet{\overline{u(-x)}}{-\overline{v(-x)}} \, ,
\end{equation} 
namely 
\begin{equation}\label{cLr:rev}
    \cL(\al,\mu,\e) \circ \varrho_c = -\varrho_c \circ \cL(\al,\mu,\e) \ , \qquad  \cB(\al,\mu,\e) \circ \varrho_c = \varrho_c \circ \cB(\al,\mu,\e)\, .
\end{equation}
\item[(iii)] {\sc Gauge covariance:} for any $k \in \Z$, 
\begin{equation}
    \label{Gmu+k}
\cG(\al,\mu+k,\e) = e^{-\im k x}\cG(\al,\mu,\e)e^{\im k x}\,  , 
\end{equation}
and consequently
\begin{equation}\label{cBLmu+k}
 \cL(\al,\mu+k,\e) = e^{-\im k x}\cL(\al,\mu,\e)e^{\im k x}\,  , 
 \qquad 
    \cB(\al,\mu+k,\e) = e^{-\im k x}\cB(\al,\mu,\e)e^{\im k x}\, .
\end{equation}
 \item[(iv)] {\sc unperturbed operators:}  
 $ \cG(\alpha, \mu, 0) $ is 
 the Fourier multiplier
 $ \cG(\alpha, \mu, 0) = |D|_{\al,\mu} = ((D+\mu)^2+\al^2)^\frac12 
 $ 
     and $\cL(\al, \mu, 0)$ is the Fourier multiplier operator \eqref{cLvero0}.
     \item[(v)]{\sc Symmetry:} 
     It results 
\begin{equation}\label{cGcLmenomu}
    \begin{aligned}
        &\overline{\cG(\al, \mu,\e)} = \cG(\al,-\mu,\e)   \, , \qquad   \overline{\cL(\al, \mu,\e)} = \cL(\al,-\mu,\e)\, , \qquad \overline{\cB(\al, \mu,\e)} = \cB(\al,-\mu,\e) \, ,\\
        &\cG(-\al, \mu,\e) = \cG(\al,\mu,\e)   \, , \qquad   \cL(-\al, \mu,\e) = \cL(\al,\mu,\e)\, , \qquad {\cB(-\al, \mu,\e)} = \cB(\al,\mu,\e) \, .
    \end{aligned}
\end{equation}
In particular  
$ \cG(\al,0,\e) $, 
$ \cL(\al,0,\e) $ 
and $ \cB (\al,0,\e) $ are real operators. 
\end{itemize}
\end{proposition}

\begin{proof}
The proof is  in  \Cref{sec:proofs.FDN}. 
\end{proof}

The regularity  properties of $(\al,\mu,\e) \mapsto \cG(\al,\mu,\e)$  and $(\al, \mu, \e) \mapsto \cL(\al, \mu, \e)$ are established  in  \Cref{DNProp1}.

\paragraph{ Complex Symplectic structure.} 
We equip the space  $ L^2 (\T, \C^2) $
with  the {\it complex symplectic} form
\begin{equation}\label{ses}
{\cal W}_c  \, \colon L^2 (\T, \C^2) \times L^2 (\T, \C^2) \to \C \, , \quad  {\cal W}_c(f, g) := (\cJ f,g) \, ,
\end{equation}
where $\cJ $ is the symplectic matrix defined in \eqref{PoissonTensor}.  
Since $ \cJ $ is skew self-adjoint,
i.e.  $ \cJ^* = - \cJ $, and invertible,   
the map $ {\cal W}_c $ is a 
{\it complex symplectic form}, 
  namely, cf. \cite[Def.  1]{EM},
\begin{itemize}
\item {\sc Sesquilinear:}  
$ {\cal W}_c (f, g) $
is linear in  $ f $ and anti-linear in $ g $; 
\item {\sc Skew-Hermitian:}
$ {\cal W}_c (f, g) = - 
\overline{{\cal W}_c (g, f)} $; 
\item {\sc Non-degenerate:}
$ {\cal W}_c (f, g) = 0 $, for any $ f, g \in L^2 (\T, \C^2)  $, 
implies that $ f = 0 $. 
\end{itemize} 
We also remind 
the following basic definition: 

\begin{definition}\label{def:sympsub}
  {\bf (Symplectic subspace)}   A subspace $\cV $ of 
     $ L^2(\T,\C^2)$ is  symplectic  if $\cW_c\vert_{\cV}$ is non-degenerate. 
 \end{definition}

The  
spectrum of $ \cL (\alpha, \mu, \epsilon )$ has the following properties, which are 
inherited by 
the Hamiltonian and reversible structure 
of $ \cL (\alpha, \mu, \epsilon )$:
\begin{itemize}
\label{symspectrum} 
 \item 
$ \sigma (\cL(\al,\mu,\e)) $ is symmetric with respect to the imaginary axis: 
if $ \lambda $ is an eigenvalue of
$ \cL (\al,\mu,\e)  $ so is  $ - \bar \lambda $. 
\item 
In view of  \eqref{cGcLmenomu} 
the spectrum  
$ \sigma (\cL(\al, -\mu,\e)) = \overline{  \sigma (\cL(\al, \mu,\e)) } $. In particular $\cL(\al,0,\e)$ is real and $\sigma(\cL(\al,0,\e))$ is symmetric also with respect to the real axis. 
\item  In view of the covariance property \eqref{cBLmu+k} 
the spectrum $  \sigma (\cL(\al,\mu,\e)) $ is a 1-periodic set with respect to $\mu$.
 \end{itemize}
We first describe the unperturbed spectrum  at 
$ \e = 0 $. 

\subsection{Unperturbed spectrum and McLean curves}\label{sec:McLeanunp}

Since the operator $\cL (\al,\mu,\e) $
in \eqref{cLame1} is complex Hamiltonian according to \eqref{cLame0}, 
if $ \lambda $ is an eigenvalue of
$ \cL (\al,\mu,\e)  $ so is  $ - \bar \lambda $,  
and therefore eigenvalues of $ \cL(\alpha,\mu,\e)   $ with non zero real part may only arise from multiple 
eigenvalues of
the  
Fourier multiplier operator  
$ \cL(\alpha,\mu, 0) $ in \eqref{cLvero0}, with actually opposite Krein signature, see e.g. \cite{HK,KP}.
Its  spectrum is formed by  
the  purely imaginary eigenvalues, for any $ k \in \Z $, $ \sigma \in \{\pm 1 \} $, 
\begin{equation}
\label{eig.0}
\lambda_k^\sigma (\alpha,\mu) = \im \omega^\sigma_k (\alpha,\mu)
 = \im \big(   \sigma k + \mu  - \sigma \Omega_{\al} (\sigma k 
 + \mu ) \big) 
 \qquad \text{where} 
 \qquad  
 \Omega_\al(\varphi) := (\varphi^2+\alpha^2)^{\frac14} \, .
\end{equation}
Note that there is freedom in parametrizing the spectrum since the  eigenvalues satisfy the 
{\it covariance}-property
\be\label{labeling.eig}
\lambda_k^\sigma (\al, \mu + \ell) = \lambda^\sigma_{k + \sigma 
\ell}(\al, \mu) \, ,
\quad \forall k, \ell \in \Z \, , \ \alpha \, , \mu \in \R \, , \sigma = \pm \, , 
\ee
(this explains the variety of different conventions encountered in the literature).  
If $\Omega_{\al}(\sigma k + \mu) \neq 0$ the 
eigenvector associated to 
the eigenvalue 
$\lambda_k^\sigma (\alpha,\mu)$ in 
\eqref{eig.0} is
\begin{equation}\label{unperturbed.eigv}
  v^\sigma_k :=   v^\sigma_k (\alpha,\mu)
= 
\frac{1}{\sqrt{2 \Omega_{\al}(\sigma k +\mu)}}
\vet{\im \sigma \, \Omega_{\al}(\sigma k + \mu)  }{1} e^{\im \sigma  k x}\,  
 \, ,
\quad
k \in \Z \, ,  \ 
\sigma = \pm \, ,
\end{equation}
note that  $ v_k^\sigma (\alpha,\mu)$
is not defined at $ (\alpha,\mu) = (0, - \sigma k )$.
These  eigenvectors satisfy  
\begin{equation}\label{symp.nonzero}
    \cW_c \big( v_k^\sigma(\al,\mu) , v_{k'}^{\sigma'}(\al,\mu) \big) = \begin{cases}
        - \im & \mbox{ if } k=k' \mbox{ and } \sigma = \sigma' = + \\
         \im & \mbox{ if } k=k' \mbox{ and } \sigma = \sigma' = -\\
         0 & \mbox{ otherwise } ,
    \end{cases} 
\end{equation}
where 
$ \cW_c $ is the symplectic form in \eqref{ses}, 
and 
\begin{equation}\label{rev.nonzero}
        \varrho_c v_k^\sigma (\al,\mu) = - v_k^\sigma (\al,\mu) \, , \qquad \quad \text{with}\ 
        \varrho_c  \mbox{ in } \eqref{def:cinvolution} \, . 
\end{equation}
The multiple eigenvalues of  $\cL(\al,\mu,0)$ may be  
only double or have 
algebraic multiplicity $ 4 $, 
as we prove in  \Cref{lem:description} and \Cref{separation_eigenvalues} below. 
Actually the only eigenvalue with  algebraic multiplicity four is  $ 0 $,  
and this happens when    
$ (\alpha, \mu) = (0,0) $, 
\begin{equation}\label{4collision}
\lambda^+_0(0,0) = \lambda^-_0(0,0) = 
\lambda^+_{1}  (0,0) = 
  \lambda^-_{1}  (0,0)   = 0 \, .
\end{equation}
Non-zero eigenvalues of $ {\cal L}(\alpha,\mu, 0)  $ 
are either simple or double.
\\[1mm]
{\bf McLean curves.}
As noted by McLean \cite{Mc2}, 
the 
 $(\al, \mu)$ for which  {\em at least} two  eigenvalues among $\{ \lambda_k^\sigma(\al, \mu)\}_{k, \sigma}$ in  \eqref{eig.0} 
 are equal,  are classified in two distinct classes, parametrized by an integer $\tm \in \N_0 := \{0 \} \cup \N $, 
\begin{subequations}\label{def:McleanI_II_m}
\begin{align}
    & \label{def:McleanI}\text{Class (I)}_\tm\colon \ \quad 
\lambda_\tm^+ (\alpha,\mu) = \lambda_\tm^- (\alpha,\mu) &&\iff \qquad 2 \tm = \Big[ (\mu+\tm)^2 + \alpha^2 \Big]^{1/4} +
\Big[ (\mu-\tm)^2 + \alpha^2 \Big]^{1/4} \, ,  
\\
&\label{def:McleanII} \text{Class (II)}_\tm \colon \ \quad 
\lambda_\tm^+ (\alpha,\mu) = \lambda_{\tm+1}^- (\alpha,\mu) &&\iff  2 \tm + 1 = \Big[ (\mu + \tm)^2 + \alpha^2 \Big]^{1/4} +
\Big[ (\mu-\tm-1)^2 + \alpha^2 \Big]^{1/4} \, .
\end{align}
\end{subequations}
Equal eigenvalues $ \lambda^{\sigma}_{k_1} (\alpha, \mu) =
 \lambda^{\sigma}_{k_2} (\alpha, \mu)$ with the same sign $ \sigma $ occur only 
at  zero, cf. \eqref{4collision} and   \eqref{samesignint}.
Any  
possible eigenvalue collision is classified by 
 \eqref{def:McleanI_II_m} up to shifting $\mu$ by an integer because, in view of 
the covariance property 
\eqref{labeling.eig}, 
\be\label{covag}
\lambda^+_{k_1}(\al, \mu) =  \lambda^-_{k_2}(\al, \mu) \quad
\Leftrightarrow
\quad
\begin{cases}
    (\al, \mu + \frac{k_1-k_2}{2})  \in (\text{I})_{\frac{k_1+k_2}{2}} & \mbox{ if } k_1 \equiv_2  k_2  \\
    (\al, \mu + \frac{k_1-k_2+1}{2})  \in (\text{II})_{\frac{k_1+k_2-1}{2}} & \mbox{ if } k_1 \not\equiv_2 k_2 \, . 
\end{cases}
\ee
It is also convenient to parametrize the classes 
 \eqref{def:McleanI_II_m}  by a single integer  
$ \tp \in \Z  $, that  physically describes the order of ``plane waves interactions",  defining the  
unperturbed McLean  curves  \begin{equation}\label{mcleanmanifoldsp0}
 {\cM}^{(\tp)} := \Big\{ (\alpha, \mu) \in \R^2 \ | \ 
 m_\tp (\alpha, \mu )  = 0 
 \Big\}  \quad
 \text{where} \quad 
 m_\tp (\alpha, \mu ) :=
 \begin{cases}
     \omega_{\frac{\tp}{2}}^+(\al, \mu) -\omega_{\frac{\tp}{2}}^-(\al, \mu)  &  \text{if} \ \tp \ \text{is even,} \\
     \omega_{\frac{\tp-1}{2}}^+ (\alpha, \mu) - \omega_{\frac{\tp+1}{2}}^- (\alpha, \mu) &  \text{if} \ \tp \ \text{is odd} \,  .
 \end{cases}
 \end{equation}
The first two McLean curves are
\begin{equation}\label{McLean1}
\Big[ (\mu+1)^2 + \alpha^2 \Big]^\frac14+
\Big[ (\mu-1)^2 + \alpha^2 \Big]^\frac14 = 2 \, ,
\quad
\Big[ (\mu-2)^2 + \alpha^2 \Big]^\frac14+
\Big[ (\mu+1)^2 + \alpha^2 \Big]^\frac14 = 3 \, .
\end{equation}
In view of the periodicity of the spectrum 
$ \sigma (\cL (\alpha,\mu,0) )$ one could regard the McLean curves 
$ \cM^{(\tp)}$ on the cylinder 
$ \R \times (\R/\Z) $.
Equivalently the set of Fourier-Bloch parameters $ (\alpha, \mu) \in \R^2 $ where two unperturbed  eigenvalues $ \lambda^{\sigma}_{k_1} (\alpha, \mu) =
 \lambda^{\sigma'}_{k_2} (\alpha, \mu)$ coincide  is 
$$
\bigcup_{\tp \in \N_0, k \in \Z}
\Big[ (0,k) + {\cM}^{(\tp)} \Big]    
$$
since, in view of \eqref{covag}, 
\begin{equation}\label{mcleanmanifoldsp0k}
(\alpha, \mu) \in (0,k) +
{\cM}^{(\tp)} 
\quad 
\iff \quad 
 \begin{cases}
     \omega_{\frac{\tp}{2}-k}^+(\al, \mu) =\omega_{\frac{\tp}{2}+k}^-(\al, \mu)  &  \text{if} \ \tp \ \text{is even,} \\
     \omega_{\frac{\tp-1}{2}-k}^+ (\alpha, \mu) = \omega_{\frac{\tp+1}{2}+k}^- (\alpha, \mu) &  \text{if} \ \tp \ \text{is odd} \,  .
 \end{cases}
  \end{equation}
The graphs of $ \cM^{(2)}$ and 
$ \cM^{(3)}$  are plotted in \Cref{fig:Mc16}.  
The next result describes
all the McLean curves. 

\begin{proposition}[{\bf Spectral collisions and McLean curves}]\label{lem:description}\hspace{0.1cm}\\
{$\bullet$ \sc Same sign wave interactions.} 
Let $ k \neq m $. For any $\sigma=\pm $,  
\begin{equation}\label{samesignint}
    \omega_k^\sigma(\al,\mu) = \omega_{m}^\sigma (\al,\mu) \quad \iff \quad 
   \begin{aligned}
       &k-m  = 1 \ \text{ and }
   \ (\al,\mu)=(0, \sigma (1-k))\, , \text{ or}\\
   &k-m = -1 \ \text{ and }
   \ (\al,\mu)=(0,-\sigma k) \, ,
   \end{aligned}
\end{equation}
and
$ \omega_k^\sigma(0,\mu) = \omega_{m}^\sigma (0,\mu)  = 0 $, 
corresponding to the quadruple 
Benjamin-Feir collision at the origin.
\\[1mm]
{$\bullet$ \sc Opposite sign wave interactions.}  The unperturbed McLean curves
    $ \cM^{(\tp) }$   are
\begin{equation}\label{oppsigninteractions}
    \cM^{(\tp)} =
\begin{cases}
\emptyset &   \text{if} \  \tp\leq -1 \\
(0,0)  &   \text{if}  \ \tp = 0   \\ 
(0,0),(0,1) &   \text{if}  \ \tp = 1   \\ 
\text{a non-trivial  compact analytic variety  with  
 a cross singularity at $(0,0)$}&   \text{if}  \ \tp = 2   \\
\text{a real  analytic 1-dimensional, connected, compact manifold} &  \text{if}  \ \tp \geq 3 \, .  
 \end{cases}
\end{equation}
All the  $ \cM^{(\tp)}$ are compact.  For any $ \tp \geq 3 $ each $ \cM^{(\tp)}$ does not intersect $ (0,0) $.
For any $\tp\geq 2$  each $ {\cM}^{(\tp)}$ is symmetric under the reflection 
    $ (\alpha, \mu) \mapsto ( -\alpha, \mu ) $.
    In addition, if $\tp$ is even,  $\cM^{(\tp)}$ is symmetric also under $(\al,\mu)\mapsto(\al,-\mu)$, resp. if $\tp$ is odd,
    $ {\cM}^{(\tp)}$ is symmetric with respect to $(\al,\mu)\mapsto (\al,1-\mu)$. Each McLean curve
    $ \cM^{(\tp) }$  intersects the axis $\{\al=0\}$ at
\begin{equation}\label{sing.ML}
        \begin{aligned}
            &\mu_*^+ (\tp) = \frac14 (1+\tp^2)\quad &&\mu_*^- (\tp) = -\frac14 (1+\tp^2)&&\text{if }\tp \text{ is even},\\
            &\mu_*^-(\tp):= -\frac{\tp-1}{2} - \Big( \frac{\tp-1}{2}\Big)^2\, , \qquad  &&\mu_*^+(\tp):=  1+\frac{\tp-1}{2} + \Big( \frac{\tp-1}{2}\Big)^2  \quad &&\text{if }\tp \text{ is odd}  \, .
        \end{aligned}
    \end{equation}
    and, on $\{ \alpha >0 \} $, is the graph of an analytic function $\al_\tp : (\mu_*^-(\tp),\mu_*^+(\tp)) \to \R_+$  ($\al_2$ is not analytic at $\mu=0$). 
    The McLean curves do not intersect each other and are nested as illustrated in \Cref{fig:Mc16}, 
    namely    \begin{equation}\label{outinunp}
        \cM^{(\tp)}\subset \cU_{\tp+1}^+:=\{m_{\tp+1}(\al,\mu)>0 \} =  \text{bounded region enclosed by }\cM^{(\tp+1)} \, , 
    \end{equation}
where each  function
$ m_\tp (\alpha, \mu ) $ is  
defined in \eqref{mcleanmanifoldsp0}.
\end{proposition}

A direct computation using \eqref{eig.0} proves that for any $ \tp \geq 2 $ the double eigenvalue on the McLean curve
$ \cM^{(\tp)} $ vanishes only at
$ 
\mu = 0 $ if $ \tp $ is even, resp. 
$ 
\mu = \tfrac12 $ if $ \tp $ is odd, 
namely 
\begin{equation} \label{diffaut}
\omega_{\frac{\tp}{2}}^+(\pm \al_*^{(\tp)}, 0) =\omega_{\frac{\tp}{2}}^-(\pm \al_*^{(\tp)},0) = 0  \quad   \text{if} \ \tp \ \text{is even} \, , \quad 
     \omega_{\frac{\tp-1}{2}}^+ (\pm \alpha_*^{(\tp)}, \tfrac12) = \omega_{\frac{\tp+1}{2}}^- (\pm \alpha_*^{(\tp)}, \tfrac12) 
     = 0 \quad 
     \text{if} \ \tp \ \text{is odd} \, ,
\end{equation}
where 
$ (\pm \alpha_*^{(\tp)},0) \in \cM^{(\tp)} $
for $ \tp $ even and 
$ (\pm \alpha_*^{(\tp)},\tfrac12) \in \cM^{(\tp)} $
for $ \tp $ odd.

\Cref{lem:description} is  proved in Appendix \ref{app:descmclean}, 
jointly with 
the next lemma which describes the separation properties of the colliding eigenvalues of $\cL(\al,\mu,0)$  when  $(\al,\mu) \in  \cM^{(\tp)}$, 
from  the remaining part of the spectrum.

\begin{lemma}[{\bf Spectral separation near McLean curves}]\label{separation_eigenvalues}  
For any $\tp \geq 2$,  there exist  a neighborhood $\cN^{(\tp)}$ of 
the McLean curves $ {\cM}^{(\tp)}$  defined in  \eqref{mcleanmanifoldsp0}, 
positive constants  $ (\tc_\tp)_{\tp \geq 3} $ and,  
 for any $\delta>0$  small, a  constant 
 $  \tc_2(\delta) > 0 $, satisfying $ \tc_2(\delta) \to 0 $
 as
$ \delta\to 0^+ $,  
 such that:  
\begin{itemize}
    \item[(a)] {\bf (Benjamin-Feir separation)} setting 
    $\Lambda := 
\{(0,+), \, (0,-),\,(1,+),\,(1,-) \} $,  \begin{equation}\label{0905:1852}
\inf_{
\substack{ (k,\sigma)\in\Lambda, (q,\sigma')\not\in \Lambda \\ (\al, \mu) \in B_{\delta}(0,0)} }
 |\omega_k^\sigma(\al,\mu)-\omega_q^{\sigma'}(\al,\mu)| \geq \tfrac14  \, .
\end{equation}
\item[(b)] {\bf (McLean separation)} 
\begin{equation}\label{0905:1850}
\begin{aligned}
&\mbox{if } \tp =2: \qquad \quad \inf_{\substack{
       (q,\sigma) \neq (1,\pm)  \\ (\al,\mu)\in\cN^{(2)}\setminus B_\delta(0,0)}}\abs{\omega_1^+ (\al,\mu)-\omega_q^{\sigma} (\al,\mu)}+ \abs{\omega_1^- (\al,\mu)-\omega_q^{\sigma} (\al,\mu)}
       \geq \tc_2(\delta) > 0 \, , \\
&\mbox{if } \tp\geq 3 \mbox{ odd: } \quad 
        \inf_{\substack{(q,\sigma)
    \notin \{ (\frac{\tp-1}{2},+) , \, (\frac{\tp+1}{2},-)\} \\ (\al,\mu)\in\cN^{(\tp)}}}\abs{\omega_{\frac{\tp-1}{2}}^+(\al,\mu)-\omega_q^{\sigma} (\al,\mu)}+
    \abs{\omega_{\frac{\tp+1}{2}}^-(\al,\mu)-\omega_q^{\sigma} (\al,\mu)}
       \geq \tc_\tp > 0   \, , \\
&\mbox{if } \tp\geq 4 \mbox{ even: } \quad \inf_{\substack{
       (q,\sigma) \neq (\frac{\tp}{2} ,\pm) \\ (\al,\mu)\in\cN^{(\tp)}}}
       \abs{\omega_{\frac{\tp}{2}}^+ (\al,\mu)-\omega_q^{\sigma} (\al,\mu)} +
       \abs{\omega_{\frac{\tp}{2}}^- (\al,\mu)-\omega_q^{\sigma} (\al,\mu)} 
       \geq \tc_\tp > 0   \, .
\end{aligned}
\end{equation}
\end{itemize}
\end{lemma}

\subsection{Perturbed 
$3d $ Benjamin-Feir spectrum}\label{sec:22}

\Cref{TeoremoneFinale} below
fully describes  the perturbed 
Benjamin-Feir spectrum of the Stokes waves under 
 $3d$-longitudinal and transverse wave
perturbations, namely the four perturbed spectral bands 
of the operator $ {\cal L}(\alpha, \mu, \epsilon) $ in  
\eqref{cLame1} near zero.

The unperturbed operator 
$\cL(0,0,0) $ has the degenerate eigenvalue  $ 0 $, with algebraic multiplicity $ 4 $, 
cf. \eqref{4collision}, and 
geometric multiplicity 3 
with three real eigenvectors
\begin{equation}\label{base3e}
f_{1}^+ := \vet{ \cos x}{\sin x} \, , \ 
f_{1}^- := \vet{-\sin x}{ \cos x} \, ,
\ \ f_0^-  := \vet{0}{1} \, , 
\quad   \text{and the generalized eigenvector} \ 
 f_0^+ := \vet{1}{0} \, , 
\end{equation}
where $ \cL(0,0,0) f_0^+ = -f_0^- $. 
The basis $\{f_1^\pm,f_0^\pm\}$ is  {\it symplectic} and {\it reversible}  according to the following definition.

\begin{definition}\label{def:symprevbasis} 
{\sc (symplectic and reversible basis)}
    A basis $\{f_k^\pm\}_{k=1,\dots, n}$ of 
   a  symplectic subspace 
    $\cV$ is 
{\it symplectic}
        if 
    \begin{equation}\label{def:symp}
        {\cal W}_c (f_k^-, f_k^+) = 1 \ , \quad {\cal W}_c ( f_k^\sigma, f_{k'}^{\sigma'}) = 0 \quad \forall \, k\neq k'\, ,\ \sigma,\sigma' =\pm \, \text{ or }\ k = k',\ \sigma = \sigma' \ , 
        \quad 
        {\cal W}_c \mbox{ in } \eqref{ses}\, ; 
    \end{equation}
and  {\em reversible}, if 
    \begin{equation}\label{def:rev}
        \varrho_c (f_k^\sigma ) =\sigma f_k^\sigma \, , \quad 
    \forall k=1, \ldots, n \, , \  
    \sigma = \pm \, , 
    \quad 
    \varrho_c  \mbox{ in } \eqref{def:cinvolution} \, . 
    \end{equation}
\end{definition}
The complex eigenvectors $v_1^{\pm}(0,0)$ of $ \cL (0,0,0) $ in \eqref{unperturbed.eigv} are related to $f_1^\pm$ in \eqref{base3e} as
\begin{align}\label{v1pv1-}
&    v_1^+(0,0) = \frac{1}{\sqrt{2}} (f_1^- + \im f_1^+) = \frac{1}{\sqrt{2}} \vet{\im}{1} e^{\im x }  \,  , 
    \quad
    v_1^-(0,0) = \frac{1}{\sqrt{2}}  (f_1^- - \im f_1^+) = \frac{1}{\sqrt{2}} \vet{-\im}{1} e^{\im x }  \, .  
\end{align}
The 
spectrum of $\cL (0,0,0)  $ decomposes as the disjoint union 
\begin{equation}\label{spettrodiviso0}
\begin{aligned}
&   \qquad \qquad \qquad \qquad \qquad  \qquad  \sigma(\cL(0,0,0)) = \sigma'(\cL(0,0,0)) \cup \sigma''(\cL(0,0,0)) \, , \\
& \sigma'(\cL(0,0,0)) = \{ 0\} =
\big\{
\lambda_k^\sigma (0,0)\, , \, k =0,1, \, \sigma = \pm \big\}
\, , 
\quad  \sigma''(\cL(0,0,0)) = \big\{
\lambda_k^\sigma (0,0)\ , \ k\in \Z\setminus \{0,1\}; \ \sigma = \pm \big\}\, .
    \end{aligned}
\end{equation}
By Kato's bifurcation theory (see  \Cref{lem:Kato1}) the perturbed spectrum $\sigma(\cL
(\al,\mu,\e))$  admits, for any $(\al,\mu,\e)$ sufficiently small, a disjoint decomposition 
$$
   \sigma(\cL(\al,\mu,\e)) = \sigma'(\cL(\al,\mu,\e)) \cup \sigma''(\cL(\al,\mu,\e))  
$$
where $\sigma'(\cL(\al,\mu,\e))$ is composed by  four eigenvalues 
$ \lambda_k^\pm(\alpha, \mu,\e) $, $ k = 0,1 $,  close to zero.
We denote by 
\begin{equation}\label{def.cV}
    \cV_{\al,\mu,\e}
    \mbox{ the spectral subspace of } \cL(\al,\mu,\e)\mbox{ associated to } \sigma'(\cL(\al,\mu,\e)) \, , 
\end{equation}
which is invariant under $\cL(\al,\mu,\e)$, has dimension four, and satisfies
$\sigma'(\cL(\al,\mu,\e)) = \sigma(\cL(\al,\mu,\e)\vert_{\cV_{\al,\mu,\e}})$.

\smallskip

We need  to introduce the class 
of ``polar-analytic" functions $ f (\alpha, \mu )$, namely functions 
which are analytic once expressed in the polar coordinates 
\begin{equation}\label{polar}
    \mu = \rho \cos \theta\, , \qquad \al = \rho \sin \theta\, , 
    \qquad \rho:= (\al^2+\mu^2)^\frac12 \, . 
\end{equation}
Given a closed set $S \subset\R^n$ and a Banach space $X$, we say that a function $f\colon S \to X$ is analytic if it is real analytic on an open set containing $S$. 
\begin{definition}{\bf (Polar-analytic function)}
\label{def:tildeM}
    Let $ X $ be a Banach space,  
    $r >0 $ and $ \e_0 \in (0,+\infty] $.
     A function 
     $$
     f: B_{r}(0,0)\setminus \{(0,0)\} \times B_{\e_0}(0) \to X \, ,
     \ (\alpha, \mu, \epsilon) \mapsto 
     f (\alpha, \mu, \epsilon) \, , 
     $$
      is  {\em polar-analytic}
      in $ \cA_P(B_{r}(0,0),\e_0;X)$,    if 
$$
F(\rho, \theta, \epsilon) := f(\rho\sin\theta ,\rho\cos \theta, \e) \, , \quad 
\forall  0 < \rho \leq r \, , 
$$  
admits an  
analytic 
extension on 
    $ \cD_{r,\epsilon_0} := \{  |\rho| \leq r \} \times \T 
     \times B_{\e_0}(0) \to X $. 
     A polar-analytic function has the expansion
$$
F(\rho, \theta, \epsilon) =
\sum_{m,n \geq 0} F_{m,n}(\theta)
\rho^m \epsilon^n \, ,
\quad \sup_{\theta \in \T} 
|F_{m,n}(\theta)| \leq 
\| F \|_{L^\infty(\cD_{r,\epsilon_0})}
r^{-(m+n)}  \, . 
$$
    {\sc Notation for remainders.} Given integers  
    $ j_1,k_1, \ldots, j_n, k_n \in \N_0  $,   
    we denote 
    by 
    $\cO_X(\rho^{j_1} \e^{k_1},\dots , \rho^{j_n} \e^{k_n}) $ a polar-analytic function 
    in $ \cA_P(B_{r}(0,0),\e_0;X)$  such that
    \be\label{remainpoa}
    F(\rho, \theta, \e) = \sum_{\ell=1}^n \rho^{j_\ell} \e^{k_\ell} g_\ell(\rho,\theta,\e)\, ,  \qquad g_\ell : \cD_{r,\epsilon_0}  \to X\ \text{ analytic} \, . 
    \ee
    If $ X = \R $ we simply denote $r (\rho^{j_1} \e^{k_1},\dots , \rho^{j_n} \e^{k_n})
    = \cO_\R (\rho^{j_1} \e^{k_1},\dots , \rho^{j_n} \e^{k_n})$. 
\\[1mm]
   For  a compact neighborhood $K^{(2)}$ of $\cM^{(2)}$ we denote by $\cA_P(K^{(2)},\e_0;X)$ the class of  functions $f:K^{(2)}\times B_{\e_0}(0) \to X$ which are polar-analytic in $\cA_P(B_r(0,0), \e_0; X)$ 
for some $r >0$ and 
 analytic 
in $\bar{[K^{(2)}\setminus B_r(0,0)]} \times B_{\e_0}(0)$.
\end{definition}

\noindent 
{\sc Remark:}
A polar-analytic function
$ f $ in $ \cA_P(B_r(0,0), \e_0; X)$ has well defined directional limits
\begin{equation}\label{dirlimits}
    \exists \lim_{\rho \to 0 } f(\rho \sin \theta, \rho \cos \theta,\e) \ ,
\quad 
\forall \theta \in \T\, , \  \e\in B_{\e_0}(0) \, ,
\end{equation}
but the limit can depend on $\theta$. 
The functions 
$$
f_1(\al, \mu) = \sqrt{\al^2+\mu^2}\, , \quad f_2(\al,\mu) =  \frac{\al^2}{\sqrt{\al^2+\mu^2}}  \, ,\quad 
f_3(\al,\mu) = \frac{\al}{\sqrt{\al^2+\mu^2}}\, ,
$$ 
are polar-analytic, but only $f_1$ and $f_2$ have a limit as $(\al, \mu) \to 0$, and actually extend as  Lipschitz  functions near  $(0,0)$.  
An analytic function $ F(\rho,\theta, \epsilon)$, once expressed in the euclidean  coordinates \eqref{polar},  does not define a function of $ (\alpha, \mu) $ at the origin $ (0,0) $. 

Instead, a polar-analytic function  $\cO_X(\rho)$ 
\text{has}  a unique continuous extension at $ (\alpha, \mu) = (0,0) $, with value $0$, which 
is Lipschitz in a whole neighborhood of 
$  (0,0)$, cf. \Cref{lem:prodintM}-($v$). 
A polar-analytic function whose directional limits \eqref{dirlimits} all coincide (i.e., are independent of $\theta$) admits a continuous (and actually Lipschitz) extension at $(\al,\mu)=(0,0)$.

\begin{theorem}\label{TeoremoneFinale}
{\bf (Benjamin-Feir spectrum for 3d water waves)}
\noindent
There exist $ \e_1, \rho_1>0 $  such that: 
\\[1mm]
$ \bullet $  {\bf Symplectic and reversible basis:}
there exists 
a  symplectic, reversible 
and polar-analytic basis 
$  \cH := \{ h_k^\sigma (\al,\mu,\e)\, : \ k=0,1,\, \sigma = \pm\}$ of $\cV_{\al,\mu,\e}$, of the form
\begin{equation}
\label{proph}
h_k^\sigma(\al,\mu,\e) = f_k^\sigma + \cO_{H^1}(\rho,\e)  \in\cA_P(B_{\rho_1}(0,0),\e_1;H^1(\T,\C^2)) 
\end{equation}
where $ f^\sigma_k $ is the unperturbed basis  \eqref{base3e}. At $ \e = 0$   the vectors $h_1^\pm (\alpha, \mu, 0) $ 
are related to  $v_1^\pm (\alpha, \mu ) $ in \eqref{unperturbed.eigv} as 
\begin{equation}
\label{basiepzero}
v_1^+(\al,\mu) = \tfrac{1}{\sqrt 2}\left(h_1^-(\al,\mu,0) + \im  h_1^+(\al,\mu,0)\right)\, , \qquad v_1^-(\al,\mu) = \tfrac{1}{\sqrt 2}\left(h_1^-(\al,\mu,0) - \im  h_1^+(\al,\mu,0)\right) \, .
\end{equation}
$ \bullet $ {\bf Matrix representation: }
 the operator $ \cL(\al,\mu,\e) : \mathcal{V}_{\al, \mu, \e} \to  \mathcal{V}_{\al, \mu, \e} $ 
 is represented on the basis $\cH $   by a $4\times 4$ matrix of the form 
 \begin{equation} \label{matricefinae}
  \tL(\al,\mu,\e)=\begin{pmatrix} \mathtt{U} & \vline & 0 \\ \hline 0 & \vline & \mathtt{S} \end{pmatrix},
 \end{equation}
 where 
 \begin{align}\label{UU}
&  \mathtt{U} := \mathtt{U}(\alpha,\mu, \e) = 
  \im \mathsf{a}(\al, \mu, \e) + 
\begin{pmatrix} 
0 & \mathsf{b}^+(\al, \mu, \e)\\
  \mathsf{b}^-(\al, \mu, \e) & 0
 \end{pmatrix}  \\
 &\label{S}  \mathtt{S} := 
  \mathtt{S}(\alpha,\mu, \e) := 
  \im  \mu \, (1+  r_9(\e^2,\rho^2))  +
 \begin{pmatrix} 
  0 & \rho (1+r_{10}(\e^2,\rho^2))\\
-1+ r_8(\e^4,\rho\e^2,\rho^3) &   0
 \end{pmatrix} ,
\end{align}
are  polar-analytic 
functions in $ \cA_P(B_{\rho_1}(0,0),\e_1;\C^{2\times 2})$, even in $\e$, with entries
\begin{align}
\label{exp:a}
&    \mathsf{a}(\al, \mu, \e) :=  \frac{\mu}{2} \big(  1 + r_2(\e^2,\rho^2) \big) \, , \\
\label{exp:b+-}
 &   \mathsf{b}^+(\al, \mu, \e):= \frac{\al^2}{4}(1+r_5'(\e^2,\rho))-\frac{\mu^2}{8} (1+r_5(\e^2,\rho)) \ , \\
  &    \mathsf{b}^-(\al, \mu, \e):= -\e^2(1+r_1(\e^2,\rho)) 
  - \frac{\al^2}{4}(1+r_1'(\e^2,\rho)) + \frac{\mu^2}{8}(1+r_1''(\e^2,\rho)) \, .  
  \label{exp:b-}
  \end{align}
The function  $ \mathsf{b}^-(\al, \mu, \e)$ admits a Lipschitz extension in a neighborhood of $(\al, \mu) = (0,0)$.
Furthermore at $ \e = 0 $
\be\label{siannu}
\mathsf{a}(\al, \mu, 0) =
\frac12 (\omega_1^+(\alpha, \mu)+
 \omega_1^-(\alpha, \mu)) \, , 
\quad 
- \mathsf{b}^+(\al, \mu, 0)   =  \mathsf{b}^-(\al, \mu, 0) = \tfrac12 (\omega_1^+(\al, \mu)- \omega_1^-(\al, \mu)) \, . 
\ee
$ \bullet $  {\bf Unstable Eigenvalues:} the eigenvalues of the matrix $ \mathtt{U} $ in \eqref{UU},
\begin{equation}\label{eignear0}
 \lambda_\pm^{(2)}(\alpha, \mu,\e):=     \lambda_1^\pm(\alpha, \mu,\e) = \im 
    \mathsf{a}(\al, \mu, \e)  \pm \sqrt{
\mathsf{b}^+(\al, \mu, \e) \, \mathsf{b}^-(\al, \mu, \e)
} \, , 
\end{equation}
are even in $ \e $,
and have a nontrivial real part
if and only if 
$ (\al,\mu)$ belong to the local {\sc instability region} 
\begin{equation}
    \label{cDu}
\cU^{(2)}_{\e,\rm{loc}} : = \big\{ (\al,\mu)\in B_{\rho_1}(0,0) \, : \ \mathsf{b}^+(\al,\mu,\e)\mathsf{b}^-(\al,\mu,\e) > 0 \big\}\, .
\end{equation}
At $ \e = 0$ the spectrum
$ 
\sigma (\tU(\al,\mu,0)) = \{ \im \omega_1^\pm 
(\alpha, \mu) \} $ 
according to 
\eqref{eig.0}.  
\\[1mm]
$ \bullet $  {\bf Local McLean perturbed curves:} 
 The boundary of $ \cU^{(2)}_{\e,\rm{loc}}  
 $ is formed by 
the local branches of perturbed McLean curves 
\begin{equation}\label{M2pm}
\cM^{(2)}_{\pm,\rm{loc}} (\e):= \big\{(\al,\mu)\in B_{\rho_1}(0,0)\, : \ \tb^\pm (\al,\mu,\e) = 0 \big\} 
\end{equation}
which are locally, in the half-plane $\{\mu\geq 0\}$,  graphs of functions 
$ \mu^+ : \{ |\alpha|
\leq \alpha_0 \} \times B_{\e_1^2} (0) \to \R\, , \ \mu^-:\{ |\alpha|
\leq \alpha_0 \} \times B_{\e_1} (0) \to \R\, $, respectively of the form 
\begin{equation}\label{graphsmclean2}
 \ \mu^+(\al, \e^2)=\sqrt{2} |\alpha| ( 1+ \tf(\al,\e^2))  \ , \quad \mu^-(\al, \e) = \sqrt{8 \e^2(1+\ell_1(\al,\e)) + 2 \al^2(1+\ell_2(\al,\e))} \, ,  
\end{equation}
where, for some $C>0$,

$ \bullet $ $\tf(\al,\e^2)$ is an analytic function of $(\al,\e^2) \in B_{\al_0}(0)\times B_{\e_1^2}(0)$ satisfying 
$\abs{\tf(\al,\e^2)} \leq C (\e^2 + |\al|)$;

$ \bullet $ the functions  
$\ell_i ( \alpha, \e )$,
$i=1,2 $ 
are real analytic in $ \alpha $ for any 
$ 0 < |\e | < \e_1  $, satisfy   $|\ell_i(\al,\e)| \leq C(|\al|+|\e|)$, and 
are Lipschitz on $ B_{\al_0}(0)\times B_{\e_1}(0)$. 
\end{theorem}

Let us make  some comments. 

\begin{enumerate}
\item 
 {\sc Stable Eigenvalues.} 
The  matrix $ \mathtt{S} $ 
in \eqref{S} has the 
purely imaginary eigenvalues
\begin{equation}\label{eignear0s}
 \lambda_0^\pm (\al, \mu, \e) = \im\mu \big(1+r_9(\e^2,\rho^2)\big)  \mp  \im\sqrt{\rho}\big(1+ r'(\e^2,\rho^2 )\big)\, . 
\end{equation}
\item 
{\sc Local perturbed McLean curves.}
The function
$$
\mathsf{b}^+(\al, \mu, \e) \, \mathsf{b}^-(\al, \mu, \e) = 
\tfrac14 \big(\text{Tr} ( \mathtt{U}(\alpha, \mu,\epsilon)) \big)^2 -
\det   \mathtt{U}(\alpha, \mu, \epsilon) 
$$
is a {\it spectral invariant}, namely it is independent of  the basis used to represent $\cL(\al,\mu,\e)_{|\mathcal{V}_{\al, \mu, \e}}$, 
as well as the instability region $ \cU^{(2)}_{\e,\rm{loc}} $
in  \eqref{cDu} and  the perturbed McLean curves in \eqref{M2pm}. 
The local McLean curve $ \cM^{(2)}_{-,\rm{loc}} (\e) $   
approaches, 
as $ \e \to 0 $,   to the corner of  
the unperturbed McLean curve $\cM^{(2)}$, as in   Figure \ref{fig:pertmctotal}-left. 
\item {\sc Analytic continuation.} \label{item:anco}
Theorem \ref{TeoremoneFinale} provides,
for any  $|(\al , \mu)| < \rho_1$ and 
$ |\e| <  \e_1 $, 
the  splitting 
of the $4$-dimensional symplectic subspace $\cV_{\al, \mu, \e}$ 
in \eqref{def.cV} 
(see Definition \ref{def:sympsub}) 
\begin{equation}\label{decosim}
\begin{aligned}
&  \qquad \qquad  \ 
\cV_{\al, \mu, \e}  = \cV^{(u)}_{\al, \mu, \e} \oplus^{\perp_{\cW_c}} \cV^{(s)}_{\al, \mu, \e}  \qquad \text{where} 
\\
& 
  \cV^{(u)}_{\al, \mu, \e}:= \textup{span}\langle h_1^\sigma(\al, \mu, \e) \rangle_{\sigma = \pm} \ , 
    \quad 
     \cV^{(s)}_{\al, \mu, \e}:= \textup{span}\langle h_0^\sigma(\al, \mu, \e) \rangle_{\sigma = \pm} \, ,
     \end{aligned}
\end{equation}
are $2$-dimensional symplectic  subspaces  
pairwise symplectic orthogonal. It results  
\be\label{biscot2}
    \sigma\big(\cL({\al, \mu, \e}) \vert_{\cV^{(u)}_{\al, \mu, \e}}\big) = \sigma(\mathtt{U}) = \{\lambda^\pm_1(\al, \mu, \e)\} \ , \quad 
     \sigma\big(\cL({\al, \mu, \e}) \vert_{\cV^{(s)}_{\al, \mu, \e}}\big) = \sigma(\mathtt{S}) = \{\lambda^\pm_0(\al, \mu, \e)\} \, , 
\ee
where $\lambda^\pm_1(\al, \mu, \e)$ are the eigenvalues  in \eqref{eignear0},  and $\lambda^\pm_0(\al, \mu, \e)$ are those  in \eqref{eignear0s}.
In  \Cref{part:II} we  shall exploit \eqref{decosim} to continue 
analytically  the eigenvalues $\lambda_1^\pm(\alpha,\mu,\e)$
in a full neighborhood of the  McLean
curve $\cM^{(2)}$, where  double collisions occur, 
while keeping $\e$ uniformly small.
\end{enumerate}

\begin{remark}\label{rem:4aut}
The  four unperturbed eigenvalues
$ \lambda^{\pm }_1 (\alpha, \mu) $
and
$ \lambda^{\pm }_0 (\alpha, \mu) $
strongly interact among them  for 
$ (\alpha, \mu) \sim (0,0) $ as their  spectral gap 
$ 
\text{dist} \big( \lambda^{\pm }_1 (\alpha, \mu), 
\lambda^{\pm }_0 (\alpha, \mu) \big) \to 0 
$ as $ (\alpha, \mu )
\to (0,0) $.  
Based on  
this information 
alone, one could obtain a decoupling
as \eqref{matricefinae} 
 taking $ \e (\alpha,\mu)
 \to 0 $  as
 $ (\alpha,\mu) \to 0 $.
 In contrast, Theorem \ref{TeoremoneFinale}
 holds for any $ |\e| \leq \e_0 $ small uniformly for 
 $ |(\alpha,\mu)| \leq \rho_1 $, thanks to the block-decoupling procedure 
 of  
\Cref{sec:block}. 
\end{remark}


\noindent 
{\sc Scheme of proof
of  \Cref{TeoremoneFinale}.} 
By a symplectic Kato reduction,
in \Cref{sec:katoBF}
we reduce to 
study the eigenvalues of a symplectic and reversible  $ 4 \times 4 $ matrix in the  
 long-wave regime 
$ (\alpha, \mu) \to (0,0) $.
By a Taylor expansion up to linear terms in $ (\alpha, \mu )$  of the Kato basis and the fiber Dirichlet-Neumann operator
$ \cG(\al,\mu,\e) $,
valid for any  $ \epsilon $, 
we obtain the expansion in \Cref{BexpG}.  Then we recognize that the top $ 2 \times 2 $ matrix $ \tJ_2 E $ 
in \eqref{BinG1} possesses a pair of Benjamin-Feir eigenvalues which 
accurately describe  \Cref{fig:pertmctotal} near the origin. In order to  prove that the real eigenvalues of 
$ \cL (\alpha,\mu,\epsilon)$ 
near zero have actually this form  we  eliminate the coupling term $ F $ in \eqref{BinG3}. The term of order $ \epsilon^3 $ in the top-left entry of $F$ is removed in \Cref{decoupling1}, in agreement 
with the fact that $ {\cal L} (0,0,\epsilon)$
has $0 $ as eigenvalues with geometric multiplicity $ 2 $. Then in  \Cref{lem:secondstep} we remove the linear terms of order $ \rho $ from each entry of $F$. This procedure  
preserves the 
polar analytic regularity, namely the block-diagonalizing transformation
leading to  $ \tL(\al,\mu,\e)$ in
\eqref{matricefinae} 
is polar analytic.

\subsection{Perturbed $3d$ McLean instabilities}
\label{sec:23}

\Cref{TeoremoneMcLean} below   describes the unstable spectrum 
of the operator
$\cL
(\al,\mu,\e) $ in  
\eqref{cLame1} also away from the origin. 
Its statement  also includes the Benjamin-Feir  \Cref{TeoremoneFinale}, thanks to the analytic continuation argument outlined in item \ref{item:anco}.

For any $\tp\geq 2 $ and  $(\alpha, \mu)$ 
near the $\tp$-McLean curve $\cM^{(\tp)}$ in   \eqref{mcleanmanifoldsp0}, 
the pair of eigenvalues of 
$ \cL(\alpha, \mu, 0) $, 
\begin{equation}\label{McL.d} 
\lambda^{(\tp)}_\pm (\al,\mu):= \im \omega^{(\tp)}
_\pm (\al,\mu) := \begin{cases}
    \im\omega_{\frac{\tp}{2}}^{\pm}(\alpha, \mu) \, \quad &\mbox{for }\tp \mbox{ even,} \\
    \im\omega_{\frac{\tp-1}{2}}^{+}(\alpha, \mu)\, ,
    \im\omega^-_{\frac{\tp+1}{2}}(\alpha, \mu) 
\quad 
&\mbox{for }\tp \mbox{ odd} \, , 
\end{cases} 
\end{equation}
are very close to each other, and actually coincide on $ \cM^{(\tp)} $. The
corresponding eigenvectors are
\begin{equation}\label{ivMC}
    v^{(\tp)}_\pm (\alpha,\mu) := \begin{cases}
        v_{\frac{\tp}{2}}^{\pm}(\alpha, \mu) \, \quad &\mbox{for }\tp \mbox{ even,} \\
    v_{\frac{\tp-1}{2}}^{+}(\alpha, \mu)\, ,
    v^-_{\frac{\tp+1}{2}}(\alpha, \mu) 
\quad 
&\mbox{for }\tp \mbox{ odd}  \, , 
    \end{cases}
\end{equation}
where $ v^\sigma_k (\alpha,\mu) $ are defined in \eqref{unperturbed.eigv}.
For any $\tp \geq 3$ 
and any $ (\alpha,\mu)$ near the McLean curves $ \cM^{(\tp)} $ the 
 eigenvalues 
$ \lambda^{(\tp)}_\pm (\alpha, \mu) $
are well separated from 
the other eigenvalues 
of $ \cL(\alpha, \mu, 0) $ 
in \eqref{eig.0}, by \Cref{separation_eigenvalues}.  
Instead 
$ \lambda^{(2)}_\pm (\alpha, \mu) $ collide
 when $(\al, \mu)\to 0 $ with $\lambda_0^\pm (0,0)  = 0$, cf. \eqref{4collision}.
Therefore we proceed as follows. By Kato perturbation theory  we first construct in \Cref{KatonearMcLean}  
the projector $ \wt P^{(2)}_{\al, \mu, \e}$ near the McLean curve $\cM^{(2)}$ away from the origin, and  then  we extend it analytically near the origin
into the projector on the subspace
$ \cV^{(u)}_{\al, \mu, \e}$ in \eqref{decosim} 
constructed via \Cref{TeoremoneFinale}.
 In any case, for any $(\al,\mu) $ in a
 sufficiently small compact neighborhood 
 of each  McLean curve $ \cM^{(\tp)}$, $ \tp \geq 2 $,  the perturbed spectrum $\sigma(\cL
(\al,\mu,\e))$ admits,  for $ \e$ enough,  a disjoint decomposition 
$$
   \sigma(\cL(\al,\mu,\e)) = \sigma'_\tp(\cL(\al,\mu,\e)) \cup \sigma''_\tp(\cL(\al,\mu,\e))   
$$
 where $\sigma'_\tp(\cL(\al,\mu,\e))$
 is composed by  two eigenvalues $ \lambda^{(\tp)}_\pm 
 (\alpha,\mu,\epsilon) $
  close to $\lambda^{(\tp)}_\pm 
 (\alpha,\mu) $. 
We denote  
\begin{equation}
\label{def.cVp}
    \cV^{(\tp)}_{\al,\mu,\e}
    \mbox{ the } 2 \text{ dimensional spectral subspace associated to } \sigma'_\tp(\cL(\al,\mu,\e)) \, , 
\end{equation}
which is invariant under $\cL(\al,\mu,\e)$ and 
$\sigma'_\tp(\cL(\al,\mu,\e)) = \sigma(\cL(\al,\mu,\e)\vert_{\cV^{(\tp)}_{\al,\mu,\e}})$. 
\\[1mm]
In order to state the main result, 
 we introduce the following class of functions.

\begin{definition}\label{def:tM}
 {\bf (Space $\cA $)}
     Let X be a Banach space, $\Omega\subset \R^2$ be a compact set and $\e_0 \in (0,+\infty] $. A function 
     $$
     A : \Omega \times B_{\e_0}(0) \to X \, , \ 
     (\al,\mu,\e) \mapsto A(\al,\mu,\e) \, , 
     $$ 
     belongs to $ \cA (\Omega,\e_0 ; X)$ if:\\[1mm]
     {\noindent\sc 1) (Regularity in $\e $)} for any $(\al,\mu) \in \Omega$ the map 
         $
          \e\to A(\al,\mu,\e) \in X
         $
         is analytic in $ B_{\e_0}(0)$. \\
         {\noindent\sc 2) (Regularity in $\al,\mu$)} \begin{itemize}
             \item for any $\e\in B_{\e_0}(0)$ the map 
             $
             A(\cdot, \cdot ,\e) :
             \Omega\setminus (\{0\}\times \Z) \to X $, 
             $
             (\al,\mu) \mapsto A(\al,\mu,\e) $,  
             is analytic;
         \item there exists $r>0$ such that, for any $j\in \Z$, for any $(\al,\mu)\in B_r(0,j)\cap \Omega$ and $\e\in B_{\e_0}(0)$, the operator $ A (\al,\mu,\e)$ decomposes as
    \begin{equation}\label{dectM}
             \begin{aligned}
             A (\al,\mu,\e) &= A^{[\mathrm{I}]}(\al^2,\mu,\e)+ (\al^2+(\mu-j)^2)^\frac12  A^{[\mathrm{II}]}(\al^2,\mu,\e)  
         \end{aligned}
         \end{equation}
 where 
 $$ 
 A^{[\mathrm{I}]}, 
 A^{[\mathrm{II}]} : 
 B_{r^2}(0)\times B_{r}(j)\times B_{\e_0} (0)
 \to X \, , \quad
 (\beta,\mu,\e) \mapsto 
 A^{[\mathrm{I}]}, 
 A^{[\mathrm{II}]} (\beta,\mu,\e) \, , 
 $$
 are  analytic functions. 
         \end{itemize}
If $X=\cL(Y,Z)$ is the space of bounded linear operators between Banach spaces $Y, Z$, we simply denote 
\be\label{notaLYZ}
\cA(\Omega,\e_0;Y,Z):= \cA(\Omega,\e_0;\cL(Y,Z))\, .
\ee
  \end{definition}
  
\noindent 
{\sc Notation for remainders.}
Let, for every $ \tp \geq 2 $, $K^{(\tp)}$  be a compact neighborhood of $\cM^{(\tp)}$ and $\e^{(\tp)}>0$.  
We denote 
$ r^{(2)}(\e^n) $ a polar-analytic function in $  \cA_P(K^{(2)},\e^{(2)};\R) $, cf. \Cref{def:tildeM},
and $ r^{(\tp)}(\e^n) $ a function in $ \cA(K^{(\tp)},\e^{(\tp)};\R) $,
$ \tp \geq 3 $, 
satisfying $|r^{(\tp)}(\e^n)| \leq C |\e|^n$ for some uniform constant $C>0$. 
\begin{remark}\label{rem:regtM_rhotheta}
     If  $A  \in  \cA (\Omega,\e_0 ; X)$
     and the open set  
     $ \Omega $ intersects 
     $ \{0\}\times \Z $ only at the origin, namely 
     $ \Omega\cap (\{0\}\times \Z) = \{(0,0)\}  \, , $      then\ the function $A(\rho\sin \theta,\rho\cos\theta,\e)$ is analytic in a neighborhood of  $\{\rho=0\}$, thus $A$ is a polar-analytic function in $\cA_P (\Omega,\e_0 ; X)$, according to \Cref{def:tildeM}. 
 \end{remark}
 
The class  $ \cA $ is closed under  composition, functional calculus and Cauchy integrals, cf.
\Cref{ap_cAF}.  

\begin{theorem}\label{TeoremoneMcLean}
{\bf (3d unstable spectral
bands)}
For any $ \tp \geq 2 $ 
there exists a compact neighborhood $K^{(\tp)} $ of each 
McLean curve
$ \cM^{(\tp)} $ in \eqref{mcleanmanifoldsp0} 
and $ \e^{(\tp)}  >0 $   
 such that, 
for any  $  (\al, \mu ) \in K^{(\tp)}  $,  any $ 0\leq |\e| < \e^{(\tp)}  $, 
 the operator $ \cL (\al, \mu,\e)  : \mathcal{V}_{\al, \mu, \e}^{(\tp)} \to  \mathcal{V}_{\al, \mu, \e}^{(\tp)} $ 
is represented by a $2\times 2$ 
matrix 
\begin{equation} \label{tocomputematrixb}
\tL^{(\tp)}(\al, \mu,\e) =\tJ\tB^{(\tp)}(\al, \mu,\e), 
\quad 
\tJ:=\begin{pmatrix} -\im & 0 \\ 0 & \im \end{pmatrix} \, ,  \quad 
\tB^{(\tp)}(\al, \mu,\e)=
\begin{pmatrix} \fa^{(\tp)}(\al, \mu,\e) & \fb^{(\tp)}(\al, \mu,\e) \\ \fb^{(\tp)}(\al, \mu,\e) & \fc^{(\tp)}(\al, \mu,\e) \end{pmatrix} 
 \, , 
\end{equation}
where 
\be\label{reguabcp}
\fa^{(\tp)}(\al, \mu , \e),  \fb^{(\tp)} (\al, \mu,\e),  \fc^{(\tp)} (\al, \mu,\e) 
\in \begin{cases}
\cA_P (K^{(2)}, \e^{(2)}; \R) 
\quad \ {\rm if} \ \tp  = 2 \, 
\\ \cA(K^{(\tp)},\e^{(\tp)};\R) \, , \quad  \ \forall \, 
\tp\geq 3 \, , 
\end{cases}
\ee
have the  form 
\begin{subequations}\label{matrixentries}
\begin{align}\label{expa}
 \fa^{(\tp)}(\al, \mu , \e) & = - \omega^{(\tp)}_+(\al, \mu) 
 + \fa_\tp (\al, \mu)  \e^2   +  r_\fa^{(\tp)} (\e^4)\, , \\
\label{expb}
  \fb^{(\tp)} (\al, \mu,\e) & = \fb_\tp (\al, \mu)  \e^\tp + \beta_\tp (\al,\mu)\e^{\tp+2} + 
  r_\fb^{(\tp)}  (\e^{\tp+4}) \, ,  \\
    \label{expc}
  \fc^{(\tp)}(\al, \mu, \e) & = \omega^{(\tp)}_-(\al, \mu) 
  + \fc_\tp(\al, \mu)  \e^2   + r_\fc^{(\tp)} (\e^4) \, .
\end{align}
\end{subequations}
The functions $ \fa^{(\tp)}(\al, \mu , \e)$, $ \fc^{(\tp)}(\al, \mu , \e) $
are even in $ \e $, while 
$ \fb^{(\tp)}(\al, \mu , \e)$ is odd in $ \e $
if $ \tp $ is odd;  $ \fb^{(\tp)}(\al, \mu , \e)$ is even in $ \e $ if $ \tp $ is even. They satisfy the symmetry properties  \eqref{symmetryapcpbp}. 
\\[1mm]
$ \bullet $ {\bf Eigenvalues:}
the matrix $ \tL^{(\tp)}(\al, \mu,\e)  $ has eigenvalues 
\begin{equation}\label{eigenvalues}
\lambda^{(\tp)}_{\pm} (\al, \mu,\e) = \tfrac\im2 
\big( \fc^{(\tp)}(\al, \mu,\e) - \fa^{(\tp)}(\al, \mu,\e)\big) 
\pm 
\tfrac12\sqrt{
D^{(\tp)}(\al, \mu,\e) 
}
\end{equation}
where
 \begin{align}
&  D^{(\tp)}(\al, \mu,\e)   = 4 
 \big( \fb^{(\tp)} (\al, \mu,\e) \big)^2 -
 \big( T^{(\tp)} (\al, \mu,\e) \big)^2  =\mathsf{d}^{(\tp)}_+(\al, \mu, \e)\mathsf{d}^{(\tp)}_-(\al, \mu, \e)\, ,\label{traccianulla} \\
\label{tracciaBep}
&  T^{(\tp)}(\al, \mu,\e) = 
\textup{Tr}\,  \tB^{(\tp)}(\al, \mu,\e)   
 :=  \fa^{(\tp)}(\al, \mu,\e) + \fc^{(\tp)}(\al, \mu,\e)  \, ,   \\
&  \mathsf{d}^{(\tp)}_\pm(\al, \mu, \e) := 2\fb^{(\tp)} (\al, \mu,\e) \pm T^{(\tp)} (\al, \mu,\e) \, .
\label{def:dppm}
 \end{align}
The  eigenvalues $ \lambda^{(\tp)}_{\pm}   (\al, \mu,\e)  $ 
have nonzero real part if and only if $(\alpha, \mu) $ belong to 
the global 
{\sc instability region} 
\begin{equation}\label{def:instaintro}
    \cU_\e^{(\tp)}:= \big\{
    (\al,\mu)\in K^{(\tp)}\ | \ D^{(\tp)}(\al, \mu,\e)>0 \big\} \, .  
\end{equation}
At $\e=0$ the eigenvalues $\lambda_\pm^{(\tp)}(\al,\mu,0)$ coincide with those in \eqref{McL.d}. \\
$ \bullet $  {\bf Global perturbed McLean curves:}  
The boundary of the instability region $\cU^{(\tp)}_\e$  is 
\be\label{boundinst}
\pa \cU^{(\tp)}_\e = \{(\al,\mu)\in K^{(\tp)}\, : \ D^{(\tp)}(\al,\mu,\e) = 0 \} = \cM^{(\tp)}_+(\e)\cup \cM^{(\tp)}_-(\e)
\ee
where $ \cM^{(\tp)}_\pm (\e )  $ are the 
$\tp$-th perturbed McLean curves
\be\label{defMepm}
\cM^{(\tp)}_\pm (\e) := 
\big\{ (\al,\mu)\in K^{(\tp)} \, : \ \mathsf{d}^{(\tp)}_\pm (\al,\mu,\e) = 0 \big\}  
\ee
which are, for any $\tp \geq 3$, 
  connected real-analytic closed curves
 satisfying, for some $C_\tp >0$, 
\begin{equation}\label{deformationmcleantpintro}
       \textup{d}_{\rm H}(\cM^{(\tp)}_+(\e), \cM^{(\tp)}_-(\e) ) \leq C_\tp  |\e|^{\tp} \ , \qquad
       \textup{d}_{\rm H}(\cM^{(\tp)}_\pm(\e), \cM^{(\tp)} ) \leq C_\tp  \e^2 \, . 
    \end{equation}
The set 
$\cM^{(2)}_+ (\e)$ is a connected closed curve,  analytic away from the origin
where it has a cross-singularity  described in \eqref{graphsmclean2} (see \Cref{fig:pertmctotal}), while  $\cM^{(2)}_-(\e)$ is the  union of two real-analytic closed curves,   satisfying, for some $C_2 >0$, 
\begin{equation}\label{asympmclean2intro}
    \textup{d}_{\rm H}(\cM^{(2)}_+(\e), \cM^{(2)} ) \leq C_2  \e^2 \ , 
    \qquad 
    \textup{d}_{\rm H}(\cM^{(2)}_-(\e), \cM^{(2)} ) \leq C_2  |\e| \, . 
\end{equation} 
Denoting by $\mathring \cM^{(\tp)}_\pm (\e)$   the interior regions enclosed by the  perturbed McLean curves, the instability region $\cU_\e^{(\tp)} $ in \eqref{def:instaintro} is
\be\label{symdifmclean}
\cU_\e^{(\tp)} = \mathring \cM^{(\tp)}_+ (\e) \, \triangle  \, \mathring \cM^{(\tp)}_-  (\e) \ .
\ee 
The perturbed McLean curves $\cM^{(\tp)}_\pm(\e)$  intersect at 
\be\label{newoffice}
\cM^{(\tp)}_+(\e) \cap \cM^{(\tp)}_-(\e) = \Big\{(\al, \mu, \e) \in K^{(\tp)} \times B_{\e^{(\tp)}}(0) \ \colon \ 
T^{(\tp)}(\al, \mu, \e) = \fb^{(\tp)}(\al, \mu, \e) = 0 
\Big\} 
\ee
where the eigenvalues $ \lambda^{(\tp)}_{\pm} (\al, \mu,\e) $ are equal and purely imaginary.

\smallskip
 \noindent{\bf $ \bullet $   Instability,  upper bounds:} for any $ \tp \geq 2 $ there is $C_\tp >0$ such that    
 the real part of the eigenvalues in
    \eqref{eigenvalues} satisfies 
    \be\label{upboundeigv}
\big| \Re \,  \lambda^{(\tp)}_\pm(\al, \mu, \e) 
\big| \leq C_\tp |\e|^{\tp} \,,
\quad \forall 
 |\e|\leq \e^{(\tp)} \, , \
 (\al, \mu) \in \cU_\e^{(\tp)} \, . 
    \ee
\noindent{$ \bullet $  \bf  Splitting of the McLean curves for $\tp = 2, 3 $:} for $\tp =2,3$,  
for any small 
$  \epsilon  \neq 0 $ 
the instability region
$\cU^{(\tp)}_\e \neq \emptyset$, and 
there are  closed analytic  curves $\cT^{(\tp)}(\e)$
  near $\cM^{(\tp)} $       and  $ C_\tp>0$, such that
\begin{quote}
    \underline{$\tp=2$:}
for any $r_2>0$,
there is $ c_2 (r_2) >0 $ such that
\begin{equation}\label{ordinstaintro2}
c_2 (r_2) \e^2 \leq 
\Re\lambda_+^{(2)}(\al,\mu,\e) \leq 
C_2 \e^2  
 \, ,
         \quad 
\forall (\al,\mu) \in \cT^{(2)}(\e)\setminus B_{r_2}(0, \pm \tfrac54) \, ,  
     \end{equation} 
     where  $ c_2 (r_2) \to 0 $
     as $ r_2 \to 0 $. 
     Near the  
points $(0,\pm \tfrac{5}{4})$, the size of $\Re\lambda^{(2)}_+
(\alpha,\mu,\epsilon) = \cO(\e^4) $.
     \\[1mm]
     \underline{$\tp=3$:} there are
     at most finitely many points 
     $(\al_j(\e),\mu_j(\e))_{j=1,\dots,n} \in \cT^{(3)}(\e)$, 
     and  for any $r_3>0$ there is $c_3(r_3)>0$
     (satisfying $c_3(r_3)\to 0$ as $r_3\to 0$) such that 
\begin{equation}\label{ordinstaintro3} 
c_3(r_3)|\e|^3  \leq \Re\lambda_+^{(3)}(\al,\mu,\e)
        \leq C_3|\e|^3 
          \, ,
         \quad
         \forall (\al,\mu) \in \cT^{(3)}(\e)\setminus \bigcup_{j=1}^n B_{r_3}(\al_j(\e),\mu_j(\e)) \, . 
     \end{equation}    
   \end{quote}
\end{theorem}

Let us make some comments. 
\begin{enumerate}
\item 
{\sc Global perturbed McLean curves.}
For any $ \tp \geq 2 $,  the discriminant  
$$
D^{(\tp)}(\al, \mu,\e)  = 
 \big( \textup{Tr}\,  \tL^{(\tp)}(\al, \mu,\e) \big)^2 
 - 4  \det\tL^{(\tp)}(\al, \mu,\e) 
 =
 \mathsf{d}^{(\tp)}_+(\al, \mu, \e)\mathsf{d}^{(\tp)}_-(\al, \mu, \e) 
$$
of $\tL^{(\tp)}(\al, \mu,\e) $ 
is a {\it spectral invariant}, as well as 
the perturbed McLean curves $\cM^{(\tp)}_\pm(\e)$
in  \eqref{defMepm}.
Actually 
since the subspace $\cV^{(u)}_{\al, \mu, \e}$ 
in \eqref{decosim} coincides with 
$\cV^{(2)}_{\al, \mu, \e}$ for any $(\al, \mu)$ small, 
the matrices $\tU$ in \eqref{UU} and 
$\tL^{(2)}(\al, \mu, \e)$ in \eqref{tocomputematrixb} are similar, 
 in particular 
\be\label{dpmbpm}
    \td^{(2)}_+(\al,\mu,\e)\td^{(2)}_-(\al,\mu,\e) = 4 \tb^+(\al,\mu,\e) \tb^-(\al,\mu,\e)\, , 
\ee
    where $\tb^\pm (\al,\mu,\e)  $ are in \eqref{exp:b+-}-\eqref{exp:b-}.
\item  {\sc Global expression of the eigenvalues.}
The  eigenvalues $\lambda^{(2)}_\pm(\al, \mu, \e)  $ of $ \cL (\alpha, \mu, \epsilon) $ near $ \cM^{(2)}$ are 
\begin{align}
    \label{lambda2.ae}
    \lambda^{(2)}_\pm(\al, \mu, \e)  = 
    \frac{\im}{2}&  \big(\omega_1^-(\al, \mu) + \omega_1^+(\al, \mu) + \e^2(\fc_2(\al, \mu) - \fa_2(\al, \mu) ) + r(\e^4) \big) \\
    \notag
    & \pm 
    \sqrt{
    4( \fb_2(\al, \mu) \e^2 + r(\e^4))^2 - \big(\omega_1^-(\al, \mu)  - \omega_1^+(\al, \mu) + (\fc_2(\al, \mu) + \fa_2(\al, \mu) )\e^2  + r(\e^4) \big)^2
    )}
    \end{align}
with $\omega_1^\pm(\al, \mu)$ in \eqref{eig.0}, and $\fa_2(\al, \mu)$, $\fb_2(\al, \mu)$, $\fc_2(\al, \mu)$ are computed explicitly in \eqref{fapappendix}, \eqref{fcpappendix} (for $\tp =2$ and $\tm = 1$) and \eqref{fb2appendix}, with \eqref{entexp}.
These formulas rigorously  imply the  spectrum of \Cref{fig:spectrum}. 
The eigenvalues in \eqref{lambda2.ae} are even in $ \alpha $ and $ \overline{\lambda_\pm^{(2)}(\alpha, \mu, \e)} = \lambda_\pm^{(2)}(\alpha, - \mu, \e) $.
\item {\sc Perturbed eigenvectors.} For any $ (\alpha,\mu)  \not \in \cM^{(\tp)}_+(\e)\cup \cM^{(\tp)}_-(\e)$  the eigenvalues $\lambda^{(\tp)}_\pm(\alpha,\mu,\epsilon) $ in  \eqref{eigenvalues} are simple, thus $ \tL^{(\tp)}(\al,\mu,\e) $ has two independent eigenvectors. If $(\al,\mu) $ belongs to one and only one perturbed McLean curve $ \cM_+^{(\tp)} (\e)$ or $ \cM_-^{(\tp)} (\e)$, the double eigenvalue 
$\lambda^{(\tp)}_+(\al,\mu,\e) = \lambda^{(\tp)}_-(\al,\mu,\e)$ is defective,  
namely $\tL^{(\tp)}(\al,\mu,\e)$ has a Jordan block. If $(\al,\mu) \in \cM^{(\tp)}_+(\e)\cap \cM^{(\tp)}_-(\e)$, the eigenvalue $\lambda^{(\tp)}_+(\al,\mu,\e) = \lambda^{(\tp)}_-(\al,\mu,\e)$ is  semi-simple and the matrix $\tL^{(\tp)}(\al,\mu,\e)$ is diagonal. This follows since  $\tL^{(\tp)}(\al,\mu,\e)$  is similar to 
\begin{equation}
    \tC \, \tL^{(\tp)}(\al,\mu,\e) 
    \, \tC^{-1} = \tfrac{\im}{2} (\fc^{(\tp)}-\fa^{(\tp)})(\al,\mu,\e) + \tfrac{1}{2}\begin{pmatrix}
        0 & \mathsf{d}^{(\tp)}_+(\al,\mu,\e)\\
         \mathsf{d}^{(\tp)}_-(\al,\mu,\e) & 0
    \end{pmatrix}
\end{equation}
where $\tC$ is the matrix in \eqref{Ccomplex}.
\item 
{\sc Size of the unstable regions and real part of  eigenvalues.} 
The  Hausdorff distance 
between the perturbed McLean curves 
$ \cM^{(2)}_\pm (\e) $ satisfy 
$$
\textup{d}_{\rm H}(\cM^{(2)}_+(\e),\, \cM^{(2)}_-(\e)) \sim 2\sqrt 2 |\e| \, , \qquad
\textup{d}_{\rm H}(\cM^{(3)}_+(\e), \, \cM^{(3)}_-(\e)) = \cO(\e^3) \, .
$$
The maximal, resp. minimal,  separation between $\cM^{(2)}_+(\e)$ and $\cM^{(2)}_-(\e)$ 
is attained along the $\mu$-axis near $ (0,0) $, 
corresponding to the Benjamin-Feir  Floquet band, resp.  near the points   
$(0,\pm \tfrac{5}{4})$ corresponding to the first longitudinal isola. 
The real part of the eigenvalues $ \Re\lambda_+^{(2)}(0,\pm \tfrac54,\e) = \cO( \e^4 ) $ and 
$ \Re\lambda_+^{(3)}(\al,\mu,\e) = 0  $
at the (finitely many) points of $\cM^{(3)}_+(\e) \cap \cM^{(3)}_-(\e)  $. 
\end{enumerate}

By  \eqref{traccianulla} we readily 
deduce a criterion  
for the emergence of an instability region. Let 
\be \label{tracciazero}
\cT^{(\tp)}(\e):= \{ (\al, \mu) \in K^{(\tp)} \colon T^{(\tp)}(\al, \mu, \e)  = 0 \ \}
\ee
where 
$ T^{(\tp)} (\al, \mu,\e) $
is the function defined in \eqref{tracciaBep}.  

\begin{proposition}\label{lem:inst} {\bf (Instability criterion)} 
For any $ \tp \geq 2 $,  for any 
$|\e|\leq \e^{(\tp)} $  the 
instability region 
$ \cU^{(\tp)}_\e $  in \eqref{def:instaintro} is not empty
if and only if the off-diagonal entry 
$  \fb^{(\tp)} (\cdot, \cdot,\e) $ in \eqref{tocomputematrixb}, 
restricted to the curve 
$ \cT^{(\tp)}(\e)$ defined in \eqref{tracciazero}, 
is not identically zero, namely  
\be\label{b.not.vanishing}
\cU^{(\tp)}_\e\neq \emptyset 
 \quad \mbox{ if and only if } 
 \quad   \fb^{(\tp)} (\cdot,\cdot,\e)\vert_{\cT^{(\tp)}(\e)} \not \equiv 0  \, . 
\ee
If there exists  $(\und \al,\und \mu) $ 
on the unperturbed McLean curve  $\cM^{(\tp)}$  in \eqref{mcleanmanifoldsp0} such that 
      \be\label{bp}
      \fb_\tp(\und\al,\und\mu) := \frac{1}{\tp!}\pa_\e^\tp \fb^{(\tp)}(\und \al,\und \mu,0) \neq 0 \ ,
      \ee
then \eqref{b.not.vanishing} occurs. 
\end{proposition}

Regarding the proof of \Cref{TeoremoneMcLean} 
we  emphasize that, unlike previous works, we apply  the Kato reduction approach  {\it globally} 
around  the McLean curves $ \cM^{(\tp) }$.
\begin{enumerate} 
\item {\sc Splitting for $\tp = 2 $.}  
By comparison with \Cref{TeoremoneFinale} we prove in \Cref{b00} that $\fb_{2}(0,0) = - \tfrac12  $, thus 
the analytic
function $\fb_{2}(\alpha,\mu) $ 
is not identically zero along 
$\cM^{(2)}$ and  $\cM^{(2)}_+(\e) \cap \cM^{(2)}_-(\e) $ consists at most of finitely many points.
 Next we show  
 that     
$  \fb_{2}(\alpha,\mu) < 0 
$ for any $ (\alpha,\mu)
 \in \cM^{(2)} \setminus \{(0, \tfrac54) \} $,  
yielding 
\eqref{intersectionmcleans}.
\item {\sc Splitting for $\tp = 3 $.} 
The matrix 
$\tL^{(3)}(\al,\mu,\e)$ in \eqref{tocomputematrixb} 
has been computed in \cite{CNS} at the special point 
$ (\al_3(0),0) $ which belongs to  the 
unperturbed McLean curve $ \cM^{(3)} $ and to the axis 
$ \{ \mu = 0  \}$, showing  that  
 $\fb_{3}(\al_3(0),0) \neq 0 $. Therefore 
the analytic function $\fb_{3}(\alpha,\mu) $  is not identically zero along the $ 1 $-dimensional McLean curve $\cM^{(3)}$ and  $\cM^{(3)}_+(\e) \cap \cM^{(3)}_-(\e) $ is discrete, 
as stated in \Cref{TeoremoneIntro}. 
\item {\sc Upper bounds for any $\tp \geq 4 $.} 
The upper bounds 
\eqref{upboundeigv}  
 follow by the  expansions \eqref{expa}-\eqref{expc}. 
\end{enumerate}

\subsection{Regularity of the fiber Dirichlet-Neumann operator}\label{sec:MDN}

A key analytical  step of the paper 
is to establish  
the regularity properties in 
\Cref{DNProp1} 
of the fiber Dirichlet-Neumann operator $\cG (\al,\mu,\e)  $ 
defined in \eqref{def:Galmu}
in the Floquet parameters $ (\al, \mu) $
and in $\e $. 
 \\[1mm]
{\bf Notation.}
We identify a linear
(possibly unbounded) 
operator $A $  acting on $ C^\infty (\T, \C)$ as the infinite matrix $ \{ A^{k_1}_{k_2} \}_{k_1,k_2\in \Z} $  with respect to the exponential basis,
$$
A^{k_1}_{k_2} := ( A e^{\im k_1 x} , \ e^{\im k_2 x } )_{L^2(\T)} \, , 
\quad (f,g)_{L^2(\T)} = \frac{1}{2\pi} \int_{\T} f(x) \overline{g(x)} \,   \de x \, . 
$$
Thus  
the action of $A$  
is  given  by  
$ h(x) = \sum_{k_1 \in \Z} h_{k_1} e^{\im k_1 x} \mapsto  
(A h)(x) = \sum_{k_2\in \Z} 
\Big( \sum_{k_1 \in \Z} A^{k_1}_{k_2} \, h_{k_1} \Big) \, e^{\im k_2 x} $. 
We do not concern about issues of convergence
as we will deal with 
finitely many matrix entries. 

Given $\kappa \in \Z $, we define its
``$\kappa$-band" operator 
$ A^{[\kappa]} \equiv \{A^{k_1}_{k_2} \}_{k_2 - k_1 = \kappa} $ 
with matrix coefficients  supported on the 
``band" $ k_2-k_1 = \kappa $. 
In other words the action of 
$A^{[\kappa]} $  is  to ``shift 
the exponential $ e^{\im j x}$  of 
$ \kappa $ harmonics", namely 
\begin{equation}\label{band.def}
A^{[\kappa]} (e^{\im j x})  =
A^{j}_{ j + \kappa} e^{\im (j + \kappa) x} \, . 
\end{equation}
If $A = \begin{bmatrix}
    A_1 & A_2 \\
    A_3 & A_4
\end{bmatrix}$ is a $2\times 2$ matrix of operators acting on
$ C^\infty (\T, \C^2)$  
we define its $\kappa$-{\it band} as the
operator 
\begin{equation}\label{Am4}
A^{[\kappa]}:= 
\begin{bmatrix}
  A^{[\kappa]}_1 & A^{[\kappa]}_2 \\
  A^{[\kappa]}_3 & A^{[\kappa]}_4
\end{bmatrix}.
\end{equation}
Given a family of linear 
operators 
$A=A(\al, \mu,\e)$, 
analytic in $ \e $,  we define, for any 
$ \ell \in \N_0 $, its jets 
$ A_\ell := \frac{1}{\ell!} \big(\pa^\ell_\e A \big)(\al, \mu) $. 
Following \cite{BMV5}, we introduce the space of operators
whose jets $ A_\ell $  have ``finite-range interactions" of order at most $ \ell $ and with the same parity of $ \ell $.

\begin{definition}{\bf (Spaces $\mathfrak{F}_\ell$ and  $\mathtt{F}$)}
\label{defFell} 
For any $\ell \in \N_0$ we define $\mathfrak{F}_\ell$ the space of operators, or  $2\times 2$ matrices of operators $ B $,  such that 
\begin{equation}\label{TFj} 
 B^{[\kappa]} =  
0 \quad   \textup{ if }|\kappa| > \ell \quad 
 \textup{ or} \quad \kappa \nequiv \ell \ (\textup{mod }2) \, . 
\end{equation}
We denote by 
$\mathtt{F}$  the space of operators, or
$2\times 2$ matrices of operators, with a formal power series

$$ 
A(\e) =     
\sum_{\ell \geq 0} A_{\ell}\, \e^\ell  \quad
\text{such that} \quad 
A_{\ell}   \ \in \mathfrak{F}_\ell \, , \ \forall \ell \in \N_0 \, . 
$$  
\end{definition}
 
The  fiber Dirichlet-Neumann operator  $\mathcal{G}(\alpha, \mu, \epsilon) $  in \eqref{def:Galmu}, 
is a first  order 
operator, analytic in \( (\alpha, \mu) \) except at points in \( \{0\} \times \mathbb{Z} \), where it exhibits algebraic singularities.

\begin{theorem}\label{DNProp1} {\bf (Fiber Dirichlet-Neumann operator)}
   For any $s\in\R$,
 there exists $\e_0 :=  \e_0 (s)>0$ such that for any $(\al,\mu,\e)\in \R\times \R \times B_{\e_0}(0)$
    the fiber Dirichlet-Neumann operator $\cG({\al,\mu, \e}) $ 
    defined in \eqref{def:Galmu} 
    maps $ H^s (\T) \to H^{s-1} (\T) $, 
 and decomposes as 
\begin{equation}\label{decoGsha}
   \cG({\alpha, \mu, \e})    = |D|_{\al,\mu}  
   + \cG^\sharp({\alpha,\mu, \e}) 
\end{equation}
    where  
\\[1mm]
$ \bullet $    
$       |D|_{\al,\mu} := ((D+\mu)^2 + \alpha^2)^{\frac12} \colon H^s(\T)   \to H^{s-1}(\T) $; 
 \\[1mm]
$ \bullet $  the operator     $\cG^\sharp({\al,\mu,\e}): H^s (\T) \to H^{s+1} (\T)  $ is 1-smoothing
and satisfies, 
    for any $ \mu_0 > 0 $, 
\begin{equation}\label{est:DNpert}
    \sup_{|\mu|\leq\mu_0}\|\cG^{\sharp}(\al, \mu, \e)\|_{\cL(H^s,H^{s+1})}\leq C_{s,\mu_0} \al^2|\e|\, , \qquad 
    \sup_{|\mu|\leq\mu_0}\|\cG^{\sharp}(\al, \mu, \e)\|_{\cL(H^s,H^{s})}\leq C_{s,\mu_0}  |\al||\e|\, .  
        \end{equation}
$ \bullet $   the operators  
\be
\begin{aligned}
\label{regGs}
 (\alpha^2 \e)^{-1}\cG^\sharp({\al,\mu,\e})\in \cA(\R\times \R,\e_0;H^s(\T),H^{s+1}(\T))\cap \tF\,   , &  \\ \cG(\al,\mu,\e)\in \cA(\R\times \R,\e_0;H^s(\T),H^{s-1}(\T))\cap\tF \, , & 
    \end{aligned}
\ee
according to \Cref{def:tM,defFell}. 
\end{theorem}

 \Cref{DNProp1} is proved  in   \Cref{ap_analiticity}. 
Let us make some comments.

\begin{enumerate}
\item 
{\sc Special cases.}
For $ \alpha = 0 $ (purely longitudinal perturbations) the fiber Dirichlet-Neumann operator 
$
   \cG({0, \mu, \e})    = |D + \mu | $ reduces to a Fourier multiplier. For $ \mu = 0 $
(purely transverse case) the 
operator $
   \cG({\alpha, 0, \e})  $ is real
   by \eqref{cGcLmenomu}. The analyticity  
of  $ \cG^\sharp({\al,0,\e}) $ with respect to $ (\alpha, \epsilon) $  for 
$ \al   $  far   from zero was  proven in \cite[Proposition 4.2]{CNS}.
\item 
{\sc Regularity of $|D|_{\al,\mu} $.}
 The unperturbed  operator $|D|_{\al,\mu} $ 
 belongs to $\cA(\R^2,+ \infty ;H^s(\T),H^{s-1}(\T))$   
for any $s\in\R$.  Indeed,   denoting $ \Pi_j $, for any $ j \in \Z $,   the projector on $e^{\im j x}$ and $ \Pi_{j}^\perp := \uno - \Pi_{j} $, we have 
\begin{equation}\label{dec:Gflat}
        |D|_{\al,\mu}  =
          \underbrace{|D|_{\al,\mu} \Pi_{-j}^\perp }_{=:|D|^{[\mathrm{I}]}(\al^2,\mu)} + (\al^2 + (\mu-j)^2)^{\frac12}\underbrace{\Pi_{-j}}_{=:|D|^{[\mathrm{II}]}}
    \end{equation}
    and $|D|_{\al,\mu} \Pi_{-j}^\perp$ is analytic in $ (\alpha^2,\mu) $ in a neighborhood  $ (0,j)$.  
\item
{\sc Jets.}
The first two jets of the  fiber Dirichlet-Neumann operator   $\cG^\sharp(\al,\mu,\e) = \sum_{\ell\geq 1} \cG_\ell(\al,\mu)\e^\ell$  are 
\begin{equation}\label{cG12action}
  \begin{aligned}
            &\cG_1(\al,\mu)[e^{\im n x}] = c_{1,n}^+(\al,\mu)e^{\im(n+1)x} + c_{1,n}^-(\al,\mu)e^{\im(n-1)x}, \\
            &\cG_2(\al,\mu)[e^{\im n x}] = c_{2,n}^+(\al,\mu)e^{\im(n+2)x} + c_{2,n}^0 (\al,\mu)e^{\im n x} + c_{2,n}^-(\al,\mu)e^{\im(n-2)x}, 
            \end{aligned}
   \end{equation}  
   with coefficients 
     \begin{equation}  \label{actioncG12}     \begin{aligned}
            &c_{1,n}^\pm (\al,\mu)= \al^2 \left( |n\pm 1|_{\al,\mu} + |n|_{\al,\mu} + 1 \right)^{-1} \\
            &c_{2,n}^\pm (\al,\mu)= \al^2 \frac{\left(2- \al^2\left((|n\pm 1|_{\al,\mu} + |n|_{\al,\mu} + 1)(|n\pm 2|_{\al,\mu} + |n\pm 1|_{\al,\mu} + 1 )\right)^{-1} \right)}{|n\pm 2|_{\al,\mu} + |n|_{\al,\mu}+2} \\
            &c_{2,n}^0(\al,\mu) = \al^2 \frac{\left(1-\al^2\left(\frac{1}{(|n+1|_{\al,\mu}+|n|_{\al,\mu}+1)^2}+\frac{1}{(|n-1|_{\al,\mu}+|n|_{\al,\mu}+1)^2}\right)\right)}{2(|n|_{\al,\mu}+1)} \, . 
        \end{aligned}
    \end{equation}
   These formulas were computed in 
  \cite[Proposition 4.2]{CNS} (in the case $ \mu = 0 $). They can be derived also by Taylor expanding at quadratic order \eqref{formulathetasharppp1} and using the Taylor expansion  \eqref{deintF}. 
    They  
    agree with the property 
that $\cG(\al,\mu,\e)$ belongs to $\cA \cap \tF $ according to  \Cref{def:tM},
as  $|k|_{\al,\mu} = ((k+\mu)^2+\al^2)^\frac12$ have the form \eqref{dectM}. Note also that  $\cG_1 \in \mathfrak{F}_1$  has only bands $\pm 1$, and $\cG_2\in \mathfrak{F}_2$ has only bands $0, \pm 2$.
\item {\sc Regularity of $ \cL(\al,\mu,\e) $. } 
  The operators 
$ \cB(\al,\mu,\e) $ and $ \cL(\al,\mu,\e) $ in \eqref{operator Bmualep}, \eqref{cLame0} satisfy 
    \begin{align}
\label{regGBL}\cB(\al,\mu,\e),\, \cL(\al,\mu,\e) \in \cA(\R\times \R,\e_0;&H^s(\T;\C^2),H^{s-1}(\T;\C^2))\cap \tF \, ,
    \end{align}
as readily follows by
\eqref{regGs} and the properties 
\eqref{apexp}-\eqref{pafparity}
of the functions $ a_\epsilon (x), p_\epsilon (x) $.     
\end{enumerate}

\noindent 
{\bf Notation:}
We  write   $a\lesssim b$, resp. 
$a \gtrsim b$,  to mean  that there is a constant $C>0$ such that $a \leq C b$, resp. $ a \geq C b$,  for any $ a, b \geq 0 $. 

 \part{Benjamin-Feir 
 instability}\label{part:I}

In this part we fully describe the transverse and longitudinal 
Benjamin-Feir instability of the four spectral bands 
for Stokes waves in deep water, for 
$ (\alpha, \mu) $ near $(0,0) $, proving Theorem \ref{TeoremoneFinale}.

\section{Perturbative approach to  Benjamin-Feir eigenvalues}\label{sec:katoBF}

First we decompose the operator
$ \cL(\al,\mu,\e)  $ in \eqref{cLame1}, \eqref{cLame0} 
as 
\begin{equation}\label{cLsL}    \cL(\al,\mu,\e) = \im\mu + \sL(\al,\mu,\e)
\end{equation}
where    $\sL(\al,\mu,\e)$ is the Hamiltonian and reversible operator 
\begin{equation}\label{sLresca}
\sL(\al,\mu,\e) 
 = 
\cJ  \underbrace{\begin{bmatrix}
        1+a_\e(x) & - (1 + p_\e(x))\pa_x  - \im \mu p_\e(x)\\
         \pa_x\circ (1 + p_\e(x))  + \im \mu p_\e(x) & 
         \cG(\al,\mu,\e) 
    \end{bmatrix}}_{=: \sB (\alpha, \mu, \epsilon)} \, . 
\end{equation}
The  operators $\sL(\al,\mu,\e)$
and $\sB (\al,\mu,\e)$
satisfy 
the same properties of 
$\cL(\al,\mu,\e)$
and $\cB (\al,\mu,\e)$ in \Cref{DNonR3} and Theorem \ref{DNProp1}. 
We regard $\sL (\al,\mu,\e) $ as an operator with domain $ H^1:= H^1(\T,\C^2)$ and range $ L^2:=L^2(\T,\C^2)$. 

\begin{lemma}\label{lem:Kato1}
{\bf (Kato theory for separated eigenvalues near $0$)}
 Let $\Gamma$ be a closed, counterclockwise-oriented circle in the complex plane, centered in $0$, separating $\sigma'\left(\sL(0,0,0)\right)=\{0\}$
  and $\sigma''\left(\sL(0,0,0)\right) = \sigma''\left(\cL(0,0,0)\right)$ in \eqref{spettrodiviso0}.
Then there exist $\rho_0, \e_0 >0$  such that for any $(\al , \mu, \e) \in B_{\rho_0}(0,0) \times B_{\e_0}(0)$   the following  holds:
\\[1mm] 
1. The curve $\Gamma$ belongs to the resolvent set of 
the operator $\sL(\al,\mu,\e) : H^1 \subset L^2 \to L^2 $ and the operators
\begin{equation}\label{Pproj}
 P_{\al,\mu,\e} := \frac{1}{2\pi\im}\oint_\Gamma (\lambda-\sL(\al,\mu,\e) )^{-1} \di \lambda : L^2 \to H^1 
\end{equation}  
are well defined projectors commuting  with $\sL(\al,\mu,\e) $,  i.e. 
$ P_{\al,\mu,\e}^2 = P_{\al,\mu,\e} $, $ P_{\al,\mu,\e}\sL(\al,\mu,\e) = 
\sL (\al,\mu,\e)  P_{\al,\mu,\e} $.   In addition $P_{\al,\mu,\e}$ is \textit{skew-Hamiltonian}
 and reversibility preserving, i.e. 
   \be\label{Psym}
    \cJ P_{\al,\mu,\e}= P_{\al,\mu,\e}^* \cJ \, , \qquad 
    \varrho_c P_{\al,\mu,\e} = P_{\al,\mu,\e}  \varrho_c \, , 
    \ee
    where $ \varrho_c $ is defined in \eqref{def:cinvolution}. 
The map 
\be\label{Pamuan}
(\al , \mu,  \epsilon)\mapsto P_{\al , \mu, \epsilon}\ \quad \text{
belongs to} \quad  
\cA(B_{\rho_0}(0,0),\e_0;L^2,H^1) \cap \tF 
\ee
    according to \Cref{def:tildeM,defFell}.
\\[1mm] 
2.
The domain $H^1$  of the operator $\sL(\al,\mu,\e) $ decomposes as  the direct sum
\begin{equation}\label{projdec}
    H^1 = \mathcal{V}_{\al,\mu,\e} \oplus \text{Ker}(P_{\al,\mu,\e}\vert_{H^1}) \, , \quad \mathcal{V}_{\al,\mu,\e}:=\text{Rg}(P_{\al,\mu,\e})=\text{Ker}(\uno-P_{\al,\mu,\e}) \, ,
\end{equation}
of   closed invariant  subspaces, namely 
$ \sL (\al,\mu,\e) : \mathcal{V}_{\al,\mu,\e} \to \mathcal{V}_{\al,\mu,\e} $ and $
\sL(\al,\mu,\e) : \text{Ker}(P_{\al,\mu,\e}\vert_{H^1}) \to \text{Ker}(P_{\al,\mu,\e}) $ and 
$$
\begin{aligned}
&\sigma(\sL(\al,\mu,\e))\cap \{ z \in \C \mbox{ inside } \Gamma \} = \sigma(\sL(\al,\mu,\e) \vert_{{\mathcal V}_{\al,\mu,\e}} )  = \sigma'(\sL(\al,\mu,\e)) \, .  
\end{aligned}
$$
3.  The projectors $P_{\al,\mu,\e}$ 
are similar one to each other: the  transformation operators
\begin{equation} \label{OperatorU} 
U_{\al,\mu,\e}   := 
\big( \uno-(P_{\al,\mu,\e}-P_{0,0,0})^2 \big)^{-1/2} \big[ 
P_{\al,\mu,\e}P_{0,0,0} + (\uno - P_{\al,\mu,\e})(\uno-P_{0,0,0}) \big] 
\end{equation}
are bounded and  invertible in $ H^1 $ and in $ L^2 $, with inverse
$$
U_{\al,\mu,\e}^{-1}  = 
 \big[ 
P_{0,0,0} P_{\al,\mu,\e}+(\uno-P_{0,0,0}) (\uno - P_{\al,\mu,\e}) \big] \big( \uno-(P_{\al,\mu,\e}-P_{0,0,0})^2 \big)^{-1/2} \, , 
$$
 and 
$ U_{\al,\mu,\e} P_{0,0,0}U_{\al,\mu,\e}^{-1} =  P_{\al,\mu,\e}  $ and  
$  U_{\al,\mu,\e}^{-1} P_{\al,\mu,\e}  U_{\al,\mu,\e} = P_{0,0,0} $. In addition 
 $U_{\al,\mu,\e}$ are symplectic
    and reversibility preserving, i.e. 
\begin{equation}\label{Usymp}
    U_{\al,\mu,\e}^* \cJ U_{\al,\mu,\e} = \cJ  \, 
    , \quad \varrho_c U_{\al,\mu,\e} = U_{\al,\mu,\e}  \varrho_c \, , 
    \end{equation}
    where $ \varrho_c $ is defined in \eqref{def:cinvolution}.  
The map 
\begin{equation}\label{Ureg}
    (\al , \mu,  \epsilon)\mapsto U_{\al , \mu, \epsilon} \quad \text{belongs to} \quad \cA(B_{\rho_0}(0,0),\e_0;Z,Z) \cap \tF\, , \qquad Z \text{ either } H^1 \text{ or } L^2\, .
\end{equation} 
4. The $ 4 $ dimensional subspaces 
$$
\mathcal{V}_{\al,\mu,\e}=\text{Rg}(P_{\al,\mu,\e}) = U_{\al,\mu,\e}\mathcal{V}_{0,0,0} \, , \quad 
\forall (\al , \mu, \e) \in B_{\rho_0}(0,0)\times B(\e_0) \, , 
$$ 
which are all isomorphic to each other,  
are  \emph{symplectic} according to \Cref{def:sympsub}.
\\[1mm]
5. The operators $P_{\al,\mu,\e} $ and $ U_{\al,\mu,\e}$ satisfy  the symmetry properties  \begin{equation}\label{symmmmUP}
\overline{P_{\al,\mu,\e}} =
P_{\al,-\mu,\e} \, , \qquad   \overline{U_{\al,\mu,\e}} 
= U_{\al,-\mu,\e} \, .\end{equation}
 In particular 
 $ P_{\al,0,\e}$ and
$  U_{\al,0,\e}$ are real
operators.
The operators  $ P_{\al,\mu,0}$ and
$  U_{\al,\mu,0}$ are Fourier 
multipliers. 
\end{lemma}

\begin{proof}
    The lemma is analogous to \cite[Lemmas 3.1
  3.2]{BMV1}, 
    proved at $\alpha = 0 $ for 
    $ \cG (0,\mu,0) = |D+\mu|$. The novelties are \eqref{Pamuan}
    and \eqref{Ureg}, which directly follow by \eqref{regGBL} and \Cref{lem:prodintM,TFjprop}, and \eqref{symmmmUP}. This follows because
\begin{align}
    \overline{P_{\al, \mu, \e}}  &= \oint_{\Gamma} \overline{(\lambda-\sL(\al,\mu,\e))^{-1}} \frac{\overline{\de\lambda}}{-2\pi \im} = -\oint_{\overline{\Gamma}} (\bar\lambda-\overline{\sL(\al,\mu,\e)})^{-1} \frac{\de\bar\lambda}{2\pi \im} \notag \\
    & \stackrel{\eqref{cGcLmenomu}}{=}  \oint_{\Gamma} (\bar\lambda-\sL(\al,-\mu,\e))^{-1} \frac{\de\bar\lambda}{2\pi \im} 
=
    \oint_{\Gamma} (\bar\lambda-\sL(\al,-\mu,\e))^{-1} \frac{\de\bar\lambda}{2\pi \im} = P_{\al, -\mu, \e} \, , \label{primsim}
\end{align}
where $\overline{\Gamma}$ is clockwise oriented.  The identity in \eqref{symmmmUP} for $U_{\al,\mu,\e}$ follows taking the complex conjugate of \eqref{OperatorU} and using \eqref{primsim}.
The subspaces 
$\mathcal{V}_{\al,\mu,\e} $ are symplectic  because $ P_{\al,\mu,\e} $ is skew-Hamiltonian and \Cref{existence:Pskewham}. 
\end{proof}

There is a  one-to-one correspondence between skew-Hamiltonian projectors and closed symplectic subspaces.

\begin{lemma} {\bf (Skew-Hamiltonian  projectors)}
\label{existence:Pskewham}
    Let $\cV\subset L^2$ be a closed symplectic subspace, equipped with the symplectic form $\cW_c $ in \eqref{ses}. Then there exists a unique skew-Hamiltonian projector $P$ on $\cV$, cf. \eqref{Psym}.
    Viceversa,  the range of any skew-Hamiltonian projector  is a closed symplectic subspace.
\end{lemma}

\begin{proof}
   \underline{Existence:} let  $\cV$ be a  closed symplectic subspace.
   Then  $L^2=\cV \oplus \cV^{\perp\cW_c}$ and we define $P$ as the unique projector  satisfying 
    \begin{equation}\label{skewhamproj}
        \ker P = \cV^{\perp\cW_c}\, , \quad \textup{Rg}P = \cV\, .
    \end{equation}
    We now prove that $P $  is skew-Hamiltonian. Consider the self-adjoint operator 
$ P^*\cJ - \cJ P$
    and a vector
    $
    v = \underline{v} + v^\perp$, $ \underline{v} \in \cV$, $ v^\perp
    \in \cV^{\perp\cW_c}$.
    Then 
    $$
    \left((P^*\cJ - \cJ P)v,v\right) = \left(\cJ v , \underline{v} \right) - \left(\cJ \underline{v}, v \right) = \cW_c(v,\underline{v} ) - \cW_c (\underline{v} , v)
    $$
    but, by definition of $\cV^{\perp\cW_c}$, $\cW_c (\underline{v} , v^\perp ) = 0$, and therefore 
    $$
    \left((P^*\cJ - \cJ P)v,v\right) =
    \cW_c(\underline{v},\underline{v} ) - \cW_c (\underline{v}, \underline{v}) = 0 \quad \forall v \in X \Longrightarrow P^*\cJ - \cJ P = 0
    $$
    so that $P$ is skew-Hamiltonian.

    \noindent\underline{Uniqueness:} now let us assume that $P$ is skew-Hamiltonian.
    Assume that 
    $v\in \ker P$, i.e., using the   invertibility of $\cJ$,  $(Pv, \cJ g)= 0$ for any $g\in L^2$. Since $P$ is skew-Hamiltonian 
    $$
   0 =  (Pv,\cJ g) = (v,P^*\cJ g) = (v,\cJ P g) = - \cW_c (v, Pg) \ 
   \ \ \forall g \in L^2 
   \quad 
   \Rightarrow \quad 
   \ker P \perp_{\cW_c} \textup{Rg}P=: \cV  \, .
    $$
    On the contrary assume $v\in \cV^{\perp\cW_c}$. Then for every $g\in L^2$
    $$
    0 = \cW_c(v, Pg) =  -(v,\cJ Pg) \stackrel{\eqref{Psym}}{=} - (Pv,\cJ g) = \cW(Pv, g) \Rightarrow v \in \ker P
    $$
    so that $\ker P = \cV^{\perp\cW_c}$.
    Thus, the skew-Hamiltonian projector on $\cV$ satisfies \eqref{skewhamproj} and therefore is uniquely determined.

    In particular, if $P$ is skew-Hamiltonian then $L^2 = \ker P \oplus \textup{Rg} P = \cV \oplus \cV^{\perp\cW_c}$, so that $\cV$ is a closed symplectic subspace.
\end{proof}

{\noindent\bf Symplectic and reversible basis of $\cV_{\al,\mu,\e}$.} 
We now choose convenient basis in $ \cV_{\alpha, \mu, \epsilon}$ to represent linear operators. 
The  symplectic and reversible basis $\{f_1^\pm,f_0^\pm\}$ of $\cV_{0,0,0}$ defined in  
\eqref{base3e}  
 is mapped 
by $ U_{\alpha, \mu, \epsilon}$,
in view of \eqref{Usymp},  into the  symplectic and reversible  basis of 
$\cV_{\al,\mu,\e}$, 
\begin{equation}\label{basisF}
    \cF := \Big\{ f_k^\sigma (\al,\mu,\e) := U_{\al,\mu,\e}f_k^\sigma\ :  \ k=0,1; \ \sigma = \pm \Big\} \, .
\end{equation}
Any  vector $f_k^\sigma$
    of a reversible basis, i.e.  satisfying \eqref{def:rev},  has the form 
\be\label{lem:revbasis}
    f_k^+ = \vet{even(x)}{odd(x)} + \im \vet{odd(x)}{even (x)}\ , \quad f_k^- = \vet{odd(x)}{even(x)} + \im \vet{even(x)}{odd(x)} \ 
\ee
where  $even (x)$ and $odd (x)$ denote respectively even and odd $2\pi$ periodic, real valued function of $x$. 


\begin{lemma}\label{lem:simplbas}
     Let $  \{ \tf_k^\sigma\}_{k=1,\dots,n}^{\sigma=\pm}$ be a symplectic basis (\Cref{def:symprevbasis,}) of the symplectic subspace $\cV \subset L^2 $. Then 
     any $f\in \cV $ is 
decomposed as  
\begin{equation}\label{coefcks0}
 f = \sum_{k=1}^n- {\cal W}_c (f, \tf_k^-) \tf_k^+ + {\cal W}_c (f, \tf_k^ +) \tf_k^- \, . 
\end{equation}
The skew-Hamiltonian projection  $ P : L^2 \to \cV$ defined  in Lemma \ref{existence:Pskewham} can be written as
\be\label{lem:simplbas2}
P f = \sum_{k=1}^n- {\cal W}_c (f, \tf_k^-) \tf_k^+ + {\cal W}_c (f, \tf_k^ +) \tf_k^- \, , \quad
\forall f\in L^2 \, . 
\ee
\end{lemma}

\begin{proof}
    Decompose $f = \sum_{k=1}^n c_k^+ \tf_k^+ +  c_k^- \tf_k^-  $, and take the scalar products $(\cJ f, \tf_k^\sigma)$. Then, by \eqref{def:symp} we deduce \eqref{coefcks0}.
      Then \eqref{lem:simplbas2}  follows by \eqref{coefcks0} and   ${\cal W}_c (Pf, \tf_k^\sigma) = {\cal W}_c (f, \tf_k^\sigma)$ as $ P $ is skew-Hamiltonian.
\end{proof}

We now represent the action of 
${\sL}(\al,\mu,\e)_{|\mathcal{V}_{\al,\mu,\e}} $. 

\begin{lemma}\label{lem:B.mat}
The $ 4 \times 4 $ matrix that represents the Hamiltonian and reversible  operator 
$\sL(\al,\mu,\e)= \cJ \sB(\al,\mu,\e) :\mathcal{V}_{\al,\mu,\e}\to\mathcal{V}_{\al,\mu,\e} $ with respect to a  symplectic and reversible basis $ \tF := \{\tf_1^+,\tf_1^-,\tf_0^+,\tf_0^-\} $ of $\mathcal{V}_{\al,\mu,\e}$ is 
\begin{align}\label{Lform}
 \tJ_4 \tB_{\al,\mu,\e} \, ,\quad  \tJ_4 = 
 \begin{pmatrix} 
 \tJ_2& \vline & 0 \\
 \hline
0  & \vline & \tJ_2
\end{pmatrix}, \quad \tJ_2 = \begin{pmatrix} 
 0 & 1 \\
-1  & 0
\end{pmatrix}, \quad \text{where } \quad \tB_{\al,\mu,\e}= \tB_{\al,\mu,\e}^* \end{align}
is the self-adjoint matrix
\begin{equation}\label{matrix22} 
\tB_{\al,\mu,\e} = 
\begin{pmatrix}
\BVe{+}{1}{+}{1} & \BVe{-}{1}{+}{1} & \BVe{+}{0}{+}{1} & \BVe{-}{0}{+}{1} \\
\BVe{+}{1}{-}{1} & \BVe{-}{1}{-}{1} &  \BVe{+}{0}{-}{1} & \BVe{-}{0}{-}{1} \\
\BVe{+}{1}{+}{0} & \BVe{-}{1}{+}{0} & \BVe{+}{0}{+}{0} & \BVe{-}{0}{+}{0} \\
\BVe{+}{1}{-}{0} & \BVe{-}{1}{-}{0} &  \BVe{+}{0}{-}{0} & \BVe{-}{0}{-}{0} \\
	\end{pmatrix}.
\end{equation}
  The entries of $\tB_{\al,\mu,\e}$ are alternatively real or purely imaginary: for any $ \sigma = \pm $, $ k = 0, 1 $, 
\begin{equation}\label{revprop0}
       \molt{\sB(\al,\mu,\e)  \, \tf^{\sigma}_{k}}{\tf^{\sigma}_{k'}}\in\R\, , \quad  \molt{\sB(\al,\mu,\e)  \, \tf^{\sigma}_{k}}{\tf^{-\sigma}_{k'}}\in\im\R \, .
    \end{equation}
\end{lemma}

\begin{proof} The structure \eqref{Lform}--\eqref{matrix22} and  \eqref{revprop0} follow by 
  \Cref{lem:simplbas}, since 
  $\sB(\al,\mu,\e)$ is  selfadjoint and reversibility-preserving  (cf. \eqref{cLr:rev}) and  the   basis $ \tF $ is reversible (cf. \eqref{def:rev}),  
  cf.  \cite[Lemma 3.8]{BMV1}.
\end{proof}

In view of the previous lemma we introduce the following definition. 

\begin{definition}\label{def:revmat}
{\bf (Hamiltonian and reversible matrices)}  
    A $4\times 4$ matrix of the form $\tJ_4 \tB $ is Hamiltonian, if $\tB$ is self-adjoint;    {\it reversible} if $\tB$ is reversibility preserving, i.e. 
     its entries satisfy $[\tB]_{k,k'}^{\sigma,\sigma}\in \R$ and  $[\tB]_{k,k'}^{\sigma,-\sigma}\in \im \R$,  for any $ \sigma = \pm $, $ k = 0, 1 $. 
\end{definition}

The transformations preserving the Hamiltonian structure are called \textit{symplectic}. They satisfy
$    Y\tJ_4 Y^* = \tJ_4 $ 
so that,  for any 
Hamiltonian matrix $\tL = \tJ_4\tB $, the conjugated matrix 
\be\label{B1conj}
\tL_1 := Y\tL Y^{-1} = Y \tJ_4 Y^* Y^{-*} \tB Y^{-1} = \tJ_4 \tB_1\, , \qquad \tB_1:= Y^{-*} \tB Y^{-1} \, , 
\ee
is still Hamiltonian. Moreover, if $Y$ is reversibility preserving, $\tL_1$ is reversible if and only if $\tL$ is. 
In section \ref{sec:block} we will use that the flow of an Hamiltonian reversibility preserving matrix is symplectic and reversibility preserving.

\section {\texorpdfstring{Matrix representation of \( {\sL}(\al,\mu,\e)\) on \( \mathcal{V}_{\alpha,\mu, \e}\)  }{matrix representation of L on V}}\label{sec:mr}

The main result of this section is 
Proposition \ref{BexpG} which provides the
expansion of the  
matrix  representing  the action of  \( {\sL}(\al,\mu,\e)\) on the subspace  \( \mathcal{V}_{\alpha,\mu, \e}\) with respect to the symplectic and reversible  basis  $ \cG $ in \eqref{basisG}.

\smallskip

The operator $\sL(\al,\mu,\e)$
in \eqref{cLsL} 
satisfies, 
in addition to the regularity properties of \Cref{def:tM}, 
the symmetry 
$ \overline{\sL(\al,\mu,\e)} =
\sL(\al,-\mu,\e) 
$.
Such  operators have the following 
structure, as we  prove  in \Cref{ap_cAF}.

\begin{lemma}\label{lem:decomposition}Assume that 
$ A\in \cA(B_r(0,0),\e_0;Z) $ for some 
$r,\e_0>0$ small, 
and a Banach space $ Z $,  satisfying  
\be\label{simmeA}
 \overline{A(\al,\mu,\e)} =
 A(\al,-\mu,\e) \, . 
\ee 
Then $ A $ admits the expansion   
\begin{equation}
    \label{decompositiontM}
             \begin{aligned}
                 A(\al,\mu,\e) &= A^{[0,0]}(\al^2,\mu^2,\e) + \im\mu A^{[0,1]}(\al^2,\mu^2,\e) 
                 + \rho  A^{[1,0]}(\al^2,\mu^2,\e) + \im\mu \rho A^{[1,1]}(\al^2,\mu^2,\e)
             \end{aligned}
         \end{equation}
         where $ \rho = 
    (\al^2+\mu^2)^{\frac12} $ and 
         $A^{[i,j]}:
         B_{r^2}(0,0)\times B_{\e_0}(0) \to Z $ are the real-to-real analytic functions, 
          uniquely determined by $A$, \begin{equation}\label{A.decomp}
             \begin{aligned}
                 &A^{[0,0]} (\al^2,\mu^2,\e) : = \frac12 (A^{[\mathrm{I}]}(\al^2,\mu,\e) + A^{[\mathrm{I}]}(\al^2,-\mu,\e))\, , \\ &A^{[0,1]} (\al^2,\mu^2,\e) : = \frac1{2\im\mu} (A^{[\mathrm{I}]}(\al^2,\mu,\e) - A^{[\mathrm{I}]}(\al^2,-\mu,\e))\, , \\
                 &A^{[1,0]} (\al^2,\mu^2,\e) : = \frac12 (A^{[\mathrm{II}]}(\al^2,\mu,\e) + A^{[\mathrm{II}]}(\al^2,-\mu,\e))\, , \\ &A^{[1,1]} (\al^2,\mu^2,\e) : = \frac1{2\im\mu} (A^{[\mathrm{II}]}(\al^2,\mu,\e) - A^{[\mathrm{II}]}(\al^2,-\mu,\e)) \, . 
             \end{aligned}
         \end{equation}
         In particular $A $ is polar-analytic according to \Cref{def:tildeM}.
\end{lemma}

\noindent 
{\bf Notation for remainders.} We  denote $\cO(\rho^n\e^m)$  a polar-analytic function of the form \eqref{remainpoa} and,  for brevity,
in this section 
we adhere to the following conventions. Remainders of $\sL(\al,\mu,\e)$ are estimated in $\cL(H^1,L^2)$,
remainders of $P_{\al,\mu,\e}$ in  $\cL(L^2,H^1)$, and remainders of $U_{\al,\mu,\e}$  in the maximal norm between $ \| \ \|_{\cL(H^1,H^1)}$ and $\| \ \|_{\cL(L^2,L^2)} $. Vectors are estimated in $H^1$ and $ n\times n$-scalar matrices in $\C^{n\times n}$.

\begin{lemma}
{\bf (Expansion of  
$\sL(\al,\mu,\e)$)} 
    The  operator $\sL(\al,\mu,\e)$ in \eqref{sLresca}
has regularity
$$
[(\al,\mu,\e)
   \mapsto
\sL(\al,\mu,\e) \in \cA(\R\times \R,\e_0;H^1,L^2)] \, , \quad \quad
\overline{\sL(\al,\mu,\e)} = 
\sL(\al,-\mu,\e) \, , 
 $$
and  the expansion
\begin{equation}\label{quattropuntosei}
    \begin{aligned}
        \sL(\al,\mu,\e)=\underbrace{\begin{bmatrix}
        \pa_x \circ (1+p_\e(x)) & |D| \\ -(1+a_\e(x)) & (1+p_\e(x))\circ \pa_x
    \end{bmatrix}}_{= \sL^{[0,0]}(0,0,\e)} &+ \rho \underbrace{\begin{bmatrix}
        0 & \Pi_0 \\ 0 & 0
    \end{bmatrix}}_{=\sL^{[1,0]}(0,0,\e)} + \im \mu \underbrace{\begin{bmatrix}
        p_\e(x) & -\im \sgn (D)\\
        0 & p_\e(x)
    \end{bmatrix}}_{=\sL^{[0,1]}(0,0,\e)} \\ &+ \al^2 \underbrace{\begin{bmatrix}
        0 & \frac12 |D|^{-1}\Pi_0^\perp \\ 0 & 0
    \end{bmatrix}}_{=:\sL^{[2,0]}}+ \cO(\rho^3,\rho^2\e)\, .
    \end{aligned}
\end{equation}
\end{lemma}

\begin{proof}
We first expand  the operator 
  $\cG(\al,\mu,\e) = |D|_{\alpha,\mu}
  + \cG^\sharp (\al,\mu,\e) $ in \Cref{DNProp1}.
In view of 
\eqref{regGs} and 
\eqref{cGcLmenomu}, \Cref{lem:decomposition} implies that the operator  
$\cG^\sharp (\al,\mu,\e)$ admits the decomposition \eqref{decompositiontM}
with 
\begin{equation}\label{decAforGsharp}
    [\cG^\sharp]^{[i,j]}(\al^2,\mu^2,\e) =  \al^2\e \wt\cG^{[i,j]}(\al^2,\mu^2,\e) = \cO(\rho^2\e) \quad  
    i,j\in\{0,1\}\, .
\end{equation}
In addition, denoting $ \Pi_0 $ the $ L^2 $ projector on the zero Fourier mode, and using the Taylor expansion 
$ \sqrt{1+z} = 1 + \tfrac{z}{2} - \tfrac{z^2}{8} + \cO(z^3)$, 
the Fourier multiplier $|D|_{\alpha,\mu} $ in \eqref{cLvero0} has the expansion
\begin{align}
|D|_{\alpha,\mu}
&  = \sqrt{\al^2+\mu^2} \Pi_0 
+ |D|_{\alpha,\mu} \Pi_0^\bot  = 
\sqrt{\alpha^2+\mu^2} \Pi_0 
+ \big( |D| + \mu (\mbox{sgn} (D) + \cO(\rho^2)) + 
\frac{\alpha^2}{2} |D|^{-1} + \cO(\rho^4) \big) 
\Pi_0^\bot \notag  \\
& =  \underbrace{|D| + \frac{\al^2}{2}|D|^{-1}\Pi_0^\perp+\cO(\rho^4)}_{= |D|_{\al,\mu}^{[0,0]}} + \im \mu \underbrace{(-\im \, \sgn (D))+ \cO(\rho^2))}_{= |D|_{\al,\mu}^{[0,1]}} + \sqrt{\al^2+\mu^2} \underbrace{\Pi_0}_{= |D|_{\al,\mu}^{[1,0]}}  \label{decAforGflat} \, . 
\end{align}
The expansions \eqref{decAforGsharp},  \eqref{decAforGflat}  imply that
$$
\cG(\al,\mu,\e)  = \cG^{[0,0]}(\al^2,\mu^2,\e)   + \rho \, \cG^{[1,0]}(\al^2,\mu^2,\e) + \im \mu \, \cG^{[0,1]}(\al^2,\mu^2,\e) + \im \mu \rho\, \cG^{[1,1]}(\al^2,\mu^2,\e) 
$$
with 
\be
\begin{aligned}\label{F5}
\cG^{[0,0]}(\alpha^2, \mu^2, \e) & =  |D| +\frac{\al^2}{2} |D|^{-1}\Pi_0^\perp  + \cO(\rho^2\e,\rho^4 ) \ , \quad 
&&\cG^{[1,0]}(\alpha^2, \mu^2, \e)  =\Pi_0 + \cO(\rho^2 \e) \,\\
\cG^{[0,1]}(\alpha^2, \mu^2, \e) &
= -\im\, \sgn (D) + \cO(\rho^2) \, \quad 
&&\cG^{[1,1]}(\al^2,\mu^2,\e)  = \cO(\rho^2\e)\, .
\end{aligned}
\ee
The decomposition \eqref{A.decomp} of $\sL(\al,\mu,\e)$ in  \eqref{sLresca}    is then 
\be\label{svisLep}
 \begin{aligned}
 \sL(\al, \mu, \e) = &  
\underbrace{ \begin{bmatrix}
        \pa_x \circ (1+p_\e(x)) & \cG^{[0,0]}(\al^2,\mu^2,\e) \\ -1-a_\e(x) & (1+p_\e(x))\circ \pa_x
    \end{bmatrix}
    }_{=\sL^{[0,0]} (\al^2,\mu^2,\e)}
    +
    \rho
   \underbrace{ \begin{bmatrix}
        0 & \cG^{[1,0]}(\al^2,\mu^2,\e) \\
        0 & 0
    \end{bmatrix}
    }_{= \sL^{[1,0]} (\al^2,\mu^2,\e)}
   \\
   & +
    \im \mu 
    \underbrace{\begin{bmatrix}
        p_\e(x) & \cG^{[0,1]}(\al^2,\mu^2,\e) \\
        0 & p_\e(x)
    \end{bmatrix}
    }_{= \sL^{[0,1]}(\al^2,\mu^2,\e)}
    + \im \mu \rho 
     \underbrace{ \begin{bmatrix}
        0 & \cG^{[1,1]}(\al^2,\mu^2,\e) \\
        0 & 0
    \end{bmatrix}}_{=  \sL^{[1,1]}(\al^2,\mu^2,\e)}
    \end{aligned}
\ee
and, inserting \eqref{F5} in \eqref{svisLep}, we deduce 
\eqref{quattropuntosei}.
\end{proof}

The following lemma  provides the   expansion of the vectors  $ f_k^\sigma(\al,\mu,\e) $ in 
 \eqref{basisF}
 in $ (\al, \mu, \e ) $ near $ (0,0,0) $. 
\\[1mm]
\noindent 
{\bf Notation.}   
We  denote $even_0(x)$ an $even(x)$ function with zero space average.  
 $\cO(\rho^{n}\e^m) \footnotesize\vet{even(x)}{odd(x)}$ denotes a $\cO(\rho^{n}\e^m)$ function with values in  $ H^1(\T, \C^2) $ 
 whose first component is $even(x)$ and the second one is  $odd(x)$.
 Analogous meaning 
for $\cO(\rho^{m}) {\footnotesize\vet{odd(x)}{even(x)}}$, $\cO(\e^{m}) \footnotesize\vet{odd(x)}{even(x)}$ etc.

\begin{lemma}\label{expansion1} 
{\bf (Expansion of the basis $ {\cal F}$)} 
The symplectic and reversible basis $ {\cal F} $ of 
$\cV_{\al,\mu,\e}$ in \eqref{basisF} 
has regularity 
\begin{equation}\label{decF}
   [(\al,\mu,\e)
   \mapsto  f_1^\pm(\al,\mu,\e), f_0^\pm(\al,\mu,\e)]  \in \cA(B_{\rho_0}(0,0),\e_0;H^1)\, , \quad  \overline{f_k^\pm (\al,\mu,\e)} 
    = 
    f_k^\pm (\al,-\mu,\e) \, , 
\end{equation}
and  the expansion 
\begin{align}
 f^+_1(\al, \mu, \e) & = \vet{\cos(x)}{\sin(x)}  + 
 \epsilon \vet{2 \cos(2x)}{\sin(2x)} + \im\frac\mu 4 \vet{ \sin (x)}{\cos (x)}+\im\mu\e \vet{odd (x)}{even(x)}\label{exf41}
 \\ &  \notag+
\cO(\e^2) \vet{even_0(x)}{odd(x)} + \cO(\rho^2)\vet{even_0(x) + \im\, odd(x)}{odd(x)+\im\, even_0(x)}  + 
\cO(\rho^2\e,\rho \e^2)   \, , \notag \\
 f^-_1(\al,\mu, \e) \label{exf42} &= \vet{-\sin(x)}{\cos(x)} + \epsilon \vet{-2 \sin(2x)}{\cos(2x)} + \im \frac\mu 4 \vet{\cos (x)}{- \sin (x)}+ \im \mu \e \vet{even(x)}{odd(x)} \\ & + \cO(\e^2) \vet{odd(x)}{even(x)} + \cO(\rho^2) \vet{odd(x)+\im\, even_0(x) }{even_0(x)+\im odd(x)}+
 \cO(\rho^2\e,\rho \e^2) \, , \notag \\
 f^+_0(\al, \mu, \e) \label{exf43} &= \vet{1}{0}+ \epsilon \vet{ \cos(x)}{-\sin(x)}  +\frac14 \rho\e \vet{\cos (x)}{- \sin (x) } + \im\mu\e \vet{odd(x)}{even_0(x)}\\ \notag &
 + \cO(\e^2) \vet{even_0(x)}{odd(x)} + 
 \cO(\rho^2\e,\rho \e^2) \, , \\
 f^-_0(\al,\mu, \e) \label{exf44} & = \vet{0}{1} +\frac12 \rho\e \vet{ \sin (x)}{\cos (x)} +\im\mu\e\vet{even_0(x)}{odd(x)}+
 \cO(\rho^2\e,\rho \e^2)\, .
\end{align}
The vectors 
$ f_k^{\sigma} (0,0,\epsilon) $, $ k = 0,1, \sigma = \pm $, are real, and 
\begin{equation}\label{nonzeroaverage}
    f_1^+(0,0,\e) = \vet{even_0(x)}{odd(x)}\, , \ f_1^-(0,0,\e) = \vet{odd(x)}{even(x)}\, , \ f_0^+(0,0,\e) = \vet{1}{0} + \vet{even_0(x)}{odd(x)}\, , \ f_0^- (0,0,\e) = \vet{0}{1}\, .
\end{equation}
For $\e = 0$ we have  
\begin{equation} \label{vtrans}
f_1^+(\al, \mu,0) = \frac{1}{\im \sqrt{2} } \big(v_1^+(\al, \mu) - v_1^-(\al, \mu) \big) \, , \quad
f_1^-(\al, \mu,0) = \frac{1}{ \sqrt{2} } \big(v_1^+(\al, \mu) + v_1^-(\al, \mu) \big)  \, ,   
\end{equation}
where 
$ v_1^\pm (\al, \mu)  =  U_{\al, \mu, 0}\,  v_1^\pm(0,0)  $
are the complex eigenvectors of 
$ {\cal L} (\alpha, \mu, 0 ) $ in  
\eqref{unperturbed.eigv}.  
\end{lemma}

\begin{proof}
See Appendix \ref{ProofExpansion}. 
\end{proof}

\noindent
{\bf Second basis of $\cV_{\al,\mu,\e}$.} 
We now define 
the symplectic and  
reversible basis 
\begin{equation}\label{basisG}
    {\cal G} := \big\{ g_k^\sigma (\al,\mu,\e) \, : \ k= 0,1\, ;\ \sigma = \pm \big\}
\end{equation}
of  $\cV_{\al,\mu,\e}$ according to \Cref{def:symprevbasis,}, where 
\begin{equation}\label{basechange}
    \begin{aligned}
        &g_1^+ (\al,\mu,\e) := f_1^+ (\al,\mu,\e) \, , \qquad g_1^- (\al,\mu,\e) := f_1^- (\al,\mu,\e) - n(\al,\mu,\e) f_0^- (\al,\mu,\e) \, ,  \\
        &g_0^+ (\al,\mu,\e) := f_0^+ (\al,\mu,\e) + n(\al,\mu,\e) f_1^+ (\al,\mu,\e)\, ,\qquad g_0^- (\al,\mu,\e) := f_0^- (\al,\mu,\e)\, ,
    \end{aligned}
\end{equation}
and 
\begin{equation}\label{namep}
n(\al,\mu, \e) := \frac{(f_1^- (\al,\mu, \e),f_0^-(\al,\mu, \e))}{\| f_0^- (\al,\mu, \e)\|^2}\, .
\end{equation}
Note that $n(\al,\mu,\e)$ is real valued for any $\al,\mu,\e$, since,  
recalling \eqref{def:cinvolution} and \eqref{scalarcomplex}, we have $\overline{(\varrho_c f, \varrho_c g)_{L^2}} = (f,g)_{L^2}$, and therefore
\begin{equation}\label{nuisreal}
    n(\al,\mu,\e) = \frac {\overline{(\varrho_c f_1^-(\al,\mu,\e),\varrho_c f_0^-(\al,\mu,\e))}}{\|f_0^-(\al,\mu,\e)\|^2} \stackrel{\eqref{def:rev}} = \frac {\overline{(  f_1^-(\al,\mu,\e),f_0^-(\al,\mu,\e))}}{\|f_0^-(\al,\mu,\e)\|^2} = \overline{n(\al,\mu,\e)} \, .
\end{equation}
The key property of the new basis is that  $g_1^-(0,0,\e)$ has zero space average, see indeed the remainder $\cO(\e^2)$ in  \eqref{exg42}.
Such property,
which is not satisfied by  
$ f_1^- (0,0,\e)$ in \eqref{exf42},  is crucially exploited in Proposition \ref{BexpG}  to prove  \eqref{sB10ongksig00},\eqref{matY2}, see \Cref{rem:newbasis}.

\begin{lemma}\label{expansion2} 
{\bf (Expansion of the basis $ {\cal G}$)} 
The symplectic and reversible basis $ {\cal G}$ of 
$\cV_{\al,\mu,\e}$ in \eqref{basisG} 
have regularity 
\begin{equation}\label{decG}
   [(\al,\mu,\e)
   \mapsto g_1^\pm(\al,\mu,\e) \, ,
    g_0^\pm(\al,\mu,\e)] \in \cA(B_{\rho_0}(0,0),\e_0;H^1(\T,\C^2))\, , \quad \overline{g_k^\sigma(\al,\mu,\e)} = g_k^\sigma(\al,-\mu,\e) \, , 
\end{equation} 
and the  expansion
\begin{align}
 g^+_1(\al, \mu, \e) & = \vet{\cos(x)}{\sin(x)}  + 
 \epsilon \vet{2 \cos(2x)}{\sin(2x)} + \im\frac\mu 4 \vet{\sin (x)}{\cos (x)}+\im\mu\e \vet{odd (x)}{even (x)}\label{exg41}
 \\ &  \notag+
\cO(\e^2) \vet{even_0(x)}{odd(x)} + \cO(\rho^2)\vet{even_0(x) + \im\, odd(x)}{odd(x)+\im\, even_0(x)}  + \cO(\rho^2\e,\rho\e^2)  \, , \notag \\
 g^-_1(\al,\mu, \e) \label{exg42} &= \vet{-\sin(x)}{\cos(x)} + \epsilon \vet{-2 \sin(2x)}{\cos(2x)} + \im \frac\mu 4 \vet{ \cos (x)}{- \sin (x)}+ \im \mu \e \vet{even(x)}{odd(x)} \\ & + \cO(\e^2) \vet{odd(x)}{even_0(x)} + \cO(\rho^2) \vet{odd(x)+\im\, even_0(x) }{even_0(x)+\im odd(x)}+\cO(\rho^2\e,\rho\e^2) \, , \notag \\
 g^+_0(\al, \mu, \e) \label{exg43} &= \vet{1}{0}+ \epsilon \vet{ \cos(x)}{-\sin(x)}  +\frac14 \rho\e \vet{ \cos (x)}{- \sin (x) } 
 + \im\mu\e \vet{odd(x)}{even_0(x)}\\ \notag &
 + \cO(\e^2) \vet{even_0(x)}{odd(x)} + \cO(\rho^2\e,\rho\e^2) , \\
 g^-_0(\al,\mu, \e) \label{exg44} & = \vet{0}{1} +\frac12 \rho\e \vet{ \sin (x)}{ \cos (x)} +\im\mu\e\vet{even_0(x)}{odd(x)}+ \cO(\rho^2\e,\rho\e^2)\, .
\end{align} 
At $\e  = 0$ the basis 
$ g_k^\sigma 
(\al, \mu,0) = f_k^\sigma (\al, \mu,0)  $ for any $ k = 0, 1$, $ \sigma = \pm $.
\end{lemma}

\begin{proof}
The scalar  $n(\alpha, \mu,\e)$ in 
\eqref{namep} satisfies, by \eqref{exf42}
and \eqref{exf44},
\begin{equation}
\label{nmuep}
n(\al,\mu,\e)  =  r(\e^2, \rho^2\e) \, . 
\end{equation}
The last claim follows since  $n(\al,\mu, 0) = 0 $.
In addition, by \eqref{decF} and \Cref{lem:prodintM} $(iii)$, we
deduce that $ n (\alpha, \mu, \epsilon ) $ belongs to 
$ \cA(B_{\rho_0}(0,0),\e_0;\C) $.
Hence, \eqref{decG} holds and, in view of \eqref{exf41}-\eqref{exf44}, 
 the vectors $ g^\sigma_k (\alpha, \mu,\e) $ in \eqref{basisG} 
have the expansions \eqref{exg41}-\eqref{exg44}. 
Furthermore,  by 
\eqref{nonzeroaverage},   $ f_0^-(0,0,\e) = \footnotesize\vet{0}{1}  $, and thus 
\begin{equation*}
g_1^-(0,0,\e) \stackrel{\eqref{basechange}, \eqref{namep}} = f_1^-(0,0,\e) - \Big( f_1^-(0,0,\e), \vet{0}{1} \Big)\vet{0}{1} \, ,
\end{equation*} 
which is $\vet{odd(x)}{even_0(x)}$ 
recalling  \eqref{nonzeroaverage}. 
\end{proof}

We now state the main result of this section.

\begin{proposition}\label{BexpG}
{\bf (Matrix representation of \( {\sL}(\al,\mu,\e)\) on \( \mathcal{V}_{\alpha,\mu, \e}\))} 
The Hamiltonian and reversible operator  
$ \sL (\alpha,\mu,\e) : \mathcal{V}_{\alpha, \mu,\e} \to \mathcal{V}_{\alpha,\mu,\e} $ is represented on the symplectic and reversible basis $\mathcal{G}$ of $\mathcal{V}_{\alpha, \mu,\e}$ in \eqref{basisG} 
by a   Hamiltonian matrix 
 $\tL_{\alpha,\mu,\e}=\tJ_4 \tB_{\alpha,\mu,\e}$, where 
 $\tB_{\alpha,\mu,\e} $
is a $ 4 \times 4$  matrix   satisfying 
\be\label{reg:TB} 
[(\al,\mu,\e)\mapsto\tB_{\al,\mu,\e}]\in\cA(B_{\rho_0}(0,0),\e_0;\C^{4\times 4}) \, ,  \qquad    \overline{\tB_{\al, \mu,\e}} =
\tB_{\al, - \mu,\e} \, , 
\ee
 self-adjoint and reversibility preserving,  of the form \begin{equation}\label{splitB}
\tB_{\alpha,\mu,\e}=
\begin{pmatrix} 
E & F \\ 
F^* & G 
\end{pmatrix} \, ,  \qquad
E  =\begin{pmatrix}
    E_{11} & \im E_{12} \\
    -\im E_{12} & E_{22}
\end{pmatrix} = E^* \, , \   \ G =\begin{pmatrix}
    G_{11} & \im G_{12} \\
    -\im G_{12} & G_{22}
\end{pmatrix} = G^* \, , 
\end{equation} 
  where $E,  G, F $  are the $ 2 \times 2 $  matrices 
\begin{align}\label{BinG1}
& E := 
\begin{pmatrix} 
  \e^2(1+r_1'(\e^2)) +\frac{\al^2}{4}(1+r_1''(\e^2,\rho))-\frac{\mu^2}{8}(1+r_1'''(\e^2,\rho))  & \im \frac12\mu\big( 1+ r_2(\e^2, \rho^2))  \\
- \im \frac12\mu\big(  1 + r_2(\e^2,\rho^2) \big) & \frac{\al^2}{4}(1+r_5'(\e^2,\rho))-\frac{\mu^2}{8}(1+r_5(\e^2,\rho))
 \end{pmatrix} \\
 & \label{BinG2} G := 
 \begin{pmatrix} 
1+ r_8(\e^4,\rho\e^2,\rho^3) &   - \im  \mu \,  r_9(\e^2,\rho^2)   \\
  \im  \mu \,  r_9(\e^2,\rho^2)   & \rho+r_{10}(\rho^2\e^2,\rho^3)
 \end{pmatrix} \\
 & \label{BinG3}
 F = \begin{pmatrix}
     F_{11} & \im F_{12}\\
     \im F_{21} & F_{22}
 \end{pmatrix}=
 \begin{pmatrix} 
 r_3(\e^3,\rho^2\e) & \im  \mu \, r_4(\e) \\
  \im  \mu \, r_6(\e)     & r_7(\rho^2\e) 
 \end{pmatrix} \, . 
 \end{align}
The matrices $E, G$ are even in $\e$ and $F$ is odd in $\e$.
At 
$\e = 0 $, 
\begin{equation}\label{E11E22}
E_{11} (\alpha, \mu, 0)
= E_{22} (\alpha, \mu, 0) = \tfrac12 
(\omega_1^-(\al,\mu)-\omega_1^+(\al,\mu))\, , 
 \quad  
    E_{12}(\al,\mu,0) = \mu -
    \tfrac12 
    (\omega_1^+ (\al,\mu) +\omega_1^-(\al,\mu)) \, . 
\end{equation}
\end{proposition}

 Note  that $ F = 0 $ vanishes at $\e = 0$.
For $ \alpha = 0 $ the matrix $\tL_{0,\mu,\e}  $ coincides with
the Hamiltonian matrix
in \cite[Proposition 4.4]{BMV1} (where the property that $ F $ vanishes in $ \epsilon = 0 $ had not been observed).

\smallskip 

The rest of this section is devoted to the proof of 
Proposition \ref{BexpG}. 

The matrix $\tB_{\al, \mu, \e}$ in \eqref{splitB} 
 is the self-adjoint and reversibility preserving matrix with entries
\be\label{matriceBame}
[\tB_{\al, \mu, \e}]_{k,k'}^{\sigma,\sigma'} = 
\left(   \sB(\al, \mu,\e)g_k^{\sigma}(\al, \mu,\e),g_{k'}^{\sigma'}(\al, \mu, \e)\right)\, ,  
\ee
  associated, as in \eqref{matrix22},  to the self-adjoint and reversibility
preserving operator $\sB(\al, \mu, \e)$ in \eqref{sLresca}.
The regularity \eqref{reg:TB} holds 
in view of \eqref{decG},  \eqref{regGBL}, \eqref{cLsL}, \eqref{cGcLmenomu} and  \Cref{lem:prodintM}$(iii)$. Thus \Cref{lem:decomposition} 
implies that  
\begin{equation}\label{decAforTB}
    \tB_{\al,\mu,\e} = \tB^{[0,0]}(\al^2,\mu^2,\e)
+ \rho\, \tB^{[1,0]}(\al^2,\mu^2,\e)
+ \im\mu\, \tB^{[0,1]}(\al^2,\mu^2,\e)
+ \im\mu\rho\, \tB^{[1,1]}(\al^2,\mu^2,\e)
\end{equation}
where each matrix $\tB^{[i,j]}(\al^2,\mu^2,\e)$ is real and analytic in its argument. 
Note also that, in view of \eqref{revprop0}, we have 
\be\label{eq:altrealim}
[\tB_{\al, \mu, \e}]_{k,k'}^{\sigma,\sigma'} = 
\begin{cases}
[\tB^{[0,0]}(\al^2,\mu^2,\e) 
+ \rho  \tB^{[1,0]}(\al^2,\mu^2,\e)]_{k,k'}^{\sigma,\sigma'}
\qquad & \sigma = \sigma', \\
\im \mu \, [\tB^{[0,1]}(\al^2,\mu^2,\e)
+ \rho  \tB^{[1,1]}(\al^2,\mu^2,\e)]_{k,k'}^{\sigma,\sigma'}  \qquad & \sigma = - \sigma'\, .
\end{cases}
\ee
We now provide the expansions of 
the matrix entries  \eqref{matriceBame}. 
\\[1mm]
{ \bf Expansion of $ \sB(\al,\mu,\e) $.}
In view of the decomposition \eqref{quattropuntosei}, we have   expansion 
\begin{equation}\label{expBflatij}
    \begin{aligned}
        \sB(\al,\mu,\e) =  -\cJ \sL(\alpha, \mu, \e)   = & \underbrace{
    \begin{bmatrix}
        1+a_\e(x) & -(1+p_\e(x))\pa_x \\
        \pa_x\circ (1+p_\e(x)) & |D|
    \end{bmatrix}}_{=\sB^{[0,0]}(0,0,\e)}
    +\rho
    \underbrace{\begin{bmatrix}
        0 & 0\\
        0 & \Pi_0
    \end{bmatrix}}_{= \sB^{[1,0]}(0,0,\e)}
     \\
    & 
    + \im\mu \underbrace{\begin{bmatrix}
        0 & -p_\e(x) \\
        p_\e(x) & -\im \, \sgn (D)
    \end{bmatrix}}_{=  \sB^{[0,1]}(0,0,\e)} 
    + \al^2 \underbrace{\begin{bmatrix}
        0 & 0 \\
        0 & \frac12 |D|^{-1}\Pi_0^\perp
    \end{bmatrix}}_{=: \sB^{[2,0]}} + \cO(\rho^3,\rho^2\e) \, .
    \end{aligned}
\end{equation}
By \Cref{expansion2} and \eqref{lem:revbasis} we get the following.
\\[1mm]
{ \bf Expansion of $ g_k^\sigma $.}
The vectors in \eqref{exg41}-\eqref{exg44} have the form, for any 
$ k = 0,1 $, $ \sigma = \pm $,  
\begin{equation}\label{expansionbasisinsideproof}
    \begin{aligned}
        g_k^\sigma(\al,\mu,\e) = g_k^{\sigma,[0,0]}(\e) &+ \rho g_k^{\sigma,[1,0]}(\e) + \im\mu g_k^{\sigma,[0,1]}(\e) 
        + g_k^{\sigma,[\geq 2]}(\al, \mu, \e)
    \end{aligned}
\end{equation}
where each $g_k^{\sigma,[i,j]}(\e)$ are real vectors.
The jets  $g_k^{\sigma,[0,0]}(\e)$ are
\begin{equation}\label{gksig00}
    \begin{aligned}
        &g_1^{+,[0,0]}(\e) = \vet{\cos (x)}{\sin (x)} + \e \vet{2\cos (2x)}{\sin (2x)} + \cO(\e^2)\vet{even_0(x)}{odd(x)}\, , \\ &g_1^{-,[0,0]}(\e) = \vet{-\sin (x)}{\cos (x)} + \e \vet{-2 \sin (2x)}{\cos (2x)} + \cO(\e^2)\vet{odd(x)}{even_0(x)}\, , \\ &g_0^{+,[0,0]}(\e) = \vet{1}{0} + \e \vet{\cos (x)}{-\sin (x)} + \cO(\e^2)\vet{even_0(x)}{odd(x)}\, ,\\
        &g_0^{-,[0,0]}(\e) = \vet{0}{1} \, ,
    \end{aligned}
\end{equation} 
\begin{equation}\label{exp:pa_rho_g}
    \begin{aligned}
        &g_1^{+,[1,0]}(\e) = \cO(\e^2)\vet{even(x)}{odd(x)}\, , \qquad &&g_0^{+,[1,0]}(\e) = \frac{\e}{4}  \vet{\cos (x)}{-\sin (x)} + \cO(\e^2)\vet{even(x)}{odd(x)}\, , \\
        & g_1^{-,[1,0]}(\e) = \cO(\e^2)\vet{odd(x)}{even(x)}\, , \qquad &&g_0^{-,[1,0]}(\e) = \frac{\e}{2}  \vet{\sin (x)}{\cos (x)} + \cO(\e^2)\vet{odd(x)}{even(x)}\, ,
    \end{aligned}
\end{equation}
and 
\begin{align}\label{exp:pa_mu_g}
 &g_1^{+,[0,1]}(\e)= \frac14 	\vet{\sin(x)}{\cos(x)}+\cO(\e) \vet{odd(x)}{even(x)} \, , \qquad  &&g_0^{+,[0,1]}(\e)= \e \vet{odd(x)}{even_0(x)} + \cO(\e^2)\vet{odd(x)}{even(x)} \, ,\\  
\notag 
 &g_1^{-,[0,1]}(\e) = \frac14\vet{\cos(x)}{-\sin(x)} + \cO(\e)\vet{even(x)}{odd(x)}\, ,\qquad &&g_0^{-,[0,1]}(\e) = \e \vet{even_0(x)}{odd(x)} +\cO(\e^2)\vet{even(x)}{odd(x)}  \, . 
\end{align}
The higher order jets $g_k^{\sigma,[\geq 2]}(\al, \mu, \e)$  satisfy 
\begin{equation}\label{gksig2}
    \begin{aligned}
        &g_1^{+,[\geq 2]}(\al, \mu, \e) = \cO(\rho^2)\vet{even_0(x)+\im odd(x)}{odd(x)+\im even_0(x)} +  \cO(\rho^2\e)\,  \, , \\
        &g_1^{-,[\geq 2]}(\al, \mu, \e) = \cO(\rho^2) \vet{odd(x)+\im even_0(x)}{even_0(x)+\im odd(x)} + \cO(\rho^2\e)\,  \, ,\\
         &g_0^{+,[\geq 2]}(\al, \mu, \e)\, , \ g_0^{-,[\geq 2]}(\al, \mu, \e) =  \cO(\rho^2\e)\,  \, .
    \end{aligned}
\end{equation}
{\bf First expansion of $ \tB_{\al, \mu, \e} $.}
The matrix \eqref{matriceBame} has the expansion 
\begin{equation}\label{TaylorexpBemu}
    \tB_{\al, \mu, \e}  = \tB^{[0,0]}_{\e} + \im \mu  \tB^{[0,1]}_{\e} +\rho  \tB^{[1,0]}_{\e} + \im \mu \rho \tB^{[1,1]}_{0} + \mu^2\tB^{[0,2]}_{0} + \al^2\tB^{[2,0]}_{0} + \tB^{[\geq 3]}_{\al,\mu,\e}
    + \cO(\rho^2\e,\rho^3)
\end{equation}
 where
 \be \label{B00}
\tB^{[0,0]}_{\e} := \left(   \sB^{[0,0]}(0,0,\e)g_k^{\sigma,[0,0]}(\e),g_{k'}^{\sigma',[0,0]}(\e)\right)_{k,k'}^{\sigma,\sigma'}\, ,
 \ee
\begin{equation}\label{MatrixX1}
\tB^{[0,1]}_{\e} := Y_1 - X_1 + X_1^* 
\qquad \text{with} \qquad\begin{aligned}&[X_1]_{k,k'}^{\sigma,\sigma'} :=
 \big(   \sB^{[0,0]}(0,0,\e) g_k^{\sigma,[0,0]}(\e),  g^{\sigma',[0,1]}_{k'}(\e) \big) \, ,\\
 &[Y_1]_{k,k'}^{\sigma,\sigma'} :=
 \big(  \sB^{[0,1]}(0,0,\e) g_k^{\sigma,[0,0]}(\e),  g^{\sigma',[0,0]}_{k'}(\e) \big) \, .
 \end{aligned}
 \end{equation}
\begin{equation}\label{MatrixX2}
\tB^{[1,0]}_{\e} = Y_2 + X_2 + X_2^* 
\qquad \text{with} \qquad\begin{aligned}&[X_2]_{k,k'}^{\sigma,\sigma'} :=
 \big(  \sB^{[0,0]}(0,0,\e) g_k^{\sigma,[0,0]}(\e),  g^{\sigma',[1,0]}_{k'}(\e) \big) \, ,\\
 &[Y_2]_{k,k'}^{\sigma,\sigma'} :=
 \big(  \sB^{[1,0]}(0,0,\e) g_k^{\sigma,[0,0]}(\e),  g^{\sigma',[0,0]}_{k'}(\e) \big) \, ;
 \end{aligned}
 \end{equation}
the second order jets, evaluated at $\e = 0$, are
 \begin{equation}\label{dersec1}
\begin{aligned}        [\tB^{[1,1]}_0]_{k,k'}^{\sigma,\sigma'} & := \underbrace{\big(\sB^{[0,0]}(0,0,0) g_k^{\sigma,[0,1]}(0),g_{k'}^{\sigma',[1,0]}(0)\big)}_{=: [X_7]_{k,k'}^{\sigma,\sigma'}} -\underbrace{\big(\sB^{[0,0]}(0,0,0) g_k^{\sigma,[1,0]}(0),g_{k'}^{\sigma',[0,1]}(0)\big)}_{= [X_7^*]_{k,k'}^{\sigma,\sigma'}} \\
        & - \underbrace{\big(\sB^{[1,0]}(0,0,0) g_k^{\sigma,[0,0]}(0),g_{k'}^{\sigma',[0,1]}(0)\big)}_{=: [Y_5]_{k,k'}^{\sigma,\sigma'}} + \underbrace{\big(\sB^{[1,0]}(0,0,0) g_k^{\sigma,[0,1]}(0),g_{k'}^{\sigma',[0,0]}(0)\big)}_{= [Y_5^*]_{k,k'}^{\sigma,\sigma'}}
        \\
        &+\underbrace{\big(\sB^{[0,1]}(0,0,0) g_k^{\sigma,[0,0]}(0),g_{k'}^{\sigma',[1,0]}(0)\big)}_{=: [Y_6]_{k,k'}^{\sigma,\sigma'}} +\underbrace{\big(\sB^{[0,1]}(0,0,0) g_k^{\sigma,[1,0]}(0),g_{k'}^{\sigma',[0,0]}(0)\big)}_{= -[Y_6^*]_{k,k'}^{\sigma,\sigma'}}\, ,
\end{aligned}
\end{equation}
\begin{equation}\label{dersec}
\begin{aligned}
     [\tB^{[0,2]}_0]_{k,k'}^{\sigma,\sigma'}
     & := \underbrace{\big(\sB^{[0,0]}(0,0,0) g_k^{\sigma,[0,1]}(0),g_{k'}^{\sigma',[0,1]}(0)\big)}_{=:[X_4]_{k,k'}^{\sigma,\sigma'}} + \underbrace{\big(\sB^{[0,0]}(0,0,0) g_k^{\sigma,[1,0]}(0,0,0),g_{k'}^{\sigma',[1,0]}(0)\big)}_{=:[X_5]_{k,k'}^{\sigma,\sigma'}}\\
     &+\underbrace{\big(\sB^{[0,1]}(0,0,0) g_k^{\sigma,[0,0]}(0),g_{k'}^{\sigma',[0,1]}(0)\big)}_{=:[Y_3]_{k,k'}^{\sigma,\sigma'}} - \underbrace{\big(\sB^{[0,1]}(0,0,0) g_k^{\sigma,[0,1]}(0),g_{k'}^{\sigma,[0,0]}(0)\big)}_{=-[Y_3^*]_{k,k'}^{\sigma,\sigma'}}\\
     &+\underbrace{\big(\sB^{[1,0]}(0,0,0) g_k^{\sigma,[0,0]}(0),g_{k'}^{\sigma',[1,0]}(0)\big)}_{=:[Y_4]_{k,k'}^{\sigma,\sigma'}} + \underbrace{\big(\sB^{[1,0]}(0,0,0) g_k^{\sigma,[1,0]}(0),g_{k'}^{\sigma',[0,0]}(0)\big)}_{=[Y_4^*]_{k,k'}^{\sigma,\sigma'}}\, ,
\end{aligned}
\end{equation}
and
\begin{equation}\label{dersec2}
\begin{aligned}
     [\tB^{[2,0]}_0]_{k,k'}^{\sigma,\sigma'}
     & :=\underbrace{\big(\sB^{[1,0]}(0,0,0) g_k^{\sigma,[0,0]}(0),g_{k'}^{\sigma',[1,0]}(0)\big)}_{=:[Y_4]_{k,k'}^{\sigma,\sigma'}} + \underbrace{\big(\sB^{[1,0]}(0,0,0) g_k^{\sigma,[1,0]}(0),g_{k'}^{\sigma',[0,0]}(0)\big)}_{=[Y_4^*]_{k,k'}^{\sigma,\sigma'}}\\
     &+ \underbrace{\big(\sB^{[2,0]} g_k^{\sigma,[0,0]}(0),g_{k'}^{\sigma',[0,0]}(0)\big)}_{=:[Z_1]_{k,k'}^{\sigma,\sigma'}}+ \underbrace{\big(\sB^{[0,0]}(0,0,0) g_k^{\sigma,[1,0]}(0),g_{k'}^{\sigma',[1,0]}(0)\big)}_{=:[X_5]_{k,k'}^{\sigma,\sigma'}}
\end{aligned}
\end{equation}
and finally 
\begin{equation}\label{Bgeq3}
    \tB^{[\geq 3]}_{\al,\mu,\e} = Z_2 + Z_2^*\, , \quad \text{where}\quad [Z_2]_{k,k'}^{\sigma,\sigma'}:= \big( \sB^{[0,0]}(\e)g_k^{\sigma,[0,0]}(\e),g_{k'}^{\sigma',[\geq 2]}(\al,\mu,\e)\big)\, .
\end{equation} 
Each $\tB^{[i,j]}_{\e}$ is real and
\begin{equation}\label{zeroentriesrev}
         \begin{aligned}
             &[\tB^{[1,0]}_{\e}]_{k,k'}^{\sigma,\sigma'} \equiv [\tB^{[0,2]}_0]_{k,k'}^{\sigma,\sigma'} = [\tB^{[2,0]}_0]_{k,k'}^{\sigma,\sigma'} = 0 \quad \forall  \sigma = -\sigma'\, , \\
             &[\tB^{[0,1]}_{\e}]_{k,k'}^{\sigma,\sigma'}\equiv [\tB^{[1,1]}_{\e}]_{k,k'}^{\sigma,\sigma'} =  0 \quad \quad \forall \sigma = \sigma'\, .
         \end{aligned}
     \end{equation}

\begin{proof} 
We insert the  expansions  \eqref{expBflatij} and  \eqref{expansionbasisinsideproof} in 
\eqref{matriceBame}.  
 To compute the expansions \eqref{TaylorexpBemu}, we collect all the terms $ (\cB^{[i_0,j_0]}(0,0,0) g_k^{\sigma,[i_1,j_1]},g_{k'}^{\sigma',[i_2,j_2]})$, and divide them into the groups in \eqref{B00}-\eqref{dersec2}, according to the resulting product of the coefficients ($1$, $\im\mu$, $\rho$, $\mu^2$, $\im\mu\rho$, $\al^2$). Note that combinations in which terms of type $[1,0]$ appear twice have to be considered both in $\tB^{[2,0]}_0$ and  $\tB^{[0,2]}_0$, since the product $\rho \cdot \rho = \rho^2 = \al^2 + \mu^2$.  We also use that $\sB^{[1,0]}$ is self-adjoint and $\sB^{[0,1]}$ is skew-adjoint (being respectively a real and an imaginary jet of a self-adjoint operator), to show, for example, that the third and fourth term in \eqref{dersec} are one the skew-adjoint of the other, while the last two terms in \eqref{dersec}  are reciprocally adjoint.
Finally \eqref{zeroentriesrev} follows recalling \eqref{TaylorexpBemu}, \eqref{reg:TB} 
     and \eqref{revprop0}.
\end{proof}

We now compute each term of \eqref{TaylorexpBemu}.

\smallskip
 
\noindent{\bf Expansion of $\tB^{[0,0]}_{\e}$ in \eqref{B00}.} Since 
 the operator $\sB^{[0,0]}(0,0,\e)$ defined in \eqref{expBflatij} coincides with the operator  $  {\cal B}_{0,\e}$  in \cite[formula (4.22)]{BMV1}, and the vectors $g_k^{\sigma,[0,0]}(\e) \equiv g_k^\sigma(0,0,\e)$ (cf.  \eqref{expansionbasisinsideproof}) coincide with the vectors $g_k^\sigma(0,\e)$ 
in \cite[formula (4.9)]{BMV1},  the matrix $\tB^{[0,0]}_\e$ 
coincides with the one computed in \cite[equation (4.30)]{BMV1}, namely 
\begin{equation}
\label{Bsoloeps}
\tB^{[0,0]}_{\e}  = 
 \begin{pmatrix}\e^2+ {r(\e^3)} & 0 & \vline & r(\e^3) & 0 \\ 
 0& 0 & \vline & 0 & 0 \\
 \hline 
  r(\e^3)  & 0 & \vline & 1+ r(\e^3) & 0 \\
 0 & 0 & \vline & 0 & 0\end{pmatrix}.
\end{equation}

\noindent{\bf Expansion of $\tB^{[0,1]}_\e$ in  \eqref{MatrixX1}.} In view of \eqref{zeroentriesrev}, the only non-zero entries $[\tB^{[0,1]}_\e]_{k,k'}^{\sigma,\sigma'}$ are those corresponding to $\sigma = -\sigma'$.
We now show that 
\begin{equation}\label{matB01}
    \tB^{[0,1]}_\e = \begin{pmatrix}
        0 & \frac12 + r(\e^2) & \rvline & 0 & r(\e)\\
        -\frac12 + r(\e^2) & 0 & \rvline & r(\e) & 0 \\
        \hline
        0 & r(\e) & \rvline & 0 & r(\e^2)\\
         r(\e) & 0 & \rvline & r(\e^2) & 0
    \end{pmatrix} \, . 
\end{equation}
 We compute the matrices  $X_1,Y_1$ in \eqref{MatrixX1}.
We start with $X_1$, noting that the matrix 
$$ \tL_\e^{[0,0]} := \tJ_4 \tB_\e^{[0,0]}  \stackrel{\eqref{Lform},\eqref{Bsoloeps}}{=} 
 \begin{pmatrix} 0 & 0 & \vline & 0 & 0 \\ 
 -\e^2+ {r(\e^3)}& 0 & \vline & r(\e^3) & 0 \\
 \hline 
 0 & 0 & \vline & 0 & 0 \\
  r(\e^3)  & 0 & \vline & -1+ r(\e^3) & 0
  \end{pmatrix}
$$ 
represents the action 
of the operator $\sL_{0,0,\e}:\mathcal{V}_{0,0,\e}\to\mathcal{V}_{0,0,\e}$ on the basis
$ \{ g^{\sigma}_k (0,0,\e) \} = \{ g_k^{\sigma,[0,0]}(\e)\} $, and therefore
\begin{align*}
& \sL(0,0,\e) g_1^{+,[0,0]}(\e) = -(\e^2 + r(\e^3))g_1^{-,[0,0]}(\e) +r(\e^3) g_0^{-,[0,0]}(\e)  \ , 
\\
& \sL(0,0,\e) g_1^{-,[0,0]}(\e)  = \sL(0,0,\e) g_0^{-,[0,0]}(\e) = 0
\\ 
& \sL(0,0,\e) g_0^{+,[0,0]}(\e) =r(\e^3) g_1^{-,[0,0]}(\e) - (1+r(\e^3))g_0^{-,[0,0]}(\e)\, .
\end{align*}
Using also $\sB^{[0,0]}(0,0,\e) = -\cJ\sL(0,0,\e)$ and \eqref{gksig00} we obtain
\begin{align}
\sB^{[0,0]}(0,0,\e)  g_1^{+,[0,0]}(\e)  &= \big(\e^2+r(\e^3)\big)\,  \cJ g_1^{-,[0,0]}(\e) + r(\e^3)\, \cJ \footnotesize\vet{0}{1} =\e^2 \vet{\cos (x)}{\sin (x)} + r(\e^3)\Big(\vet{1}{0} + \vet{even_0(x)}{odd(x)} \Big) \, , \notag \\
\sB^{[0,0]}(0,0,\e)  g_0^{+,[0,0]}(\e) &= r(\e^3)\cJ g_1^{-,[0,0]}(\e) +\big(1+r(\e^3)\big)\cJ \footnotesize\vet{0}{1}= \vet{1}{0}+r(\e^3)\Big(\vet{1}{0} + \vet{even_0(x)}{odd(x)} \Big) \, ,\notag\\
 \sB^{[0,0]}(0,0,\e)  g_1^{-,[0,0]}(\e) &= 0\, , \quad \sB^{[0,0]}(0,0,\e)  g_0^{-,[0,0]}(\e)   = 0\, ,
 \label{exp:Bg}
\end{align}
 so in particular  the second and fourth columns of the matrix $X_1$ in \eqref{matX1} are zero. 
 By  \eqref{exp:pa_mu_g} and \eqref{exp:Bg}, the other two columns of the matrix $X_1$  in \eqref{MatrixX1} have the expansion    
 \begin{equation}\label{matX1}
   X_1 = \begin{pmatrix} 
 0 & 0 & \vline & 0 & 0 \\ 
 r(\e^3)  & 0 & \vline & r( \e) & 0 \\
 \hline
 0 & 0 & \vline & 0 & 0 \\
 r(\e^3) & 0 & \vline & r(\e^2) & 0
 \end{pmatrix}.
 \end{equation}
 We now turn the attention to the matrix $Y_1$. We first compute the action of the operator  $\sB^{[0,1]}(0,0,\e)$ in \eqref{expBflatij} on the vectors $g_k^{\sigma,[0,0]}(\e)$ in \eqref{gksig00}: using that  $p_\e = - 2 \e \cos(x) + \cO(\e^2)[even(x)]$ (cf.  \eqref{apexp}), and that 
 \begin{equation}\label{action-isgnD}
     -\im  \, \sgn(D)\cos (kx) = \sin (kx)\, , \quad -\im \, \sgn(D)\sin (kx) = -\cos (kx) \quad \forall k\in\Z\, ,
 \end{equation}
 we obtain
 \begin{equation}\label{Bflat01ong00}
     \begin{aligned}
         \sB^{[0,1]}(0,0,\e) g_1^{+,[0,0]}(\e) &= \vet{0}{-\cos (x)} + \e \vet{\sin (2x)}{-1-2\cos (2x)} + \cO(\e^2)\vet{odd(x)}{even(x)}\, , \\
         \sB^{[0,1]}(0,0,\e) g_1^{-,[0,0]}(\e) &= \vet{0}{\sin (x)} + \e \vet{1+\cos (2x)}{2\sin (2x)} + \cO(\e^2)\vet{even(x)}{odd(x)} \, , \\
         \sB^{[0,1]}(0,0,\e) g_0^{+,[0,0]}(\e) &= \e \vet{0}{-\cos (x)} + \cO(\e^2)\vet{odd(x)}{even(x)}\,  ,\\
         \sB^{[0,1]}(0,0,\e) g_0^{-,[0,0]}(\e) & = \vet{- p_\epsilon(x)}{0} = \e \vet{2\cos (x)}{0} + \cO(\e^2)\vet{even(x)}{0}\, .
     \end{aligned}
 \end{equation}
 Taking the scalar products of the vectors in \eqref{Bflat01ong00} with those in \eqref{gksig00} as for $Y_1$ in \eqref{MatrixX1}, we obtain
 \begin{equation}\label{matY1}
     Y_1 = \begin{pmatrix}
    0 & \frac12 + r(\e^2) & \rvline & 0 & r(\e) \\
    -\frac12 + r(\e^2) &0 &\rvline  & r(\e) & 0\\
    \hline
    0 & r(\e ) &\rvline & 0 & r(\e^2)\\
    r(\e) & 0 &\rvline & r(\e^2) & 0
\end{pmatrix}\, .
 \end{equation}
 Then \eqref{matB01} follows from \eqref{MatrixX1}, \eqref{matX1} and \eqref{matY1}.
\\[1mm]
{\bf Expansion of $\tB^{[1,0]}_\e$  in \eqref{MatrixX2}.} Recalling \eqref{zeroentriesrev}, the only possibly non-zero entries $[\tB^{[1,0]}_\e]_{k,k'}^{\sigma,\sigma'}$ are those corresponding to $\sigma = \sigma'$.
We now show that 
\begin{equation}\label{matB10}
    \tB^{[1,0]}_\e = \begin{pmatrix}
        r(\e^4) & 0 & \rvline & r(\e^2) & 0\\
    0 & 0 & \rvline & 0 & 0\\
    \hline
    r(\e^2) & 0 &\rvline & r(\e^2) & 0\\
    0 & 0 & \rvline & 0 & 1
    \end{pmatrix} \, . 
\end{equation}
Let us first focus on the matrix $X_2$ in \eqref{MatrixX2}:
again by \eqref{exp:Bg}  the second and the fourth column of $X_2$ are zero. 
 Regarding the other entries, taking the scalar products of the vectors in \eqref{exp:Bg} and \eqref{exp:pa_rho_g} as for $X_2$ in \eqref{MatrixX2}, the first and third columns of the matrix $X_2$ have the expansion  
 \begin{equation}\label{matX2}
   X_2= \begin{pmatrix} 
 r(\e^4) & 0 & \vline & r(\e^2) & 0 \\ 
0 & 0 & \vline & 0 & 0 \\
 \hline
 r(\e^4) & 0 & \vline & r(\e^2) & 0 \\
 0 & 0 & \vline & 0 & 0
 \end{pmatrix} .
 \end{equation}
 Regarding the matrix $Y_2$ in \eqref{MatrixX2} instead, using \eqref{expBflatij} and \eqref{gksig00} we compute 
 \begin{equation}\label{sB10ongksig00}
     \begin{aligned}
        &  \sB^{[1,0]}(0,0,\e) g_1^{+,[0,0]}(\e)\equiv\sB^{[1,0]}(0,0,\e) g_1^{-,[0,0]}(\e)\equiv \sB^{[1,0]}(0,0,\e) g_0^{+,[0,0]}(\e) \equiv 0\, , \\
         & \sB^{[1,0]}(0,0,\e) g_0^{-,[0,0]}(\e) \equiv \vet{0}{1} \ . 
     \end{aligned}
 \end{equation}
 In the computation of $\sB^{[1,0]}(0,0,\e) g_1^{-,[0,0]}(\e)$, we  use that $g_1^{-,[0,0]}(\e) \equiv g_1^-(0,0,\e)$ has zero average, cf. \eqref{exg42}.
 Therefore taking the scalar products of the vectors in \eqref{sB10ongksig00} with \eqref{gksig00} for $Y_2$ as in \eqref{MatrixX2},  we obtain 
 \begin{equation}\label{matY2}
     Y_2 =\begin{pmatrix}
       0 & 0 & \rvline & 0 & 0\\
       0 & 0 & \rvline & 0 & 0\\
       \hline
       0 & 0 & \rvline & 0 & 0\\
       0 & 0 & \rvline & 0 & 1\\
   \end{pmatrix} \, .
 \end{equation}
 Therefore summing \eqref{matX2} and \eqref{matY2} one obtains \eqref{matB10}.
 
\begin{remark}\label{rem:newbasis}
We use that the second component of $g_1^{-,[0,0]}(\e)$ has zero average, in order to show that the $(2,2)$ entry of the matrix $\mathtt{B}_{\al,\mu,\e}$ in \eqref{splitB} contains no terms of order $\cO(\rho\e^k)$. This property is fundamental for verifying that the $(2,2)$ entry of the matrix $E$ in \eqref{BinG1} starts with $\frac{\alpha^2}{4}$ and is therefore positive for small $\alpha$. Such a property does not hold for the first basis ${\cal F}$ defined in \eqref{basisF}, motivating 
the use of the second basis ${\mathcal G}$.
\end{remark}

\vspace{1mm}
\noindent{\bf Expansion of $\tB^{[0,2]}_0$ in \eqref{dersec}.}
Recalling \eqref{zeroentriesrev}, the  only possibly non-zero entries $[\tB^{[0,2]}_0]_{k,k'}^{\sigma,\sigma'}$ are those corresponding to $\sigma = \sigma'$. We now show that
\begin{equation} \label{B02_0}
    \tB^{[0,2]}_0 = \begin{pmatrix}
    -\frac18 & 0 & \rvline & 0&0\\
    0 & -\frac18 & \rvline & 0 & 0 \\
    \hline
    0 & 0 & \rvline & 0 & 0 \\
    0 & 0 & \rvline & 0 & 0
\end{pmatrix} \, . 
\end{equation}
In view of \eqref{dersec}, we compute 
 the matrices $X_4$,  $X_5$, $Y_3$ and $Y_4$.
 We start with $X_4$ by  computing  the action of $\sB^{[0,0]}(0,0,0) = \begin{bmatrix}
     1 & - \pa_x \\
     \pa_x & |D|
 \end{bmatrix}$    in 
 \eqref{expBflatij} on the vectors $g_k^{\sigma,[0,1]}(0) $ in  \eqref{exp:pa_mu_g}, getting
$$
\begin{aligned}
& \sB^{[0,0]}(0,0,0)g_1^{+,[0,1]}(0) = \frac12 \vet{\sin (x)}{\cos (x)}\, ,   \quad \sB^{[0,0]}(0,0,0)g_1^{-,[0,1]}(0) = \frac12 \vet{\cos (x)}{-\sin (x)}\, ,
\\ 
& \sB^{[0,0]}(0,0,0)g_0^{+,[0,1]}(0) = \sB^{[0,0]}(0,0,0)g_0^{-,[0,1]}(0) = 0 \ . 
\end{aligned}
$$
Therefore, computing the scalar products for $X_4$ as in \eqref{dersec} with $g_k^{\sigma,[0,1]}(0)$ in \eqref{exp:pa_mu_g}, we obtain
\begin{align}\label{matZ}
 X_4= \begin{pmatrix} 
 \frac{1}{8} & 0 & \vline & 0 & 0 \\
 0 & \frac{1}{8} & \vline & 0 & 0 \\
 \hline
 0 & 0 & \vline & 0 & 0 \\
 0 & 0 & \vline & 0 & 0 \\
 \end{pmatrix} \, .
\end{align}
Regarding the matrices $X_5$ and $Y_4$ in \eqref{dersec}, by \eqref{exp:pa_rho_g} the vectors $g_k^{\sigma,[1,0]}(0) = 0$,  $ \forall k=0,1$ and $\sigma = \pm$, hence 
\be
\label{X5Y5}
X_5 = Y_4 = 0 \ . 
\ee
It remains to compute the matrix $Y_3$ in \eqref{dersec}: taking the scalar product of the vectors $\sB^{[0,1]}(0,0,0)g_0^{+,[0,0]}(0)$ in \eqref{Bflat01ong00} with $g_0^{+,[0,1]}(0)$ in \eqref{exp:pa_mu_g}, one obtains
\begin{equation} \label{matY3}
    Y_3 = \begin{pmatrix}
    -\frac18 & 0 & \rvline & 0&0\\
    0 & -\frac18 & \rvline & 0 & 0 \\
    \hline
    0 & 0 & \rvline & 0 & 0 \\
    0 & 0 & \rvline & 0 & 0
\end{pmatrix} \, . 
\end{equation}
Then \eqref{B02_0} follows from its definition in \eqref{dersec} and \eqref{matZ}, \eqref{X5Y5}, \eqref{matY3}.

\smallskip

\noindent{\bf Expansion of  $\tB^{[2,0]}_0$ in \eqref{dersec2}.} Recalling \eqref{zeroentriesrev}, the  only possibly non-zero entries $[\tB^{[2,0]}_0]_{k,k'}^{\sigma,\sigma'}$ are those corresponding to $\sigma = \sigma'$. 
We now show that 
\be\label{B200}
\tB^{[2,0]}_0 =
\begin{pmatrix}
        \frac14 & 0 & \rvline & 0 & 0\\
        0 & \frac14 & \rvline & 0 & 0\\
        \hline
        0 & 0 & \rvline & 0 & 0\\
        0 & 0 & \rvline & 0 & 0
    \end{pmatrix}\, .
\ee
By \eqref{X5Y5}, $X_5 = Y_4 = 0$.
Next we compute $Z_1$. The operator $\sB^{[2,0]}$ in 
\eqref{expBflatij}
acts on the vectors  $g_k^{\sigma,[0,0]}(0)$ in  \eqref{gksig00} as 
\begin{equation}\label{actionsB20}
   \begin{aligned}
       &\sB^{[2,0]} g_1^{+,[0,0]}(0) = \frac12\vet{0}{ \sin (x)}\, , \quad \sB^{[2,0]} g_1^{-,[0,0]}(0) = \frac12\vet{0}{ \cos (x)}\, , \\ &\sB^{[2,0]} g_0^{+,[0,0]}(0) =\sB^{[2,0]} g_0^{-,[0,0]}(0) = 0\, .
   \end{aligned}
\end{equation}
Therefore taking the scalar products of the vectors in \eqref{actionsB20} and in \eqref{gksig00} for $Z_1$ as in \eqref{dersec2} we obtain 
\begin{equation}\label{matZ1}
    Z_1 = \begin{pmatrix}
        \frac14 & 0 & \rvline & 0 & 0\\
        0 & \frac14 & \rvline & 0 & 0\\
        \hline
        0 & 0 & \rvline & 0 & 0\\
        0 & 0 & \rvline & 0 & 0
    \end{pmatrix}\, .
\end{equation}
\noindent{\bf Expansion of $\tB^{[1,1]}_0$ in \eqref{dersec1}.} 
We show that 
\begin{equation}\label{B11iszero}
    \tB^{[1,1]}_0 = 0\, .
\end{equation}
Indeed, evaluating \eqref{exp:pa_rho_g} at $\e=0$, it results $g_k^{\sigma,[1,0]}(0) = 0 $ for any $k=0,1$, $\sigma = \pm$, and therefore, recalling \eqref{dersec1}, $X_7 = Y_6 = 0$. It is therefore only left to compute $Y_5$. Computing the scalar products between the vectors in \eqref{sB10ongksig00} with those in \eqref{exp:pa_mu_g} evaluated at $\e=0$ as for $Y_5$ in \eqref{dersec1}, we get that also $Y_5 = 0$.
\\[1mm]
\noindent{\bf Expansion of $\tB^{[\geq 3]}_{\al,\mu,\e}$ in \eqref{Bgeq3}.} Finally we show that \begin{equation}\label{Bgeq3pert}
    \tB^{[\geq 3]}_{\al,\mu,\e} = \cO(\rho^2\e) \, . 
\end{equation}
Indeed, evaluating \eqref{exp:Bg} at $\e=0$, yields  $\sB^{[0,0]}(0,0,0)g_k^{\sigma,[0,0]}(0) = 0$ for all $(k,\sigma)\in \{ (1,+),(1,-),(0,-)\}$ and 
$\sB^{[0,0]}(0,0,0)g_0^{+,[0,0]}(0)  = \vet{1}{0}$, whereas evaluating  
$g_k^{\sigma,[\geq 2]}(\al,\mu,\e)$ in \eqref{gksig2} at  $\e=0$ yields vectors with  zero average of order $\cO(\rho^2)$; consequently, by  \eqref{gksig2}, the matrix $Z_2$ in \eqref{Bgeq3} is of order $ \cO(\rho^2\e)$ and \eqref{Bgeq3pert} follow.
    \\[1mm]
  \noindent{\bf Preliminary expansion of the matrix $\tB_{\al,\mu,\e}$.}  
By \eqref{TaylorexpBemu}
and summing up the expansions \eqref{Bsoloeps}, \eqref{matB01},     \eqref{matB10},  \eqref{B02_0},  \eqref{B200},  \eqref{B11iszero}, \eqref{Bgeq3pert}    we deduce the expansion $\tB_{\al,\mu,\e}$ as in \eqref{splitB}, 
with 
\begin{equation}\label{expBpreliminary}
    \begin{aligned}
    &E = \begin{pmatrix}
    \e^2(1+r'(\e)) + \frac{\al^2}4-\frac{\mu^2}8 + r_1''(\rho^2\e,\rho^3) & \frac{\im\mu}{2}+\im \, r_2(\rho\e^2,\rho^2\e,\rho^3)\\
    -\frac{\im\mu}{2}-\im \, r_2(\rho\e^2,\rho^2\e,\rho^3)) & \frac14\al^2-\frac18\mu^2 + r_5(\rho^2\e,\rho^3)
\end{pmatrix}\\
    &G = \begin{pmatrix}
        1 + r_8(\e^3,\rho\e^2,\rho^2\e,\rho^3) & \im\, r_9(\rho\e^2,\rho^2\e,\rho^3) \\
        -\im \, r_9(\rho\e^2,\rho^2\e,\rho^3) & \rho + r_{10}(\rho^2\e,\rho^3)
    \end{pmatrix} \, ,  \quad F = \begin{pmatrix}
        r_3(\e^3,\rho\e^2,\rho^2\e,\rho^3) & \im \, r_4(\rho\e,\rho^3) \\
        \im \, r_6(\rho\e,\rho^3) & r_7(\rho^2\e,\rho^3)
    \end{pmatrix} \, . 
    \end{aligned}
\end{equation}
The next lemma proves 
that 
$ E, G $ are even in $ \e $ and $ F $ is odd in 
$ \epsilon $.

\begin{lemma}\label{lem:orderEFG}
\begin{equation}
\label{EFGTayolr}
E(\al,\mu,\e)=\!\! 
\sum_{\substack{\ell\geq 0, 
\ell \text{even}} } \!\! \!\! E_\ell(\al,\mu)\e^\ell \, , \quad 
G(\al,\mu,\e)= \!\! \sum_{\substack{\ell\geq 0,\ell \text{even}} }\!\!\!\! G_\ell(\al,\mu)\e^\ell \, , \quad
F(\al,\mu,\e)
= \!\! \sum_{\substack{\ell\geq 1,\,  
\ell \text{odd}} } \!\! \!\!  F_\ell(\al,\mu)\e^\ell  \, .  
\end{equation} 
\end{lemma}

\begin{proof}
     The matrices $ E, G, F $ in \eqref{splitB} can be written, 
     recalling    \eqref{matriceBame},
     \eqref{basechange}, \eqref{basisF},  as
        \begin{align}
        &[ E(\al,\mu,\e)]_{\sigma,\sigma'} = \left(\mathfrak{B}(\al,\mu,\e) 
    m_1^\sigma(\al,\mu,\e), 
    m_1^{\sigma'}(\al,\mu,\e)  \right) \, , \quad [ G(\al,\mu,\e)]_{\sigma,\sigma'} = \left(\mathfrak{B}(\al,\mu,\e) 
    m_0^\sigma(\al,\mu,\e), 
    m_0^{\sigma'}(\al,\mu,\e)  \right) \notag \\
    &[F(\al,\mu,\e)]_{\sigma,\sigma'} =\left(\mathfrak{B}(\al,\mu,\e) 
    m_0^\sigma(\al,\mu,\e), 
    m_1^{\sigma'}(\al,\mu,\e)  \right)\, , \qquad \sigma,\sigma'=\pm\, ,\label{EFGsigmasigmap}
    \end{align}
    where $\mathfrak{B}(\al,\mu,\e) := U_{\al,\mu,\e}^* \sB(\al,\mu,\e) U_{\al,\mu,\e} $ and 
    \begin{equation}\label{Uinvgisf}
        \begin{aligned}
        & m_1^+ (\al,\mu,\e) := f_1^+  \, , \qquad m_1^- (\al,\mu,\e) := f_1^-  - n(\al,\mu,\e) f_0^-  \, ,  \\
        & m_0^+ (\al,\mu,\e) := f_0^+  + n(\al,\mu,\e) f_1^+ \, ,\qquad m_0^- (\al,\mu,\e) := f_0^- \, ,
    \end{aligned} \text{with} \quad  f_k^\sigma \text{ in }\eqref{base3e}
    \end{equation}
By \Cref{TFjprop}, \eqref{Ureg}, \eqref{regGBL}, the operator
$  \mathfrak{B}(\al,\mu,\e)     $  belongs to $ \tF$, as well as $U_{\al,\mu,\e}^*U_{\al,\mu,\e}$. Moreover $f_1^\pm$ are supported on the Fourier harmonics $\pm 1$, while $f_0^\pm$ are supported on the harmonic $0$, so that
    \be\label{n.exp}
n(\al,\mu,\e) \stackrel{\eqref{namep}}{=} \frac{(U_{\al,\mu,\e}^*U_{\al,\mu,\e} f_1^-,f_0^-)}{(U_{\al,\mu,\e}^*U_{\al,\mu,\e} f_0^-,f_0^-)} \stackrel{\eqref{AFlele}}= \sum_{m \geq 0} n_{2m+1}(\al,\mu)\e^{2m+1}\,  \, , 
\ee
is odd in $ \epsilon $.
Inserting \eqref{n.exp},\eqref{Uinvgisf} in \eqref{EFGsigmasigmap} we deduce \eqref{EFGTayolr} using again \eqref{AFlele}.
 \end{proof}

\noindent 
{\bf Proof of \Cref{BexpG} concluded.} In view of \eqref{eq:altrealim} and the  preliminary expansion \eqref{expBpreliminary} the entries $E_{12}$, $F_{12}$, $F_{21}$, $G_{12}$ are proportional to $\mu$, so that
$$
\begin{aligned}
& E_{12}(\al,\mu,\e) = 
\frac{\mu}{2}(1+ r_2(\e^2,\rho\e,\rho^2)) \stackrel{\eqref{EFGTayolr}} =
\frac{\mu}{2}(1+ r_2(\e^2,\rho^2))
\\ 
& F_{12}(\al,\mu,\e)   = \mu r_4(\e,\rho^2) \stackrel{\eqref{EFGTayolr}}= \mu r_4(\e) \, , \quad 
F_{21}(\al,\mu,\e) 
= \mu r_6(\e,\rho^2) \stackrel{\eqref{EFGTayolr}} = \mu r_6(\e) \, , \\ 
& G_{12}(\al,\mu,\e)
= \mu r_9(\e^2,\rho\e,\rho^2)
\stackrel{\eqref{EFGTayolr}} = \mu r_9(\e^2,\rho^2) \, ,
\end{aligned}
$$
because  the matrices $E,G$  have only even powers of $ \e $, while $F$ has only odd powers of $\e$.
Similarly 
$$
\begin{aligned}
    E_{11}(\al,\mu,\e) &= \e^2(1+r'(\e)) +\frac{\al^2}4 - \frac{\mu^2}8 + r_1''(\rho^2\e,\rho^3) = \e^2(1+r(\e)) +\frac{\al^2}4 - \frac{\mu^2}8 + (\al^2+\mu^2)r(\e,\rho)\\
    &\stackrel{\eqref{EFGTayolr}} =\e^2(1+r_1'(\e^2)) +\frac{\al^2}4(1+r_1''(\e^2,\rho)) - \frac{\mu^2}8(1+r_1'''(\e^2,\rho)) 
\end{aligned}
$$
and 
$$
E_{22}(\al,\mu,\e) = 
\frac{\al^2}{4} -\frac{\mu^2}{8} + r_5(\rho^2\e,\rho^3)
\stackrel{\eqref{EFGTayolr}} = \frac{\al^2}{4}(1+r_5'(\e^2,\rho))-\frac{\mu^2}{8}(1+r_5(\e^2,\rho))\, .
$$
The expansions of 
$  G_{11}, G_{22}, F_{11}, F_{22}  $
in 
\eqref{BinG1}–\eqref{BinG3}
follow similarly.
The proof of \Cref{BexpG} is 
completed by the following lemma. 
\begin{lemma}
Property \eqref{E11E22}  holds. 
\end{lemma}

\begin{proof}
  Since $g_1^+(\al,\mu,0) = f_1^+(\al,\mu,0) = \frac{1}{\im\sqrt2}(v_1^+(\al,\mu) - v_1^-(\al,\mu))$ and $g_1^-(\al,\mu,0) = f_1^-(\al,\mu,0) = \frac{1}{\sqrt2}(v_1^+(\al,\mu) + v_1^-(\al,\mu))$ by  \eqref{vtrans}, 
  using that 
  $\sB(\al, \mu,0) = \cJ(\cL(\al,\mu,0) - \im\mu)$ in \eqref{cLsL} is a Fourier multiplier and the vectors $v_1^+(\al,\mu), v_1^-(\al,\mu)$ in \eqref{unperturbed.eigv} are supported on different Fourier harmonics, we get 
    \be
    E_{11}(\al,\mu,0)  
    =  \frac12(\sB(\al, \mu, 0)v_1^+(\al,\mu) , v_1^+(\al,\mu)) +  \frac12(\sB(\al, \mu, 0)  v_1^-(\al,\mu),  v_1^-(\al,\mu))
    \label{2310:1615} \, . 
    \ee
By \eqref{eig.0}, \eqref{unperturbed.eigv}, \eqref{cLsL} we have $ \sB(\al,\mu,0) v_1^\pm (\al,\mu) = -\im (\omega_1^\pm(\al,\mu)-\mu)\cJ v_1^\pm (\al,\mu)$
 which, inserted in \eqref{2310:1615} and exploiting the symplectic relations  \eqref{symp.nonzero}, yields  
    $
    E_{11}(\al,\mu,0)=  \tfrac12( \omega_1^-(\al,\mu)-\omega_1^+(\al,\mu)) $. The proof of \eqref{E11E22} for  $ E_{22}(\al,\mu,0)$ and $E_{12}(\al,\mu,0)$ is analogous. 
\end{proof}

\section{Block-decoupling}\label{sec:block}

By \Cref{BexpG} 
the $ 4 \times 4 $ Hamiltonian and reversible 
matrix $\tL_{\alpha,\mu,\e} = \tJ_4 \tB_{\alpha,\mu,\e} $  has the form 
\begin{equation}\label{Lepmu}
\tL_{\alpha,\mu,\e}= 
\tJ_4 \begin{pmatrix} 
 E &  F \\ 
 F^* &  G 
\end{pmatrix}
=
\begin{pmatrix} 
\tJ_2 E & \tJ_2 F \\ 
\tJ_2 F^* &\tJ_2 G 
\end{pmatrix} \, ,
\quad 
\forall (\al,\mu,\e)\in B_{\rho_0}(0,0)\times B_{\e_0}(0) \, , 
\end{equation}
where $ E, G, F $ are the $ 2 \times 2 $ matrices in \eqref{BinG1}-\eqref{BinG3}.
The goal of this  section  
is to  coniugate the matrix $\tL_{\al,\mu,\e}$ into a block-diagonal one, cf. \Cref{{completedec}}.

\begin{lemma}\label{decoupling1}
 {\bf (First step of block-decoupling)} 
 Conjugating the Hamiltonian and reversible matrix 
 $\tL_{\alpha,\mu,\e} = \tJ_4 \tB_{\alpha,\mu,\e} $ in \eqref{Lepmu}, through the symplectic, reversibility-preserving 
 $ 4 \times 4 $-matrix 
 \begin{align}\label{Ychange}
 Y = \uno_4 + m \begin{pmatrix} 0 & - P \\ Q & 0 \end{pmatrix}
 \  \text{with} \  \ 
 Q:=\begin{pmatrix} 1 & 0 \\ 0 & 0\end{pmatrix}  \, , 
 \   P:=\begin{pmatrix} 0 & 0 \\ 0 & 1\end{pmatrix} \, , 
 \ m := m_{\alpha,\mu,\e}:=-\frac{(F_{11})_{\alpha,\mu,\e}
 }{(G_{11})_{\alpha,\mu,\e}} \, ,  
\end{align} 
where  
\be\label{reg:m} 
[(\al,\mu,\e)\mapsto m_{\al,\mu,\e}]\in\cA(B_{\rho_0}(0,0),\e_0;\R) \, ,  \qquad    m_{\al, \mu,\e} =
m_{\al, - \mu,\e} \, , 
\qquad
 m(\alpha,\mu,\e) =  r(\e^3,  \rho^2\e ) \, , 
\ee
we obtain the Hamiltonian and   reversible matrix 
\begin{align}
&\tL^{(1)}_{\al,\mu,\e} := Y^{-1} \tL_{\al,\mu,\e} Y = \tJ_4\tB^{(1)} =
\begin{pmatrix}   \tJ_2 E^{(1)}  & \tJ_2  F^{(1)}  \\  
\tJ_2 [F^{(1)}]^*  &  \tJ_2 G^{(1)}  \end{pmatrix}  \, ,
\quad \overline{\tL^{(1)}_{\al,\mu,\e}} = \tL^{(1)}_{\al,-\mu,\e} \, ,  
\label{LinH}
  \end{align}
satisfying 
$ [(\al,\mu,\e) \mapsto \tL^{(1)}_{\al,\mu,\e}] \in \cA(B_{\rho_0}(0,0),\e_0;\C^{4\times 4}) $, 
 where the $ 2 \times 2 $  matrices  $E^{(1)} $, $ G^{(1)} $  have the expansions
 \eqref{BinG1}-\eqref{BinG2}   of $ E, G $ and
\begin{equation}\label{BinH}
 F^{(1)} = 
 \begin{pmatrix} 
 0 & \im  r_4(\rho\e)  \\
  \im  r_6(\rho\e)  & r_7(\rho^2\e) 
 \end{pmatrix}\, . 
 \end{equation} 
The matrices $E^{(1)}, G^{(1)}$ are even in $\e$ and $  F^{(1)}$ is odd in $\e$.
\end{lemma}

\begin{proof}
The function $m_{\al,\mu,\e}$ in \eqref{Ychange} satisfies   \eqref{reg:m} by \eqref{BinG2}-\eqref{BinG3} and \Cref{lem:prodintM}, which also implies that $[(\al,\mu,\e) \mapsto \tL^{(1)}_{\al,\mu,\e}] \in \cA(B_{\rho_0}(0,0),\e_0;\C^{4\times 4})$.
The matrix $Y^{-1} $ is symplectic, i.e.   $ Y^{-1}\tJ_4 Y^{-*} = \tJ_4 $,
and, since $m$ is real, it is also reversibility preserving according to  \Cref{def:revmat}.  
By   \eqref{B1conj} we have 
\begin{equation}\label{B1forma}
  -\tJ_4\tL^{(1)} =: \tB^{(1)} = Y^* \tB_{\al, \mu,\e} Y = \begin{pmatrix}  E^{(1)} & F^{(1)} \\ [F^{(1)}]^* &  
   G^{(1)}  \end{pmatrix},
\end{equation}
where, by \eqref{Ychange} and \eqref{LinH},  the self-adjoint matrices 
$E^{(1)}, G^{(1)}  $ are 
\begin{equation}\label{E1G1}
\begin{aligned}
& E^{(1)} := E +
\begin{pmatrix}  2 m F_{11} + m^2 G_{11} & - \im m F_{21} \\ \im m F_{21} &  
0  \end{pmatrix}   \, , 
\quad 
G^{(1)} := G + 
\begin{pmatrix}  0 &  \im m F_{21}  \\ 
- \im m F_{21}  &   - 2m F_{22} +m^2 E_{22}  \end{pmatrix} \, .
\end{aligned}
\end{equation}
Similarly,  the off-diagonal $ 2 \times 2 $ matrix $F^{(1)} $ is 
\begin{equation}\label{F1}
F^{(1)} := F + m (QG -  EP) - m^2 QF^*P =
 \begin{pmatrix}  0 &  \im (F_{12} + mG_{12} - mE_{12} + m^2 F_{21}) \\ 
 \im F_{21}  &  F_{22} - m E_{22}  \end{pmatrix}  \, , 
 \end{equation}
where  the first matrix entry   is $F_{11} + m G_{11} = 0 $, 
by the definition of $ m $ in \eqref{Ychange}. 
By \eqref{B1forma}-\eqref{F1} and \eqref{BinG1}-\eqref{BinG3}
we deduce  \eqref{LinH}, with $F^{(1)}$ in \eqref{BinH}. Since $ m (\alpha, \mu, \e)$ is odd in $ \epsilon $,  the matrices $E^{(1)}, G^{(1)}$ are even in $\e$ and $  F^{(1)}$ is odd in $\e$, as $ E, G, F$ are.
\end{proof} 

Note that  the entry $ F^{(1)}_{11} $ in \eqref{BinH} is identically $ 0 $, and
the other entries of $ F^{(1)}$ have  the same size as the corresponding ones in \eqref{BinG3}. Note also that 
the matrix 
$ \tL^{(1)}_{0,0,\e} $  in 
\eqref{LinH} is already block-diagonal.

\subsection{Second step of Block-decoupling}\label{sec:secondBlock}
We now perform a further step of block-decoupling, obtaining a new Hamiltonian and reversible matrix $\tL^{(2)} $, whose top left block $\tJ_2 E^{(2)}$ still carries the couple of Benjamin-Feir unstable eigenvalues, while the size of the off-diagonal block $\tJ_2 F^{(2)}$ is reduced compared to $\tJ_2 F^{(1)}$.
It is now crucial to use polar coordinates, since 
the block-decoupling procedure does not preserve the 
 class   $\cA(B_{\rho_0}(0,0),\e_0;\C^{4\times 4})$, but it  does preserve  
    the wider class of polar-analytic functions $ \cA_P (B_{\rho_0}(0,0),\e_0;\C^{4\times 4})$ in  \Cref{def:tildeM}.

\begin{lemma}\label{lem:secondstep}
 {\bf (Second step of block-decoupling)} 
 There exists 
     a Hamiltonian, reversibility preserving, polar-analytic  matrix
     of the form   \begin{equation}
    \label{def:S}
        S^{(1)} = \tJ_4 \begin{pmatrix}
        0 & \rvline & \Sigma \\
        \hline 
        \Sigma^* & \rvline & 0 
    \end{pmatrix}\, , \quad \Sigma = \tJ_2 X \, , \quad 
        X = \begin{pmatrix}
        x_{11} & \im x_{12}\\
        \im x_{21} & x_{22}
    \end{pmatrix} = \begin{pmatrix}
        r(\rho\e) & \im r(\rho\e)\\
        \im r(\e) & r(\rho\e)
    \end{pmatrix} \, , \quad x_{ij}\in\R\, ,
    \end{equation}
    odd in $\e$,
    such that      \begin{equation}\label{defL2}
        \tL^{(2)} := \exp (S^{(1)}) \tL^{(1)} \exp{(-S^{(1)})} = 
        \tJ_4 \begin{pmatrix}
             E^{(2)} & \rvline &  F^{(2)} \\
            \hline
             [F^{(2)}]^* & \rvline &  G^{(2)}
        \end{pmatrix}
    \end{equation}
is a Hamiltonian, reversible and polar-analytic matrix    
    where 
    \begin{equation}\label{EG1toEG2}
        E^{(2)} =E^{(1)} + \begin{pmatrix}
        r_1(\rho\e^2) & \im r_2(\rho^2\e^2)\\ 
        -\im r_2(\rho^2\e^2) & r_5(\rho^2\e^2)
    \end{pmatrix} \, , \qquad G^{(2)}=G^{(1)} + \begin{pmatrix}
        r_8(\rho\e^2) & \im r_9(\rho^2\e^2)\\ 
        -\im r_9(\rho^2\e^2) & r_{10}(\rho^2\e^2)
    \end{pmatrix}
    \end{equation}
    where
    $E^{(1)}$, $G^{(1)}$ in \Cref{decoupling1} have the expansions in 
    \eqref{BinG1}, \eqref{BinG2}
    of $ E, G $, 
    and 
    \begin{equation}\label{Fexp2}
        F^{(2)} =  \begin{pmatrix}
            r_3(\rho^2\e^3) & \im r_4(\rho^2\e^3)\\
             \im r_6(\rho^2\e^3) &   r_7(\rho^2\e^3)
        \end{pmatrix}\, .
    \end{equation}
    The matrices $E^{(2)}, G^{(2)}$ are even in $\e$ and $  F^{(2)}$ is odd in $\e$.
\end{lemma}

The rest of the section is devoted to the proof of Lemma \ref{lem:secondstep}. 

The matrix $S^{(1)}$ in \eqref{def:S}  is Hamiltonian and reversibility preserving, and so 
 $\exp (S^{(1)})$ is symplectic (cf. \cite[Lemma 3.13]{BMV1}) and reversibility preserving. Thus $\tL^{(2)}$ in \eqref{defL2} is Hamiltonian and reversible as $\tL^{(1)}$. 
 First we split $\tL^{(1)}$ in \eqref{LinH} into its block diagonal and off-diagonal parts
\begin{equation}\label{def:DR}
     \tL^{(1)}= D^{(1)} + R^{(1)}\, , \qquad  D^{(1)} :=\begin{pmatrix}
    \tJ_2 E^{(1)} & \vline & 0 \\
    \hline
    0 & \vline & \tJ_2 G^{(1)}
\end{pmatrix}\, , \qquad R^{(1)} :=  \begin{pmatrix}
    0 & \vline & \tJ_2 F^{(1)} \\
    \hline
    \tJ_2 [F^{(1)}]^* & \vline & 0
\end{pmatrix}\, .
\end{equation}
We expand in Lie series the matrix $\tL^{(2)} = \exp (S) \tL^{(1)} \exp{(-S)}$, denoting simply  $S:=S^{(1)}$, as 
\begin{equation}\label{Lieser1}
    \begin{aligned}
        \tL^{(2)}&=D^{(1)} + R^{(1)} + [S,D^{(1)}] + [S,R^{(1)}] + \frac12 [S,[S,D^{(1)}]] \\
        &\quad + \int_{0}^1 (1-\tau) \exp (\tau S) \textup{ ad}^2_S(R^{(1)}) \exp (-\tau S) \de\tau +
    \frac12 \int_{0}^1 (1-\tau)^2 \exp (\tau S) \textup{ ad}^3_S(D^{(1)}) \exp (-\tau S) \de\tau
    \end{aligned}
\end{equation}
where $\textup{ad}_A(B) := [A,B] = AB-BA$ denotes the commutator. We look for a matrix $ S $ as in \eqref{def:S} that solves the homological equation
\begin{equation}\label{homoeqst1}
  R^{(1)} +  [S,D^{(1)}]   = 0\, ,
\end{equation}
which, recalling \eqref{def:DR},  amounts to 
\begin{equation}\label{homo2}
    \begin{pmatrix}
        0 & \rvline & \tJ_2 F^{(1)} + \tJ_2 \Sigma \tJ_2 E^{(1)} - \tJ_2 G^{(1)} \tJ_2 \Sigma\\
        \hline
        \tJ_2 [F^{(1)}]^* + \tJ_2 \Sigma^* \tJ_2 G^{(1)}  - \tJ_2 E^{(1)}\tJ_2 \Sigma^* & \rvline & 0
    \end{pmatrix} = 0 \, . 
\end{equation}
Note that the  equations $\tJ_2 F^{(1)} + \tJ_2 \Sigma \tJ_2 E^{(1)} - \tJ_2 G^{(1)} \tJ_2 \Sigma = 0$ and $\tJ_2 [F^{(1)}]^* + \tJ_2 \Sigma^* \tJ_2 G^{(1)}  - \tJ_2 E^{(1)}\tJ_2 \Sigma^* = 0$ are equivalent. In particular, writing $\Sigma = \tJ_2 X$, equation \eqref{homo2} amounts to solve the ``Sylvester'' equation
\begin{equation}\label{sylveq0}
    \tJ_2 G^{(1)}  X - X \tJ_2 E^{(1)} = - \tJ_2 F^{(1)} \, . 
\end{equation}
Recalling \eqref{def:S}, this corresponds to solve
\begin{equation}\label{sylveq}
\underbrace{\begin{pmatrix}
        G^{(1)}_{12}-E^{(1)}_{12} & G^{(1)}_{11} & E^{(1)}_{22} & 0 \\
    G^{(1)}_{22} & G^{(1)}_{12} -E^{(1)}_{12} & 0 & -E^{(1)}_{22} \\
    E^{(1)}_{11} & 0 & G^{(1)}_{12}-E^{(1)}_{12} & - G^{(1)}_{11}\\
    0 & -E^{(1)}_{11} & -G^{(1)}_{22} & G^{(1)}_{12}-E^{(1)}_{12}
\end{pmatrix}}_{=:\cA}
\underbrace{\begin{pmatrix}
        x_{11}\\
        x_{12}\\
        x_{21}\\
        x_{22}
    \end{pmatrix}}_{=:\vec x}
    = \underbrace{\begin{pmatrix}
        -F_{21}^{(1)}\\
        F_{22}^{(1)}\\
        -F_{11}^{(1)}\\
        F_{12}^{(1)}
    \end{pmatrix}}_{=:\vec f} 
\end{equation}
where $F_{11}^{(1)} = 0$ 
by \eqref{BinH}. 

\begin{lemma}
The  matrix $ \cA $ in \eqref{sylveq} has determinant   
$ \det \cA = \rho^2(1+r(\e^2,\rho)) $. Thus for any $ \rho > 0 $
it is invertible 
and $\cA^{-1}$ has the expansion 
\begin{equation}\label{inverseA}
    \cA^{-1} = \frac1{\rho }
    \begin{pmatrix}
        \frac12 \rho\cos\theta   & 1 & r(\rho^2) & r(\rho^2) \\
        \rho & \frac12 \rho \cos\theta  & r(\rho^3) & r(\rho^2)\\
        r(\e^2,\rho^2) & r(\e^2,\rho^2) & \frac12 \rho \cos\theta & -1 \\
        r(\rho\e^2,\rho^3) & r(\e^2,\rho^2) & -\rho  & \frac12 \rho\cos\theta 
    \end{pmatrix}(1+r(\e^2,\rho))\, . 
\end{equation}
\end{lemma}

\begin{proof}
The matrix $ \cA $ in \eqref{sylveq} has the form  
$$ \cA
=\begin{pmatrix}
    a & b & c & 0\\
    d & a & 0 & -c\\
    e & 0 & a & -b\\
    0 & -e & -d & a
\end{pmatrix} 
 \qquad 
  \begin{aligned}
        &a :=G_{12}^{(1)}-E_{12}^{(1)} =-\frac12 \rho\cos(\theta) (1+r(\e^2,\rho^2))\, , \quad b := G_{11}^{(1)}=1+ r(\e^4,\rho\e^2,\rho^3)\, , \\
    &c := E_{22}^{(1)}= r(\rho^2)\, , \quad d := G_{22}^{(1)} = 
    \rho +r(\rho^2 \e^2,\rho^3) \, ,
    \quad 
    e := E_{11}^{(1)}=r(\e^2,\rho^2)\, , 
\end{aligned}
$$
using the expansions of $ E^{(1)}, G^{(1)}$ in \eqref{BinG1}, \eqref{BinG2}, cf. \Cref{decoupling1}. 
Thus, cf. \cite[Lemma 5.4]{BMV1}, its determinant is  
$ \det \cA 
= a^4 -2a^2 (bd+ce) +(bd-ce)^2  
= \rho^2(1+r(\e^2,\rho))  $
and, using \Cref{lem:prodintM},  
$$
    \cA^{-1} = \frac{1}{\det \cA}\begin{pmatrix}
        a(a^2-bd-ce) & b(-a^2+bd-ce) & -c(a^2+bd-ce) & -2abc\\
        d(-a^2+bd-ce) & a(a^2-bd-ce) & 2acd & c(a^2+bd-ce)\\
        -e(a^2+bd-ce) & 2abe & a(a^2-bd-ce) & b(a^2-bd+ce) \\
        -2ade & e(a^2+bd-ce) & d(a^2-bd+ce) & a(a^2-bd-ce)
    \end{pmatrix} \, . 
$$
Using the expansions of the coefficients $a,b,c,d,e $  we deduce \eqref{inverseA}.  
\end{proof}

\begin{lemma}\label{lemmaX}
For any $\rho\neq 0 $, there exists a unique solution  $\vec x = \cA^{-1}\vec f$ of \eqref{sylveq}, namely a solution $X$ of the Sylvester equation \eqref{sylveq0}
which is polar-analytic with expansion \eqref{def:S}.
Furthermore $ X $ is odd in $\e$.
\end{lemma}

\begin{proof}
 The coefficients $x_{ij} = [\cA^{-1}\vec f]_{ij}$ in  \eqref{sylveq} satisfy  \eqref{def:S}, by  \eqref{inverseA} and \eqref{BinH}. The entries $x_{ij}$ are odd in $\e$ since $\cA$  and $\vec f$ in \eqref{sylveq} are respectively even and odd in $\e$.
\end{proof}

Since  $S$ solves the homological equation \eqref{homoeqst1}, we deduce by \eqref{Lieser1} that 
\begin{equation}
\label{Lieser2}
    \tL^{(2)} = D^{(1)} + \frac12 [S,R^{(1)}] +\frac12 \int_0^1 (1-\tau^2)\exp (\tau S) \textup{ad}_S^2 (R^{(1)})\exp (-\tau S) \de \tau \, .
\end{equation}
In view of  \eqref{def:S} and  \eqref{def:DR}, 
\begin{equation}
\label{1stcorr}
    \frac12 [S,R^{(1)}] = 
 \tJ_4 \begin{pmatrix}
        \wt E & 0 \\
        0 & \wt G
    \end{pmatrix} \, , \quad 
  \wt E := \text{Sym}(
\tJ_2 X \tJ_2 [F^{(1)}]^* ) \, , \quad 
 \wt G = 
 \text{Sym}(
X^* F^{(1)}) \, ,  
\end{equation}
denoting $  \text{Sym} (A) := \tfrac12 (A+A^*)$. 

\begin{lemma}\label{lemmatEG}
    The  $ 2 \times 2 $ matrices $\wt E$, $\wt G$ in \eqref{1stcorr} are self-adjoint,  reversibility preserving 
    and polar-analytic  with expansion
\begin{equation}\label{expwtEG}
        \wt E = \begin{pmatrix}
            r_1(\rho\e^2) & \im r_2(\rho^2\e^2) \\
            -\im r_2(\rho^2\e^2) & r_5(\rho^2\e^2) 
        \end{pmatrix}\,  , 
        \quad  \wt G  = \begin{pmatrix}
            r_8(\rho\e^2) & \im r_9(\rho^2\e^2) \\
            -\im r_9(\rho^2\e^2) & r_{10}(\rho^2\e^2) 
        \end{pmatrix}\, . 
    \end{equation} 
\end{lemma}

 \begin{proof}
     Use \eqref{def:S} and  \eqref{BinH}.
 \end{proof}
 
We now show that the last term in \eqref{Lieser2}, which is Hamiltonian and reversible, is very small.

\begin{lemma}\label{lemmahatEFG}
    The  Hamiltonian, reversible and polar-analytic
    matrix 
    $$
        \frac12 \int_0^1 (1-\tau^2)\exp (\tau S) \textup{ad}_S^2 (R^{(1)})\exp (-\tau S) \de \tau = \tJ_4 \begin{pmatrix}
            \hat E & F^{(2)} \\
            [F^{(2)}]^* & \hat G
        \end{pmatrix}
    $$
    where  $ \hat E, \hat G $
    are $ 2 \times 2 $ self-adjoint,  reversibility preserving and polar-analytic  matrices 
\begin{equation}
\label{exphatEG}
\hat E = \begin{pmatrix}
        \hat E_{11} & \im\hat E_{12}\\
        -\im\hat E_{12} & \hat E_{22}
    \end{pmatrix} \, , \ \hat G = \begin{pmatrix}
        \hat G_{11} & \im\hat G_{12}\\
        -\im\hat G_{12} & \hat G_{22}
    \end{pmatrix}
\, , \quad 
       \hat E_{ij},\hat G_{ij} =  r(\rho^2\e^3)\, ,
       \ i,j =1,2 \, , 
  \end{equation}
   whereas $F^{(2)}$ is a  reversible $ 2 \times 2 $ polar-analytic matrix  with the expansion \eqref{Fexp2}. 
\end{lemma}

\begin{proof}
In view of  \eqref{def:S} and \eqref{1stcorr}
we get 
$$
        \textup{ad}_S^2 (R^{(1)}) = \tJ_4 \begin{pmatrix}
            0 & \wt F\\
            \wt F^* & 0
        \end{pmatrix} \qquad \text{ where }\qquad \wt F := 2(\Sigma \tJ_2 \wt G - \wt E \tJ_2 \Sigma)  \, .
$$
    By \eqref{expwtEG} and \eqref{def:S}, we deduce that
     $\wt F = r (\rho^2 \e^3)$.
    Since $\exp(\tau S)$ is bounded
    we deduce \eqref{exphatEG} and \eqref{Fexp2}.
\end{proof}

Lemma \ref{lem:secondstep} follows by \Cref{lemmaX,lemmatEG,lemmahatEFG}. The only off-diagonal terms in \eqref{Lieser2} are given by $F^{(2)}$ in  \Cref{lemmahatEFG}, while the corrections to the diagonal part $ D^{(1)}$ in   \eqref{def:DR} given by \Cref{lemmatEG,lemmahatEFG} are already perturbative
(recall that 
$E^{(1)} $, $ G^{(1)}$
have the  expansions
 \eqref{BinG1},\eqref{BinG2} by  \Cref{decoupling1}).
 The parity properties in $\e$ of $E^{(2)},G^{(2)},F^{(2)}$ follow from the following consideration. The product of matrices 
 $  X(\e),Y(\e)$ of the form 
 $$
 \begin{pmatrix}
     {\rm Even}_{2\times2}(\e) & \rvline & {\rm Odd}_{2\times2}(\e)\\
     \hline 
     {\rm Odd}_{2\times2}(\e) & \rvline & {\rm Even}_{2\times2}(\e)
 \end{pmatrix}$$
    where ${\rm Even}_{2\times2}(\e) , {\rm Odd}_{2\times2}(\e)$ are $2\times 2$ matrices, respectively even and odd in $\e$, has the same form, specifically
\begin{equation}\label{evod:properties}
        X(\e)\cdot Y(\e) = \begin{pmatrix}
     {\rm Even}_{2\times2}(\e) & \rvline & {\rm Odd}_{2\times2}(\e)\\
     \hline 
     {\rm Odd}_{2\times2}(\e) & \rvline & {\rm Even}_{2\times2}(\e)
 \end{pmatrix}\, , \qquad \exp (X(\e)) = \begin{pmatrix}
     {\rm Even}_{2\times2}(\e) & \rvline & {\rm Odd}_{2\times2}(\e)\\
     \hline 
     {\rm Odd}_{2\times2}(\e) & \rvline & {\rm Even}_{2\times2}(\e)
 \end{pmatrix}\, .
    \end{equation}
\subsection{Complete Block-decoupling} \label{sec:5.2}

We now block diagonalize the matrix $ \tL^{(2)}$ in \eqref{defL2} which we  split 
as
\begin{equation}
\label{L2D2R2}
     \tL^{(2)} = D^{(2)} + R^{(2)}
\end{equation}
with 
\begin{equation}
\label{D2E2f}
D^{(2)} :=\begin{pmatrix}
    \tJ_2 E^{(2)} & \vline & 0 \\
    \hline
    0 & \vline & \tJ_2 G^{(2)}
\end{pmatrix}\, , \qquad R^{(2)} :=  \begin{pmatrix}
    0 & \vline & \tJ_2 F^{(2)} \\
    \hline
    \tJ_2 [F^{(2)}]^* & \vline & 0
\end{pmatrix}\, .
\end{equation}  
The matrices $D^{(2)}$ and $R^{(2)} $ are Hamiltonian and reversible. 

\begin{lemma}\label{completedec}
   {\bf (Complete block-decoupling)} There exist $\rho_1, \epsilon_1 > 0 $ such that for any $(\al,\mu,\e)\in B_{\rho_1}(0,0)\times B_{\e_1}(0)$ there exists a $4\times 4$ Hamiltonian,  reversibility preserving, polar-analytic matrix
\begin{equation}\label{def:Scomplete}
        S^{(2)} = \tJ_4 \begin{pmatrix}
        0 & \vline & \Sigma^{(2)} \\ 
        \hline
        [\Sigma^{(2)}]^* & \vline & 0 
    \end{pmatrix}\, , \quad \Sigma^{(2)} = \tJ_2 X^{(2)}\, , \quad  X^{(2)} = \begin{pmatrix}
        x_{11}^{(2)} & \im x_{12}^{(2)} \\
        \im x_{21}^{(2)} & x_{22}^{(2)}
    \end{pmatrix}\, , \quad x_{ij}^{(2)} = r_{ij}(\e^2) \, , 
    \end{equation}
 even in $\e$,    such that 
     $$
        \tL^{(3)} := \exp{(\rho\e S^{(2)})} \tL^{(2)} \exp{(-\rho\e S^{(2)} )} 
  = 
        \tJ_4 \begin{pmatrix}
             E^{(3)} & \rvline &  0 \\
            \hline
           0 & \rvline &  G^{(3)}
        \end{pmatrix}
  $$
is a Hamiltonian, reversible and polar-analytic matrix    of the form   \begin{equation}\label{EG2toEG3}
        E^{(3)} =E^{(2)} +  \begin{pmatrix}
        r_1(\rho^3 \e^6) & \im r_2(\rho^3 \e^6)\\ 
        -\im r_2(\rho^3 \e^6) & r_5(\rho^3 \e^6)
    \end{pmatrix} \, , \qquad G^{(3)}=G^{(2)} + \begin{pmatrix}
        r_8(\rho^3 \e^6) & \im r_9(\rho^3 \e^6) \\ 
        -\im r_9(\rho^3 \e^6) & r_{10}(\rho^3 \e^6)
    \end{pmatrix} \, , 
    \end{equation}
    with
    $E^{(2)}$, $G^{(2)}$ in 
    \eqref{EG1toEG2}. The matrices $E^{(3)}, G^{(3)}$ are even in $\e$. 
\end{lemma}

\begin{proof}
We prove that 
\begin{equation}
\label{diagonalizationins}
      \tL^{(3)}  = 
      \exp{(\rho\e S^{(2)})} \tL^{(2)} \exp{(-\rho\e S^{(2)} )} 
 =   D^{(2)} + P  
\ee
where  $D^{(2)}$ is 
in \eqref{L2D2R2} and  $ P $ 
    is  a $4\times 4$ Hamiltonian, reversible 
    and polar-analytic block diagonal matrix
\begin{equation}\label{PsmallBD}
    P=\begin{pmatrix}
        \cO(\rho^3\e^6)  & \rvline & 0\\
        \hline 
        0 & \rvline & \cO(\rho^3\e^6)
    \end{pmatrix} \, , \quad \text{even in} \ \e \, .
    \end{equation}
    For simplicity we 
    denote $ S := S^{(2)} $. Equation \eqref{diagonalizationins} is equivalent to the system
    \begin{equation}\label{diagequiv}
        \begin{cases}
            \Pi_D (\exp{(\rho\e S)} \tL^{(2)} \exp{( -\rho\e S )}) = D^{(2)} + P \\
            \Pi_\varnothing (\exp{(\rho\e S)} \tL^{(2)} \exp{( -\rho\e S )}) = 0  
        \end{cases}
    \end{equation}
    where $\Pi_D$ and $\Pi_\varnothing$ are the projections respectively  on the block diagonal and off diagonal  matrices. 
 Expanding the second equation in \eqref{diagequiv} in Lie series, we get by \eqref{L2D2R2}, 
\begin{equation}\label{nonlin_homo}
        R^{(2)} + \rho\e [S,D^{(2)}] + \rho^2 \e^2 \underbrace{\Pi_\varnothing \int_0^1 (1-\tau) \exp{(\tau\rho\e S)} \textup{ad}_S^2 ( 
        D^{(2)} + R^{(2)}) \exp{(-\tau\rho\e S )}\de \tau }_{=:\cR(\rho,\theta,\e;S)}= 0
    \end{equation}
    where the remainder $\cR$ is quadratic in $S$. 
    Since $S$ is  reversibility preserving, $[S,D^{(2)}]$ and $\cR$ are block off-diagonal Hamiltonian and reversible matrices. 
    By denoting
    \begin{equation}\label{2dsub}
       \vec x := (
        x_{11},
        x_{12},
        x_{21},
        x_{22})^\top\, , \quad \rho^2\e^2 \vec v(\rho, \theta, \e) := (
        -F_{21}^{(2)},
        F_{22}^{(2)},
        -F_{11}^{(2)},
        F_{12}^{(2)})^\top\, ,
    \end{equation}
    and $\vec g(\rho, \theta, \e, \vec x)$ the 
    vector associated to the entries of the Hamiltonian and reversible block off-diagonal matrix $\cR$, using the same ordering as for $\vec v(\rho, \theta, \e)$ in \eqref{2dsub}. 
    The vectors $\vec v(\rho, \theta, \e)$ and $\vec g(\rho,\theta,\e;x)$ are analytic in every entries, $\vec v$ is of size $\cO(\e)$ thanks to \eqref{Fexp2}, while $\vec g$  is quadratic in $\vec x$. Equation \eqref{nonlin_homo} is equivalent to the four dimensional system
    \begin{equation}\label{reducedhomoeq}
        \cA \vec x +\rho\e \, \vec v(\rho, \theta, \e) +\rho\e \,  \vec g(\rho, \theta, \e, \vec x) = 0\, ,
    \end{equation}
    where $\cA$ is the matrix representing the linear application $S\to [S,D^{(2)}]$ in \eqref{sylveq}, with inverse 
    $\cA^{-1}=\frac1\rho \cT (\rho,\theta,\e)$, with $\cT$ polar-analytic, as in \eqref{inverseA}.  Multiplying \eqref{reducedhomoeq} by $\cA^{-1}$, we get 
    \be\label{eqfinab}
    \vec x = -  \e \cT(\rho,\theta,\e)\vec v(\rho, \theta, \e) - \e \cT(\rho,\theta,\e)\vec g(\rho, \theta, \e, \vec x) \, .
    \ee
    Since  $\vec v  = \cO(\e)$,  by the analytic implicit function theorem, for $\rho$ and $\e$ sufficiently small,  there is a unique small solution $ \vec x= \cO(\e^2) $ of \eqref{eqfinab}, i.e. we obtained $ S $ as in \eqref{def:Scomplete}. We now show \eqref{PsmallBD}. 
   We  Lie expand \eqref{diagonalizationins} as
    $$
    \tL^{(3)} = 
    D^{(2)} + 
    \underbrace{R^{(2)}
     + \rho \e[S,D^{(2)}]}_{
    \stackrel{\eqref{nonlin_homo}} =- \rho^2 \e^2 {\cal R}(S)}
     + \rho \e[S,R^{(2)}] + \frac12 \rho^2\e^2\textup{ad}_S^2(\tL^{(2)}) + \rho^3\e^3\cO(S^3)
    $$ 
    and thus
\begin{equation}\label{est:Pcompdec}
        P \stackrel{\eqref{diagonalizationins}} = \tL^{(3)} - D^{(2)} = \underbrace{\rho\e[S,R^{(2)}]}_{ =\cO(\rho^3\e^6) \mbox{ by } \eqref{D2E2f}, \eqref{Fexp2}} + \frac12 \rho^2\e^2 \textup{ad}_S^2(D^{(2)}) + \underbrace{\rho^3\e^3\cO(S^3)}_{\cO(\rho^3\e^9)} \, ,
    \end{equation}
    since by \eqref{def:Scomplete} 
    we have $S=\cO(\e^2)$. 
    Finally we estimate $ \rho^2\e^2 \textup{ad}_S^2(D^{(2)})$.
    Taking the commutator of \eqref{nonlin_homo} with $\rho\e S$, using again \eqref{Fexp2}, \eqref{def:Scomplete} and that $\cR$ in \eqref{nonlin_homo} is quadratic in $S$, and isolating the second term we get
    $$
    \rho^2\e^2 \textup{ad}_S^2(D^{(2)}) = -\rho\e [S,R^{(2)}] - \rho^3\e^3 [S,\cR(\rho,\theta,\e;S)] = \cO(\rho^3\e^6)\, .
    $$
    Hence the matrix $P$  in \eqref{est:Pcompdec} satisfies \eqref{PsmallBD}.    Finally, the solution $\vec x$ of \eqref{reducedhomoeq} is odd in $\e$, since $\vec v$ in \eqref{2dsub} is odd in $\e$ and, using \eqref{evod:properties} for $\cR(\rho,\theta,\e;S)$ in \eqref{nonlin_homo},  $\vec g (\rho,\theta,-\e,\vec x) = - \vec g (\rho,\theta,\e,\vec x)$. Thus, using again \eqref{evod:properties}, one deduces that $P$ in \eqref{PsmallBD} is even in $\e$.
\end{proof}

\subsection{Proof of Theorem \ref{TeoremoneFinale}}\label{sec:proofteoremone}

The basis $\{h_1^\pm(\al,\mu,\e),h_0^\pm(\al,\mu,\e)\}$ is obtained transforming the symplectic and reversible basis $\cG$ in \eqref{basisG}-\eqref{basechange}, which satisfies \eqref{decG},  with  the changes of variables  of  \Cref{decoupling1,lem:secondstep,completedec}, which are the identity for $ \epsilon = 0 $. Then \eqref{proph}  follows by 
\eqref{nmuep}, \eqref{reg:m}, \eqref{def:S},  \eqref{def:Scomplete}
and \Cref{expansion1}. The identity  \eqref{basiepzero}
follows by 
\eqref{vtrans}.
The action of the  operator $ \cL(\al,\mu,\e) = \im \mu + \sL(\al,\mu,\e) $  in \eqref{cLsL} on $  \mathcal{V}_{\al, \mu, \e}  $, is represented,   
in the basis $ \cH $, 
by  the matrix 
$$ 
\tL(\alpha,\mu,
\epsilon) = \im \mu + \tL^{(3)}(\al,\mu,\e) \, , \quad \tL(\al,\mu,0) =
 \im \mu + \tJ_4 \tB_{\al,\mu,0} \, , 
$$ 
where $ \tL^{(3)}(\al,\mu,\e) $ is given by \Cref{completedec}
and $\tB_{\al,\mu,\e}$ is the matrix   \eqref{splitB}. 
Thus the matrices in \eqref{matricefinae} 
are $ \mathtt U = \im \mu  +   
\tJ_2  E^{(3)} $ and
$ \mathtt S = \im \mu  + 
\tJ_2  G^{(3)} $ where 
$ E^{(3)}, G^{(3)} $ 
are in \Cref{completedec}, \eqref{EG2toEG3}, \eqref{EG1toEG2} and \Cref{decoupling1}.
Thus \eqref{UU}, \eqref{S}, 
\eqref{exp:a}, \eqref{exp:b+-}-\eqref{exp:b-}, 
\eqref{siannu} 
follow using also \eqref{E11E22}. 
In particular the matrix  $E^{(3)} =  E^{(1)} + \cO(\rho \e^2)$
has a Lipschitz extension in a neighborhood of $(\al,\mu) = (0,0)$ by \Cref{lem:prodintM} $(v)$ and recalling that $\tL^{(1)}_{\al,\mu,\e}$ in \eqref{LinH} belongs to $\cA(B_{\rho_0}(0,0),\e_0;\C^{4\times 4})$ and has all Lipschitz entries. 
The form of the eigenvalues \eqref{eignear0} is a direct consequence of  \eqref{UU}.  

Finally we prove \eqref{graphsmclean2}. 
In view of \eqref{exp:b+-}  the equation $\mathsf{b}^+(\al, \mu, \e) = 0$ reads,  in  polar coordinates  \eqref{polar}, 
\be\label{tanteta}
\tan^2  (\theta) =
\frac12 (1 + r(\e^2,\rho))  \,  \quad   \Rightarrow  \quad 
\tan  (\theta) = \pm 
\frac{1}{\sqrt{2}} 
(1 + r(\e^2,\rho)) \, . 
\ee
By the analytic implicit function theorem we solve 
\eqref{tanteta}  getting
$ \theta(\rho,\e) = \pm \text{arctan}(\frac{1}{\sqrt{2}}) + r(\e^2,\rho) $. Thus
$ \mu = \rho(\al,\e) \cos (\theta(\rho(\al,\e),\e))$
where $\rho(\al,\e)$ is an analytic function defined for 
$|\alpha | <  \alpha_0 $ small,  solving implicitly the equation  
$\al = \rho \sin (\theta(\rho,\e))$    for $ (\al,\e) $ small. 
This proves that $
\cM^{(2)}_{+,\rm{loc}} (\e)$ is described 
 in \eqref{graphsmclean2}.

Regarding $\cM^{(2)}_{-,{\rm loc}}(\e)$, we look for solutions of 
\begin{equation}\label{b-e}
    -\tb^-(\al,\mu,\e) = \e^2 (1+r_1(\e^2,\rho)) + \frac{\al^2}4 (1+r_1'(\e^2,\rho)) - \frac{\mu^2}{8}(1+r_1''(\e^2,\rho)) = 0\, .
\end{equation}
The  set of  solutions $ (\alpha, \mu ) $ 
of \eqref{b-e}
is close to an hyperbola with two symmetric components, one with $ \mu \geq 0 $, and another one with $ \mu < 0 $. Let us consider the case 
$ \mu \geq 0 $. 
Dividing \eqref{b-e} by $ 1+r_1(\e^2,\rho) $ and isolating the last term, this amounts to solve
$$
G (\al,\mu,\e) := \mu (1 + r'(\e^2,\rho)) - \sqrt{ 8\e^2 + 2 \al^2 (1+ r''(\e^2,\rho)) }  = 0 \, .
$$
Note that the function 
$ G (\al,\mu,\e) $ has Lipschitz extension setting  $ G (0,0,\e) := 2 \sqrt{2}|\epsilon |$.
We now apply a  Lipschitz Implicit Function Theorem, see e.g. \cite[Theorem 4.8]{simader}. 
We compute the 
derivatives of a polar analytic function $r (\e^2,\rho) = \e^2 g_1(\rho,\theta,\e) + \rho g_2 (\rho,\theta,\e)$,
where  $g_1$, $g_2$ are 
analytic functions of their arguments, bounded in $(\rho,\theta,\e) \in B_{\rho_1}(0)\times \T \times B_{\e_1}(0)$, obtaining  
$$
    \pa_\mu r(\e^2,\rho) 
    = \frac{\mu}{\rho} \left( \e^2 \pa_\rho g_1  + g_2 + \rho \pa_\rho g_2  \right) - \frac{\al}{\rho^2}\left( \e^2 \pa_\theta g_1  + \rho \pa_\theta g_2 \right)  = \frac{1}{\rho} r(\e^2,\rho)\, .
$$
Similarly $\pa_\al r(\e^2,\rho) = \frac{1}{\rho} r(\e^2,\rho)$, $\pa_\e r(\e^2,\rho) = r(\e,\rho)$.
Thus 
$$
\pa_ \mu G (\al,\mu,\e) = 1 + r(\e^2,\rho) + \frac{\mu}{\rho}r(\e^2,\rho) - \frac{2\frac{\al^2}{\rho}r(\e^2,\rho)}{2\sqrt{ 8\e^2 + 2 \al^2 (1+ r''(\e^2,\rho)) }}\quad \Rightarrow \quad  1/2 \leq \pa_ \mu G (\al,\mu,\e) \leq 2 \, ,$$
and  
$$ 
\begin{aligned}
    &|\pa_\al G (\al,\mu,\e)| = \left|\frac{\mu}{\rho}r(\e^2,\rho) -  \frac{4\al(1+r(\e^2,\rho)) + 2\frac{\al^2}{\rho}r(\e^2,\rho)}{2\sqrt{ 8\e^2 + 2 \al^2 (1+ r''(\e^2,\rho)) }}\right| \leq 2 \ , \\ &|\pa_\e G (\al,\mu,\e)| = \left|\mu r(\e,\rho) - \frac{16\e + 2\al^2 r(\e,\rho)}{2\sqrt{ 8\e^2 + 2 \al^2 (1+ r''(\e^2,\rho)) }} \right| \leq 5\, .
\end{aligned}
$$
The Lipschitz Implicit Function \cite[Theorem 4.8]{simader} implies the  existence of a  unique Lipschitz function $\mu^-(\al,\e)$ 
defined for 
$ |(\alpha, \e) |$ small enough, 
such that 
$ G(\al,\mu^-(\al,\e) ,\e) = 0 $
satisfying  $\mu^-(0,0) = 0 $. 
Actually $\mu^-(\cdot,\e)$ is analytic for any 
$ \e \neq 0 $.  
Indeed, by  \eqref{b-e} we have  
$  \e^2+\al^2 \leq \mu^-(\al,\e) $ and so, for any $\e\neq 0$, the graph of $\mu^-(\al,\mu) > 0 $  belongs to the region where $\tb^- (\alpha, \mu, \e)$ is analytic. So 
$ \mu^-(\al,\e) $ is analytic 
by the implicit function theorem.
By \eqref{b-e}, we have 
$
\mu^-(\al,\mu) = \sqrt{8\e^2 (1 + r'(\e^2,\rho)) + 2\al^2 (1+r''(\e^2,\rho))}
$, 
and substituting $\mu = \mu^-(\al,\e) 
\leq 10(\e^2+\al^2)  $ in the remainders we obtain the expression in \eqref{graphsmclean2}, with remainders $|\ell_i(\al,\e)|\lesssim |\al|+|\e|$, analytic for every fixed $0 < |\e| < \e^{(1)}$. The Lipschitzianity of $(\al,\e)\mapsto\ell_i (\al,\e)$ follows by taking the $\al$ and $\e$ derivatives and using that $\rho = (\al^2+ (\mu^-(\al,\e))^2)^\frac12 \geq |\mu^-(\al,\e)| \geq |\e|$.
This concludes the proof of \Cref{TeoremoneFinale}.

 \part{McLean instabilities}\label{part:II}
 
In this part we prove \Cref{TeoremoneMcLean}, describing the $ 3 d $  spectral bands  for 
any $ (\alpha, \mu) $ 
in  a {\it whole} neighborhood 
of the   unperturbed McLean curves $ \cM^{(\tp)}$,  $ \tp \geq 2 $.

\section{Perturbative approach to  McLean eigenvalues}\label{sec:away}

We  now construct the spectral Kato projectors $ P^{(\tp)}_{\al,\mu,\e} $ 
associated to the $2$-dimensional spectral subspaces relative to the eigenvalues \eqref{McL.d}  for  
any $ (\alpha,\mu)$ 
in  a whole neighborhood $ K^{(\tp)} $  of $ \cM^{(\tp)} $, see \eqref{Pproj:all}, for any  $\tp \geq 2$. 
 A special attention  is required near the  McLean curve 
 $ \cM^{(2)} $ where the quadruple eigenvalue collision 
 \eqref{4collision} occurs at the origin. 
In the next  lemma we first construct the Kato projectors  for  $ (\alpha,\mu)$   {\it away} from zero. 


\begin{lemma}\label{KatonearMcLean}
{\bf (Kato theory for separated eigenvalues away from $ 0 $)} 
    For any $\tp \geq 2$ there exists a compact neighborhood $K^{(\tp)}$ of the McLean curve $\cM^{(\tp)}$ defined in \eqref{mcleanmanifoldsp0}, such that the following hold true:

    \vspace{2mm}
    \noindent $\bullet$ {\sc Case $\tp=2$.} For any $\rho_2>0$ there exist $r_2>0$ and $  \wt \e^{(2)} (\rho_2) > 0 $  such that for any $(\al,\mu)\in  K^{(2)}\setminus B_{\rho_2}(0,0)$ and 
        $ |\e| < \wt \e^{(2)} (\rho_2 )$, the curve $\Gamma^{(2)}(\al,\mu) := \pa B_{r_2}(\lambda^{(2)}_+(\al,\mu))$, where $\lambda^{(2)}_+ (\al,\mu) $ is defined in \eqref{McL.d},  belongs to the resolvent set of $\cL(\al,\mu,\e)$ 
  and the operators 
  \begin{equation}\label{Pproj:2}
            \wt P^{(2)}_{\al,\mu,\e} := \frac{1}{2\pi\im}\oint_{\Gamma^{(2)}
(\al,\mu)} (\lambda - 
\cL(\al,\mu,\e))^{-1} \de\lambda : L^2 \to H^1   
        \end{equation}
      are  well defined projectors 
      commuting with 
        $\cL(\al,\mu,\e)$, i.e. 
  $  [\wt P^{(2)}_{\al,\mu,\e}]^2 
 = 
 \wt P^{(2)}_{\al,\mu,\e} $, 
 $  \wt P^{(2)}_{\al,\mu,\e}\cL(\al,\mu,\e) = \cL(\al,\mu,\e) 
 \wt P^{(2)}_{\al,\mu,\e} $. 
 In addition  
        $ \wt P^{(2)}_{\al,\mu,\e} $
       is  skew-Hamiltonian and reversibility-preserving, i.e. 
$$
 \cJ \wt P^{(2)}_{\al,\mu,\e} = [\wt P^{(2)}_{\al,\mu,\e}]^* \cJ\, ,  \quad  
  \wt P^{(2)}_{\al,\mu,\e} \circ \varrho_c  = \varrho_c\circ \wt P^{(2)}_{\al,\mu,\e} \,   ,
  \quad 
   \varrho_c \ 
   \text{in} \  \eqref{def:cinvolution} \, . 
$$
  The map 
  $$
    (K^{(2)}\setminus B_{\rho_2}(0,0)) \times B_{\wt \e^{(2)}(\rho_2 )}(0) \to  \cL(L^2,H^1) \, ,\ (\al,\mu,\e)
        \mapsto \wt P^{(2)}_{\al,\mu,\e}\, , 
        $$ 
        is analytic
        and belongs to $ \tF $,
            according to \Cref{defFell}.
        The subspace $\wt \cV^{(2)}_{\al,\mu,\e}:=\textup{Rg} \wt P^{(2)}_{\al,\mu,\e}$ is $ 2 $-dimensional, symplectic, invariant under $\cL(\al,\mu,\e)$, and
        \be\label{biscot3}
        \sigma (\cL(\al,\mu,\e))\cap \big\{ z\in \C \ \text{inside } \Gamma^{(2)}(\al,\mu) \big\} = \sigma(\cL(\al,\mu,\e)\vert_{\wt \cV^{(2)}_{\al,\mu,\e}}) \, . 
        \ee
        \noindent $\bullet$ {\sc{Cases $\tp\geq 3$. }} There exist $r_\tp>0$ and $\e^{(\tp)}>0$ such that for any $(\al,\mu)\in K^{(\tp)}$ and $|\e| \leq  \e^{(\tp)} $ the curve $\Gamma^{(\tp)}(\al,\mu) := \pa B_{r_\tp}(\lambda^{(\tp)}_+(\al,\mu))$, where $\lambda^{(\tp)}_+ (\al,\mu) $ is defined in \eqref{McL.d},  belongs to the resolvent set of 
        $\cL(\al,\mu,\e)$ and the operators 
        \begin{equation}\label{Pproj:tp}
            P^{(\tp)}_{\al,\mu,\e} := \frac{1}{2\pi\im}\oint_{\Gamma^{(\tp)}
(\al,\mu)} (\lambda - 
\cL(\al,\mu,\e))^{-1} \de\lambda : L^2 \to H^1  
        \end{equation}
        are well defined  projectors,  commuting with 
        $\cL(\al,\mu,\e)$, skew-Hamiltonian and reversibility-preserving.  
        Moreover
    \be\label{defPp} 
        P^{(\tp)}_{\al,\mu,\e} \in  \cA(K^{(\tp)},\e^{(\tp)};L^2,H^1) \cap \tF 
        \ee
          according to \Cref{def:tildeM,defFell}. 
        Each subspace $\cV^{(\tp)}_{\al,\mu,\e} :=\textup{Rg} P^{(\tp)}_{\al,\mu,\e} $ is  $2$-dimensional,  symplectic, invariant under $\cL(\al,\mu,\e)$,  and
        $$
        \sigma (\cL(\al,\mu,\e))\cap \big\{ z\in \C \ \text{inside } \Gamma^{(\tp)}(\al,\mu) \big\} = \sigma(\cL(\al,\mu,\e)\vert_{\cV^{(\tp)}_{\al,\mu,\e}})
        \, . 
        $$

\end{lemma}

\begin{proof}
We deal first with the case  $\tp\geq 3$.  By Lemma \ref{separation_eigenvalues}, there is an open neighborhood $\cN^{(\tp)}$ of  $\cM^{(\tp)}$ such that, for any 
    $(\al, \mu) \in \cN^{(\tp)}$, the  eigenvalues 
$\lambda^{(\tp)}_\pm(\al,\mu)$ in \eqref{McL.d}  are at distance at least $ \tc_\tp > 0 $ from the rest of the spectrum according to   \eqref{0905:1850}.
We then define the 
     compact set $K^{(\tp)}\subset \cN^{(\tp)}$ so that  
     \begin{equation}\label{separationinproof}
         |\lambda^{(\tp)}_+(\al,\mu)-\lambda^{(\tp)}_-(\al,\mu)|\leq \tc_\tp/3 \, , \quad \forall (\al, \mu) \in K^{(\tp)} \quad   \text{with} \  
\tc_\tp > 0 \ \text{defined in} \  \eqref{0905:1850} \, . 
\end{equation}
In this way,  setting 
$r_\tp :=  \frac23\tc_\tp $, the closed, counterclockwise oriented circuit 
     \be\label{Gpincl}
     \Gamma^{(\tp)}(\al,\mu) = \pa B_{r_\tp}(\lambda^{(\tp)}_+(\al,\mu)) 
     \subset \rho(\cL(\al,\mu,0)) \quad (\text{= the  resolvent set of }\cL(\al,\mu,0))\, ,  
     \ee
     isolate the eigenvalues $\lambda^{(\tp)}_\pm(\al,\mu)$ from the rest of the spectrum
     of $\cL(\al,\mu,0)$,
     for any $ (\al, \mu) \in K^{(\tp)}$.

We now analyze  the resolvent    $(\lambda-\cL(\al,\mu,\e))^{-1}$ for any $\lambda \in \Gamma^{(\tp)}(\al, \mu)$. In view of \eqref{decoGsha},
   we decompose $ \cL(\al,\mu,\e) $ in  \eqref{cLame1} as 
    $$
\lambda-\cL(\al,\mu,\e) = \lambda - \cL(\al,\mu,0) -\cR(\al,\mu,\e)  \quad \text{where} \quad     
    \cR(\al,\mu,\e) := \begin{bmatrix}
        (\pa_x+\im\mu)p_\e(x) & \cG^\sharp (\al,\mu,\e) \\
        -a_\e(x) & p_\e(x)(\pa_x+\im\mu)  
    \end{bmatrix} 
    $$
    satisfies 
by \eqref{apexp}, \eqref{est:DNpert},  
for any  $ |\e|\leq \e_0 $ in \Cref{DNProp1}, 
    $$
    \sup_{(\al,\mu) \in K}\|\cR(\al,\mu,\e) \|_{\cL(H^1,L^2)}\lesssim |\e|\, , \qquad \forall K\subset \R^2 \ \text{compact} \,.
    $$  
   For any $\tp\geq 3$, by \eqref{Gpincl} 
   and since $K^{(\tp)}$ is compact,  
$   \max_{\lambda\in\Gamma^{(\tp)} (\al,\mu)  \atop (\al,\mu)\in K^{(\tp)} } \|(\lambda-\cL(\al,\mu,0))^{-1}\|_{\cL(L^2, H^1)} < + \infty $ 
and  therefore 
there is $\e^{(\tp)}>0$ such that 
$$\sup_{\lambda\in\Gamma^{(\tp)} (\al,\mu)  \atop (\al,\mu)\in K^{(\tp)} } \|\cR(\al,\mu,\e)(\lambda-\cL(\al,\mu,0))^{-1}\|_{\cL(L^2, L^2)}\lesssim_\tp |\e| <1  \ , \quad \forall |\e|\leq \e^{(\tp)} \ . 
$$
Hence  for any $(\al, \mu) \in K^{(\tp)}$, $\lambda \in \Gamma^{(\tp)}(\al,\mu)$   and $|\e| \leq \e^{(\tp)}$,  
the resolvent 
of  $(\lambda-\cL(\al,\mu,\e))^{-1} $ 
is well defined as     \begin{equation}\label{def.risol}
        (\lambda-\cL(\al,\mu,\e))^{-1}
        = (\lambda - \cL(\al,\mu,0)  )^{-1}(\uno - \cR(\al,\mu,\e)(\lambda - \cL(\al,\mu,0))^{-1} )^{-1}  \in \cL(L^2,H^1) \, , 
    \end{equation}
and  consequently   $\Gamma^{(\tp)}(\al,\mu)\subset \rho(\cL(\al,\mu,\e)) $.

For $\tp = 2$ we repeat the same procedure away from the origin, exploiting the first of \eqref{0905:1850}. We  define  a compact set $K^{(2)}\subset \cN^{(2)}$ such that for any $(\al,\mu)\in K^{(2)}\setminus B_{\rho_2}(0,0)$ one has $
    | \lambda^{(2)}_+(\al,\mu)- \lambda^{(2)}_-(\al,\mu)|\leq \mathsf{c}_2(\rho_2)/3 $.
The closed, counterclockwise oriented circuit   
$$
    \Gamma^{(2)}(\al,\mu)
   = \pa B_{r_2}(\lambda^{(2)}_+(\al,\mu))   
   \subset \rho(\cL(\al,\mu,0))  \, , \quad
   r_2 := 2\mathsf{c}_2(\rho_2)/3 \, , 
$$
   separates $\lambda^{(2)}_\pm(\al,\mu)$ from the rest of the spectrum.    Then we obtain \eqref{def.risol} for  any $(\al,\mu)\in K^{(2)}\setminus B_{\rho_2}(0,0)$ and $|\e|\leq \wt \e^{(2)}(\rho_2)$ small. Note that 
   $ \wt \e^{(2)}(\rho_2) \to 0 $ as $\rho_2 \to 0$.
    
  This shows that  the projectors   \eqref{Pproj:2} and \eqref{Pproj:tp} are well defined bounded operators $L^2\to H^1$.
  The projectors $\wt P^{(2)}_{\al,\mu,\e}$ and $P^{(\tp)}_{\al,\mu,\e}$ commute with $\cL(\al,\mu,\e) $,
  are skew-Hamiltonian and reversibility preserving properties as  follow as in \Cref{lem:Kato1}. The range of a skew-Hamiltonian projector is a symplectic subspace by Lemma \ref{existence:Pskewham}.

We now prove \eqref{defPp}. 
To this goal, first note that, by the continuity of $(\al,\mu,\e) \mapsto \cL(\al,\mu,\e)$, for every fixed $(\und \al, \und \mu) \in K^{(\tp)}$ and $|\und\e|\leq \e^{(\tp)}$, there exists a neighborhood $\und{\cU} \ni (\und \al, \und \mu)$ such that
$$
\sigma (\cL(\al,\mu,\e))\cap B_{r_\tp}(\lambda^{(\tp)}_+(\al,\mu)) = \sigma (\cL(\al,\mu,\e) )\cap B_{r_\tp}(\lambda^{(\tp)}_+(\und \al, \und \mu))  \ , 
\quad \forall (\al, \mu,\e) \in \und{\cU} \ , 
$$
  in view of  the continuity of separated eigenvalues of $\cL(\al,\mu,\e)$, see for instance \cite[section 4.3.3]{Kato}.
    Therefore  the curve $\Gamma^{(\tp)}(\al,\mu)$ can be continuously deformed into $\Gamma^{(\tp)}(\und\al,\und\mu)$ inside the resolvent set of $\cL(\al,\mu,\e)$, verifying the continuity assumption of  \Cref{lem:prodintM} $(iv)$.
    Applying \Cref{lem:prodintM} $(iv)$ we 
    deduce  \eqref{defPp}.
    
  The McLean curve  $\cM^{(2)}$ intersects $\{0\}\times \Z$ only at $(0,0)$. Thus 
if the neighborhood $ K^{(2)} $ of $ \cM^{(2)} $  is sufficiently 
small, 
$  K^{(2)}\setminus B_{\rho_2}(0,0) $ does not intersect $\{0\}\times \Z $ as well, and the  operator $\cL(\al,\mu,\e)$ is analytic on $[K^{(2)}\setminus B_{\rho_2}(0,0)]\times B_{\e^{(2)}}(0)$, as well as the operator $ \wt P^{(2)}_{\al,\mu,\e}$.

    Last we prove that $ P^{(\tp)}_{\al,\mu,\e}\in\tF$. This follows immediately by integrating  the resolvent operator $(\lambda - \cL(\al,\mu,\e))^{-1}$, which, by \Cref{TFjprop} $(iii)$, belongs to $\tF$  for every $(\al,\mu) \in K^{(\tp)}$ and $\lambda\in \Gamma^{(\tp)}(\al,\mu)$. The same holds for $\wt P^{(2)}_{\al,\mu,\e}$ and  $(\al,\mu)\in K^{(2)}\setminus B_{\rho_2}(0,0) $.
\end{proof}

The  construction of  
\Cref{KatonearMcLean} is {\it not}  uniform in $\e$ when $ (\alpha, \mu) $ tends to $  (0,0) $ since 
$ \wt \e^{(2)} ( \rho_2) \to 0  $
if  $ \rho_2  \to 0 $. 
To   extend the projector $\wt P^{(2)}_{\al,\mu,\e}$ in \eqref{Pproj:2}, which is defined for any $ (\alpha, \mu) $ in $  
K^{(2)}\setminus B_{\rho_2}(0,0) $,  to a whole neighborhood of $ \cM^{(2) } $
we rely on Theorem \ref{TeoremoneFinale},
where we proved that 
for any $(\al, \mu) \in B_{\rho_1}(0,0)$
and 
$|\epsilon | < \epsilon_1 $, 
the $ 4 $-dimensional symplectic subspace $\cV_{\al,\mu,\e}$ in \eqref{projdec} has the symplectic decomposition  
\eqref{decosim}. In view of  \Cref{existence:Pskewham}, this  uniquely
defines  the skew-Hamiltonian projectors 
\( P^{(u)}_{\alpha,\mu,\epsilon} \) and \( P^{(s)}_{\alpha,\mu,\epsilon} \) 
onto 
\( \mathcal{V}^{(u)}_{\alpha,\mu,\epsilon} \) and \( \mathcal{V}^{(s)}_{\alpha,\mu,\epsilon} \) respectively.  
Recalling  \eqref{lem:simplbas2},
and \eqref{proph}  we have 
\begin{equation}\label{Punsta}
 P^{(u)}_{\alpha, \mu, \epsilon} =
 - \cW_c(\cdot , h_1^-) h_1^+ 
   + \cW_c(\cdot , h_1^ +) h_1^- \, , 
   \quad 
    P^{(u)}_{\alpha, \mu, \epsilon}  
    \in \cA_P(B_{\rho_1}(0,0),\e_1;H^1,L^2) \, . 
\end{equation}
We now show that   $ P^{(u)}_{\alpha, \mu, \epsilon}  $ defines an 
analytic 
continuation of $ \wt P^{(2)}_{\al, \mu, \e} $ for any $(\alpha,\mu) $ in a neighborhood of $ (0,0)$.

\begin{lemma}\label{prop.ident}
{\bf (Analytic continuation)}
Let $\rho_1 >0$ given by \Cref{TeoremoneFinale}. Fix   $ \rho_2 \in (0, \rho_1 )$ such that $B_{\rho_2}(0,0) \subset K^{(2)}$ and  set   $\e^{(2)} := 
    \min\{\wt\e^{(2)}(\rho_2), \e_1 \} $ (where $\wt\e^{(2)}(\rho_2) >0 $ is given by Lemma \ref{KatonearMcLean}
    and $\e_1 > 0 $ 
    by \Cref{TeoremoneFinale}).
    Then the  projector valued map $  (K^{(2)} \cup B_{\rho_1}(0,0)) \times B_{\e^{(2)}}(0) \to \cL(L^2, H^1)   $ defined by  
    \begin{equation}\label{P2}
(\alpha, \mu, \epsilon ) \mapsto P^{(2)}_{\alpha, \mu, \epsilon}  := \begin{cases}
        \wt P^{(2)}_{\alpha, \mu, \epsilon} \text{ in }\eqref{Pproj:2} \quad & \forall \,  (\al,\mu, \e)\in
        (K^{(2)}\setminus B_{\rho_2}(0,0)) \times  B_{\e^{(2)}}(0)  \ , 
        \\
            P^{(u)}_{\alpha, \mu, \epsilon}\text{ in }\eqref{Punsta} \quad & \forall \,   (\al,\mu, \e)\in B_{\rho_1}(0,0) \times B_{\e^{(2)}}(0)  \, , 
    \end{cases} 
\end{equation}
is polar-analytic in 
$ \cA_P (K^{(2)},\e^{(2)}; L^2,H^1) $, according to \Cref{def:tildeM}, and belongs to $\tF$. 
\end{lemma}

\begin{proof}
Setting $\cU:= K^{(2)}\cap 
    (B_{\rho_1}\setminus B_{\rho_2})(0,0)$,  
the choice of $\rho_2$ guarantees that the set $\cU\times B_{\e^{(2)}}(0)$ is connected. 
    We shall prove first that
    $P^{(u)}_{\al,\mu,\e} \equiv \wt P^{(2)}_{\al,\mu,\e}$  in an open set contained in  $\cU\times B_{\e^{(2)}}(0) $.  
    Fix  $(\und\al, \und\mu) \in \cU\setminus \cM^{(2)}$,  so  the unperturbed eigenvalues 
    $
     \{ 
    \lambda_1^\pm(\und\al, \und\mu,0), \ \lambda_0^{\pm}(\und\al, \und\mu,0)\}$ are  simple and purely imaginary.
  Then take a sufficiently 
 small neighbourhood $\cU_0 \subset \cU\times B_{\e^{(2)}}(0) $ of $(\und \al, \und \mu, 0)$ so that the perturbed eigenvalues 
  $ \{ 
    \lambda_1^\pm(\al, \mu,\e), \ \lambda_0^{\pm}(\al, \mu,\e)\}$  are still  simple; by \Cref{TeoremoneFinale}, the  2-dimensional subspace 
    $\cV^{(u)}_{\al, \mu, \e}$ in \eqref{decosim} is spanned by the two eigenvectors corresponding to
    $\{\lambda_1^\pm(\al, \mu,\e)\}$, cf. also  \eqref{biscot2}.
   At the same time, by  \Cref{KatonearMcLean}, the simple  eigenvalues
   $
    \lambda^{(2)}_\pm(\al, \mu, \e)\equiv \lambda_1^\pm(\al, \mu,\e)$ on $\cU_0$, since they  coincide at $(\und \al, \und \mu, 0)$ and then by continuity of simple eigenvalues on $\cU_0$. 
    Therefore,  $ \wt \cV^{(2)}_{\al,\mu,\e}\equiv  \cV^{(u)}_{\al, \mu, \e}\ $ for any $(\al, \mu, \e) \in \cU_0$ as well as 
    $ P^{(u)}_{\al,\mu,\e} \equiv \wt P^{(2)}_{\al,\mu,\e}$.
Since both projectors are analytic on $\cU \times \cB_{\e^{(2)}}(0)$ and coincide on an open subset, they agree everywhere by analytic continuation.
    Then by \eqref{Punsta} the projector \eqref{P2} is also polar-analytic. 
\end{proof}

\smallskip 
\noindent{\bf Projectors in whole neighborhood $ K^{(\tp)}$ of the McLean curve $ \cM^{(\tp)}$.}
For any $ \tp \geq 2 $, 
we denote the projectors  defined in  \Cref{KatonearMcLean,prop.ident}
as
\begin{equation}\label{Pproj:all}
    P^{(\tp)}_{\al,\mu,\e} := \begin{cases}
        \eqref{P2} \quad & \text{ if }\tp = 2 \, , \\
        \eqref{Pproj:tp} \quad &\text{ if } \tp \geq 3 \, ,
    \end{cases} \qquad \forall (\al,\mu,\e)\in K^{(\tp)} \times B_{\e^{(\tp)}}(0)\, .
\end{equation}
By 
 \eqref{defPp} and  Lemma \ref{prop.ident}
\be\label{regPX}
P^{(2)}\in\cA_P(K^{(2)},\e^{(2)};L^2,H^1)\cap\tF \, , \quad
P^{(\tp)}\in\cA(K^{(\tp)},\e^{(\tp)};L^2,H^1)\cap\tF \, , \quad \forall \tp\geq 3 \, . 
\ee

\begin{lemma}\label{katoisomMcLean}
    For any $\tp\geq 2$, for any $(\al,\mu,\e)\in K^{(\tp)}\times B_{\e^{(\tp)}}(0)$
    (by possibly shrinking $\e^{(\tp)} $)
    the operator   
\begin{equation}\label{Umclean}
            U^{(\tp)}_{\al,\mu,\e}:= \big( \uno-(P^{(\tp)}_{\al,\mu,\e}-P^{(\tp)}_{\al,\mu,0})^2 \big)^{-1/2} \big[ 
P^{(\tp)}_{\al,\mu,\e}P^{(\tp)}_{\al,\mu,0} + (\uno - P^{(\tp)}_{\al,\mu,\e})(\uno-P^{(\tp)}_{\al,\mu,0}) \big] 
        \end{equation}
        is a symplectic and reversibility preserving isomorphism of both $L^2$ and $H^1$,  i.e. 
\begin{equation}\label{Upreversible}
        [U^{(\tp)}_{\al,\mu,\e}]^* \cJ U^{(\tp)}_{\al,\mu,\e}  =\cJ \, , \qquad \varrho_c \circ U^{(\tp)}_{\al,\mu,\e} = U^{(\tp)}_{\al,\mu,\e} \circ \varrho_c \, , 
    \end{equation}
   and  
    \be\label{U.P}
    U^{(\tp)}_{\al,\mu,\e} P^{(\tp)}_{\al,\mu,0}[U^{(\tp)}_{\al,\mu,\e}]^{-1} =  P^{(\tp)}_{\al,\mu,\e}  \, .
    \ee
The subspaces 
$\mathcal{V}_{\al,\mu,\e}^{(\tp)}= \textup{Rg} P^{(\tp)}_{\al,\mu,\e}$  are all  isomorphic.  
  Denoting   $ Z $   either
     $L^2$ or $H^1$, 
\be\label{regUX}
    U^{(2)}\in\cA_P(K^{(2)},\e^{(2)};Z,Z)\cap\tF \, , \quad
    U^{(\tp)}\in\cA(K^{(\tp)},\e^{(\tp)};Z,Z)\cap\tF \, , \quad \forall \tp\geq 3 \, . 
\ee
\end{lemma}

\begin{proof}
The transformation 
operator 
$ U^{(\tp)}_{\al,\mu,\e} $ in \eqref{Umclean}
  satisfies  \eqref{U.P} and  \eqref{Upreversible}
 as in Lemma \ref{lem:Kato1}.   
Properties 
    \eqref{regUX}  follow by \eqref{regPX} and applying the functional calculus \Cref{lem:prodintM}$(i),(ii)$ and  \Cref{TFjprop} $(i)$. 
\end{proof}

The  unperturbed basis 
$ \{ v^{(\tp)}_\sigma (\al,\mu), \sigma = \pm  \} $ 
of $ \mathcal{V}_{\al,\mu,0}^{(\tp)}$ defined in  \eqref{ivMC}  is mapped by the operator 
$U^{(\tp)}_{\al,\mu,\e}$ in \eqref{Umclean}
into the  basis  of 
$ \mathcal{V}_{\al,\mu,\epsilon}^{(\tp)} $, 
\begin{equation}\label{basisFp}
{\cal F^{(\tp)}} := 
\big\{ v^{(\tp)}_\sigma (\al, \mu,\e)
= U_{\al, \mu,\e}^{(\tp)} v^{(\tp)}_\sigma (\alpha, \mu)
\,, 
\ \sigma = \pm \big\} \, . 
\end{equation}

\begin{lemma}{\bf (Matrix representation of $ \cL(\al, \mu,\e) $ on $ \mathcal{V}_{\al, \mu,\e}^{(\tp)}$)}
\label{lem:katored} 
For any $ \tp \geq 2 $ and  
$ (\al, \mu, \e) \in K^{(\tp)}\times B_{\e^{(\tp)}}(0) $ the matrix that represents  the Hamiltonian and reversible operator $\cL
(\al, \mu,\e): \mathcal{V}_{\al,\mu,\e}^{(\tp)}\to\mathcal{V}_{\al,\mu,\e}^{(\tp)} $ 
with respect to the  basis ${\cal F}^{(\tp)} $ 
in \eqref{basisFp},  is  
\be\label{Lbpagain}
\tL^{(\tp)}(\al, \mu,\e) =\tJ\tB^{(\tp)}(\al, \mu,\e) \, , \qquad \tJ =\begin{pmatrix} -\im & 0 \\ 0 & \im \end{pmatrix} \, , 
\ee
where 
$ \tB^{(\tp)}(\al, \mu,\e) $ is the $2\times 2$ real symmetric matrix 
\begin{align}\label{tocomputematrixB}
\tB^{(\tp)}(\al, \mu,\e) := \begin{pmatrix}
( \mathfrak{B}^{(\tp)}
(\al, \mu,\e) v^{(\tp)}_+(\al,\mu), v^{(\tp)}_+(\al,\mu)) 
& 
(\mathfrak{B}^{(\tp)}(\al, \mu,\e) v^{(\tp)}_-(\al,\mu), v^{(\tp)}_+(\al,\mu)) \\
(\mathfrak{B}^{(\tp)}(\al, \mu,\e) v^{(\tp)}_+(\al,\mu), v^{(\tp)}_-(\al,\mu)) 
&
(\mathfrak{B}^{(\tp)}(\al, \mu,\e) v^{(\tp)}_-(\al,\mu), v^{(\tp)}_-(\al,\mu))
\end{pmatrix}\, ,
\end{align}
and
\begin{equation}\label{Bgotico}
\mathfrak{B}^{(\tp)} (\al, \mu,\e) := [P_{\al, \mu,0}^{(\tp)}]^*
[U_{\al, \mu,\e}^{(\tp)}]^* \, \cB(\al, \mu,\e) \, U_{\al, \mu,\e}^{(\tp)} P_{\al, \mu,0}^{(\tp)}    \in \mathtt{F} \, . 
\end{equation} 
Moreover
\begin{equation}\label{regBtp}
            \tB^{(2)}(\al, \mu,\e )\in \cA_P (K^{(2)},\e^{(2)};\R^{2\times 2})  \, , \qquad \tB^{(\tp)}(\al, \mu,\e)\in \cA(K^{(\tp)},\e^{(\tp)};\R^{2\times 2}) \ \ \forall \tp\geq 3\, , 
\end{equation}
    according to \Cref{def:tildeM,defFell}.
\end{lemma}

\begin{proof}
In view of \eqref{Umclean}
the vectors 
$ v^{(\tp)}_\pm  (\al, \mu,\e)$
satisfy the same simplecticity property 
\eqref{symp.nonzero} 
of the unperturbed vectors $ v^{(\tp)}_\pm (\al, \mu) $. 
Then  any vector $v \in \mathcal{V}_{\al,\mu,\e}^{(\tp)}$ decomposes as 
$$v = \im \cW_c( v, v^{(\tp)}_+(\al,\mu,\e)) \,  v^{(\tp)}_+(\al,\mu,\e) - \im  \cW_c( v, v^{(\tp)}_-(\al,\mu,\e)) \,  v^{(\tp)}_-(\al,\mu,\e) \ . $$
Applying this formula with $v \leadsto  \cL(\al, \mu, \e) v^{(\tp)}_\sigma(\al,\mu,\e)$, $\sigma \in \{\pm\}$, recalling \eqref{ses}, \eqref{cLame0}, $\cJ^2 = -\uno$, and using \eqref{Bgotico} shows  that 
$\cL
(\al, \mu,\e): \mathcal{V}_{\al,\mu,\e}^{(\tp)}\to\mathcal{V}_{\al,\mu,\e}^{(\tp)} $ 
 is represented in   the  basis ${\cal F}^{(\tp)} $ in \eqref{basisFp} by the matrix \eqref{Lbpagain}.
The  operator $\mathfrak{B}^{(\tp)}
(\al, \mu,\e) $ in \eqref{Bgotico}
is  reversibility preserving because 
$\cB(\al,\mu,\e)$,
$P_{\al, \mu,\e}^{(\tp)}$ and 
$ U_{\al, \mu,\e}^{(\tp)}$
are reversibility preserving by 
\eqref{cLr:rev},  \Cref{KatonearMcLean}
and \eqref{Upreversible}.
Therefore,     using \eqref{rev.nonzero} 
and  $(\varrho_c f, \varrho_c g) = \overline{(f,g)}$, 
we deduce 
    $$
        \overline{( \mathfrak{B}^{(\tp)}
 v^{(\tp)}_\sigma, v^{(\tp)}_{\sigma'}) } = (\varrho_c \mathfrak{B}^{(\tp)}
 v^{(\tp)}_\sigma, \varrho_c v^{(\tp)}_{\sigma'}) = ( \mathfrak{B}^{(\tp)} \varrho_c
 v^{(\tp)}_\sigma,\varrho_c v^{(\tp)}_{\sigma'}) = ( \mathfrak{B}^{(\tp)} 
 v^{(\tp)}_\sigma, v^{(\tp)}_{\sigma'})  
    $$
proving that the entries are real. 
Property
\eqref{Bgotico} 
and 
\begin{equation}\label{fBanalytic}
       \mathfrak{B}^{(2)}(\al,\mu,\e)\in
\cA_P(K^{(2)},\e^{(2)};H^1,L^2)\, , \qquad \mathfrak{B}^{(\tp)}(\al,\mu,\e)\in\cA(K^{(\tp)},\e^{(\tp)};H^1,L^2)\,  , \ \tp\geq 3 \, ,  
   \end{equation}
follow by    \eqref{regGBL}, 
\eqref{regPX}, 
\eqref{regUX}
and \Cref{TFjprop,lem:prodintM}. 
For any  $\tp \geq 2$, the unperturbed eigenvectors  $ v^{(\tp)}_\pm (\al,\mu) $ in \eqref{unperturbed.eigv} 
are analytic except at the  points $(\al,\mu) = (0,\pm \frac{\tp}{2})$ if $\tp$ is even, respectively $(\al,\mu) \in \{ (0, -\frac{\tp-1}{2}), (0, \frac{\tp+1}{2})\}$ if $\tp$ is odd. In view of \eqref{sing.ML}  
these points do not stay on the unperturbed  McLean curve $\cM^{(\tp)}$, and therefore neither in a sufficiently small neighborhood  
$ K^{(\tp)} $ of $\cM^{(\tp)}$.  
We conclude that  the matrix 
$
\tB^{(\tp)}(\al, \mu,\e) $ in \eqref{tocomputematrixB} satisfies
\eqref{regBtp}  in view of    \eqref{fBanalytic} and Lemma~\ref{lem:prodintM}$(iii)$. 
\end{proof}

In view of the symmetry properties 
of the McLean curves stated in  \Cref{lem:description}, 
we may assume that each 
neighborhood 
$K^{(\tp)}$ of the McLean curve $\cM^{(\tp)}$  is invariant both under the reflection $(\al,\mu) \leadsto (-\al,\mu)$ and, if $\tp$ is even, under $(\al,\mu) \leadsto (\al,-\mu)$, respectively if $\tp$ is odd under $(\al,\mu) \leadsto (\al,1-\mu)$. 

\begin{lemma}
{\bf (Symmetries)}
For any $ \tp \geq 2 $ and  
$ (\al, \mu, \e) \in K^{(\tp)}\times B_{\e^{(\tp)}}(0) $ the operator 
$ \mathfrak{B}^{(\tp)} (\al, \mu,\e)$
in \eqref{Bgotico} satisfies 
        \begin{align}
&        
\mathfrak{B}^{(\tp)}(\al,-\mu,\e) = \overline{\mathfrak{B}^{(\tp)}(\al,\mu,\e)}\, , \ 
\text{if }\tp \text{ is even;}
\quad 
\mathfrak{B}^{(\tp)}(\al,1-\mu,\e) = e^{-\im x}\overline{\mathfrak{B}^{(\tp)}(\al,\mu,\e)} e^{\im x}\, , \ 
\text{if }\tp \text{ is odd}\, , 
\label{symmBgotico} \\  
& \mathfrak{B}^{(\tp)}(-\al,\mu,\e) = \mathfrak{B}^{(\tp)}(\al,\mu,\e) \label{symmBgotico1} \, . 
        \end{align}
\end{lemma}

\begin{proof}
We first prove \eqref{symmBgotico} if  $\tp$ is odd. 
    By \eqref{cBLmu+k} and \eqref{cGcLmenomu}
    \begin{equation}\label{cLsymmetry6bis}
        \cL(\al, \mu, \e)  = e^{-\im x} \, \overline{\cL(\al, 1-\mu, \e)} \, e^{ \im x} 
    \quad 
    \Rightarrow 
    \quad 
    \big(\lambda - \cL(\al, \mu, \e) \big)^{-1}= e^{-\im x} \, \big(\lambda -  \overline{\cL(\al, 1-\mu, \e)} \big)^{-1} \, e^{\im x}\, .
    \end{equation}
  Then, denoting 
    $\tm=\frac{\tp-1}{2}$, the projector $ P^{(\tp)}_{\alpha,\mu,\e}$
    in     \eqref{Pproj:all} is equal to  
    \begin{align}\label{transfPpbis}
        e^{-\im x}\overline{P^{(\tp)}_{\al,1-\mu,\e}}e^{\im x} &
        \stackrel{\eqref{Pproj:tp},\eqref{McL.d}} = \overline{\oint_{\pa B_{r_\tp}(\lambda_\tm^+(\al,1-\mu))}e^{\im x}(\lambda-\cL(\al,1-\mu,\e))^{-1}e^{-\im x} \frac{\de\lambda}{2\pi \im}} \\
       \notag
       & \stackrel{\eqref{cLsymmetry6bis}} = \oint_{\pa B_{r_\tp}(\overline{\lambda_\tm^+(\al,1-\mu)})}(\zeta-\cL(\al,\mu,\e))^{-1}\frac{\de\zeta}{2\pi \im}  \\
              \notag
       & \stackrel{\eqref{eig.0}} =  \oint_{\pa B_{r_\tp}(\lambda_{\tm+1}^-(\al,\mu))} (\zeta-\cL(\al,\mu,\e))^{-1}\frac{\de\zeta}{2\pi\im} 
       =  \oint_{\pa B_{r_\tp}(\lambda_{\tm}^+(\al,\mu))} (\zeta-\cL(\al,\mu,\e))^{-1}\frac{\de\zeta}{2\pi\im} 
         \stackrel{\eqref{Pproj:tp}}  = P^{(\tp)}_{\al,\mu,\e} 
    \end{align}
    noting that the curves $ 
    \pa B_{r_\tp}(\lambda_\tm^+(\al,\mu))$ and $\pa B_{r_\tp}(\lambda_{\tm+1}^-(\al,\mu))$ are homotopic inside $\rho(\cL(\al,\mu,\e))$, and 
    encircles the pair of nearby 
    eigenvalues $ \{ \lambda_{\tm}^+(\al,\mu)), 
    \lambda_{\tm+1}^-(\al,\mu)) \} $
    for any 
    $ (\alpha, \mu) $ near the McLean curve $ \cM^{(\tp) }$. 
    
By \eqref{transfPpbis}, the operator $ U^{(\tp)}_{\al,\mu,\e} $ in 
\eqref{Umclean} satisfies  
$ e^{-\im x} \overline{U^{(\tp)}_{\al,1-\mu,\e}} e^{ \im x} =  U^{(\tp)}_{\al,\mu,\e} $ as well, 
and, since
$ \cB (\al, \mu, \e)  = e^{-\im x} \, \overline{\cB(\al, 1-\mu, \e)} \, e^{ \im x} $ (cf. \eqref{cLsymmetry6bis}),  
we deduce 
\eqref{symmBgotico} when $ \tp  $ is odd.
In the case $\tp$ is even, the proof of \eqref{symmBgotico} outside a small neighborhood of the origin follows similarly. 
By analyticity 
$ \mathfrak{B}^{(2)}(\al,-\mu,\e) = \overline{\mathfrak{B}^{(2)}(\al,\mu,\e)}$  
extends for any $ (\alpha, \mu) \neq (0,0 )$.

Finally, since $ \cL(\alpha, \mu, \epsilon)$ is even in $ \alpha $
by \eqref{cGcLmenomu}, the projector 
$ P^{(\tp)}_{\alpha,\mu,\e}$
    in     \eqref{Pproj:all} 
and the similarity transformation operators  $U^{(\tp)}_{\al,\mu,\e}$ in \eqref{Umclean} are even in $\al$ as well, 
and  $ \mathfrak{B}^{(\tp)}(\al,\mu,\e) $ in \eqref{Bgotico} satisfies \eqref{symmBgotico1}. 
\end{proof}

 We now provide the Taylor expansion of the matrix $\tB^{(\tp)}(\al,\mu,\e)$ in  \eqref{tocomputematrixB}, 
 which is  \eqref{tocomputematrixb}. 
The regularity  property \eqref{reguabcp} is proved by \eqref{regBtp}.
 In view of \eqref{Bgotico},
\be\label{Bgotpex}
\mathfrak{B}^{(\tp)} (\al, \mu,\e)
= 
\sum_{\ell \in \N_0} 
\mathfrak{B}_\ell^{(\tp)} (\al, \mu) \e^\ell \, ,
\quad 
\mathfrak{B}_\ell^{(\tp)}(\al, \mu)  \in 
\mathfrak{F}_\ell \, ,
\quad  \forall \ell \in \N_0 \, , \  (\alpha,\mu)
\in K^{(\tp)} \, , 
\ee
where the spaces $ \mathfrak{F}_\ell $ are  introduced in \Cref{defFell}. 

\begin{lemma}{\bf (Taylor expansion of  $\tB^{(\tp)}(\al, \mu,\e)$)}
\label{prop:katored} 
For any $\tp\geq 2$  the  
real valued functions 
$\fa^{(\tp)}(\al, \mu,\e) $, 
$ \fb^{(\tp)}(\al, \mu,\e)  $,
$\fc^{(\tp)}(\al, \mu,\e)$ in \eqref{tocomputematrixb} 
have the expansions 
\eqref{matrixentries} with coefficients,  for any 
$ (\alpha, \mu) \in K^{(\tp)}\setminus \{(0,0)\} $,  
\be\label{abceta}
    \begin{aligned}
   &  \fa_\tp(\al,\mu) := ({\frak B}^{(\tp)}_2(\al,\mu) v^{(\tp)}_+
    (\al,\mu),v^{(\tp)}_+(\al,\mu)) \, , \\
&     \fb_\tp(\al,\mu) := ({\frak B}^{(\tp)}_\tp(\al,\mu) v^{(\tp)}_-(\al,\mu) ,v^{(\tp)}_+(\al,\mu))  \, ,\\
&     \fc_\tp(\al,\mu) := ({\frak B}^{(\tp)}_2(\al,\mu) v^{(\tp)}_-(\al,\mu) ,v^{(\tp)}_-(\al,\mu))\, ,\\
& \beta_\tp(\al, \mu):= ({\frak B}^{(\tp)}_{\tp+2}(\al,\mu) v^{(\tp)}_-(\al,\mu) ,v^{(\tp)}_+(\al,\mu)) \, , 
    \end{aligned}
\ee
where $v^{(\tp)}_\pm(\alpha, \mu) $ are the unperturbed eigenvectors in \eqref{ivMC}.
Furthermore  $ \fa^{(\tp)}(\al, \mu , \e)$, $ \fc^{(\tp)}(\al, \mu , \e) $
are even in $ \e $, while 
$ \fb^{(\tp)}(\al, \mu , \e)$ is odd in $ \e $, resp. even, if $ \tp $ is odd, resp. even, and satisfy
\begin{equation}\label{symmetryapcpbp}
    \begin{aligned}
        &\text{if }\tp \text{ is even: }\quad  \fa^{(\tp)}(\al,-\mu,\e) = \fc^{(\tp)}(\al,\mu,\e)\, , \quad \fb^{(\tp)}(\al,-\mu,\e) = \fb^{(\tp)}(\al,\mu,\e) \, , \\ &\text{if }\tp \text{ is odd: }\quad \fa^{(\tp)}(\al,1-\mu,\e) = \fc^{(\tp)}(\al,\mu,\e)\, , \quad \fb^{(\tp)}(\al,1-\mu,\e) = \fb^{(\tp)}(\al,\mu,\e)\, ,\\
              & \fa^{(\tp)} (-\al,\mu,\e) = \fa^{(\tp)} (\al,\mu,\e)\, , \quad \fb^{(\tp)} (-\al,\mu,\e) = \fb^{(\tp)} (\al,\mu,\e)\, , \quad \fc^{(\tp)} (-\al,\mu,\e) = \fc^{(\tp)} (\al,\mu,\e)\, .
        \end{aligned}
\end{equation}
\end{lemma}

\begin{proof}
We denote for brevity 
$ v^{(\tp)}_\pm := v^{(\tp)}_\pm (\al,\mu) $.
We have 
\be\label{bpam}
\fb^{(\tp)}(\al,\mu,\e) \stackrel{ \eqref{tocomputematrixB}}{=} (\mathfrak{B}^{(\tp)}(\al, \mu,\e) v^{(\tp)}_- , v^{(\tp)}_+) 
\stackrel{\eqref{Bgotpex}} = 
\sum_{\ell=0}^{\tp +3}
\underbrace{(\mathfrak{B}_\ell^{(\tp)}(\al, \mu) v^{(\tp)}_- , v^{(\tp)}_+)}_{=:\fb^{(\tp)}_\ell (\al,\mu)} \e^{\ell}  
+ {r_\fb^{(\tp)}(\epsilon^{\tp+4})}
\, . 
\ee
The coefficients  $\fb^{(\tp)}_\ell (\al, \mu)  = 0$ for any $ \ell<\tp$ and  $\ell \not\equiv   \tp $ $ \textup{mod } 2$, 
since $ {\frak B}^{(\tp)}_\ell \in  
\mathfrak{F}_\ell $ by 
\eqref{Bgotpex}, \eqref{AFlele} and since  the vectors  $v^{(\tp)}_\pm(\al,\mu)$ in \eqref{ivMC} are separated by $\tp$ harmonics according to \eqref{unperturbed.eigv}. In view of \eqref{bpam}, we deduce \eqref{expb}
with the coefficients
$\fb_\tp(\al,\mu)$ 
and $ \beta_\tp(\al,\mu) $ defined in \eqref{abceta}.

We now prove \eqref{expa}, \eqref{expc}.  Arguing as above, by \eqref{AFlele} and \eqref{unperturbed.eigv},  we deduce $({\frak B}^{(\tp)}_1(\al,\mu) v^{(\tp)}_-,v^{(\tp)}_-) = ({\frak B}^{(\tp)}_3(\al,\mu) v^{(\tp)}_-,v^{(\tp)}_-) = 0$, and therefore
\be\label{daac}
    \begin{aligned}
        &\fa^{(\tp)}(\al,\mu,\e) 
\stackrel{\eqref{tocomputematrixB}} 
=
        ({\frak B}^{(\tp)}(\al,\mu,\epsilon) v^{(\tp)}_+ ,v^{(\tp)}_+)  
       \stackrel{\eqref{tocomputematrixB}}   = ({\frak B}^{(\tp)}(\al,\mu,0) v^{(\tp)}_+ ,v^{(\tp)}_+)  + 
\underbrace{({\frak B}^{(\tp)}_2(\al,\mu) v^{(\tp)}_+ ,v^{(\tp)}_+)}_{=: \fa_{\tp}(\al, \mu)} \e^2  + r_\fa^{(\tp)}(\e^4)\\
        &\fc^{(\tp)}(\al,\mu,\e)   \stackrel{\eqref{tocomputematrixB}} 
        = 
        ({\frak B}^{(\tp)}(\al,\mu,\epsilon) v^{(\tp)}_-,v^{(\tp)}_-) \stackrel{\eqref{tocomputematrixB}}  = ({\frak B}^{(\tp)}(\al,\mu,0) v^{(\tp)}_-,v^{(\tp)}_-)  + 
\underbrace{({\frak B}^{(\tp)}_2(\al,\mu) v^{(\tp)}_-,v^{(\tp)}_-)}_{=: \fc_{\tp}(\al, \mu)} \e^2 + r_\fc^{(\tp)}(\e^4)\, .
    \end{aligned}
\ee
   At $\e=0$ 
 we have  $U_{\al,\mu,0}^{(\tp)} = \uno$ in \eqref{Umclean}, so  \eqref{Bgotico} reduces to  ${\frak B}^{(\tp)}(\al,\mu,0) = [P^{(\tp)}_{\al,\mu,0}]^* \cB (\al,\mu,0)
    P^{(\tp)}_{\al,\mu,0} $, and therefore
    \be\label{zerotermac}
    \begin{aligned}
    ({\frak B}^{(\tp)}(\al,\mu,0) v^{(\tp)}_\pm,v^{(\tp)}_\pm) & = 
    (\cB (\al,\mu,0) v^{(\tp)}_\pm ,v^{(\tp)}_\pm ) = -(\cJ \cL(\al,\mu,0) v^{(\tp)}_\pm,v^{(\tp)}_\pm)\\
    & \stackrel{\eqref{eig.0},\eqref{unperturbed.eigv},\eqref{McL.d}}{=} -\im \omega^{(\tp)}_\pm(\al,\mu) (\cJ  v^{(\tp)}_\pm,v^{(\tp)}_\pm) \stackrel{\eqref{symp.nonzero}}{=} \mp \omega^{(\tp)}_\pm(\al,\mu)\, .
\end{aligned}
\ee
By \eqref{daac}, \eqref{zerotermac} 
we deduce 
\eqref{expa} and  \eqref{expc}. 

We finally prove \eqref{symmetryapcpbp}. The eigenvectors  \eqref{unperturbed.eigv} satisfy 
\begin{equation}\label{invareigvec}
        v_j^\sigma(\al,-\mu) = \overline{v_j^{-\sigma}(\al,\mu)} \, , \quad v_j^\sigma(\al,1-\mu) = \overline{v_{j+\sigma }^{-\sigma}(\al,\mu)}e^{-\im x} \, , \quad 
        v_j^\sigma(-\al,\mu) = v_j^\sigma(\al,\mu)  \,. 
    \end{equation}
    Then, in view of \eqref{tocomputematrixB} and \eqref{ivMC}, denoting $\tm = \tfrac{\tp-1}{2}$  for $\tp$ odd, 
    using \eqref{symmBgotico}, 
    \eqref{invareigvec} 
   we have
    $$
    \begin{aligned}
       &\fa^{(\tp)}(\al,1\!-\!\mu,\e) \!=\! ( \mathfrak{B}^{(\tp)}
(\al, 1 \! - \! \mu,\e) v^+_{\tm}(\al,1-\mu), v^+_{\tm}(\al,1-\mu))\!  = \! (\frak{B}^{(\tp)}(\al,\mu,\e) v_{\tm+1}^-(\al,\mu), v_{\tm+1}^-(\al,\mu)) \! = \! \fc^{(\tp)}(\al,\mu,\e), \\
& \fb^{(\tp)}(\al,1\!-\!\mu,\e) \!=\! ( \mathfrak{B}^{(\tp)}
(\al, 1\!-\!\mu,\e) v^+_{\tm}(\al,1\!-\!\mu), v^-_{\tm+1}(\al,1-\mu))\! =\! (\mathfrak{B}^{(\tp)} (\al,\mu,\e) v^-_{\tm+1}(\al,\mu), v^+_{\tm}(\al,\mu)) \!=\!\fb^{(\tp)}(\al,\mu,\e).
    \end{aligned}
    $$
    Similarly we get 
\eqref{symmetryapcpbp} for 
$\tp$ even. Each 
$ \fa^{(\tp)}(\al, \mu , \e)$,
$ \fb^{(\tp)}(\al, \mu , \e)$, $ \fc^{(\tp)}(\al, \mu , \e) $
is  even in $ \alpha $ by 
\eqref{symmBgotico1}
and \eqref{invareigvec}.
\end{proof}

The matrix 
$\tL^{(2)}(\al, \mu, \e)$
 in \eqref{Lbpagain} is similar to 
the matrix $\tU ( \al, \mu, \e ) $ 
in \eqref{UU} for any 
$ (\alpha, \mu )$ near $ (0,0)$.  This identification allows to determine 
the second order coefficients.

\begin{lemma}\label{b00} 
The coefficients 
$\fa_2(\alpha, \mu), \fb_2
(\alpha, \mu), \fc_2 (\alpha, \mu)$ 
in \eqref{matrixentries} extend to continuous functions at the origin with \be\label{dir.limit.abc}
\fb_2(0,0)  = 
- \fa_2(0,0) = - \fc_2(0,0)  =   - \frac12
 \, .
    \ee 
\end{lemma}

\begin{proof}
For $(\al,\mu,\e)$ sufficiently small we have constructed two basis on the $2 $-dimensional symplectic subspace 
$
 \cV^{(2)}_{\al, \mu, \e} = \wt   \cV^{(2)}_{\al, \mu, \e}
 = 
 \cV^{(u)}_{\al, \mu, \e}  
$. 
The first is 
$\{h_1^\pm (\al,\mu,\e) \} $ 
defined in 
\eqref{decosim} on which $\cL(\al,\mu,\e)$ is represented by the matrix $\mathtt{U} (\al,\mu,\e) $ in \eqref{UU}; the second is 
$\{ v^{(2)}_\pm (\al,\mu,\e)\}$, 
defined in  \eqref{basisFp}, on which $\cL(\al,\mu,\e)$ is represented by $\tL^{(2)}(\al,\mu,\e)$ in \Cref{lem:katored}.   
A third basis 
is  
\begin{equation}\label{tildev}
    w^{(2)}_+(\al,\mu,\e) := \frac{1}{\sqrt 2} \left( h_1^-(\al,\mu,\e) + \im  h_1^+(\al,\mu,\e)\right) , \qquad w^{(2)}_- (\al,\mu,\e):= \frac{1}{\sqrt 2} \left( h_1^-(\al,\mu,\e) - \im  h_1^+(\al,\mu,\e)\right)\,  .  
\end{equation}
 The matrix associated to the change of basis $\{h_1^\pm(\al,\mu,\e) \}\leadsto \{w^{(2)}_\pm(\al,\mu,\e)\} $ in \eqref{tildev} is 
\begin{equation}\label{Ccomplex}
    \tC = \frac{1}{\sqrt2} \begin{pmatrix}
        \im  & -\im \\
        1  & 1
    \end{pmatrix} \ , \qquad 
    \tC^{-1} = \tC^* = \frac{1}{\sqrt2} \begin{pmatrix}
        - \im  & 1  \\
        \im   &  1
    \end{pmatrix} \,
\end{equation}
which maps coordinates $[\begin{smallmatrix}
    a \\ b\end{smallmatrix}]$ in the basis $\{ w^{(2)}_+(\al, \mu), w^{(2)}_-(\al, \mu)\}$ to coordinates $\tC[\begin{smallmatrix}
    a \\ b\end{smallmatrix}] $ in the basis $\{h_1^+(\al, \mu), h_1^-(\al, \mu)\}$. Thus, $\cL(\al,\mu,\e)$ is represented 
in the basis $\{ w^{(2)}_\pm(\al,\mu,\e)\}$  
by the matrix 
\begin{equation}
\label{similarcomplex}
\tW (\al,\mu,\e) := \tC^{-1} \mathtt{U}
(\al,\mu,\e) \tC = 
    \begin{pmatrix}
        - \im & 0 \\
        0 & \im 
    \end{pmatrix}
    \begin{pmatrix}
    -   \mathsf{a} 
    + \tfrac{\mathsf{b}^+ - \mathsf{b}^-}{2} &  
    \tfrac{\mathsf{b}^+ + 
    \mathsf{b}^-}{2} \\
   \tfrac{\mathsf{b}^+ + \mathsf{b}^-}{2} &   
   \mathsf{a} + \tfrac{\mathsf{b}^+  - \mathsf{b}^-}{2}
    \end{pmatrix}(\al,\mu,\e) \, ,
\end{equation}
    with $\mathsf{a}(\al,\mu,\e)
    $ in \eqref{exp:a} and $\mathsf{b}^\pm (\al,\mu,\e) $ in \eqref{exp:b+-}-\eqref{exp:b-}. 
In  view of \eqref{basiepzero}, 
  \eqref{tildev}, and 
  \eqref{ivMC} 
$$
v^{(2)}_\pm (\al,\mu,0) = w^{(2)}_\pm (\al,\mu,0) \, . 
$$
The matrix associated to the change of basis  $\{ w^{(2)}_\pm(\al,\mu,\e) \}\leadsto\{v^{(2)}_\pm(\al,\mu,\e) \}$ is 
\be\label{T.matrix}
{\rm T} (\al,\mu,\e) := \begin{pmatrix}
    \im \cW_c (w^{(2)}_+(\al,\mu,\e) ,
v^{(2)}_+ (\al,\mu,\e)) & \im \cW_c(w^{(2)}_-(\al,\mu,\e) ,
v^{(2)}_+(\al,\mu,\e) )\\
- \im \cW_c(w^{(2)}_+ (\al,\mu,\e) ,
v^{(2)}_- (\al,\mu,\e) )& - \im \cW_c(w^{(2)}_-(\al,\mu,\e) ,
v^{(2)}_- (\al,\mu,\e))
\end{pmatrix}
\in \cA_{P}(K^{(2)}, \e^{(2)}, \C^{2\times 2}) \, , 
\ee
which by \eqref{tildev}, \eqref{proph}, \eqref{basiepzero} 
has the form 
\begin{equation}\label{samebase0}
    {\rm T}(\al,\mu, \e) = \uno + \cO_{\C^{2\times 2}}
    (\e)\, , 
\end{equation}
and, in view of \eqref{similarcomplex},  the two representations are related by 
    \begin{equation}\label{simUL}
        \mathtt{L}^{(2)}(\al, \mu, \e) = {\rm T}(\al,\mu,\e)\, 
        \tW (\al,\mu,\e)
        {\rm T}^{-1} (\al,\mu,\e) \, . 
    \end{equation}
    Taking the limit of  \eqref{simUL} along any straight line  $(\al,\mu) =
    (\rho \sin (\theta), \rho  \cos (\theta) ) $ as $ \rho \to 0 $, in view of   \eqref{exp:a}, \eqref{exp:b+-}, \eqref{exp:b-},  the matrix in 
\eqref{similarcomplex}
reduces to 
$$
    \tW (0,0,\e) 
    = \frac12 
    \tb^- (0,0,\epsilon) 
    \begin{pmatrix}
        - \im & 0 \\
        0 & \im 
    \end{pmatrix}
    \begin{pmatrix}
    -1 & 1   
     \\
   1 &   
   - 1
    \end{pmatrix} \, ,
    \quad 
    \mathsf{b}^-(0,0,\e) = -\e^2(1+r(\e^2)) \, , 
$$
and  by \eqref{samebase0}, \eqref{T.matrix},  
    we deduce  that, for any $\theta \in \T$, the directional limits exist and are given by 
    \begin{equation}\label{expansionL2}
      \lim_{\rho \to 0 \atop \substack{\al = \rho \sin \theta \atop \mu = \rho \cos \theta} }\begin{pmatrix}
        -\im & 0\\
        0 & \im
    \end{pmatrix}
   \begin{pmatrix} \fa^{(2)}(\al, \mu,\e) & \fb^{(2)}(\al, \mu,\e) \\ \fb^{(2)}(\al, \mu,\e) & \fc^{(2)}(\al, \mu,\e) \end{pmatrix}  = \begin{pmatrix}
        -\im & 0\\
        0 & \im
    \end{pmatrix}
    \begin{pmatrix}
        \frac12 \e^2 (1+r(\e)) & -\frac12 \e^2 (1+r(\e))\vspace{1mm}\\
        -\frac12 \e^2 (1+r(\e)) & \frac12 \e^2 (1+r(\e)) 
    \end{pmatrix}\, . 
    \end{equation}
Identifying the entries \eqref{expansionL2} with those in \eqref{expa}-\eqref{expc} for $ \tp = 2 $, one deduces 
$$
\lim_{\rho \to 0}
\fa_2(\rho \sin(\theta), \, \rho \cos(\theta))  =
\lim_{\rho \to 0}
\fc_2(\rho \sin(\theta), \, \rho \cos(\theta)) = 
\lim_{\rho \to 0}
-\fb_2(\rho \sin(\theta), \, \rho \cos(\theta))  =  \frac12  \, .
$$
Since $\fa_2 (\alpha,\mu) $, $\fb_2 (\alpha,\mu) $ and $\fc_2 (\alpha,\mu) $ are polar analytic functions 
whose  directional limits are all the same, 
they extend with continuity at $(0,0)$.
\end{proof}

In order to detect the emergence of the instability region 
   $ \cU_\e^{(\tp)} $ in \eqref{def:instaintro}
we first describe  the perturbed McLean curves $\cM^{(\tp)}_\pm(\e)$ in \eqref{defMepm}, which form  the boundary $ \partial 
\cU_\e^{(\tp)} $ in \eqref{boundinst}. 
    
\section{Perturbed McLean curves $\cM^{(\tp)}_\pm(\e)$}\label{sec:Mclean.des}
The  perturbed McLean curves $\cM^{(\tp)}_\pm(\e)$ in \eqref{defMepm}
satisfy the following properties.  

\begin{proposition}\label{pertmcleandesc3456789} {\bf (Perturbed McLean curves)}
For any $\tp\geq 2$ there exist $\e^{(\tp)}>0 $ and  $C_\tp >0$ such that, for any $|\e|\leq \e^{(\tp)} $:

\noindent$\bullet$ {\sc Case $\tp=2 $. } The set $\cM^{(2)}_+ (\e)$ is a   connected curve, analytic   away from the origin,  with a cross-singularity at the origin described in \eqref{graphsmclean2}, while   $\cM^{(2)}_-(\e)
$ and $ \cT^{(2)}(\e)$  in \eqref{tracciazero} are real-analytic curves, with  two connected components each, see \Cref{fig:pertmctotal}. 
Moreover \eqref{asympmclean2intro} holds. 
$\cT^{(2)}(\e)$ is a real analytic graph over $\cM^{(2)}$, and    
 $\cM_\pm^{(2)}(\e)$ are real analytic graphs over  $\cT^{(2)}(\e)$ outside the origin. Near the origin  $\cM_\pm^{(2)}(\e) \equiv \cM^{(2)}_{\pm, {\rm loc}}(\e)$ in \eqref{M2pm}-\eqref{graphsmclean2}.\\
\noindent$\bullet$ {\sc Cases $\tp\geq 3$.} The sets 
$\cM_\pm^{(\tp)}(\e)$ in \eqref{defMepm} 
and  
$\cT^{(\tp)}(\e)$ in \eqref{tracciazero}   are connected real-analytic curves satisfying 
\begin{equation}\label{deformationmcleantp}
      \textup{d}_{\rm H}(\cT^{(\tp)}(\e),\cM^{(\tp)})\leq C_\tp \e^2 \, , \qquad   \textup{d}_{\rm H}(\cT^{(\tp)}(\e), \cM^{(\tp)}_\pm(\e) ) \leq C_\tp  |\e|^{\tp} \, .
    \end{equation}
 The set $\cT^{(\tp)}(\e)$ is a real analytic graph over $\cM^{(\tp)}$ and    
 $\cM_\pm^{(\tp)}(\e)$ are real analytic graphs over  $\cT^{(\tp)}(\e)$.
\end{proposition}

The rest of this section is devoted to prove \Cref{pertmcleandesc3456789}. 
The proof is divided into three  steps. 
First, 
locally close to $ (0,0) $, 
in  view of  \eqref{dpmbpm},  the sets $\cM^{(2)}_\pm (\e)$ are  described  in \eqref{graphsmclean2}.
Next  we exclude small neighborhoods of the origin and of the singular points $(0,\mu_*^\pm(\tp)) \subset  \{0\} \times \Z $
where the functions  $\mathsf{d}^{(\tp)}_\pm(\al,\mu,\e)$ 
and 
$T^{(\tp)}(\al,\mu,\e)$  defined in \eqref{def:dppm}
and \eqref{tracciaBep}, have a  Lipschitz singularity, of the form \eqref{dectM}.
Finally we argue  near such singular points. We first use a Lipschitz  implicit function theorem to show that  $\cT^{(\tp)}(\e)$ is a Lipschitz graph over $\cM^{(\tp)}$, and $\cM^{(\tp)}_\pm(\e)$ are  Lipschitz  graphs over $\cT^{(\tp)}(\e)$. Then   we  show that, for any 
$\e\neq 0 $,   the perturbed curves never cross the singular points, deducing their analyticity properties.


\smallskip 

\noindent{\bf {Step 1: 
$\cM^{(2)}_\pm (\e)$ and 
$ \cT^{(2)} (\e) $ near the origin.}}  
We claim that 
$$
\cM^{(2)}_\pm 
(\e) \cap B_{\rho_2} (0,0) = \big\{ (\al,\mu)\in B_{\rho_2} (0,0)\, : \ \mathsf{d}^{(\tp)}_\pm(\al,\mu,\e)=0 \big\}
= \cM^{(2)}_{\pm,{\rm loc}}(\e)
$$ 
defined in \eqref{M2pm}. Indeed, by  \eqref{def:dppm} 
$$
\mathsf{d}^{(2)}_-(0,0,\e) = \e^2(2\fb_2(0,0)- \fa_2(0,0)-\fc_2(0,0)) + r(\e^3) \stackrel{\eqref{dir.limit.abc}}{=} - 2 \e^2 + r(\e^4)
\neq 0 \, , \quad \forall
\e \neq 0 \, , 
$$
implying  that $\cM^{(2)}_-(\e)$ does not cross the origin and thus locally coincide with  $\cM^{(2)}_{-,{\rm loc}}(\e)$.  
Last, we describe $\cT^{(2)}_{\rm loc}(\e)$ close to zero with $ \mu \geq 0 $, which,  by \eqref{exp:tracetp}, \Cref{b00} and \eqref{eig.0}, is the set of solutions of 
\be\label{T2zero}
T^{(2)}(\al,\mu,\e) = -\frac{\mu^2}{4}(1+r(\rho))+\frac{\al^2}2(1+r(\rho)) + \e^2(1+r(\rho,\e^2)) = 0 \, .
\ee
The same procedure used for solving $\tb^- = 0 $ in \eqref{b-e} shows that the 
solutions to \eqref{T2zero}  
  are graphs of 
  Lipschitz functions
  on $ \alpha $ small.

\smallskip 

\noindent
\noindent{\bf {Step 2: non-singular points.}}
The set 
$ \cT^{(\tp)}(\e)$ in  \eqref{tracciazero} is the level set of the function 
$ T^{(\tp)}(\al,\mu,\e)$
in \eqref{tracciaBep} that, 
by \eqref{expa} and \eqref{expc}, has the form 
\begin{equation}\label{exp:tracetp}
    \begin{aligned}
        &T^{(\tp)}(\al,\mu,\e) = 
- \omega^{(\tp)}_+(\al, \mu) 
+
 \omega^{(\tp)}_-(\al, \mu)
 +\e^2 t_\tp (\al, \mu) + r^{(\tp)}(\e^4) = -m_\tp(\al,\mu) + \e^2 t_\tp (\al,\mu) +  r^{(\tp)}(\e^4)\, , \\
 &t_\tp(\al,\mu) = \fa_\tp (\al,\mu)+\fc_\tp(\al,\mu)\, .
    \end{aligned}
 \end{equation}
The  functions $T^{(\tp)}$ and $\mathsf{d}^{(\tp)}_\pm$ 
belong to
$\cA (K^{(\tp)} ,\e^{(\tp)};\R)$, by  \eqref{tracciaBep}, \eqref{def:dppm}  and  \eqref{regBtp}.
At $\e=0$ each 
 $$
 \cT^{(\tp)}(0) = \left\lbrace
(\al, \mu) \in K^{(\tp)} \ \vert \ 
m_\tp(\al, \mu) =  \omega^{(\tp)}_+(\al, \mu) - \omega^{(\tp)}_-(\al, \mu)    = 0
 \right\rbrace  = \cM^{(\tp)}\, ,  
 $$
reduces to the unperturbed McLean curve described in \Cref{lem:description}.  
We consider  an analytic parametrization   $\nu^{(\tp)} \colon S^1 \to \R^2 $ 
of $\cM^{(\tp)}$, such that 
$ m_\tp(\nu^{(\tp)} (\tau )) = 0 $ for any $ \tau \in S^1 $, and we define the 
outward  normal at any  $\nu^{(\tp)}(\tau)\in \cM^{(\tp)}$, 
$$
\vec{n}^{(\tp)}(\tau) := -\nabla_{\al,\mu} m_\tp (\nu^{(\tp)}(\tau))/ |\nabla_{\al,\mu} m_\tp(\nu^{(\tp)}(\tau))|\, ,
\quad \forall 
\tau \in S^1 \, ,
$$  
except 
at the point  $\nu^{(2)} (\tau) = (0,0) $, where $\vec{n}^{(2)}(\tau)$
is not defined. 
The set $ \cT^{(\tp)} (\e) $ can be described as 
\be\label{2002:1727}
\breve T^{(\tp)}(\tau,s,\e) := T^{(\tp)} ( \nu^{(\tp)}(\tau) + s \vec n^{(\tp)}(\tau) , \e ) = 0\, ,\quad 
\breve T^{(\tp)}(\tau,0,0)
\equiv 0\,  \, . 
\ee
  The  function $T^{(\tp)}$  belong to
$\cA (K^{(\tp)} ,\e^{(\tp)};\R)$
and so  
they
are analytic on $K^{(\tp)} \times B_{\e^{(\tp)}}(0)$ for any $\tp$ even,  
$ \tp \neq 2 $, while, if $\tp=2$, resp. $\tp$ odd, they  are analytic at any $(\al,\mu,\e)\in ( K^{(2)}\setminus [B_{\rho_2}(0,0)])\times B_{\e^{(2)}}(0)$, resp. at any $(\al,\mu,\e)\in (K^{(\tp)}\setminus [B_{\rho_\tp}(0,\mu_*^+ (\tp))\cup B_{\rho_\tp}(0,\mu_*^- (\tp))])\times B_{\e^{(\tp)}}(0)$ for some $\rho_\tp>0$, with $\mu_*^\pm(\tp)$ as in \eqref{sing.ML} which are integers for every $\tp$ odd. For unity of  exposition, we denote
$$
\cK^{(\tp)}:=\begin{cases}
   K^{(2)}\setminus B_{\rho_2}(0,0) \quad &\text{if }\tp = 2\, ,     \\
    K^{(\tp)}\setminus [B_{\rho_\tp}(0,\mu_*^+ (\tp))\cup B_{\rho_\tp}(0,\mu_*^- (\tp))]\quad &\text{if }\tp \text{ odd } \\
K^{(\tp)} \quad &\text{if }\tp \text{ even}, \tp\neq 2\, .
    \end{cases}
$$
\begin{lemma}
For any $ \tp \geq 2 $
there is $\e^{(\tp)}>0$  such that for any $|\e|\leq \e^{(\tp)}$, the set $\cT^{(\tp)}(\e)\cap \cK^{(\tp)}$ is a real analytic manifold 
with a analytic 
parametrization of the form 
\begin{equation}\label{eq:expparamtr}
    \nu_{\cT^{(\tp)}}(\tau,\e) = \nu^{(\tp)}(\tau) - \left(  \frac{t_\tp(\nu^{(\tp)}(\tau))}{|\nabla_{\al,\mu}m_\tp(\nu^{(\tp)}(\tau))|} \e^2 + r(\e^4)\right)\vec n^{(\tp)}(\tau)\, , 
\end{equation}
for any $\tau$ such that $\nu^{(\tp)}(\tau)\in \cK^{(\tp)}$. 
The sets  $\cM_\pm^{(\tp)}(\e) \cap \cK^{(\tp)} $ are graphs over $\cT^{(\tp)}(\e) \cap \cK^{(\tp)} $
of the form 
\begin{equation}\label{McLean.e2}
 \nu_{\cM^{(\tp)}_\pm}(\tau,\e)  = \nu_{\cT^{(\tp)}}(\tau,\e) \mp 
 \left(  2 \frac{\fb_{\tp}(\nu^{(\tp)}(\tau))}{|\nabla_{\al,\mu} m_\tp(\nu^{(\tp)}(\tau))|}\e^\tp + r(\e^{\tp+1}) \right) \vec{n}_{\cT^{(\tp)}}(\tau,\e) \,   
\end{equation}
where $ \fb_{\tp} (\alpha, \mu )$ is defined in \eqref{abceta}.
\end{lemma}

\begin{proof}
The function $\breve T^{(\tp)}(\tau,s,\e)$ 
in \eqref{2002:1727} is analytic on the set $\{ (\tau,s)\, : \ \nu^{(\tp)}(\tau) + s \vec n^{(\tp)}(\tau)\in \cK^{(\tp)} \}\times B_{\e^{(\tp)}}(0)$ for $s_0$ sufficiently small, and satisfies
$$
\pa_s  \breve T^{(\tp)}(\tau,0,0)  = 
|\nabla_{\al,\mu}m_\tp(\nu^{(\tp)}(\tau))|  \neq 0  \quad
\text{for every }\tau \in I.$$
Then, by the analytic Implicit Function Theorem
and \eqref{exp:tracetp},  we deduce  \eqref{eq:expparamtr}. 

We now construct $\cM_\pm^{(\tp)}(\e)$ as graphs over $\cT^{(\tp)}(\e)$. In view of \eqref{defMepm} we 
have to solve  the implicit equation
\be\label{dpmame=0}
\breve{\mathsf{d}}(\tau, s, \e):= \mathsf{d}^{(\tp)}_+(\nu_{\cT^{(\tp)}}(\tau,\e) + s \vec{n}_{\cT^{(\tp)}}(\tau,\e), \e) \equiv 0  
\ee
where 
$ \vec{n}_{\cT^{(\tp)}}(\tau,\e) $ is the outward normal to $\cT^{(\tp)}(\e)$, 
which is analytic in $(\tau,\e)$.
Since by \eqref{eq:expparamtr}  $\nu_{\cT^{(\tp)}}(\tau,0)  = \nu^{(\tp)}(\tau)$ and since by \eqref{expb} $\fb^{(\tp)}(\al,\mu,0)\equiv 0$, we get
$$\breve{\mathsf{d}}(\tau, 0,0) \stackrel{\eqref{def:dppm},\eqref{expb}}{=} T^{(\tp)}(\nu^{(\tp)}(\tau), 0) \stackrel{\eqref{exp:tracetp}}{=} 0 \quad \text{and} \quad
(\pa_s \breve{\mathsf{d}})(\tau, 0, 0)  \stackrel{\eqref{def:dppm},\eqref{expb},\eqref{exp:tracetp}}{=} |\nabla_{\al,\mu}m_\tp(\al,\mu)| \neq 0 \ , 
$$
 thus the implicit function theorem guarantees the existence of a unique  analytic function $s(\tau, \e)$, which, expanding in $s$ and employing \eqref{def:dppm} and \eqref{expb}, satisfies
    \begin{align*}
       0 \equiv \breve{\mathsf{d}}(\tau, s(\tau, \e), \e) =
       2\fb_\tp(\nu_{\cT^{(\tp)}}(\tau,\e))\e^{\tp} + 
       s(\tau, \e)  \underbrace{\grad_{(\al, \mu)} T^{(\tp)}(\nu_{\cT^{(\tp)}}(\tau,\e), \e) \cdot \vec{n}_{\cT^{(\tp)}}(\tau,\e)}_{=|\grad_{(\al, \mu)} T^{(\tp)}(\nu_{\cT^{(\tp)}}(\tau,\e), \e)|} +  r(\e^{\tp+1}, (s(\tau, \e))^2)  \ .
   \end{align*}
 In particular, solving for $s(\tau,\e)$ we find
   $$
        s(\tau, \e) = -2  \e^\tp\frac{\fb_{\tp}(\nu^{(\tp)}(\tau))}{|\grad_{(\al, \mu)} T^{(\tp)}(\nu_{\cT^{(\tp)}}(\tau,\e), \e)|} + r(\e^{\tp+1}) 
        = - 2  \e^\tp\frac{\fb_{\tp}(\nu^{(\tp)}(\tau))}{|\grad_{(\al, \mu)} m_\tp (\nu^{(\tp)}(\tau))|} + r(\e^{\tp+1}) 
   $$
 for any $ \tau $ such that  $ \nu^{(\tp)}(\tau)\in \cK^{(\tp)}$. 
   By setting
   $
   \nu_{\cM^{(\tp)}_\pm}(\tau,\e) := \nu_{\cT^{(\tp)}}(\tau,\e) + s(\tau,\e) \, \vec n_{\cT^{(\tp)}}(\tau,\e)
   $
   we obtain
   \eqref{McLean.e2}. 
   \end{proof}

 \noindent{\bf {Step 3: singular points for $\tp \geq 3 $ odd}.} For any $\tp$ odd,   the functions $T^{(\tp)}(\cdot ,\cdot ,\e)$ and $\mathsf{d}^{(\tp)}_\pm (\cdot ,\cdot  ,\e)$ are only  Lipschitz in $ B_{\rho_\tp}(0,\mu_*^\pm(\tp))$. 
 Suppose with no loss of generality that $\nu^{(\tp)}(0) = (0,\mu_*^+(\tp))$.

 \begin{lemma}
For any  $\tp \geq 3 $ odd,  
there is $ \e^{(\tp)} > 0 $ such that 
$\cT^{(\tp)}(\e)\cap 
B_{\rho_{\tp} }(0, \mu_*^\pm(\tp))$
is described by a
Lipschitz parametrization 
of the form
\be\label{singLip}
\nu_{\cT^{(\tp)}}(\tau,\e) := \nu^{(\tp)}(\tau) + \Big( -\e^2 \frac{t_\tp(\nu^{(\tp)}(\tau))}{|\nabla_{\al,\mu} m_\tp(\nu^{(\tp)}(\tau))|} + r^{(\tp)}(\e^4)\Big) 
\vec n^{(\tp)}(\tau)\, , \quad (\tau,\e)\in B_{\tau^{(\tp)}}(0)\times B_{\e^{(\tp)}}(0)\, .
\ee
The sets  $\cM_\pm^{(\tp)}(\e) \cap B_{\rho_{\tp} }(0, \mu_*^\pm(\tp)) $ 
are Lipschitz graphs over $\cT^{(\tp)}(\e) \cap B_{\rho_{\tp} }(0, \mu_*^\pm(\tp)) $
of the form 
\be\label{nuMep}
\nu_{\cM^{(\tp)}_\pm}(\tau,\e) = \nu_{\cT^{(\tp)}}(\tau,\e) \mp \left( 2  \frac{\fb_\tp (\nu^{(\tp)}(\tau))}{|\nabla_{\al,\mu}m_\tp(\nu^{(\tp)}(\tau))|}  \e^\tp + r^{(\tp)}(\e^{\tp+1})\right)\vec n^{(\tp)}(\tau) \, .
\ee
 \end{lemma}

 \begin{proof}
 We 
use the  Lipschitz  Implicit Function Theorem  in \cite[Theorem 4.8]{simader}.
  By direct inspection, using \eqref{mcleanmanifoldsp0} and \eqref{sing.ML}, we have that 
 $$
 |\pa_s m_\tp(\nu^{(\tp)}
(\tau)+s \vec n^{(\tp)}(\tau))|\vert_{\tau=s=0} = |\nabla_{\al,\mu} m_\tp(0,\mu_*^+(\tp)) |  =:C_\tp> 0\, . $$
By continuity, there exist $\tau^{(\tp)},s^{(\tp)}>0$ such that
$$
|\pa_s m_\tp(\nu^{(\tp)}
(\tau)+s \vec n^{(\tp)}(\tau))| \geq \frac{C_\tp}{2}\qquad \forall (\tau,s)\in B_{\tau^{(\tp)}}(0)\times B_{s^{(\tp)}}(0) \, .
$$
Let $L_{\tp}$ be a Lipschitz constant of $t^{(\tp)}$ in $B_{\rho_\tp}(0,\mu_*^+(\tp))\times B_{\e^{(\tp)}}(0)$.  Then, taking $\e^{(\tp)}$ sufficiently small,  the function  $\breve{T}^{(\tp)}(\tau,s,\e)$ in
 \eqref{2002:1727} satisfies, recalling \eqref{exp:tracetp},
$$
\left|\breve{T}^{(\tp)}(\tau,s,\e) - \breve{T}^{(\tp)}(\tau,s',\e) \right| \geq \underbrace{\left(\frac{C_\tp}{2}-\e^2 L_{\tp}\right)}_{>0}|s-s'|\, , \quad 
\forall \  \tau \in B_{\tau^{(\tp)}}(0),  \   \e\in B_{\e^{(\tp)}}(0), \  s,s'\in B_{s^{(\tp)}}(0) \ .
$$
Since  $\breve{T}^{(\tp)} (\cdot, \cdot, \cdot ) $ is also Lipschitz  in $B_{\tau^{(\tp)}}(0)\times B_{s^{(\tp)}}(0) \times B_{\e^{(\tp)}}(0) $, the Lipschitz implicit function  \cite[Theorem 4.8]{simader} guarantees the existence of a unique Lipschitz function 
$B_{\tau^{(\tp)}}(0)\times  B_{\e^{(\tp)}}(0) \to  B_{s^{(\tp)}}(0) $,   
$ (\tau,\e) \mapsto s(\tau,\e) $, of the form \eqref{singLip},  solving $\breve{T}^{(\tp)}(\tau,s(\tau,\e),\e) \equiv 0$ for all $(\tau,\e)\in B_{\tau^{(\tp)}}(0)\times B_{\e^{(\tp)}}(0) $.

A similar procedure can be used to describe the curves $\cM^{(\tp)}_\pm(\e)$ near the singular points.
\end{proof}

 Actually  the perturbed 
 curve $\cT^{(\tp)}(\e)$ does not cross the singular point $(0,\mu_*^\pm(\tp))$.
We need the following lemma. 

\begin{lemma}\label{prop:values}
 For any odd $\tp \geq 3 $, the functions  $\fa_\tp (\alpha, \mu) $ and $\fc_\tp (\alpha, \mu) $ in \eqref{expa}, \eqref{expc}, evaluated at the points $(0, \mu_*^\pm(\tp))$ in \eqref{sing.ML} where the unperturbed McLean curve
 $ \cM^{(\tp)} $ intersect $ \{ \alpha = 0 \} $, are,
 setting $ \tm := \frac{\tp-1}{2}  $, 
\begin{equation}
\label{apcpatextremes}
\fa_\tp(0,\mu_*^+(\tp))= 
        \tfrac12 (1+\tm)^2\, , \ \fc_\tp(0,\mu_*^+(\tp))=-\tfrac{\tm^2}{2}\, ,  \  
        \fa_\tp(0,\mu_*^-(\tp))= -\tfrac{\tm^2}2 \, , \  \fc_\tp(0,\mu_*^-(\tp))= \tfrac12 (1+\tm)^2 \, .
\end{equation}
\end{lemma}
 
\begin{proof}
In  Appendix \ref{sec:5}.
\end{proof}

As a consequence 
we obtain the following 
corollary.

\begin{lemma}
For any  $\tp \geq 3 $ odd, 
$\cT^{(\tp)}(\e)\cap B_{\rho_\tp}(0,\mu_*^\pm(\tp))$ and 
$\cM^{(\tp)}_\pm (\e)\cap B_{\rho_\tp}(0,\mu_*^\pm(\tp))$ 
are  analytic manifolds. 
\end{lemma}

\begin{proof}
For any $\e\neq 0$ small enough,
 $$
 T^{(\tp)}(0,\mu_*^+ (\tp),\e) \stackrel{\eqref{exp:tracetp}}= -\underbrace{m_\tp (0,\mu_*^+(\tp))}_{=0} + \e^2 t_\tp(0,\mu_*^+(\tp)) + r^{(\tp)}(\e^4)  \stackrel{\eqref{apcpatextremes}}{=} \frac{\tp}{2}\e^2 + r(\e^4) \neq 0  \, .
 $$
As a result, any point of $\cT^{(\tp)}(\e)$ has a neighborhood in which the function $T^{(\tp)}(\al,\mu,\e)$ is analytic in $(\al,\mu)$ and the gradient $\nabla_{\al,\mu}T^{(\tp)}(\al,\mu,\e) = -\nabla_{\al,\mu} m_\tp(\al,\mu) + \cO(\e^2) \neq 0$ does not vanish. Thus, by the analytic implicit function theorem, $\cT^{(\tp)}(\e)\cap B_{\rho_\tp}(0,\mu_*^+(\tp))$ is an analytic manifold.
In view of \eqref{nuMep} the Lipschitz sets 
$\cM^{(\tp)}_\pm(\e)$ are $\cO(\e^\tp) $ perturbations of $\cT^{(\tp)}(\e)$, which are
$ \cO(\e^2)$-distant from the singular point $(0,\mu_*^+(\tp))$. Therefore also $\cM^{(\tp)}_\pm(\e)$ do not intersect $(0,\mu_*^+(\tp))$, and are therefore are analytic curves as well. 
\end{proof}

\Cref{pertmcleandesc3456789} follows by 
gluing the local analytic branches of $\cM^{(\tp)}_\pm(\e)$, $\cT^{(\tp)}(\e)$ constructed in the previous steps. 
In particular the bounds  \eqref{asympmclean2intro}
 follow by 
\eqref{eq:expparamtr}
and \eqref{singLip} outside $ B_{\rho_2}(0,0)$ and noting that, 
by \eqref{graphsmclean2},   
inside $B_{\rho_2}(0,0)$, 
it results   $|\mu^+(\al,\e^2)-\mu^+(\al,0)| \lesssim \e^2$, $|\mu^-(\al,\e)-\mu(\al,0)|\lesssim |\e|$. 

This concludes the proof of \Cref{pertmcleandesc3456789}.

\subsection{Instability criterion}

 We may now prove the instability criterion of \Cref{lem:inst}.

\begin{proof}[Proof of \Cref{lem:inst}.]
Note  that 
\be\label{tagliaalberi}
\begin{aligned}
& \{(\al,\mu)\in K^{(\tp)}\, : \ \mathsf{d}^{(\tp)}_+(\al,\mu,\e)>0 \} =   K^{(\tp)}\setminus \mathring\cM^{(\tp)}_+(\e)\, \\
& \{(\al,\mu)\in K^{(\tp)}\, : \ \mathsf{d}^{(\tp)}_-(\al,\mu,\e)>0 \} = \mathring\cM^{(\tp)}_-(\e) \cap K^{(\tp)} \, ,
\end{aligned}
\ee
 since, in view of  \eqref{tracciaBep}, \eqref{def:dppm} and \Cref{prop:katored}, 
$$
\begin{aligned}
    &\{(\al,\mu)\in K^{(\tp)}\, : \ \mathsf{d}^{(\tp)}_+(\al,\mu,\e)>0 \}\stackrel{\e\sim 0}{\sim } K^{(\tp)}\cap \{(\al,\mu)\, : \ m_\tp(\al,\mu)<0  \} \quad \text{the {\it outer} region to } \cM^{(\tp)}\, , \ \text{ see \eqref{outinunp}}\\
    &\{(\al,\mu)\in K^{(\tp)}\, : \ \mathsf{d}^{(\tp)}_-(\al,\mu,\e)>0 \}\stackrel{\e\sim 0}{\sim } K^{(\tp)} \cap \{(\al,\mu)\, : \ m_\tp(\al,\mu)>0  \} \quad \text{the {\it enclosed} region inside } \cM^{(\tp)} \, . \\
\end{aligned}
$$
Now we prove \eqref{b.not.vanishing}.
If $\fb^{(\tp)} (\al, \mu,\e)\vert_{\cT^{(\tp)}(\e)} \not \equiv 0$, there exists $ (\alpha, \mu) \in  \cT^{(\tp)}(\e)$ where $\fb^{(\tp)} (\al, \mu,\e) \neq 0$ and thus  $D^{(\tp)}(\al, \mu, \e)$ in \eqref{traccianulla} is strictly positive, so 
  $  \cU_\e^{(\tp)}$ in \eqref{def:instaintro}
is nonempty. Viceversa, suppose that   $\cU_\e^{(\tp)} \neq \emptyset$ and assume by contradiction that $\fb^{(\tp)} (\al, \mu,\e)\vert_{\cT^{(\tp)}(\e)}  \equiv 0$.
Then $\cT^{(\tp)}(\e) \subset \cM^{(\tp)}_\pm (\e)$, hence by \eqref{newoffice}
$$
\cM^{(\tp)}_+(\e) \cap \cM^{(\tp)}_- (\e) \subset \cT^{(\tp)}(\e)   \subset  \cM^{(\tp)}_\pm(\e) \quad \Rightarrow \quad \cM^{(\tp)}_+(\e) \cap \cM^{(\tp)}_- (\e) = \cT^{(\tp)}(\e) \, . 
$$
But then, since $\cM^{(\tp)}_\pm(\e) $ are graphs over $\cT^{(\tp)}(\e)$,
necessarily $\cM^{(\tp)}_+(\e) \equiv \cT^{(\tp)}(\e) \equiv \cM^{(\tp)}_-(\e)$.
Consequently,  by \eqref{tagliaalberi}, $D^{(\tp)}  =\mathsf{d}^{(\tp)}_+\mathsf{d}^{(\tp)}_- \leq 0$ in $K^{(\tp)}$ and  $\cU_\e^{(\tp)} = \emptyset$. This is a contradiction.

We prove now that 
       \eqref{bp} implies \eqref{b.not.vanishing}.
       Indeed, in view of \eqref{expb}, the  condition  \eqref{bp} implies  the 
   existence of a neighborhood $\cU(\und \al,\und \mu)\ni (\und \al,\und \mu)$ and  $\e(\und \al,\und \mu)>0$ such that $\fb^{(\tp)} (\al,\mu,\e) \neq 0$ for every $(\al,\mu,\e)\in \cU(\und \al,\und \mu) \times [B_{\e(\und \al,\und \mu)}(0)\setminus\{0\}]$. By \Cref{pertmcleandesc3456789} the zero-trace set \eqref{tracciazero} lies closer to the corresponding McLean curve $\cM^{(\tp)}$ the more $\e$ is close to  $0$. As soon as $\cT^{(\tp)}(\e)$ intersects the neighborhood $\cU(\und \al,\und \mu)$,
   the function $\fb^{(\tp)} (\al, \mu,\e)\vert_{\cT^{(\tp)}(\e)} \not \equiv 0$ and thus by \eqref{b.not.vanishing} the instability region $\cU_\e^{(\tp)}\supset \cT^{(\tp)}(\e) \cap \cU(\und \al,\und \mu)$ is not empty. 
   \end{proof}

In the next lemma we verify the 
instability criterion 
for $\tp = 3$.

\begin{lemma}\label{b00_3}
{\bf (Instability criterion for $ \tp = 3 $)} The coefficient 
$ \fb_3(\al_3(0),0) \neq 0 $  where $\al_3(0)$ is the unique positive number such that   $(\al_3(0),0) \in  \cM^{(3)}$.
Therefore  $\cM^{(3)}_+(\e)\cap \cM^{(3)}_-(\e)$ is discrete.
\end{lemma}

\begin{proof}
We claim that our matrix 
$\tL^{(3)}(\al_3(0)+ \delta,0,\e)$ coincides with the matrix $\tL_{\e,\delta}$ of  \cite[lemmata 6.1 and 6.2]{CNS},  i.e.
$$
\tL^{(3)}(\al_3(0)+ \delta,0,\e) = \tL_{\e,\delta} =
\begin{pmatrix}
     -\im & 0 \\
     0 & \im
 \end{pmatrix} 
 \begin{pmatrix}
      -\omega_1^+(\al_3(0),0) -A (\e,\delta) & -B(\e,\delta)  \\
     -B(\e,\delta)  &   \omega_1^+(\al_3(0),0)+ C(\e,\delta)  
 \end{pmatrix} 
$$
where 
 $ A,B,C : 
\{ |\delta | < \delta_0 \} \times \{ |\epsilon | < \epsilon_0 \} \to \R $ 
are real analytic functions and 
$$
 B(\e,\delta)  = b_{3,0}\e^3 + \cO(\e^4 + \delta^4 )
 \quad \text{with} \quad b_{3,0}\neq 0 \, .
$$
Indeed our  perturbed 
basis \eqref{basisFp} is the same one used in 
\cite{CNS} up to a scaling factor, 
 $$
 v^{(3)}_+(\al_3(0)+\delta,0,\epsilon) = \frac{V_1^{\e,\delta}}{\sqrt{2}} \, , \quad 
 v^{(3)}_-(\al_3(0)+\delta,0,\epsilon) = \frac{V_2^{\e,\delta}}{\sqrt{2}}  \,,  
 $$ 
 where $ V_m^{\e,\delta} $, 
 $ m = 1,2 $ are defined in 
 \cite[equation (3.25)]{CNS}.  
Thus  $ \fb_3(\al_3(0),0) =  -b_{3,0}  \neq 0 $. 
Thus the analytic function  
$\fb^{(3)}(\cdot,\cdot,\e)$  vanishes at most at finitely many points on $\cT^{(3)}(\e)$,  and 
\Cref{lem:inst} guarantees that 
the analytic manifolds $\cM^{(3)}_+(\e)\cap \cM^{(3)}_-(\e)$ intersect at most finitely many times. 
\end{proof}

\subsection{Proof of \Cref{TeoremoneMcLean,TeoremoneIntro}  
}\label{sec:verification}

\noindent{\bf Proof of \Cref{TeoremoneMcLean}.} 
The matrix representation 
\eqref{tocomputematrixb} is proved in  \Cref{lem:katored}
and the expansions \eqref{matrixentries}  in 
\Cref{prop:katored}.  All the claimed properties \eqref{boundinst}-\eqref{asympmclean2intro} of the McLean curves $\cM^{(\tp)}_\pm(\e)$ are proved in \Cref{pertmcleandesc3456789}.
Property \eqref{symdifmclean} follows directly by \eqref{tagliaalberi}. Equation \eqref{newoffice} follows by \eqref{def:dppm}
and \eqref{upboundeigv} by  \eqref{eigenvalues}, \eqref{matrixentries}.
By \Cref{TeoremoneFinale} the instability region  
$\cU_\e^{(2)}\neq \emptyset $. 
\Cref{b00_3} proves that $\cU_\e^{(3)}\neq \emptyset$.
In order to prove 
\eqref{ordinstaintro2} and 
\eqref{intersectionmcleans}
we rely on the following result.

\begin{lemma}\label{lem:coeff.b2}
The coefficient $ \fb_2(\al,\mu) $ in  
\eqref{abceta} 
satisfies
\be\label{b2.alongMcLean}
\fb_2(\al,\mu)<0 \  \ \   \forall (\al, \mu) \in \cM^{(2)}\setminus \{(0, \pm \tfrac54)\} \ , 
\quad 
\fb_2(0,\tfrac54) = 0 \ , 
\quad 
\pa_\mu \fb_2(0,\tfrac54) = \tfrac{1}{2\sqrt 3} \, . 
\ee 
Furthermore, recalling \eqref{abceta} and \eqref{tracciaBep},
\begin{equation}\label{b2onmclean}
 \beta_2 (0,\tfrac54) = \tfrac{39\sqrt 3}{512}\, , \qquad \begin{aligned}
        &T_1:=\pa_\mu T^{(2)}(0,\tfrac54,0)= |\nabla_{\al,\mu} m_2(0,\tfrac54)|=\tfrac43\, , \\ &T_2:= \fa_2(0,\tfrac54) + \fc_2(0,\tfrac54) = \tfrac{19}{16}\ . 
 \end{aligned}
\end{equation}
\end{lemma}

\begin{proof}
    The long proof is in  Appendix \ref{sec:5}.
\end{proof}
\medskip
\noindent
\underline{\sc Proof of the lower bound in \eqref{ordinstaintro2}, \eqref{ordinstaintro3}.} 
Taking $ r>0$ sufficiently small, there exists $ c_r >0$, such that, for any $(\al,\mu) \in \cT^{(2)}(\e) \setminus B_r(0,\tfrac54) $, cf. \eqref{tracciazero},  
one has 
\begin{align*}
& \left|\Re\lambda^{(2)}_\pm  (\al,\mu,\e)\right|  \stackrel{\eqref{eigenvalues},\eqref{traccianulla}}= |\fb^{(2)}(\al,\mu,\e)| \stackrel{\eqref{expb},\eqref{b2.alongMcLean}}\geq c_r |\e|^2\, , \qquad \forall |\e|\leq \e^{(2)}  \, ,
\end{align*}
proving  \eqref{ordinstaintro2}.
In the case $\tp = 3$, considering an analytic parametrization 
$\tau \mapsto \nu_{\cT^{(3)}}(\tau, \e)$ of $\cT^{(3)}(\e)$,  one has 
$$
\left|\Re\, \lambda^{(3)}_\pm(\nu_{\cT^{(3)}}(\tau;\e),\e)\right| = |\fb^{(3)}(\nu_{\cT^{(3)}}(\tau;\e),\e)| \stackrel{\eqref{expb}, \eqref{eq:expparamtr}} \geq  \tfrac12 |\fb_3(\nu^{(3)}(\tau)) +
r(\e^2)|\e^3
$$
where $\nu^{(3)}(\tau)$ is an analytic parametrization of $\cM^{(3)} = \cT^{(3)}(0)$. The proof of
\eqref{ordinstaintro3} follows because the analytic function $\fb_3(\nu^{(3)}(\tau))$ is different from zero except for finitely many $ \tau $ by 
\Cref{b00_3}. 
\\[2mm]
\noindent{\bf Proof of \Cref{TeoremoneIntro}.}
It remains only to prove 
\eqref{intersectionmcleans}.

\smallskip
\noindent\underline{\sc Proof that $\cM^{(2)}_+(\e) \cap \cM^{(2)}_-(\e) = \emptyset$  for  any $\e \neq 0$.}
In view of \eqref{newoffice}
it is sufficient to show that,  for any  $
0 < |\e| \leq \e^{(2)}$ small enough, the function $\fb^{(2)}(\al,\mu,\e)$ never vanishes on the curve $\cT^{(2)}(\e)$ defined in  \eqref{tracciazero} by $ T^{(2)} (\al, \mu,\e) = 0 $. In view of the symmetry \eqref{symmetryapcpbp}  it is sufficient to  consider the right component of $\cT^{(2)}(\e)$. By  \eqref{b2.alongMcLean}, the function $\fb_2(\nu_{\cT^{(2)}}(\tau,0)) < 0 $  for any $\tau$ such that  $\nu_{\cT^{(2)}}(\tau;0) \neq  (0,\tfrac54 )$. Therefore,  
by compactness and \eqref{eq:expparamtr},  for $|\e| $ small enough,
the function $\fb^{(2)}(\cdot,\cdot,\e)\vert_{\cT^{(2)}(\e)} < 0$ outside  a small neighborhood of $(0,\tfrac54)$.  
We now prove that 
$\fb^{(2)}(\cdot,\cdot,\e) < 0$ on $ \cT^{(2)}(\e) \cap B_{\rho_2}(0,\tfrac{5}{4}) $ for some $ \rho_2 >0 $. Consider a local  analytic parametrization $(-\tau_0,\tau_0) \to\cT^{(2)}(\e)\cap B_{\rho_2}(0,\tfrac54) $, $ \tau \mapsto \nu_{\cT^{(2)}}(\tau,\e) $ 
as in \eqref{eq:expparamtr} fulfilling  $\nu_{\cT^{(2)}} (0,\e) = (0, \underline{\mu}(\e)) 
= \cT^{(2)} \cap \{ \alpha = 0 \} 
$ in the half plane $ \mu > 0 $.
It results 
$ \underline{\mu}(\e) = \tfrac54 + \cO(\epsilon) $. 
By Taylor expansion in $ (\e, \tau)$ at $ (0,0) $, using 
\eqref{expb},   \eqref{eq:expparamtr}
for $ \tp = 2 $, $ \nu_{\cT^{(2)}}(\tau,0) = \nu^{(2)}(\tau)$, 
noting that $ \vec n^{(2)}(0) $ is the versor along the $ \mu $ axis, 
\eqref{exp:tracetp} and \eqref{b2onmclean}
one gets
\begin{equation}\label{b2onT2f}
    \fb^{(2)}(\nu_{\cT^{(2)}}(\tau,\e),\e)  = \e^2\Big(\underbrace{\fb_2(\nu_{\cT^{(2)}}(\tau,0))}_{\leq 0\text{ by }\eqref{b2.alongMcLean}}
   + \e^2 
 \big(\underbrace{\beta_2(0,\tfrac54) - \pa_\mu \fb_2(0,\tfrac54) \frac{T_2}{T_1}}_{=-\frac{37\sqrt3}{512}\text{ by }\eqref{b2onmclean}} + \cO(\tau,\e)  \big) \Big) < -
 \e^4 \tfrac{37\sqrt{3}}{1024}
\end{equation}
for any $ |\tau | $ small, uniformly in $ \epsilon $.

\appendix

\section{Conformal flattening} 
\label{ApC}

According to \eqref{LC}, the deviation of the conformal diffeomorphism $(X,Y,Z) \mapsto (x,y,z)$ which flattens the domain 
$\cD_{\eta_\e} $ in 
\eqref{flatdom} from the identity decays exponentially as $Z \to -\infty$. 
To characterize this asymptotic behavior, we introduce the following function spaces: given $s \in \mathbb{R}$, $b \in \mathbb{N}_0$, and $a \ge 0$, we define
\begin{equation}\label{L2ascalar}
    L^{2,a}:=L^{2,a}(\R_-,\C) := L^2(\R_-,e^{-2az}\de z;\C) \, , \qquad  
            \langle f,g\rangle_{L^{2,a}} := \int_{-\infty}^0 \bar f(z) g(z)e^{-2az}\de z\ 
\end{equation}
and 
\begin{equation}
\label{spaziok}
H^{s, b}_a := H^{s, b}_a(\T\times\R_-):= 
\Big\{ 
u(x,z) = \sum_{k \in \Z}   u_k (z) e^{\im k  x}
\, \colon \, \T \times (-\infty, 0] \to \C \ \ \mbox{:} \ \    \| u \|_{s,b, a} < \infty 
\Big\}  \,
\end{equation}
endowed with  the norm 
\begin{equation}
\label{usba}
        \| u \|_{ s,b, a}^2  := 
\sum_{j = 0}^b 
\| \pa_z^j u \|_{L^{2,a}( \R_-;H^{s-j})}^2 
= 
\sum_{j = 0}^b \sum_{k \in \Z}
\langle k \rangle^{2(s-j)}
\| \pa_z^j u_k \|_{L^{2,a}}^2  \, .
\end{equation}
Each mode $ u_k(z) $ of a  function $ u $ in $ H^{s,b}_a $ 
exhibit uniform exponential decay   as $ z \to - \infty $.

\begin{remark}\label{remBMV}
The key difference
with respect to 
the spaces $H^{\sigma, s, a}$  in \cite[equation 2.7]{BMV2},
is that the norm 
\eqref{usba} distinguishes the horizontal space regularity $s$ from the vertical one $b$ (and we drop the $\sigma$-analyticity in $x$). For any $s\in \N_0 $ the norm  $\| \cdot \|_{s,s,a}$ in \eqref{usba} coincides with the norm $\|\cdot \|_{0,s,a}$ in \cite[equation (2.7)]{BMV2}.
\end{remark}

The next result is essentially  established in \cite[Proposition~3.3]{NS},   \cite{BMAMS}[Section 2.4]. We revisit the construction  here to prove the quantitative decay estimate  \eqref{est:d}, roughly stating that $ |d_\e (x,z)| \lesssim \epsilon e^{z} $,  and  property \eqref{deintF}. 

\begin{lemma}
\label{LeviCivita}
{\bf (Levi-Civita conformal flattening)} 
There is $ \e_0 > 0 $ such that  for any $ |\e| \leq \e_0  $ there exist a constant $ \tc_\e$ and a smooth $2\pi$-periodic odd function  $ \mathfrak{p}(X):=\mathfrak{p}_\e(X) = \sum_{k\in \Z} \mathfrak{p}_k e^{\im k X}$,
analytic in $ \e $
as a map 
$B(\e_0) \to H^{s}(\T)$ for any $ s \in \R $,  satisfying
\begin{equation}\label{est:fp}
\| \mathfrak{p}_\e \|_{H^s(\T)}= \cO (\e) \, , 
\end{equation}
such that:
\\[1mm]
(i) the change of coordinates
\begin{equation}
\label{LC}
    \begin{cases}
     x=  U(X, Z) = X + \sum_{k \neq 0} \mathfrak{p}_k \, e^{|k| Z } \, e^{\im k X} \, , 
     \\ 
     y = Y\, , \\
        z =  V(X, Z) = Z + \tc_\e + \sum_{k \neq 0} \im \, \textup{sign}(k) \mathfrak{p}_k \, e^{|k| Z} \, e^{\im k X} \, ,
    \end{cases}
\end{equation}
    is a conformal diffeomorphism  from  $\R^2 \times (- \infty, 0)$ to $\{(x,y,z)\in \R^3\, :\ z < \eta_\e(x)\}$ that transforms  the boundary 
    $ \R^2 \times \{ 0\} $ into the boundary $ \{(x,y,z)\in\R^3\, : \ z = \eta_\e (x) \} $. 
    The  functions $U,V$ satisfy the Cauchy--Riemann equations
$ U_X = V_Z $ and $
U_Z = - V_X $. 
    The map
    $ X \mapsto x = X+ \mathfrak{p}(X) $ defines a diffeomorphism of  $ \R $ satisfying 
    \begin{equation}\label{LCb}
    U(X,0) = X + \mathfrak{p}(X) \ , 
    \quad 
    V(X, 0) = 
        \eta_\e (U(X,0)) =  
        \eta_\e (X + \mathfrak p(X)) 
       \, .
    \end{equation}
    The function $ U(X,Z) $ is odd in $ X $ and $V(X,Z)$ is even in $X$.
\\[1mm]
(ii) The real valued function
\begin{equation}\label{def:d}
        d_\e(X, Z):=  U_X^2 (X,Z) +  U_Z^2 (X, Z) - 1  
    \end{equation}
    is even in $ X $ and analytic in $\e $ as a map $ B_{\e_0}(0) \to H^{s,b}_a(\T\times\R_-)$ for any $a\in (0,1)$, $ s \geq  3 $, $b \in \N_0 $, 
    and 
\begin{equation}
\label{est:d}
        \|d_\e\|_{s,b,a} \leq C_{s,b} |\e| \, , \qquad \forall \e \in B_{\e_0}(0) \, .
    \end{equation}
    For any $Z \in \R_-$, 
    \begin{equation}\label{deintF}
        d_\e(X,Z) = \sum_{\ell \geq 1} \e^\ell d_{\ell}(X,Z) \ , \quad d_\ell(X,Z) = \sum_{\substack{k = 0,\dots,\ell \\
        k\equiv \ell \ \textup{mod }2 }} d_\ell^{[k]}(Z) \cos (kX)\, , 
    \end{equation}
with   $ d_1 (X,Z) = 2 e^Z \cos (X)  $ and $ d_2 (X,Z) = e^{2Z} (1+4 \cos (2X))$. 
\end{lemma}

\begin{proof}\smallskip
\noindent
\textit{(i)} 
The function $\mathfrak{p} $ is determined as the fixed point of
\begin{equation}\label{def:fp}
    \mathfrak{p} = \cH\big[\eta_\e \circ (\mathrm{Id} + \mathfrak{p})\big] 
\end{equation}
where $\cH = - \im \, \textup{sgn} (D)$ is the Hilbert transform, see \cite[Proposition~3.3]{NS}, \cite{BMAMS}[(2.125)].  
In view of 
\Cref{PTW} 
the function $\e \mapsto \eta_\e \in H^{\sigma,s} (\T) $ is analytic for any $\sigma > 0$ and $s > 5/2$, and therefore the map
$$
F : B_{\epsilon_0}(0) \times H^s (\T) \to H^s(\T)\, , \quad
(\epsilon, \mathfrak{p}) \mapsto 
F(\epsilon, \mathfrak{p}) := 
\mathfrak{p} - 
\cH\big[\eta_\e \circ (\mathrm{Id} + \mathfrak{p})\big] \, , 
$$
is analytic. 
Actually  the composition operator $p \mapsto \eta(x + p(x))$ induced by an analytic function $\eta(x) \in H^{\sigma_1,s_1} (\T) $ is analytic on $H^{\sigma_0,s_0}(\T)$ for any $ 0 \leq \sigma_0 < \sigma_1$ and $s_0 > 1/2$.  Since $ F(0,0) = 0 $
and $ \pa_{\mathfrak{p}} F(0,0) = \rm{Id} $,
the analytic implicit function theorem implies the existence of a unique analytic solution $ \mathfrak{p}_\epsilon $
of \eqref{def:fp}
satisfying \eqref{est:fp} (the same holds for  $\mathfrak{p}_\epsilon $ in $ H^{\sigma_0,s_0}$). 

The fixed point $\mathfrak{p}_\e$ is odd  since $\eta_\e$ is even. 
The functions  $U(X,Z)$ and $V(X,Z)$ are respectively odd and even in~$X$, hence $d_\e (X,Z)$ in~\eqref{def:d} is even in~$X$. 
Defining
$ \tc_\epsilon = \frac{1}{2\pi} \int_{\T} \eta_\e(X + \mathfrak{p}_\e(X))\,\de X $ 
and using~\eqref{def:fp}, we have  formula~\eqref{LCb}. 
\\[1mm]
\textit{(ii)} 
Differentiating \eqref{LC} we have
\begin{equation}
\label{UXUZ}
    U_X(X,Z) = 1 +  e^{Z|D|}\mathfrak{p}'(X) \, ,
\qquad 
U_Z(X,Z) =  e^{Z|D|} |D| \mathfrak{p}(X) 
\end{equation}
and then, by \eqref{def:d}, 
\be\label{formadep}
d_\epsilon (X,Z) = 
(e^{Z|D|}\mathfrak{p}'(X))^2 
+ 2 e^{Z|D|}\mathfrak{p}'(X) + 
(e^{Z|D|} |D| \mathfrak{p}(X))^2 \, . 
\ee
The propagator estimate in~\cite[Lemma~2.5]{BMV2}
and \Cref{remBMV} 
show that for any 
$s\geq 3$ and $a\in (0,1)$,  
\be\label{propa} 
\| e^{Z|D|} \Pi_0^\perp g \|_{\lfloor s+\frac12\rfloor,\lfloor s+\frac12\rfloor,a}
\lesssim_s \| g \|_{H^s(\T)} 
\, .
\ee
The function $ d_\epsilon $ in \eqref{formadep} satisfies  \eqref{est:d} 
by \eqref{propa},  
the tame estimates in~\cite[Lemma~2.4]{BMV2} satisfied by the norm 
$ \| \ \|_{s,s,a}$, and 
\eqref{est:fp}.   
According to 
\cite[Lemma 3.13]{BMV5}  the function $\fp_\e(X)$ has the expansion 
$$
    \fp_\e(X) = \sum_{ \ell\geq 1 }\e^\ell \fp_\ell(X)\, , \quad \text{where}\quad \fp_\ell(X) = \sum_{\substack{k = 1,\dots,\ell \\
    k\equiv \ell \textup{ mod }2} }\fp_\ell^{[\kappa]} \sin (kX)\, ,
$$
and therefore the function $ d_\epsilon $ in 
\eqref{formadep} has the form \eqref{deintF} using \cite[lemmata 3.2, 3.3]{BMV5}.
The first terms  $ d_1 (X,Z) = 2 e^Z \cos (X)  $ and $ d_2 (X,Z) = e^{2Z} (1+4 \cos (2X))$ are computed in \cite[page 25]{CNS}.
\end{proof}

\paragraph{Proof of \eqref{G.id}.}
If  $\Phi (X,Y,Z) $ is a solution of \eqref{Phi.ell} then  $\vartheta(X,Y, Z):= \Phi (U(X,Z),Y, V(X,Z))$, with $U$ and $V$ defined in \eqref{LC},  
solves \eqref{transf_laplace} with $ \varsigma = \mathfrak{P}\psi $. Then, using  the identities
\begin{equation}\label{1010:1754}   \Phi_x(U(X,Z), Y, V(X, Z)) = \frac{\vartheta_X U_X + \vartheta_Z U_Z}{U_X^2 + U_Z^2} \, , \quad \Phi_z(U(X,Z), Y, V(X, Z)) = \frac{\vartheta_Z U_X - \vartheta_X U_Z}{U_X^2 + U_Z^2} \,,  
\end{equation}
we obtain
\begin{align}
\notag
    G(\eta_\e)[\psi](x,y) 
   & \stackrel{\eqref{DN1}}{ =}  (\pa_z \Phi) (x,y, \eta_\e(x)) - (\eta_\e)_x(x)   
   (\pa_x \Phi)(x,  y,  \eta_\e(x)) \\
   \notag
   & \stackrel{\eqref{LCb}, \eqref{diffeo}}{=} 
   \mathfrak{P}^{-1} \left[ 
   (\pa_z \Phi) (U(X,0), Y, V(X,0)) - 
   (\mathfrak{P}(\eta_\e)_x)(X)
   (\pa_x \Phi)(U(X,0),Y,   V(X,0))
   \right]  \\
   & \stackrel{\eqref{1010:1754}}{ =} \mathfrak{P}^{-1}
   \Big\{ \frac{1}{U_X^2 + U_Z^2}\left[\vartheta_Z U_X - \vartheta_X U_Z -  (\eta_\e)_x(U(X,0))  \left(\vartheta_X U_X + \vartheta_Z U_Z \right) \right] \vert_{(X,Y,0)} \Big\} \, . 
   \label{1110:1109}
\end{align}
Using the Cauchy-Riemann equations for $ U, V  $ and differentiating \eqref{LCb} we deduce  
$ 
-U_Z(X,0) = V_X(X,0) = $ $  (\eta_\e)_x(U(X,0)) U_X(X,0) 
$, 
and  \eqref{1110:1109} becomes 
$$
\begin{aligned}
G(\eta_\e)[\psi](x,y) 
& = \mathfrak{P}^{-1} \left\lbrace\frac{(1 + (\eta_x\circ U)^2)U_X }{U_X^2 + U_Z^2} \vartheta_Z \vert_{(X,Y,0)}\right\rbrace
 = 
\mathfrak{P}^{-1}
 \left\lbrace
 \frac{\vartheta_Z(X,Y,0)}{U_X(X,0)}  \right\rbrace  \\
 & =
\mathfrak{P}^{-1}
 \left\lbrace
\frac{\vartheta_Z(X,Y,0)}{1+ {\mathfrak p}'(X)} 
\right\rbrace \stackrel{\eqref{def:cG}} = 
\mathfrak{P}^{-1}
 \left\lbrace
\frac{ (\cG_{\e} \circ \mathfrak{P}  \psi)(X,Y)}{1+ {\mathfrak p}'(X)} 
\right\rbrace 
\end{aligned}
$$
which is the first identity in \eqref{G.id}. 

\section{McLean curves and spectral separations}\label{app:descmclean}

In this Appendix we provide the proofs of the 
properties of the McLean curves stated in 
\Cref{lem:description} and the separation properties of 
the spectrum in Lemma \ref{separation_eigenvalues}. 
We first prove \eqref{samesignint}. 
\\[1mm]
\noindent{\sc Same sign wave interactions.} 
 Using \eqref{eig.0} we have 
\begin{align} 
    \omega_k^\sigma(\al,\mu)-\omega_m^\sigma (\al,\mu) 
    & = \sigma (k-m) - \sigma [
    \Omega_\alpha (\sigma k + \mu) -
    \Omega_\alpha (\sigma m + \mu) 
     ] \notag \\
    & = \tq - \sigma [ 
(\varphi^2+\al^2)^\frac14 - ((\varphi-\tq)^2+\al^2)^\frac14] \label{ref:collisionsigsig}
    \end{align}
where $\tq := \sigma(k-m)$ and $\varphi := \sigma k + \mu$. 
In view of \eqref{ref:collisionsigsig} 
we have 
\begin{equation}\label{ref:collisionsigsigin}
|\omega_k^\sigma (\al,\mu)- \omega_m^\sigma(\al,\mu)| \geq |\tq| - \left|S_\tq(\varphi, \alpha)\right| \quad \text{where} 
\quad 
S_\tq(\varphi, \alpha) := \sqrt{d[(\varphi,\al),(0,0)] } - \sqrt{d[(\varphi,\al),(\tq,0)]} 
\end{equation}  
and $d$ denotes the euclidean distance on $\R^2$.
Triangular inequality implies that 
\begin{equation}\label{triangineqB}
    \left| d[(\varphi,\al),(0,0)]  - d[(\varphi,\al),(\tq,0)]\right|\leq |\tq| 
\end{equation}
and thus 
$$
|S_{\tq} (\varphi, \alpha)| = 
|\sqrt{d[(\varphi,\al),(0,0)] } - \sqrt{d[(\varphi,\al),(\tq,0)]}|  \leq \sqrt{|d[(\varphi,\al),(0,0)] -d[(\varphi,\al),(\tq,0)]|} \leq \sqrt{|\tq|}\, ,
$$
proving, in view of  \eqref{ref:collisionsigsigin}, that
for any 
    $ (\alpha, \mu ) \in \R^2 $, for any $ \sigma = \pm $, for any $k, m 
    \in \Z $, 
\begin{equation}\label{sepeigenvalues:dist2}
        |\omega_k^\sigma (\al,\mu)- \omega_m^\sigma(\al,\mu)| \geq 
        |k-m|-\sqrt{|k-m|} 
\end{equation}
which is strictly positive  for any $ |k-m|\geq 2 $.  In order to prove \eqref{samesignint}
it remains to analyze
$ | k - m  | = 1 $. 
So assume $\omega_k^\sigma (\al,\mu)= \omega_m^\sigma(\al,\mu)$. In this case 
\be\label{k-m=1}
k - m = \sqrt{d[(\varphi,\al),(0,0)]}-\sqrt{d[(\varphi,\al),(\tq,0)]} \, . 
\ee
We now use that, if   $x,y\geq 0$ and $\sqrt{x} - \sqrt{y}=1 $ then $ x-y \geq 1$. Moreover if  $\sqrt{x} - \sqrt{y}= 1 $, $ x-y \leq  1 $ then  $(x,y)= (1,0) $. 
\\[1mm]
If 
$ k-m = 1 $, by  \eqref{k-m=1} 
and  \eqref{triangineqB} we deduce
$$
d[(\varphi,\al),(0,0)]= 1 \, , \qquad d[(\varphi,\al),(\tq,0)] = 0 \, , \quad \varphi = \sigma k + \mu \, , \quad
\tq = \sigma (k-m) \, , 
$$
and so $ \alpha = 0 $ and 
$ \mu   = - \sigma m  = \sigma (1-k)  
 $. Similarly we deduce that if 
$ k - m = - 1 $ then $ \mu = - \sigma k $. Thus  
 \eqref{samesignint} is proved.
\\[1mm]
\noindent{\sc Opposite sign wave interactions.}
Let us prove \eqref{oppsigninteractions}. 
The fact that  $
\cM^{(\tp)} = \emptyset $ for any  $\tp\leq-1$,
and $
\cM^{(0)} = \{(0,0)\} $
directly follow by 
 \eqref{def:McleanI} and \eqref{def:McleanII}.  If  $(\alpha, \mu) $
 belongs to 
 $ \cM^{(1) }$ then, by 
\eqref{def:McleanII} with $ \tm = 0 $,
$$
1 = (\al^2 + \mu^2)^\frac14 + [(\mu-1)^2+\al^2]^\frac14 = \sqrt{d[(\al,\mu),(0,0)]} +
\sqrt{d[(\al,\mu),(0,1)] } \, , 
$$
which, jointly with the triangular inequality 
$ 1 \leq  d[(\al,\mu),(0,0)] +
d[(\al,\mu),(0,1)]  $, 
implies that 
$$
d[(\al,\mu),(0,0)] = 1\, , \ d[(\al,\mu),(0,1)] = 0 \,  \quad \text{or} \quad d[(\al,\mu),(0,1)] = 1\, , \ d[(\al,\mu),(0,0)] = 0 \, .  
$$
Thus $\cM^{(1)}= \{(0,0),(0,1)\}$. 

By   \eqref{def:McleanI_II_m} for any 
$ \tp \geq 3 $ each $ \cM^{(\tp)} \cap \{ (0,0) \} = \emptyset $. 
The symmetry properties of 
the McLean curves $ \cM^{(\tp)}$,  $\tp \geq 2 $,  directly follow by \eqref{mcleanmanifoldsp0}
and \eqref{eig.0}. 
Then  \eqref{sing.ML} is a direct computation using \eqref{mcleanmanifoldsp0}. The curves $ \cM^{(\tp) } $ are analytic since the function
$
(\al,\mu) \to  m_\tp (\al,\mu)
$  in \eqref{mcleanmanifoldsp0}
is analytic at any point, except $ (0,\pm \frac\tp 2)$ (that are not on the curve $\cM^{(\tp)}$), and its gradient vanishes only at $(0,0)$. Since $(0,0)\in\cM^{(\tp)}$ only for $\tp=2$, the curves $\cM^{(\tp)}$ are analytical manifolds for any $\tp\geq 3$ even, and $\cM^{(2)}$ is an analytical variety with a unique singular point at $(0,0)$.
By the explicit expression \eqref{McLean1} we recognize that 
the two tangent lines of $  \cM^{(2)}$ at $(0,0)$ 
are $ \mu = \pm \sqrt{2} \alpha $.
Moreover for any $\tp\geq 3$ 
$$
m_\tp (0,\mu) >0 \ \ \forall \mu \in (\mu_*^-(\tp),\mu_*^+(\tp))\, , \qquad m_\tp (0,\mu) <0 \ \  \forall \mu \in \R\setminus (\mu_*^-(\tp),\mu_*^+(\tp))\, ,
$$
and  $\al \in \R_+ \to m_\tp(\al,\mu)$ is strictly decreasing for any $\mu$, with $\lim_{\al\to\infty}m_\tp(\al,\mu)=-\infty$, so that $\cM^{(\tp)} \cap \{\al>0\}$ is the graph of an analytic function $\al_\tp:(\mu_*^-(\tp),\mu_*^+(\tp))\to \R_+$.   For every even $\tp\geq 4$, $\cM^{(\tp)}$ does not intersect the origin and is symmetric with respect both axes with horizontal tangents at $(\al,\mu)=(\pm  \sqrt{\left(\frac{\tp}{2}\right)^4 -  \left(\frac{\tp}{2}\right)^2} , 0)$ and vertical tangent at $(\al,\mu) = (0, \pm \mu_*(\tp))$. These are the unique intersections with the axes, thus $\cM^{(\tp)}$ is a simple closed curve. This also proves that the regions $\cU^+_\tp$  and $\cU^-_\tp$ are respectively bounded and unbounded. 

We now prove that the McLean curves do not intersect each other, namely 
\begin{equation}
\label{noint}
\cM^{(\tp)}\cap \cM^{(\tp+1)} = \emptyset \, , 
\quad \forall \tp \geq 2 \, . 
\end{equation}
Suppose first that $\tp$ is even. If $(\al,\mu)\in \cM^{(\tp)}\cap \cM^{(\tp+1)}$, then  $\omega^+_{\frac{\tp}{2}}(\al,\mu)=\omega^-_{\frac{\tp}{2}}(\al,\mu)$ and $\omega^+_{\frac{\tp}{2}}(\al,\mu) = \omega^-_{\frac{\tp}{2}+1}(\al,\mu)$, thus $\omega^-_{\frac{\tp}{2}}(\al,\mu) = \omega^-_{\frac{\tp}{2}+1}(\al,\mu)$. Then, by \eqref{samesignint}, $(\al,\mu) = (0,\frac{\tp}{2})$.  But by \eqref{sing.ML} this point belongs neither to $\cM^{(\tp)}$ nor to $\cM^{(\tp+1)}$, obtaining a contradiction. The case $\tp$ odd is similar. By \eqref{noint}, and by direct inspection $m_{\tp+1}((0,\mu^+_*(\tp)))>0$, then
$
\cM^{(\tp)}\subset \cU_{\tp+1}^+,
$
proving \eqref{outinunp}.

\begin{proof}[Proof of Lemma \ref{separation_eigenvalues}]

\noindent{\sc Proof of \eqref{0905:1852}}.  Consider first $\omega_0^+(\alpha, \mu)$.
We claim that   
for any  $ \delta > 0 $ small enough  
\begin{equation}\label{0995:1916+}
\abs{\omega_0^+ (\al, \mu) - \omega_q^+ (\al, \mu)} \geq \tfrac12   \, , \quad 
\forall q \neq 0,1 \, , \quad  
\forall (\al, \mu) \in B_\delta(0,0) \, . 
\end{equation}
Indeed, if $|q| \geq 2$, by \eqref{sepeigenvalues:dist2}, we have 
$ |\omega_0^+(\al, \mu) - \omega_q^+(\al, \mu) | \geq 2 - \sqrt{2} $. 
If  $q =- 1$,
since $\omega_{-1}^+ (\al, \mu) = - 1+ \mu - \Omega_{\al}(-1 + \mu)\stackrel{(\al,\mu)\sim (0,0)}{\sim} -2$, hence for any $ \delta >0 $ small it results 
$ | \omega_0^+(\al, \mu) - \omega_{-1}^+(\al, \mu) | \geq 1 $ for any 
$ (\al, \mu) \in B_\delta (0,0) $. 
This proves \eqref{0995:1916+}.

Similarly, by \eqref{sepeigenvalues:dist2} and since 
$\omega_{-1}^- (0, 0) = 2$,
for any $ \delta > 0 $ small enough 
\begin{equation}\label{0995:1916}
\abs{\omega_0^-(\al, \mu) - \omega_q^-(\al, \mu)} \geq \tfrac12   \, , \quad 
\forall q \neq 0,1 \, , \quad  
\forall (\al, \mu) \in B_\delta(0,0) \, . 
\end{equation}
Then, for any $\delta>0$ so small that $\abs{\omega_0^+(\al, \mu) - \omega_0^-(\al, \mu)}\leq \frac{1}{4}$ for $(\al, \mu) \in  B_\delta (0,0)$, we deduce,
for any $ q \neq 0,1 $,  
\begin{equation}\label{1612:1554}
    \begin{aligned}
\abs{\omega_0^+(\al, \mu) - \omega_q^-(\al, \mu)} \geq
\abs{\omega_0^-(\al, \mu) - \omega_q^-(\al, \mu)} - \abs{\omega_0^+(\al, \mu) - \omega_0^-(\al, \mu)}
\stackrel{\eqref{0995:1916}}{\geq} \tfrac14 \ , \\
\abs{\omega_0^-(\al, \mu) - \omega_q^+(\al, \mu)} \geq
\abs{\omega_0^+(\al, \mu) - \omega_q^+(\al, \mu)} - \abs{\omega_0^+(\al, \mu) - \omega_0^-(\al, \mu)}
\stackrel{\eqref{0995:1916+}}{\geq} \tfrac14\ .
    \end{aligned}
\end{equation}
Then \eqref{0995:1916+} and 
 \eqref{1612:1554} yield \eqref{0905:1852}  for $(k,\sigma) \in \{ (0,+),
 (0,-) \} $.
 The bound \eqref{0905:1852} for $(k, \sigma) = (1,\pm) $ follows analogously,  exploiting that $ |\omega_2^\pm (\al, \mu) | $ is close to $2-\sqrt{2} $ for $(\al, \mu) \in B_\delta(0,0)$.
\\[2mm]
\noindent{\sc proof of  \eqref{0905:1850}.} 
Let us consider first the case $ \tp \geq 4 $ even. We claim that  
there is $ \tc_\tp >0$  such that for any small neighborhood 
$ \cN^{(\tp)} $ of 
$ \cM^{(\tp)} $  
\begin{align}\label{0905:1835}
       & \big| \omega_{\frac{\tp}{2}}^\sigma (\al,\mu)-\omega_q^\sigma (\al,\mu)
       \big| 
       \geq 2 \tc_\tp > 0  \, ,\qquad \forall  q\neq \frac{\tp}{2}\, , 
       \  \sigma = \pm \, , 
       \ (\al,\mu)\in \cN^{(\tp)} \, . 
    \end{align}
This follows by \eqref{sepeigenvalues:dist2} for $ |q- \frac{\tp}{2}| \geq 2 $, by 
\eqref{samesignint} 
if $ |q- \frac{\tp}{2}| = 1 $, using that $(0,\sigma(1-\tfrac{\tp}{2})),\, (0,\sigma\tfrac\tp2)\not\in \cM^{(\tp)}$ for $\tp\geq 4$ (or use \eqref{diffaut}),
for any neighborhood 
$ \cN^{(\tp)}$
of the McLean curve
$\cM^{(\tp)}$  
not intersecting $(0,0)$ 
 (recall that $\cM^{(\tp)}$, $\tp \geq 3 $, does not intersect $(0,0)$). 

Now we examine interactions of opposite sign. 
 For any $(\al,\mu)\in \cM^{(\tp)}$ we have $\omega^+_{\frac{\tp}{2}}(\al,\mu) = \omega^-_{\frac{\tp}{2}}(\al,\mu)$ and  therefore, up to shrinking the neighborhood $\cN^{(\tp)}$ of $ \cM^{(\tp)} $, we have (recall that 
 $\cM^{(\tp)}$ is compact) 
\begin{equation}\label{maxcp}
 \max_{(\al,\mu)\in  \cN^{(\tp)}} |\omega^+_{\frac{\tp}{2}}(\al,\mu) - \omega^-_{\frac{\tp}{2}}(\al,\mu)| <  \tc_\tp   \, . 
\end{equation}
 Thus, for any  $\tq\neq \frac\tp 2$, 
 $ \sigma = \pm $, for any 
 $ (\alpha, \mu) \in \cN^{(\tp) }$, \begin{equation}\label{deleven}
        |\omega^\sigma _{\frac{\tp}{2}}(\al,\mu) - \omega^{-\sigma} _{\tq}(\al,\mu)| \geq  |\omega^{-\sigma} _{\frac{\tp}{2}}(\al,\mu) - \omega^{-\sigma} _{\tq}(\al,\mu)| - |\omega^\sigma _{\frac{\tp}{2}}(\al,\mu) - \omega^{-\sigma} _{\frac{\tp}{2}}(\al,\mu)| 
        \stackrel{\eqref{0905:1835},\eqref{maxcp}} \geq \tc_\tp \, . 
    \end{equation}
 Then  \eqref{0905:1835}, \eqref{deleven} prove \eqref{0905:1850} when $\tp \geq 4 $ is even. The cases $\tp \geq 3 $ odd and $ \tp = 2 $ 
 follow similarly.
\end{proof}

\section{Properties of the classes $ \cA $ and $ \mathtt F $}\label{ap_cAF}

In this appendix we prove  key  properties of the classes $\cA$ and $\tF$ introduced in  \Cref{sec:MDN}.
First we prove \Cref{lem:decomposition}.
\begin{proof}[Proof of \Cref{lem:decomposition}]
    Define for a function $A\in\cA(B_r(0,0),\e_0;X)$ its even an odd components in $\mu$: $$A_{even}(\al,\mu,\e) := \frac{A(\al,\mu,\e) + A(\al,-\mu,\e)}{2} \, , \qquad   A_{odd}(\al,\mu,\e) :=\frac{A(\al,\mu,\e) - A(\al,-\mu,\e)}{2}\, ,$$ 
where $A_{even}$ is actually analytic in $\mu^2$ since it is analytic and even in $\mu$, while $A_{odd}$ has a zero of order 1 at $\mu = 0$, so that, gathering out a factor $\mu$, one obtains an analytic function of $\mu^2$. In addition, by \eqref{simmeA} the function $A_{even}$ is real and $A_{odd}$ is purely imaginary.  
Hence starting from \eqref{dectM} one has 
$$
A(\al,\mu,\e) = \underbrace{A^{[\mathrm{I}]}_{even}(\al^2,\mu,\e)}_{=:A^{[0,0]}(\al^2,\mu^2,\e)} + \im\mu \underbrace{\frac{-\im}{\mu}A^{[\mathrm{I}]}_{odd}(\al^2,\mu,\e)}_{=:A^{[0,1]}(\al^2,\mu^2,\e)} + (\al^2+\mu^2)^\frac12 \underbrace{A^{[\mathrm{II}]}_{even}(\al^2,\mu,\e)}_{=:A^{[1,0]}(\al^2,\mu^2,\e)} + \im\mu (\al^2+\mu^2)^\frac12 \underbrace{\frac{-\im}{\mu}A^{[\mathrm{II}]}_{odd}(\al^2,\mu,\e)}_{=:A^{[1,1](\al^2,\mu^2,\e)}}\, , 
$$
and all the operators $A^{[i,j]}$ are well-defined on the set $B_{r^2}(0,0)\times B_{\e_0}(0)$ for some sufficiently small $r$, analytic in every variable and real-to-real, proving \eqref{decompositiontM}. 

In order to prove that the decomposition in \eqref{decompositiontM} is unique let us show that if
$$
A^{[0,0]}(\al^2,\mu^2,\e) + \rho A^{[1,0]}(\al^2,\mu^2,\e) + \im\mu A^{[0,1]}(\al^2,\mu^2,\e) + 
\im\mu \rho A^{[1,1]}
(\al^2,\mu^2,\e) \equiv 0 
$$
then 
$ A^{[i,j]}(\al^2,\mu^2,\e)\equiv 0 $ for any  $ i,j  \in \{ 0,1\} $. 
Considering first only the real part, one has for every fixed $t\in \R$
$$
A^{[0,0]}(t^2\mu^2,\mu^2,\e) + (1+t^2)^\frac12 |\mu| A^{[1,0]}(t^2\mu^2,\mu^2,\e) \equiv 0\, .
$$
Taylor expanding in $\mu>0$, the function $A^{[0,0]}(t^2\mu^2,\mu^2,\e)$ has only even degree monomials in $\mu$, while $\mu A^{[1,0]}(t^2\mu^2,\mu^2,\e)$ has only monomials of odd degree, and therefore $A^{[0,0]}(t^2\mu^2,\mu^2,\e) \equiv A^{[1,0]}(t^2\mu^2,\mu^2,\e) \equiv 0$, and letting $t$ vary we deduce that   $A^{[0,0]}\equiv A^{[1,0]} \equiv 0$. The argument for the imaginary part is analogous.
\end{proof}

Next we consider 
the class $ \cA $ in 
\Cref{def:tM}.
\begin{lemma}\label{lem:prodintM} Let $\Omega \subset \R^2 $ be an open set, $ X,Y,Z$ be Banach spaces,  and $\e_0 \in (0,+\infty] $. The following properties hold:
    \begin{enumerate}
        \item[(i)]{\bf Composition:} If $A\in \cA(\Omega,\e_0;Y,Z)$, 
        $B\in \cA(\Omega,\e_0;X,Y)$ (recall notation \eqref{notaLYZ}) then $AB\in \cA(\Omega,\e_0;X,Z)$. 
    \item[(ii)]{\bf Functional calculus:} 
    Let   $f(z) = \sum_{k\geq 0} f_kz^k$ be an analytic function on $\{z\in\C\, : \ |z|<a\}$ for some $ a > 0 $. 
    Let  $A\in \cA(\Omega,\e_0;X,X)$ satisfy  $\|A(\al,\mu,\e)\|_{\cL(X,X)}\le a$ for any $ (\alpha, \mu, \epsilon)$ in $\Omega\times B_{\e_0}(0)$, and suppose that for any $j\in \Z$
    $$
\sup_{(\al,\mu,\e)\in(\Omega \cap B_r(0,j))\times B_{\e_0}(0)}\left( \|A^{[\mathrm{I}]}(\al^2,\mu,\e)\|_{\cL(X,X)} + (\al^2+(\mu-j)^2)^\frac12 \|A^{[\mathrm{II}]}(\al^2,\mu,\e)\|_{\cL(X,X)} \right) < a\, .
$$
Then 
    \be\label{funcalcap}
    (\al, \mu, \e) \mapsto f(A)(\al, \mu, \e) := \sum_{k\geq 0} f_k A^k(\al, \mu, \e) \in \cA(\Omega,\e_0;X,X)\, .
    \ee
    \item[(iii)]{\bf Scalar product:} Let $Y $ be an Hilbert space with scalar product $(\cdot,\cdot)_Y $. If $v\in \cA(\Omega,\e_0;X)$, $w\in \cA(\Omega,\e_0;Y)$ and $A\in \cA(\Omega,\e_0;X,Y)$, then 
    $$
    (Av,w)_Y \in \cA(\Omega,\e_0;\C)\, .
    $$
    \end{enumerate}
\!\! \!\!\! \!\!\! \!\!\! \!
\!\! \!\!\! \! Properties $(i)$-$(iii)$ hold analogously for the class of polar-analytic functions $\cA_P (\Omega,\e_0;X)$ in  \Cref{def:tildeM}.
\begin{enumerate}
    \item[(iv)]{\bf Spectral projectors:} Let $A(\al,\mu,\e) \in 
\cA (\Omega,\e_0;Y,X)$ and  $\Gamma(\al,\mu,\e)$ is a 
family of closed smooth curves,
counterclockwise-oriented, 
belonging  to the resolvent set of $ A(\al,\mu,\e)$ and  satisfying 
\begin{itemize}
    \item {\bf (continuity property)}  for any $(\und\al,\und\mu,\und\e) \in \Omega\times B_{\e_0}(0)$ there exists a small neighborhood $\cU(\und\al,\und\mu,\und\e)$ such that for any $(\al,\mu,\e)\in\cU(\und\al,\und\mu,\und\e)$ the curve $\Gamma(\al,\mu,\e)$ can be continuously deformed into $\Gamma(\und\al,\und\mu,\und\e)$ inside the resolvent set of $A(\al,\mu,\e)$. 
\end{itemize}
        Then
        $$\Omega\times B_{\e_0} (0) \ni (\al,\mu,\e) \mapsto \oint_{\Gamma(\al,\mu,\e)} (\lambda-A(\al,\mu,\e))^{-1}\frac{\de\lambda}{2\pi\im}\in \cA(\Omega,\e_0;X,Y) \, . 
        $$
        \item[(v)]{\bf (Lipschitz functions)} Let 
        $A(\al,\mu,\e) $ be a polar-analytic function in $ \cA_P( B_{r}(0,0),\e_0;X)$, $r, \e_0 >0$,  satisfying  $A(\al,\mu,\e) = \cO(\rho)$ according to  \Cref{def:tildeM}. Then $A(\al,\mu,\e)$ extends to a Lipschitz function of $(\al,\mu)$ in a neighborhood of the origin. 
        \end{enumerate}
\end{lemma}

\begin{proof} 
 For  the operator $AB$ in item $(i)$, the operator $f(A)$ in item $(ii)$ and the scalar function $(Av,w)_Y $ in item $(iii)$ it is immediate to verify condition 1) and the first bullet of 2)  in Definition \ref{def:tM}.
It remains to show that these functions satisfy the condition in the second bullet of item 2). For simplicity we do this only for $j=0$. 
\\[1mm]
$(i)$ Since $A$ and $B$ decompose 
as in \eqref{dectM}, we have 
\begin{equation}\label{gatheringtM}
    \begin{aligned}
    AB = \underbrace{A^{[\mathrm{I}]}B^{[\mathrm{I}]}  + (\al^2+\mu^2)A^{[\mathrm{II}]}B^{[\mathrm{II}]}}_{=: (AB)^{[\mathrm{I}]}}+ (\al^2+\mu^2)^\frac12 \underbrace{ \left(A^{[\mathrm{II}]}B^{[\mathrm{I}]} + A^{[\mathrm{I}]}B^{[\mathrm{II}]}\right)}_{=:(AB)^{[\mathrm{II}]}}\, .
\end{aligned}
\end{equation}
Clearly $(AB)^{[\mathrm{I}]}$ and $(AB)^{[\mathrm{II}]}$ are analytic functions of $\al^2$ and $\mu$, as well as $A^{[\mathrm{I}]},B^{[\mathrm{I}]},A^{[\mathrm{II}]}$ and $B^{[\mathrm{II}]}$.
\\[1mm]
$(ii)$ Define the following functions depending on an additional free variable $\rho$:
\begin{equation}\label{addfreevar}
    \begin{aligned}
    &g_I(A^{[\mathrm{I}]},A^{[\mathrm{II}]},\rho^2) := \frac12 \left( f(A^{[\mathrm{I}]}+\rho A^{[\mathrm{II}]}) + f(A^{[\mathrm{I}]}-\rho A^{[\mathrm{II}]}) \right)\, , \\ &g_{II}(A^{[\mathrm{I}]},A^{[\mathrm{II}]},\rho^2) := \frac1{2\rho} \left( f(A^{[\mathrm{I}]}+\rho A^{[\mathrm{II}]}) - f(A^{[\mathrm{I}]}-\rho A^{[\mathrm{II}]}) \right)\, .
\end{aligned}
\end{equation}
Defining the open set $\cU :=\left\{(A^{[\mathrm{I}]},A^{[\mathrm{II}]},\rho^2)\in\cL(X,X)\times \cL(X,X) \times \C \, : \ \|A^{[\mathrm{I}]}\|+|\rho|\|A^{[\mathrm{II}]}\|<a \right\} $, we prove now that the functions 
$  g_{\mathrm{I}}\, , \ g_{\mathrm{II}}\, :\cU \to \cL(X,X) $
are holomorphic.
Indeed, since  the function $ \{A\in \cL(X,X)\, : \ \|A\|_{\cL(X,X)}<a \} \ni A \to f(A)\in\cL(X,X)$ is analytic, the functions $g_{\mathrm{I}},g_{\mathrm{II}}$ depend analytically on $\rho $ for every fixed $A^{\mathrm{[I]}}$, $A^{\mathrm{[II]}}$, and are also even in $\rho$. Thus, they are holomorphic in $\rho^2$. The holomorphicity of the maps  $A^{[\mathrm{I}]}\mapsto g_{\star}(A^{[\mathrm{I}]},A^{[\mathrm{II}]},\rho^2) $, 
$A^{[\mathrm{II}]}\mapsto g_{\star}(A^{[\mathrm{I}]},A^{[\mathrm{II}]},\rho^2) $, $\star = \mathrm{I}, \mathrm{II}$, 
follow directly by the analyticity of $f$. By the Hartogs theorem in Banach spaces -- see \cite[Theorem 36.8]{muj} -- separate holomorphicity in every variable implies holomorphicity of the full function.
Therefore, considering a function $A(\al,\mu,\e) = A^{[\mathrm{I}]}(\al^2,\mu,\e) + (\al^2+\mu^2)^\frac12 A^{[\mathrm{II}]}(\al^2,\mu,\e) \in \cA(\Omega,\e_0;X,X)$, we define
$$
\begin{aligned}
    &f(A)^{[\mathrm{I}]}(\al^2,\mu,\e):=g_{\mathrm{I}}(A^{[\mathrm{I}]}(\al^2,\mu,\e),A^{[\mathrm{II}]}(\al^2,\mu,\e),\al^2+\mu^2)\, ,\\
    & f(A)^{[\mathrm{II}]}(\al^2,\mu,\e):=g_{\mathrm{II}}(A^{[\mathrm{I}]}(\al^2,\mu,\e),A^{[\mathrm{II}]}(\al^2,\mu,\e),\al^2+\mu^2)\, ,
    \end{aligned}
$$
which are analytic in the variables $\al^2,\mu$ and $\e$, and satisfy, by \eqref{addfreevar}, the identity
$$
f(A)(\al,\mu,\e) = f(A)^{[\mathrm{I}]}(\al^2,\mu,\e) + (\al^2+\mu^2)^\frac12 f(A)^{[\mathrm{II}]}(\al^2,\mu,\e)\, ,
$$
so that $f(A)\in \cA(\Omega,\e_0;X,X)$. 
\\[1mm]
$(iii)$ We now prove the decomposition \eqref{dectM} for the scalar product $(Av,w)_F$ of  item $(iii)$.  Expanding $A,v,w$ as in \eqref{dectM} we decompose
\begin{alignat*}{2}
    (Av,w)_Y &= \overbrace{(A^{[\mathrm{I}]}v^{[\mathrm{I}]},w^{[\mathrm{I}]})_Y + (\al^2+\mu^2)\bigg[(A^{[\mathrm{I}]}v^{[\mathrm{II}]},w^{[\mathrm{II}]})_Y 
    +  (A^{[\mathrm{II}]}v^{[\mathrm{I}]},w^{[\mathrm{II}]})_Y + (A^{[\mathrm{II}]}v^{[\mathrm{II}]},w^{[\mathrm{I}]})_Y \bigg]}^{=:(Av,w)_Y^{[\mathrm{I}]}}\\
    &+ (\al^2+\mu^2)^\frac12 \underbrace{\left[(A^{[\mathrm{I}]}v^{[\mathrm{I}]},w^{[\mathrm{II}]})_Y + (A^{[\mathrm{I}]}v^{[\mathrm{II}]},w^{[\mathrm{I}]})_Y + (A^{[\mathrm{II}]}v^{[\mathrm{I}]},w^{[\mathrm{I}]})_Y + (\al^2+\mu^2)(A^{[\mathrm{II}]}v^{[\mathrm{II}]},w^{[\mathrm{II}]})_Y \right]}_{=:(Av,w)_Y^{[\mathrm{II}]}}.
\end{alignat*}
Clearly both $(Av,w)_Y^{[\mathrm{I}]}$ and $(Av,w)_Y^{[\mathrm{II}]}$ are analytic in $\al^2$ and $\mu$ in a neighborhood of the origin, since $A^{[\mathrm{I}]},A^{[\mathrm{II}]},$ $v^{[\mathrm{I}]},v^{[\mathrm{II}]},w^{[\mathrm{I}]}$ and $w^{[\mathrm{II}]}$ have the same property.
\\[1mm]
 $(iv)$ 
We now prove that the spectral projectors belong to the class $\cA$. 
For every fixed $(\und\al,\und\mu)\in \Omega\setminus \{ (0,0)\}$ and $\und\e\in B_{\e_0}(0)$ there exists a sufficiently small complex neighborhood $\cU\equiv \cU(\und\al,\und\mu,\und\e)\subset \C^3$ on which $A(\al,\mu,\e)$ admits an analytic extension. If instead $(\und\al,\und\mu) = (0,0) $ we choose the neighborhood $\cU(0,0,\und\e) = B_{r^2}(0)\times B_r(0) \times B_{\und\e_0}(\und\e)\subset\C^3$ for some small $r,\und\e_0>0$, so that $\cU(0,0,\und\e)\ni(\al^2,\mu,\e) \mapsto A^{[\star]}(\al^2,\mu,\e)$ is analytical, for both $\star = \mathrm{I},\mathrm{II}$.
In particular, from hypothesis $(a)$, one can choose $\cU$ so that
for every $(\al, \mu, \e)$ in $\cU\cap \R^3$
the curve $\Gamma(\und\al,\und\mu,\und\e)$ can be continuously deformed into $\Gamma(\und\al,\und\mu,\und\e)$ inside the resolvent set $A(\al,\mu,\e)$. 
Then  for   any $(\al,\mu,\e)\in \cU\cap \R^3$,
$$
\oint_{\Gamma(\al,\mu,\e)} (\lambda-A(\al,\mu,\e))^{-1} \frac{\de\lambda}{2\pi \im} = \oint_{\Gamma(\und\al,\und\mu,\und\e)} (\lambda-A(\al,\mu,\e))^{-1} \frac{\de\lambda}{2\pi \im}\, .
$$
Moreover, since $(\al,\mu,\e)\to A(\al,\mu,\e)$ belongs to $\cA(\Omega,\e_0;Y,X)$,
$$
\begin{aligned}
    &\Gamma(\und\al,\und\mu,\und\e) \times \cU\ni (\al,\mu,\e) \mapsto (\lambda-A(\al,\mu,\e))^{-1} \in \cL(X,Y) \ \ \text{is continuous,}\\
    &\cU \ni (\al,\mu,\e) \mapsto (\lambda-A(\al,\mu,\e))^{-1}\in \cL(X,Y) \ \ \text{is analytic for every }\lambda\in \Gamma(\und\al,\und\mu,\und\e)\, .
\end{aligned}
$$
Instead, if $(\und\al,\und\mu)=(0,0)$, then by item $(ii)$ also 
$$
\begin{aligned}
    &\Gamma(0,0,\und\e) \times B_{r^2}(0)\times B_r(0)\times B_{\und\e_0}(\und\e) \ni (\al^2,\mu,\e) \mapsto  [(\lambda-A)^{-1}]^{[\star]} (\al^2,\mu,\e) \ \ \text{is continuous,}\\
    &B_{r^2}(0)\times B_r(0)\times B_{\und\e_0}(\und\e) \ni (\al^2,\mu,\e) \mapsto [(\lambda-A)^{-1}]^{[\star]} (\al^2,\mu,\e) \ \ \text{is analytic for every }\lambda\in \Gamma(0,0,\und\e)\, ,
\end{aligned} \quad \star = \mathrm{I}, \ \mathrm{II}\, ,
$$
both with values in $\cL(X,Y)$. Then, the fact that the maps 
$$
\begin{aligned}
    [\Omega \setminus \{(0,0)\}]\times B_{\e_0}(0) \ni (\al,\mu,\e) &\mapsto \oint_{\Gamma(\al,\mu,\e)} (\lambda-A(\al,\mu,\e))^{-1} \frac{\de\lambda}{2\pi \im}\in\cL(X,Y)\, ,\\
    B_{r^2}(0)\times B_r(0)\times B_{\e_0}(0) \ni (\al^2,\mu,\e) &\mapsto \oint_{\Gamma(\al,\mu,\e)} [(\lambda - A)^{-1}]^{[\star]}(\al^2,\mu,\e) \frac{\de\lambda}{2\pi\im}\in\cL(X,Y)\, , \quad \star = \mathrm{I}, \ \mathrm{II}\,, 
\end{aligned}
$$
are analytic  is a consequence of the following general result: let $ X$ and $Y$ be two complex Banach spaces, and $\cU\subset X $ open. If $\Gamma\times \cU
\ni (\lambda,y)
\mapsto f(\lambda,y)\in Y $ is continuous and for every fixed $\lambda$  the map $\cU
\ni y
\mapsto f(\lambda,y)\in Y $ is holomorphic, then $g(y) = \oint_\Gamma f(\lambda,y)\de\lambda$ is holomorphic on $\cU$.

\noindent $(v)$ 
Since 
$A(\al,\mu,\e) = \cO(\rho)$,  the function 
 $A(\rho\sin\theta,\rho\cos\theta,\e) = \rho  F(\rho, \theta, \e)$ for an analytic function $F(\rho,\theta,\e)$.
 Since  $\mu = \rho \cos \theta$ and $\al = \rho \sin \theta$, we have $\pa_\mu \rho = \cos \theta$ and $\pa_\mu \theta = - \sin\theta / \rho$, and thus 
$$
   \pa_\mu A(\al,\mu,\e) =(\rho \pa_\rho F + F) \cos \theta - \pa_\theta F \sin \theta \, , 
\quad 
\pa_\alpha A(\al,\mu,\e) =(\rho \pa_\rho F + F) \sin \theta + \pa_\theta F \cos \theta  \, , 
$$
which are uniformly bounded in a neighborhood of the origin. Thus  
$A$  is Lipschitz. 

This concludes the proof of the lemma for the class $\cA$.
Properties $(i)$-$(iii)$ hold analogously for the class of polar-analytic functions $ \cA_P (\Omega,\e_0;X)$ in \Cref{def:tildeM}.
\end{proof}

We  now consider  the class $\tF$ in \Cref{defFell}. 
Note  that  
 the $ \kappa $-band operator 
 \eqref{band.def} 
 associated to 
 the multiplication operator 
  for a function $ a(x) $ is the multiplication operator for the $\kappa $-harmonic  of $ a(x)$, namely 
\be\label{armoMOP}
[a(x)]^{[\kappa]}\equiv a_{\kappa} e^{\im \kappa x } :\,  \  
h(x) \mapsto a_{\kappa} e^{\im \kappa x } h(x) 
\qquad \text{where} \qquad 
a_\kappa := \frac{1}{2 \pi}
\int_{\T} a(x) e^{-\im \kappa x}
\de x \, . 
\ee
For a Fourier  multiplier $ g(D) $, 
for any $ \kappa \in \N $ the $ \kappa $-band operator 
$ [g(D)]^{[\kappa]} $  is zero. 

If  $ A $ is a matrix operator as in \eqref{Am4}, then  
  the adjoint of the $ \kappa $-band  operator $ A^{[\kappa]} $ 
  with respect to the scalar product 
  \eqref{scalarcomplex} is  
\begin{equation}\label{starij}
\big[A^{[\kappa]}\big]^* = (A^*)^{[-\kappa]} \, . 
\end{equation}
Formally each $ A$ is the sum of its $ \kappa $-bands
$    A = \sum_{\kappa \in \Z} A^{[\kappa]} $ 
and the $ \kappa $-band of the composed operator is 
\begin{equation}\label{bandAB}
    (A\circ B)^{[\kappa]} = \sum_{\kappa_1 + \kappa_2 = \kappa} A^{[\kappa_1]} \circ B^{[\kappa_2]} \, . 
\end{equation}

\begin{lemma}\label{TFjprop}
Let $ \ell, \ell' \in \N_0 $
and $ A_\ell \in \mathfrak{F}_\ell$, $ B_{\ell'} \in \mathfrak{F}_{\ell'}$. If  $A, B \in \mathtt{F}$ then 
\begin{itemize}
    \item[(i)] {\bf Composition:} $A_\ell \circ B_{\ell'} \in \mathfrak{F}_{\ell+ \ell'}$ and $A \circ B \in  \mathtt{F}$. 
    \item[(ii)] {\bf Adjoint:} $A_\ell^* \in \mathfrak{F}_{\ell}$ and $A^* \in \mathtt{F}$. 
    \item[(iii)] {\bf Functional calculus:} Let $f(z) = \sum_{k\geq 0}f_k z^k$ be a formal power series. Then 
    \begin{equation}\label{FCFeq}
        f(A) := \sum_{k\geq 0}f_k A^k \in \mathtt{F} \, .
    \end{equation}
    If $ A(\lambda) \in F $ for any  $\lambda $ on a closed smooth curve $\Gamma \subset \C$, then 
    $
    \oint_\Gamma f(A(\lambda))\de\lambda \in \tF$.
    \item[(iv)]  {\bf Finite range interaction:} If $ 
    A \in \mathfrak{F}_\ell$, then for any $ v_1, v_2 \in \C^2$,
    \begin{equation}\label{AFlele}
( A\,  v_1 e^{\im j_1 x} , 
v_2 e^{\im j_2 x} ) = 0 \ \ \mbox{ if } |j_1 - j_2| > \ell \mbox{ or } j_1 - j_2 \not\equiv \ell \mbox{ mod } 2 \, . 
\end{equation}
\end{itemize}
\end{lemma}

\begin{proof}
Items $(i),\, (ii)$ and $(iv)$ readily follow by \eqref{bandAB}, \eqref{starij} and \Cref{defFell}, cf.  \cite[Lemma 5.5]{BMV5}. We now prove $(iii)$.
 By  functional calculus
    $$
    f(A) = \sum_{k\in\N_0} f_k \Big( 
    \sum_{\ell\in\N_0} \e^\ell A_\ell \Big)^k = \sum_{\ell\geq 0}\e^\ell f(A)_\ell \, , \qquad f(A)_\ell: =\sum_{k\in\N_0}f_k \sum_{
    \scriptscriptstyle \ell_1+\dots+\ell_k = \ell, \ell_1, \ldots, \ell_k \in \N_0} \!\!
    \!\!\!\! \! A_{\ell_1}\dots A_{\ell_k}
    $$
    and each $f(A)_\ell \in {\frak F}_{\ell}$
    by  item $(i)$. Hence $f(A) \in \mathtt{F}$, proving \eqref{FCFeq}.  If $A=A(\lambda)$, $\lambda\in\Gamma$, then 
   $f(A(\lambda))\in \tF$ for any  $\lambda\in \Gamma$ by \eqref{FCFeq},  and then   
    $
    \Big[ \oint_\Gamma f(A(\lambda))_\ell\de\lambda\Big]^{[\kappa]} = \oint_\Gamma f(A(\lambda))_\ell^{[\kappa]}\de\lambda = 0 $ if 
    $|\kappa|>\ell $ or  $ \kappa\not\equiv \ell \mod{2} $.  
\end{proof}

\section{Fiber Dirichlet-Neumann operator $\cG (\al,\mu,\e)$}\label{ap_analiticity}

In this Appendix we analyze the fiber Dirichlet-Neumann operator $\cG(\al, \mu, \e)$   defined in \eqref{def:Galmu} and establish 
\Cref{DNProp1}.  The core analysis   focuses on the elliptic problem 
\eqref{ellprobtransf}, detailed  in 
\Cref{sub:ell.f,sec:proofs.FDN} 
contains the proofs of both 
\Cref{DNProp1,DNonR3}.

\subsection{The elliptic problem \eqref{ellprobtransf}}\label{sub:ell.f}
We look for a  solution $ \Theta := \Theta_g $ of   \eqref{ellprobtransf}
in the form  
\begin{equation}\label{decompsolution}
    \Theta_g (x,z) = \Theta_{g}^\flat (x,z) + \Theta^\sharp_g 
(x,z) 
\end{equation}  where   
\begin{equation}\label{eq0}
\begin{cases}
    \pa_z^2\Theta^\flat_g  (x,z) +(\pa_x+\im\mu)^2
    \Theta_g^\flat (x,z) -  \al^2 \Theta^\flat_g (x,z) = 0\\
    \Theta^\flat_g (x,z)|_{z=0}=g(x) \, , \quad 
    \lim_{z\to-\infty} \pa_z \Theta^\flat_g (x,z) = 0   \, , 
\end{cases}    
\end{equation}
and $ \Theta^\sharp_g (x,z) $ solves the non-homogeneus problem
\begin{equation}\label{eqsharp}
    \begin{cases}
    \pa_z^2\Theta^\sharp_g (x,z) +(\pa_x+\im\mu)^2\Theta_g^\sharp (x,z) -  \al^2 \Theta^\sharp_g
    (x,z) = 
    \al^2 d_\e (x,z) (\Theta^\flat_g + \Theta^\sharp_g)
    (x,z)\\
\Theta^\sharp_g (x,z)|_{z=0}=0 \, , \quad 
    \lim_{z\to-\infty} \pa_z \Theta^\sharp_g
    (x,z)= 0  \, . 
\end{cases}
\end{equation}
The solution of
the homogeneous elliptic problem  
\eqref{eq0}
is  
\begin{equation}\label{propg}
\Theta^\flat_g (x,z) :=  
( e^{z|D|_{\al,\mu}}g )(x) := \sum_{k\in \Z} g_k \, e^{z|k|_{\al,\mu}} \, e^{\im k x} 
 \end{equation} 
 where
\begin{equation}\label{disprel}
        |k|_{\al,\mu}:= ((k+\mu)^2+\al^2)^\frac12 \, .
\end{equation}
The  propagator solution \eqref{propg}   does not have,
on the mode
$k=0$, a fast decay as $ z \to - \infty $, uniformly  for $\al,\mu$ close to $(0,0)$. We therefore introduce 
the following Hilbert spaces
of functions:
given $s\in\R,\, b \in \N_0$, $a >0 $, $a_0 \in \R$, we define 
\begin{equation}\label{def:expmod0}
    H^{s, b}_{-a_0, a}:= 
\Big\{ 
u(x,z) = \sum_{k \in \Z}   u_k (z) e^{\im k  x}
\, \colon \, \T \times (-\infty, 0] \to \C \ \ \mbox{:} \ \   
\| u \|_{\subalign{ &\scriptscriptstyle s,b\\  & \scriptscriptstyle -a_0, a}} < \infty 
\Big\}  \,
\end{equation}
endowed with the norm 
\begin{align}
\| u \|_{\subalign{ &\scriptscriptstyle s,b\\  &\scriptscriptstyle-a_0, a}}^2 & := 
\sum_{j = 0}^b 
\Big( 
 \| \pa_z^j \Pi_0 u \|_{L^{2,-a_0}}^2   
+ 
 \,  
 \| \pa_z^j \Pi_0^\bot u \|_{L^{2,a}(\R_-,H^{s-j})}^2   
\Big) 
\notag  \\
& = \sum_{j = 0}^b 
\Big( 
 \| \pa_z^j u_0 \|_{L^{2,-a_0}}^2   
+ 
\sum_{k \neq  0}
 \, |k|^{2(s-j)} \,  \| \pa_z^j u_k \|_{L^{2,a}}^2   
\Big)  
\label{migliore4}
\end{align}
where $\Pi_0$ is the projector on the zero mode, and  $\Pi_0^\perp := \uno-\Pi_0$.
        Note that, if $a_0>0$, a function   $ u(x,z) \in H^{s,b}_{-a_0,a}$
may exhibit a  first mode $ u_0 (z) $ that grows  
 exponentially as $z\to - \infty$. When  $a_0 = -a$ the space $H^{s,b}_{a,a}$ 
 coincides with 
the space defined in  \eqref{spaziok}, namely  
$H^{s,b}_{a,a} = H^{s,b}_a(\T\times\R_-) $, equipped with norm $\|\cdot\|_{s,b,a} \equiv  \|\cdot\|_{\subalign{ &\scriptscriptstyle s,b\\  &\scriptscriptstyle a, a}}$. 

\begin{remark}\label{rem:decay}
We shall prove  in \eqref{ezDA} 
that the free propagator $e^{z|D|_{\al,\mu}}$ in \eqref{propg}
is analytic 
in a full complex 
neighborhood of $ (\alpha,\mu) $ near $ (0,0) $ (in the sense of \Cref{def:tM}) as an operator valued in $ H^{s,0}_{-a_0,a} $. Then in
\Cref{lem:Thetasharp} 
we shall be able to construct a solution $ \Theta^\sharp_g (x,z) $ of the elliptic problem  \eqref{eqsharp} 
in such spaces 
thanks to the fact that $ d_\epsilon (x,z) = \cO(\epsilon e^z)$ as $ z \to - \infty $, cf. 
\eqref{est:d}.  
\end{remark}

\subsubsection{The linear propagator}

We now provide the analytic properties of the linear propagator 
\eqref{propg} which, 
given  the 
 covariance property \eqref{Gmu+k}, is sufficient 
to consider for  
$ (\alpha, \mu) \in   \R\times (-\tfrac23,\tfrac23)$. 
To establish analytic dependence on
$ (\alpha, \mu) $, we now extend $ e^{z|D|_{\al,\mu}} $ to a complex neighborhood
$$
\R\times (-\tfrac23,\tfrac23) \subset \cU_1\cup\cU_2 \, , 
$$
where
    \begin{equation}\label{0710:1740}
        \begin{aligned}
            &\cU_1 := \cU_1(c) := \big\{ (\al,\mu)\in \C^2 \, : \ |\textup{Im}\al|< c(|\Re \,\al| + |\Re\mu|)\, , \ |\textup{Im}\mu|< c(|\Re \,\al| + |\Re\mu|)\, , \ |\Re\mu|<\tfrac23 \big\} \, , 
            \\
            &\cU_2 := \cU_2(c) := \big\{ (\al,\mu)\in \C^2\, : |\al|^2+|\mu|^2 <c^2 \big\} \, ,
        \end{aligned}
    \end{equation}
and $ c \in (0,1) $. We shall choose the constant $ c $ such that  
\eqref{choice} is satisfied.   
    Note that $ (\R\times (-\tfrac23,\tfrac23))\setminus\{(0,0)\} \subset  \cU_1 $.   

    We now define an
    analytic extension of the function 
$ |k|_{\al,\mu}$ in \eqref{disprel} 
to the complex neighborhood $\cU_1 \cup \cU_2$. 
  For any   
 $\alpha := \alpha_1 + \im \alpha_2$ and $\mu := \mu_1 + \im \mu_2$ we define
     \be\label{k.ana}
    (\al,\mu) \mapsto |k|_{\al,\mu}
    := \sqrt{(k+\mu)^2 + \al^2} =  
    \big( (k+\mu_1)^2  - \mu_2^2 + \alpha_1^2 - \alpha_2^2
   + \im 2 ( k \mu_2  + \mu_1\mu_2 + \alpha_1\alpha_2)\big)^{1/2} 
    \ee
where we fix the branch of the complex square root 
\begin{equation}\label{defsqrt}
\C \setminus \{0\} \ni z = r e^{\im\theta} 
\;\longmapsto\; \sqrt z :=
r^{1/2} e^{\im\theta/2} \, ,
\quad 
r \in (0,\infty),\quad \theta \in [-\pi,\pi) \, . 
\end{equation}
Note that $ |k|_{\alpha,\mu} $
is a function of $ \beta = \alpha^2 $.

\begin{lemma}\label{lemmaassest}  There are $ c_0 > 0 $ and $ \tC \in (0,1)$   such that for any $c \in (0, c_0)$, for any $k\in\Z$,  
     \begin{equation}\label{assest}  
     \Re \, |k|_{\al,\mu} \geq 
     \tC ( |k| +|\al|)   \mbox{ in } \cU_1(c) \ , \qquad 
    \Re \, |k|_{\al,\mu} \geq 
     \tC  |k|   \mbox{ in } \cU_2(c) 
     \, .
    \end{equation}
For any $k \in \Z$ the  range of the map      
      $\cU_1(c)\cup \cU_2(c) \to \C $, $ (\al,\mu) \mapsto |k|_{\al,\mu}$, is included in 
 \be\label{rangek}
 \{\lambda\in\C\, : \ \Re\lambda \geq \tC \}
 \quad \text{if} \ 
 k \neq 0 \, , \qquad \text{resp.} \ \  
 \{\lambda\in\C\, : \ \Re\lambda \geq 0\} \quad \text{if} \ k = 0 \, , 
\ee
 and 
    is analytic on $\cU_1(c)$. For any $k\neq 0$, the function
    $$
    \{  (\al^2,\mu) \, : \ (\al,\mu) \in \cU_1(c)\cup\cU_2(c) \} \to \C \, , \ 
    (\al^2,\mu) = (\beta,\mu) \mapsto ((k+\mu)^2+\beta)^\frac12\  \ \text{is analytic.}
    $$
    \end{lemma}
    
    \begin{proof} 
   \underline{Case $(\al, \mu) \in \cU_1(c)$.} There is $  c_0 > 0 $ such that, for any $ 0 < c < c_0 $,  any $k \in \Z $, 
the real part of $(k+\mu)^2+\al^2$ is positive and satisfies    
 \begin{align}\notag
\Re\left((k+\mu)^2+\al^2\right)&  \geq 
     (k + \mu_1)^2 - 4 c^2 \mu_1^2 + (1-4c^2)\al_1^2   =  
    \big(k + (1-2c)\mu_1\big) \big(k + (1+2c) \mu_1\big) 
    + (1-4c^2)\al_1^2\\ 
    & \gtrsim
    \max(k^2, \mu_1^2) + |\al|^2  \label{D.13} \, 
\end{align}
 where to pass from the first to the second line we added and subtracted $\al_2^2$ and used the bound  $|\al_2| \leq c(|\al_1| + |\mu_1|)$. 
As a consequence, there is $ C_1 > 0 $ such that for any 
$ 0 < c < c_0 $, for any $ (\al, \mu) \in \cU_1(c)$, any $k\in \Z$, 
\be\label{tanabs}
\frac{\big| \Im \left((k+\mu)^2+\al^2\right) \big|}
     {\Re \left((k+\mu)^2+\al^2\right) } \stackrel{\eqref{k.ana},\eqref{D.13}} \lesssim 
     \frac{2 |k| c\,(  |\al_1|+|\mu_1|)  + 4c\,( |\al_1|+|\mu_1|)^2}{\max ( k^2, \mu_1^2)   + \al_1^2}  < c\,  C_1 \, .
\ee
By \eqref{tanabs}, 
there is a 
compact set $ K \subset (-\frac{\pi}{2}, \frac{\pi}{2}) $ such that 
$ \arg ( (k+\mu)^2+\al^2 ) \in K $
for any $ (\al,\mu)\in\cU_1(c) $, 
$ k \in \Z $, and therefore 
each $  \Re\,  |k|_{\al,\mu} > 0 $  and  
\begin{align*}
    \Re\,  |k|_{\al,\mu} & \simeq ||k|_{\al,\mu}| \simeq \sqrt{\Re\left((k+\mu)^2+\al^2\right) }  \stackrel{\eqref{D.13}}{\gtrsim} 
    \max (|k|, |\mu_1|)+|\al|
    \gtrsim  |k| +|\al|
    \, ,
    \end{align*}
proving  \eqref{assest} on $\cU_1(c)$. \\
\underline{Case $(\al,\mu)\in \cU_2(c)$.} There is $ c_0 > 0 $ such that, for any $ 0 < c < c_0 $, any $k\neq 0$, 
any  $
|\alpha|^2 + |\mu|^2 < c^2$,  we have 
$ \Re \left((k+\mu)^2+\al^2\right) 
\gtrsim k^2 $ and 
$ \Im \left((k+\mu)^2+\al^2\right) 
\lesssim c |k|  $ 
 and so 
\begin{align*}
  \frac{\big| \Im \left((k+\mu)^2+\al^2\right) \big|}
     {\Re \left((k+\mu)^2+\al^2\right) } 
\leq 
c C_1\, .
\end{align*}
Hence 
$ \Re\, |k|_{\al,\mu} \simeq ||k|_{\al,\mu}| \gtrsim |k| $ for any 
$  
(\al,\mu)\in\cU_2 (c) $ and 
$ k\neq 0 $. 
If $k = 0$, $\Re |0|_{\alpha,\mu} \geq 0 $ in view of  \eqref{defsqrt}, concluding the proof of  \eqref{assest}.
The inclusions \eqref{rangek} follow by \eqref{assest}. 
Finally $ |k|_{\alpha,\mu} $
is analytic on 
$\cU_1(c)$ because, 
by  \eqref{D.13}-\eqref{tanabs},    
the range of the analytic function 
$ \cU_1(c) \to \C $, 
$ (\al,\mu)\mapsto (k+\mu)^2+\al^2 $  does not intersect the semiaxis $(-\infty,0]\subset \C $, 
on which the square root in \eqref{defsqrt} fails to be analytic.
    \end{proof}

\noindent{\bf Parameters:} From now on we fix 
\be \label{choice}
0<  a < \tC< 1 \ ,  
\quad 
0 < a_0 < \min(a, 1-a) \ , 
\quad 0 < c < \min \Big(c_0, \frac{a_0}{2}\Big)  \, , 
\ee
so that  by \eqref{assest},
for any $(\al,\mu)\in\cU_1\cup \cU_2 $, 
\be\label{eq:D14}
\Re |0|_{\al, \mu} + a_0 \geq \frac{a_0}{2} + \tC |\alpha|>0  \, , \qquad \Re |k|_{\al,\mu} - a > (\tC -a) (|k| +  |\al|) \quad \forall k\neq 0\,.
\ee

\begin{proposition}[{\bf Propagator}]\label{lem:und.varphi}
There is $ C >0 $ such that, 
for any    $s\in \R$, any 
$ (\al,\mu)\in \R\times (-\tfrac23,\tfrac23)
$, 
    \begin{equation}
\label{mappaBa}
 \|e^{z|D|_{\al,\mu}}\|_{\cL\big(H^{s},  
 H^{s+\frac12 , 0}_{-a_0, a}\big)} \leq C  \, ,  \qquad \|e^{z|D|_{\al,\mu}}\|_{\cL\big(H^{s},  
 H^{s , 0}_{-a_0, a}\big)}\leq C \langle\al\rangle^{-\frac12} \, .
\end{equation}
Furthermore, 
recalling Definition \ref{def:tM}, 
\begin{equation}\label{ezDA}
e^{z|D|_{\al,\mu}} \in \cA\left(\R\times \left(\scriptstyle{-\frac23},\scriptstyle\frac23\right),+\infty;H^{s},  
 H^{s+\frac12 , 0}_{-a_0, a} \right) \, .
\end{equation} 
\end{proposition}

\begin{proof}
Let $g \in H^{s}(\T)$.
For any  $d\in\R$ and $k \in \Z $, if  $ \Re |k|_{\al,\mu} > d $ then 
\begin{equation}\label{suifour}
\|e^{ z|k|_{\al,\mu}}g_k \|_{L^{2,d}}^2 = \frac{|g_k|^2}{2(\Re|k|_{\al,\mu}- d)} 
\, . 
\end{equation}
Recalling \eqref{migliore4}
we deduce by \eqref{suifour} that,
for any $ s \in \R $,   
\begin{equation}\label{sobolevnorm}
    \|e^{z|D|_{\al,\mu}} g \|_{\subalign{ &\scriptscriptstyle s,0\\  &\scriptscriptstyle-a_0, a}}^2 = \frac{|g_0|^2}{2(\Re |0|_{\al, \mu}+a_0)}+\sum_{k\neq 0}  | k|^{2s}   \frac{|g_k|^2}{2(\Re|k|_{\al,\mu} - a)} \, . 
\end{equation}
By  \eqref{sobolevnorm} and \eqref{eq:D14} there is $ C > 0 $ such that, for any 
$(\al,\mu)\in \cU_1 \cup \cU_2 $,
\begin{equation}
\label{est:propagatorfarzero}
        \|e^{z|D|_{\al,\mu}} g \|_{\subalign{ &\scriptscriptstyle s+\frac12,0\\  &\scriptscriptstyle-a_0, a}}^2 
        \leq C  \|g\|_{H^s}^2 \, , \qquad \|e^{z|D|_{\al,\mu}} g \|_{\subalign{ &\scriptscriptstyle s,0\\  &\scriptscriptstyle-a_0, a}}^2 \leq C  \langle \al\rangle^{-1} \|g\|_{H^s}^2 \, ,
\end{equation}
proving in particular \eqref{mappaBa}.
We now prove \eqref{ezDA}. 
\Cref{def:tM}-Item 1)  trivially holds since the operator $e^{z|D|_{\al,\mu}}$ does not depend on $\e$. Furthermore,
 the map  $\cU_1 \to \cL(H^{s},  
H^{s+\frac12 , 0}_{-a_0, a}) $,
$ (\al,\mu )\mapsto 
e^{z|D|_{\al,\mu}} $, is analytic 
because it is locally  bounded by  \eqref{est:propagatorfarzero}
and weakly analytic, namely 
for any $k,k'\in \Z$,
and any $\chi \in \cC^{\infty}(\R_-,\C)$ with compact support,  
$$
 \cU_1 \to \C \, , \quad   (\al,\mu) \mapsto (e^{\im k x}\chi(z) , e^{z|D|_{\al,\mu}} e^{\im k' x})_{L^2(\R_- \times \T)} \, , 
$$
is analytic. This  proves the first bullet in Definition \ref{def:tM}-item $2$. 
In order to check that the propagator 
$ e^{z|D|_{\al,\mu}} $ has the structure \eqref{dectM} on  some small ball 
$ B_r(0,0) = \{ |\alpha|^2 + |\mu |^2 < r^2 \} $, we decompose  $$
    e^{z|D|_{\al,\mu}} = 
        \underbrace{e^{z|D|_{\al,\mu}}\Pi_0^\perp + \cosh(z(\al^2+\mu^2)^\frac12)\Pi_0}_{=: [e^{z|D|_{\al,\mu}}]^{[\mathrm{I}]}} + (\alpha^2 + \mu^2)^{\frac12} \underbrace{\frac{\sinh(z(\al^2+\mu^2)^\frac12)}{(\al^2+\mu^2)^\frac12}\Pi_0}_{=:[e^{z|D|_{\al,\mu}}]^{[\mathrm{II}]}}
$$
where $\Pi_0$ is the projector on the zero mode and $\Pi_0^\perp := \uno-\Pi_0 $.
Let us prove that 
 the functions  $ B_r(0,0) \to \cL(H^{s},  
H^{s+\frac12 , 0}_{-a_0, a}) $,
$ (\al^2 ,\mu )\mapsto 
[e^{z|D|_{\al,\mu}}]^{[\mathrm{I}]},
[e^{z|D|_{\al,\mu}}]^{[\mathrm{II}]} $ 
are analytic  for $ r $ small enough.
We shall use that both $\cosh(\lambda)$ and $\frac{\sinh(\lambda)}{\lambda}$ are entire functions of $\lambda^2$, satisfying the bounds
$$
\abs{\cosh(\lambda)} \leq \cosh(|\lambda|) \ , \quad 
\abs{\frac{\sinh(\lambda)}{\lambda}} \leq 
\frac{\sinh(|\lambda|)}{|\lambda|} \ , \quad \forall \lambda \in \C \, . 
$$
First,  the maps $B_r(0,0) \ni  (\al^2 ,\mu )\mapsto 
[e^{z|D|_{\al,\mu}}]^{[\mathrm{I}]},
[e^{z|D|_{\al,\mu}}]^{[\mathrm{II}]}$ are weakly analytic  since  the maps 
$$
(\al^2,\mu)\to (e^{\im k x}\chi(z) , [e^{z|D|_{\al,\mu}}]^{[*]} e^{\im k' x})_{L^2(\R_- \times \T)}, \qquad * \in \{I, II\} \, , 
$$
are analytic for any $k,k'\in \Z$ and any 
  smooth function $\chi(z)$ with compact support   in $(-\infty, 0]$.
They are also locally bounded since   there is $ C > 0 $ such that, for any 
$(\al, \mu) \in B_r(0,0)$ and    
$0 < r < a_0 $,  then 
$$
\| \cosh(z(\al^2+\mu^2)^\frac12) \|_{L^{2,-a_0}} \, , \  
\Big\| \frac{\sinh(z(\al^2+\mu^2)^\frac12)}{(\al^2+\mu^2)^\frac12} \Big\|_{L^{2,-a_0}} \leq C \, . 
$$
As a consequence, recalling\eqref{def:expmod0}-\eqref{migliore4}, we get  that 
  $
  \cosh(z(\al^2+\mu^2)^\frac12)\Pi_0 $, $
  \ \frac{\sinh(z(\al^2+\mu^2)^\frac12)}{(\al^2+\mu^2)^\frac12}\Pi_0 \in 
\cL(H^s, H^{s+\frac12, 0}_{-a_0, a})
$ 
with operatorial norm uniformly bounded for any  $(\al, \mu) \in B_r(0,0)$. 
Finally 
also $e^{z|D|_{\al,\mu}}\Pi_0^\perp$ is bounded by  \eqref{est:propagatorfarzero},  concluding the proof 
 that 
$ [e^{z|D|_{\al,\mu}}]^{[\mathrm{I}]},
[e^{z|D|_{\al,\mu}}]^{[\mathrm{II}]} $ 
are locally bounded near $ (0,0) $. 
\end{proof}

\subsubsection{The perturbative elliptic problem}
The study of the solution $ \Theta^\sharp_{g}$ of \eqref{eqsharp} is more complicated.

\begin{proposition}\label{lem:Thetasharp} {\bf (Solution of \eqref{eqsharp})}
For any $s \in \R$ there exist $\e_0 := \e_0 (s)>0$  such that, for any $(\alpha,\, \mu,\, \e) \in \R \times (-\frac23,\frac23)\times B_{\e_0}(0)$ and any $g\in H^s(\T)$, 
the elliptic problem  
\eqref{eqsharp} 
has a unique solution $\Theta^\sharp_{g} \in H^{s+\frac52,2}_{-a_0,a}$ 
satisfying 
\begin{equation}\label{est:thetasharpnonzero}
    \begin{aligned}
    &\|\Theta^\sharp_g\|_{\subalign{ &\scriptscriptstyle s+\frac52,2\\  &\scriptscriptstyle-a_0, a}}\leq C_s\al^2|\e|\|g\|_s  \ , \quad &&\|\Theta^\sharp_g\|_{\subalign{ &\scriptscriptstyle s+\frac12,0\\  &\scriptscriptstyle-a_0, a}} \leq C_s |\e|\|g\|_s\, ,\\
       &\|\pa_z \Theta^\sharp_g\vert_{z=0}\|_{s+1} \leq C_s \al^2|\e|\|g\|_s \, , \quad &&\|\pa_z \Theta^\sharp_g\vert_{z=0}\|_{s} \leq C_s |\al||\e|\|g\|_s \, .
    \end{aligned}
\end{equation}
Moreover, recalling \Cref{def:tM,defFell}, 
\begin{equation}\label{pazThetainA}
    [g\mapsto   (\alpha^2 \e)^{-1}\pa_z\Theta^\sharp_{g}\vert_{z=0}] \in \cA \Big(\R\times (-\tfrac23,\tfrac23),\e_0; H^s(\T), H^{s+1}(\T) \Big) \cap  \tF\, .
\end{equation}
\end{proposition}

The rest of the section is devoted to the proof of \Cref{lem:Thetasharp}. 

We
first provide  properties 
of the  spaces $H^{s, b}_{-a_0, a}$. 
In view of  \eqref{migliore4}  we directly get the following  lemma. 

\begin{lemma}\label{lem:prop}
For any  $ s \in \R $,  $ a_0, a > 0$, $b \in \N$,  
the map
 $ \pa_z :
 H^{s , b}_{-a_0, a}  \mapsto H^{s-1 , b-1}_{-a_0, a} $
is continuous.
\end{lemma}

The  next  lemma provides a simple  characterization of  the spaces $H^{s , b}_c$ in \eqref{spaziok}.
\begin{lemma}\label{scalarexp}
    Let $s\in \R$, $b\in \N_0 $ and $c\in \R$. Then a function $u(x,z) $ belongs to $ H^{s,b}_c(\T\times\R_-)$ if and only if $\wt u(x,z):=e^{-cz}u(x,z) $ belongs to $  H^{s,b}_0 (\T\times\R_-)$  
     with equivalence of the norms
    \begin{equation}\label{eq:eq}
        \|u\|_{s,b,c} \sim_{b,c} \| \wt u \|_{s,b,0}\, .
    \end{equation}
\end{lemma}

\begin{proof}
    In view of \eqref{usba} and the inequality $(A_1+\dots + A_j)^2 \lesssim_j A_1^2 +\dots +A_j^2$ we obtain 
    \begin{equation*}
        \begin{aligned}
            \| \wt u \|_{s,b,0}^2 &= \sum_{j=0}^b \sum_{k\in\Z} \langle k \rangle^{2(s-j)} \int_{-\infty}^0 |\pa_z^j(e^{-cz} u_k(z))|^2\de z \\ 
            &\lesssim_{b,c} \sum_{j=0}^b \sum_{k\in\Z} \langle k \rangle^{2(s-j)}  \sum_{l=0}^j \int_{-\infty}^0 \left|\pa_z^{l}u_k(z)\right|^2 e^{-2cz} \de z
            \lesssim_{b,c} \|u \|_{s,b,c}^2 \, . 
        \end{aligned}
    \end{equation*}
Similar estimates show that $\|u\|_{s,b,c}^2 = \|e^{cz} \wt u \|_{s,b,c}^2 \lesssim_{b,c} \| \wt u \|_{s,b,0}^2$, yielding the equivalence \eqref{eq:eq}.
\end{proof}

The next lemma proves the existence of a well defined trace operator in the spaces $H^{s , b}_{-a_0, a}$. 
\begin{lemma}[{\bf Trace}]\label{lem:trace} 
Let  $s\in \R$, 
$ a_0 > 0 $. Then 
the  trace operator 
\begin{equation} \label{optraccia}
\Gamma (u):= u(\cdot, 0) := u\vert_{z= 0}
\end{equation}
extends to a linear bounded 
 map between $ H^{s, 1}_{-a_0, a} \to H^{ s-\frac12}(\T)$, 
satisfying
\begin{equation}
\label{est:trace}
\| \Gamma (u) \|_{H^{ s- \frac12}} 
\lesssim_{s,a_0,a}  \| u \|_{\subalign{ &\scriptscriptstyle s,1\\  &\scriptscriptstyle-a_0, a}} \, . 
\end{equation}
\end{lemma}

\begin{proof} 
By  \cite[Lemma 2.3]{BMV2}
and recalling \eqref{usba}, 
the trace operator $ \Gamma $ extends to a bounded map satisfying 
\be\label{Gammavtr}
    \|\Gamma (v) \|_{s-\frac12}\leq \|v \|_{s,1,0} \ , \quad \forall v \in H^{s,1}_0 (\T\times\R_-)\ .
\ee
For any  continuous function $u(x, z) $ the function 
$ v(x,z) = e^{a_0z} u(x,z)$ 
satisfies 
    $ u(x,0) = v (x,0) $ and we get 
$$
\| \Gamma (u) \|_{H^{s-\frac12}}
=
\| \Gamma (v) \|_{H^{s-\frac12}}
\stackrel{\eqref{Gammavtr}} \leq \| v \|_{s,1,0}
\stackrel{\eqref{eq:eq}}\sim \|u \|_{s,1,-a_0} \stackrel{\eqref{usba},\eqref{migliore4}} \leq \|u\|_{\subalign{ &\scriptscriptstyle s,1\\  &\scriptscriptstyle-a_0, a}}
$$
proving \eqref{est:trace}.  
\end{proof}

 We now show that the multiplication operator by the function $d_\e (x,z) $ 
in \eqref{def:d} 
 maps the space
$H^{ s, b}_{-a_0, a}$ into the space $H^{s,b}_a$, namely it 
improves the decay on the zero mode.

\begin{lemma} \label{prop:mult.d}
Let $ a,a_0 $ as in \eqref{choice}.
 For any $s \in \R $, $b\in \N_0$ there is $C_{s,b} >0 $ such that for any $\e \in B_{\e_0}(0)$ 
 (with $\e_0$ in \Cref{LeviCivita})
\begin{equation}
\label{algebrassa}
\| d_\e u \|_{s,b, a} \leq C_{s,b} |\e| 
\| u \|_{\subalign{ &\scriptscriptstyle s,b\\  &\scriptscriptstyle-a_0, a}}   \, , \quad 
\forall u  
\in H^{ s, b}_{-a_0, a}  \, , 
\end{equation}
where  $d_\e(x,z) $ is the function in \eqref{def:d}.
The map $\e \mapsto [u \mapsto {d_\e}u]$ is analytic from $B_{\e_0}(0)$ to $\cL(H^{s,b}_{-a_0,a}, H^{s,b}_{a,a})$. Moreover, for any fixed $z \in \R_-$, the multiplication operator $ [u(\cdot,z)\to d_\e(\cdot,z)u(\cdot,z)] $ belongs to the class $\tF$ in \Cref{defFell}.
\end{lemma}

\begin{proof}
    By  \Cref{LeviCivita} the  function $d_\e\in H^{s_0,b_0}_{c}(\T\times\R_-)$ for any $s_0 \geq 0$,  $b_0\in \N_0$ and $c\in(0,1)$. 
     In view of \Cref{scalarexp}
     the functions 
      $ \wt d_\e (x,z)  := e^{-cz}  d_\e (x,z)  $ and 
      $\wt u (x,z)  := e^{a_0z}  u (x,z) $ 
   belong to  
       $H^{s_0,b_0}_0(\T\times\R_-)$ 
       and $H^{s,b}_0(\T\times\R_-)$ respectively, and
     \begin{equation}\label{tameproof}
      \|\wt d_\e \|_{s_0,b_0,0} \simeq \|d_\e \|_{s_0,b_0,c}\stackrel{\eqref{est:d}}{\lesssim_{s_0,b_0}} |\e|\, ,
      \qquad 
      \|\wt u\|_{s,b,0} \simeq \| u\|_{s,b,-a_0} \lesssim \| u\|_{\subalign{ &\scriptscriptstyle s,b\\  &\scriptscriptstyle-a_0, a}}\, .
     \end{equation}
Let   $c = a + a_0 \in (0,1)$.  Using  \Cref{scalarexp}
and recalling \eqref{usba} we have
     $$
     \begin{aligned}
      \| d_\e \, u  \|_{s,b,c-a_0}  & 
      \simeq_b \| \wt d_\e  \, \wt u  \|_{s,b,0}  \lesssim_b 
      \| \wt d_\e  \|_{|s|+1,b,0}\| \wt u  \|_{s,b,0}  
    \stackrel{\eqref{tameproof}} {\lesssim_{s}} |\e| \| u\|_{\subalign{ &\scriptscriptstyle s,b\\  &\scriptscriptstyle-a_0, a}}
      \end{aligned}
     $$
 proving \eqref{algebrassa}.
   The multiplication operator $[u(\cdot,z)\to d_\e(\cdot,z)u(\cdot,z)] $ belongs to $ \tF$, in view of  \eqref{armoMOP} and since the  $\ell$'th jet $d_\ell(x,z)$ of $d_\e(x,z) $ has the form  \eqref{deintF}. 
\end{proof}

In order to study \eqref{eqsharp} 
we   
consider the auxiliary
non-homogeneous 
elliptic problem 
\begin{equation}\label{non-pert}
        \begin{cases}
            \pa_z^2 u (x,z)  +  (\pa_x+\im\mu)^2u (x,z)- \al^2 u(x,z) = \al^2  q (x,z)  \\
            u(x,z)|_{z=0} = 0 \, , \quad 
            \lim_{z\to -\infty}\pa_z u (x,z) = 0 \, .
        \end{cases}
    \end{equation}

\begin{lemma}\label{solution_BVP}
    There is $C>0$ such that, for any $s \in \R$,  any $(\al,\mu)\in \R\times(-\tfrac{2}{3},\tfrac23) $ and
    any $ q \in H^{s,0}_a (\T\times\R_-) $, the elliptic problem \eqref{non-pert}
    admits a unique solution 
    $u := L_{\alpha,\mu }q  $ in
    $   H^{s+2,2}_{-a_0,a}$ satisfying  
    \begin{equation}\label{estlal}
        1) \ \norm{L_{\al, \mu} q}_{\subalign{ &\scriptscriptstyle s+2,2\\  &\scriptscriptstyle-a_0, a}} \leq C \alpha^2 \norm{q}_{s,0,a} \, , \qquad 2) \ \norm{L_{\al, \mu} q}_{\subalign{ &\scriptscriptstyle s,0\\  &\scriptscriptstyle-a_0, a}} \leq C  \norm{q}_{s,0,a} \, , \qquad 3) \ \|\pa_z L_{\al,\mu}q\vert_{z=0}\|_s \leq C |\al|^\frac32\|q\|_{s,0,a} \, . 
    \end{equation}
Furthermore, 
recalling Definition \ref{def:tM},
the map  
\be \label{lemmad8ana}
(\alpha, \mu ) \mapsto \alpha^{-2} L_{\alpha,\mu } \in  \cA(\R\times(-\tfrac{2}{3},\tfrac23),+\infty;H^{s,0}_a,  H^{s+2,2}_{-a_0,a}) \, . 
\ee
\end{lemma}

\begin{remark}\label{decaysolv}
The estimate \eqref{estlal}-2) is implied by the 1) for any $ |\alpha | \lesssim 1 $, not as $ |\alpha | \to \infty $. 
Note that $q$ has exponential decay, roughly  $ \cO( e^{a z}) $ as $z \to - \infty $
along all non-zero Fourier modes, while the zero mode of the solution $L_{\al,\mu}q$ 
decays only as $ \cO ( e^{z \sqrt{\al^2+\mu^2}} ) $ as $z \to - \infty $ for any $ (\al,\mu) \neq (0,0) $, cf. \eqref{sol}.  
\end{remark}

\begin{proof}
    Expanding \eqref{non-pert} in Fourier series in the $x $-variable we get
   \begin{equation}\label{2nonhomo}
      \begin{cases}
  \pa_z^2 u_k(z) - 
    \underbrace{((k+\mu)^2+\al^2)}_{= |k|_{\al,\mu}^2}u_k (z)  = \alpha^2 q_k (z) \, , \quad \forall k \in \Z \, ,  \\
            u_k(0) = 0 \, , \quad 
            \lim\limits_{z\to -\infty}\pa_z u_k (z) = 0 \, .
        \end{cases}
    \end{equation}    
    By the variation of constants method, for any $(\al,\mu)\in \R\times (-\tfrac{2}{3},\tfrac{2}{3})$, the unique solution of 
    \eqref{2nonhomo}  is, if  $ |k|_{\al,\mu} \neq 0$, 
    \begin{equation}\label{sol}
         u_k(z) := [L_{\al,\mu} q]_k (z) = 
            \frac{\alpha^2}{2|k|_{\al,\mu}}\left(-[T_{|k|_{\al,\mu}} q_k] (z) - 
            [\wt T_{|k|_{\al,\mu}}q_k](z) + e^{|k|_{\al,\mu}z} 
            [T_{|k|_{\al,\mu}} q_k] (0)\right)
    \end{equation}
    where, 
    for any  $ \lambda  \in \C $,
\begin{equation}\label{opT}
        [T_\lambda p](z) :=e^{-\lambda z}\int_{-\infty}^z e^{\lambda t} p(t) \, \de t\ ,
        \qquad
        [\wt T_\lambda p] (z) :=e^{\lambda z}\int_{z}^0 e^{-\lambda t} p(t) \, \de t  \, .
    \end{equation} 
We have   $ u_k (0) = 0 $ and    
    \begin{equation}\label{solder}
         \pa_z u_k(z) = 
            \frac12 \al^2 \Big((T_{|k|_{\al,\mu}} q_k)  (z) - (\wt T_{|k|_{\al,\mu}}q_k)(z) + e^{|k|_{\al,\mu}z} (T_{|k|_{\al,\mu}} q_k) (0)\Big)
            \, .
    \end{equation}
By  \Cref{Test} if 
$ q_k \in L^{2,a}$ the function $ u_k (z) $ in \eqref{sol}-\eqref{opT}
is well defined 
and the  inequalities on the right of \eqref{item1d}, \eqref{item2d} with $\lambda = |k|_{\al,\mu}$, imply  also the decay property  $ \pa_z u_k (z) \to 0 $ as $ z \to - \infty $ (note that $ |k|_{\alpha,\mu} >0  $ for any $ (k,\alpha,\mu) \neq (0,0,0) $).  
For $(\al,\mu)=(0,0)$ and $ k = 0 $
the solution of  \eqref{2nonhomo} is zero (whereas 
\eqref{sol} is not well defined) 
and  
$ L_{0,0} q = 0 $.  
The operators in \eqref{opT} satisfy 
    $$
    \overline{T_\lambda} = T_{\bar \lambda} \, , \quad 
    \overline{\wt T_\lambda} = \wt T_{\bar \lambda} \, . 
    $$ 
We now rewrite 
    \eqref{sol} as \begin{equation}\label
    {altformsolBVP}
        [L_{\al,\mu}q]_k(z) := \frac{\al^2}{|k|_{\al,\mu}} [F(|k|_{\al,\mu})q_k] (z) 
        = \al^2 \big[ \wt F(|k|_{\al,\mu})q_k \big] (z)
    \end{equation}
where, for any $ \lambda \in \C $,  
\begin{equation}\label{def:Flemap}
        [F(\lambda)q_k] (z) := \frac12\left(-[T_{\lambda} q_k] (z)- [\wt T_\lambda q_k ](z) + e^{\lambda z} [T_\lambda q_k](0)\right)\, \, , \quad
F(\lambda) = \lambda \wt F(\lambda) \, . 
    \end{equation} 
Note that, by \eqref{opT}, we have  $F(0) = 0$ and 
$\overline{F(\lambda)} = F(\bar \lambda)$.
\\[1mm]
    {\noindent\bf Step 0: Estimates on $ F $ and $ \tilde F $.}
By 
    \Cref{Test} the maps
    \begin{equation}\label{randomF}
        \{\lambda\in\C\, : \ \Re\lambda> -a_0  \} \ni \lambda\mapsto F(\lambda) \in \cL(L^{2,a},L^{2,-a_0})\, , \ \  \{\lambda\in\C\, : \ \Re\lambda> a  \}\ni\lambda \mapsto F(\lambda) \in \cL(L^{2,a},L^{2,a})\, ,
    \end{equation}
    are analytic, and satisfy, since $ a_0 < a  $  by \eqref{choice}, 
\begin{equation}\label{est:FFF}
    \begin{aligned}
        &\|  F(\lambda)q\|_{L^{2,-a_0}} \lesssim  \frac{\|q\|_{L^{2,a}}}{a_0+\Re\lambda} \quad \forall \ \Re\, \lambda > -a_0\, , 
        \qquad 
        \|F(\lambda)q\|_{L^{2,a}} \lesssim  \frac{\|q\|_{L^{2,a}}}{ \Re\lambda-a} \quad \forall \ \Re\lambda > a\, .
    \end{aligned}
    \end{equation}
     Thus, since $F(\lambda)$ is analytic in $\lambda$, we have that 
$
F(\lambda) = \lambda \wt F(\lambda) 
$ where
    \begin{equation}\label{est:derF}
    \begin{aligned}
    & \wt F(\lambda) \mbox{ is an analytic   function with range and domain as in } \eqref{randomF}, \mbox{ with } 
    \overline{\wt F(\lambda)} = \wt F(\bar \lambda) \mbox{ and } \\
        &\| \wt F(\lambda)q\|_{L^{2,-a_0}} \lesssim  \frac{\|q\|_{L^{2,a}}}{\langle\lambda\rangle(a_0+\Re\lambda)} \quad \forall \, \Re\lambda > -a_0\, , \quad \|\wt F(\lambda)q\|_{L^{2,a}} \lesssim  
        \frac{\|q\|_{L^{2,a}}}{ \langle\lambda\rangle(\Re\lambda-a)} \quad \forall \, \Re\lambda > a\, .
    \end{aligned}
    \end{equation}
    
    {\noindent\bf Step 1: uniform boundedness of $L_{\al,\mu}$.}
  We now show the uniform boundedness of $L_{\al,\mu}$ for complex-valued $(\al, \mu)$ in the complex neighborhood $\cU_1\cup \cU_2$ defined   in \eqref{0710:1740}. 
\\[1mm]
    \noindent{\sc Proof of \eqref{estlal}-2)}
    For any  $(\al,\mu)\in\cU_1\cup\cU_2$
    we have that $ \Re |k|_{\alpha,\mu}> a $ by \eqref{eq:D14} and $\Re |0|_{\al,\mu} \geq   0$ by \eqref{rangek},\eqref{choice}. Thus, applying \eqref{est:derF}, we have, taking $\star = a$   if $k\neq 0$ and  $\star = -a_0$ if $ k = 0 $, 
\begin{equation}\label{proofsamespace}
        \|[L_{\al,\mu}q]_k \|_{L^{2,\star}} \stackrel{\eqref{altformsolBVP},\eqref{est:derF}}{\lesssim} \frac{|\al|^2
        \|q_k\|_{L^{2,a}} }{\langle |k|_{\al,\mu}\rangle (-\star + \Re|k|_{\al,\mu})}  \stackrel{\eqref{eq:D14}}{\lesssim} \min\{1,\al^2\langle k \ra^{-2}\}\|q_k\|_{L^{2,a}}\, .
    \end{equation}
     We deduce
    \begin{equation}
        \|L_{\al,\mu}q\|_{\subalign{ &\scriptscriptstyle s,0\\  &\scriptscriptstyle-a_0, a}}^2 = \|[L_{\al,\mu}q]_0\|_{L^{2,-a_0}}^2 + \sum_{k\in \Z\setminus\{0\}}| k|^{2s} \|[L_{\al,\mu}q]_k\|_{L^{2,a}}^2 \stackrel{\eqref{proofsamespace}}{\lesssim} \sum_{k\in\Z} \langle k\rangle^{2s}\|q_k\|_{L^{2,a}}^2=\|q\|_{s,0,a}^2\, ,
    \end{equation}
    proving in particular the second estimate in \eqref{estlal} uniformly in $(\al,\mu)\in \cU_1\cup\cU_2$.
    \\[1mm]
    \noindent{\sc Proof of \eqref{estlal}-1)}. By \eqref{2nonhomo}, for any $(\al,\mu)\in\cU_1\cup\cU_2$, using the first inequality in \eqref{proofsamespace},   \eqref{eq:D14}, and the estimates $ |0|_{\alpha,\mu} \lesssim \langle \alpha \rangle $, $ |k|_{\alpha,\mu} \lesssim 1+ | \alpha| + |k|$, for any $ k \neq 0 $,  we deduce 
    \begin{equation}\label{eqpar2}
        \|\pa_z^2[L_{\al,\mu}q]_k \|_{L^{2,\star}} = \||k|_{\al,\mu}^2[L_{\al,\mu}q]_k + \al^2 q_k\|_{L^{2,\star}} \lesssim|\al|^2\|q_k\|_{L^{2,a}} 
    \end{equation}
    where $\star = -a_0$ if $k=0$, and 
$\star = a$ if $k \neq 0$. 
    For the first $z$-derivative, using \eqref{solder}, \eqref{assest} and Lemma \ref{Test}, for every $k\in\Z$ and $(\al,\mu)\in\cU_1\cup\cU_2$ one has
    \begin{equation}\label{eqpar3}
        \|\pa_z [L_{\al,\mu}q]_k\|_{L^{2,\star}} \lesssim |\al|^2 \frac{1}{-\star + \Re|k|_{\al,\mu}} \|q_k\|_{L^{2,a}} \lesssim |\al|^2\langle k \rangle^{-1}\|q_k\|_{L^{2,a}} 
    \end{equation}
    where $\star = -a_0$ if $k = 0$ and $\star = a $ if $k \neq 0$. 
    Recall also that by \eqref{proofsamespace}, for every  $(\al,\mu)\in\cU_1\cup\cU_2$ one has 
    \begin{equation}\label{eqpar1}
        \|[L_{\al,\mu}q]_k\|_{L^{2,\star}} \lesssim |\al|^2 \langle k \rangle^{-2} \|q_k\|_{L^{2,a}} 
    \end{equation}
    where $\star = -a_0$ if $k = 0$ and $\star = a $ if $k \neq 0$. 
    Recalling \eqref{migliore4}, 
    we obtain, summing up \eqref{eqpar2}-\eqref{eqpar1}  for any $(\al,\mu)\in\cU_1\cup\cU_2$, 
    \begin{equation}\label{0710:1729}
        \begin{aligned}
        \|L_{\al,\mu}q\|_{\subalign{ &\scriptscriptstyle s+2,2\\  &\scriptscriptstyle-a_0, a}} = \Big(\sum_{j=0}^{2}\Big[\|\pa_z^j [L_{\al,\mu}q]_0\|_{L^{2,-a_0}}^2 + \sum_{k\in\Z\setminus \{0\}} |k|^{2(s+2-j)}\|\pa_z^j [L_{\al,\mu}q]_k\|_{L^{2,a}}^2 \Big] \Big)^{\frac12}& 
        \\
        \lesssim |\al|^2\Big(\sum_{k\in\Z}\langle k\rangle^{2s}\|q_k\|_{L^{2,a}}^2\Big)^\frac12  & \lesssim |\al|^2 \|q\|_{s,0,a} \, , 
    \end{aligned}
    \end{equation}
    proving in particular also the first estimate in \eqref{estlal}.  
    \\[1mm]
    \noindent{\sc Proof of \eqref{estlal}-3)} In view of  \eqref{solder}, Recalling  that $[\wt T_\lambda \cdot](0) = 0$ by \eqref{opT},  we get
    $$
    |\pa_z [L_{\al,\mu}q]_k\vert_{z=0}| = |\al|^2 |(T_{|k|_{\al,\mu}}q_k)(0)| \stackrel{\eqref{item1d}} \leq  \frac{|\al|^2 \|q_k\|_{L^{2,a}} }{\sqrt{2(\Re |k|_{\al,\mu}+a)}} \stackrel{\eqref{eq:D14}}{\lesssim} |\al|^\frac32 \|q_k\|_{L^{2,a}} \, . 
    $$
    Summing up in $k$  we obtain
    $$
    \|\pa_z [L_{\al,\mu}q]\vert_{z=0}\|_s^2 \leq \sum_{k\in \Z}\langle k \rangle^{2s} |\pa_z [L_{\al,\mu}q]_k\vert_{z=0}|^2 \lesssim |\al|^3\|q\|_{s,0,a}^2\, , \quad \forall (\al,\mu)\in\cU_1\cup\cU_2 \ . 
    $$
We have proved that the map $(\al,\mu)\in \cU_1\cup\cU_2 \to L_{\al,\mu} \in \cL(H^{s,0}_a,H^{s+2,2}_{-a_0,a})$ satisfies the estimates \eqref{estlal}. 
\\[1mm]
{\bf \noindent Step 2: proof of \eqref{lemmad8ana}.} Clearly property 1 of \Cref{def:tM} is automatic since the operator is $\e$-independent. \\
{\sc Regularity in $(\alpha , \mu)$:} First we show that the map $L_{\al,\mu}$ is analytic in $(\al,\mu)\in \cU_1 \supset (\R\times (-\tfrac23,\tfrac23))\setminus \{(0,0)\}$. By \eqref{0710:1729}, for all $q\in H^{s,0}_a$ the map   
\begin{equation}\label{operatorrrrrr}
   \cU_1 \ni  (\al,\mu)  \to [\alpha^{-2} L_{\al, \mu} q] (x,z) \stackrel{\eqref{altformsolBVP}}{\equiv }\sum_{k\in \Z}[\wt F(|k|_{\al,\mu}) q_k](z)e^{\im k x} \in H^{s+2,2}_{-a_0,a} 
\end{equation}
is uniformly bounded on every bounded subset of $\cU_1$.
Moreover, by Lemma \ref{lemmaassest}, $\cU_1 \ni (\al,\mu) \mapsto |k|_{\al,\mu}$ is analytic, with range $\{\lambda\in\C\, : \ \Re\lambda > a\}$ whenever $k\neq 0$, and $\{\lambda\in\C\, : \ \Re\lambda > -a_0\}$ if $k=0$. Since by \eqref{est:derF}  $\lambda \mapsto \wt F(\lambda)$ is analytic with domain and range as in \eqref{randomF}, the map \eqref{operatorrrrrr} is weakly analytic, therefore analytic.

We now analyze the regularity of $ L_{\al,\mu}$ for any $(\al,\mu)\in\cU_2$ defined in  \eqref{0710:1740}. Let 
$$
\wt F_{odd}(\lambda) := \frac{\wt F(\lambda)-\wt F(-\lambda)}{2}\, , \quad \wt F_{even}(\lambda) := \frac{\wt F(\lambda)+\wt F(-\lambda)}{2}\, , 
$$ 
be respectively the odd and even components of $\wt F$, both analytic functions with the same domain and range as in \eqref{randomF}. 
 We write 
$$
\wt F_{even}(\lambda) = F^{[\mathrm{I}]}(\lambda^2) \, , \qquad \wt F_{odd}(\lambda) = \lambda F^{[\mathrm{II}]}(\lambda^2)
$$
where  $ F^{[\mathrm{I}]}(\eta)$ and $ F^{[\mathrm{II}]}(\eta)$ are  analytic  functions with domain $B_{c^2}(0)$ and range $\cL(L^{2,a},L^{2,-a_0})$. 
Then $L_{\al,\mu}$ decomposes as 
\begin{equation}
    \begin{aligned}
       \alpha^{-2} L_{\al, \mu}& = 
          \wt F_{even} ((\al^2+\mu^2)^\frac12)\Pi_0 + \alpha^{-2} L_{\al,\mu}\Pi_0^\perp  +   \wt F_{odd} ((\al^2+\mu^2)^\frac12)\Pi_0\\
         &= \underbrace{ F^{[\mathrm{I}]} (\al^2+\mu^2)\Pi_0 + \widetilde{F}(|D|_{\al,\mu})\Pi_0^\perp }_{=: L^{[\mathrm{I}]}}
        + (\alpha^2 + \mu^2)^{\frac12}
        \underbrace{  F^{[\mathrm{II}]} (\al^2+\mu^2) \Pi_0}_{=:L^{[\mathrm{II}]} } \, ,
    \end{aligned}
\end{equation}
where $\Pi_0$ is the orthogonal projection on the zero mode,  $\Pi_0^\perp = \uno-\Pi_0$ and we substituted $|0|_{\al,\mu} = (\al^2+\mu^2)^\frac12$.
By step $1$, the map $\widetilde{F}(|D|_{\al,\mu})\Pi_0^\perp$ is uniformly bounded on $\cU_2$, and testing it against functions of the form $\chi(z)e^{\im k x}$, where $\chi(z)\in\cC^\infty_c(\R_-)$, one proves   that it is  weakly analytic in $\al^2$ and $\mu$. Therefore both maps
$$
    \cU_2\ni (\al,\mu) \mapsto L^{[\mathrm{I}]}, L^{[\mathrm{II}]} \in \cL(H^{s,0}_a,H^{s+2,2}_{-a_0,a})
$$
are uniformly bounded and  weakly analytic in $\al^2$ and $\mu$, hence analytic.
\end{proof}

\begin{proof}[\bf Proof of \Cref{lem:Thetasharp}]
In view of
Lemma \ref{solution_BVP}
and  \eqref{propg}, 
any  solution $\Theta^\sharp_{g} $ of the elliptic problem \eqref{eqsharp} solves 
\begin{equation}
\label{invimpl}
(\uno -  L_{\al,\mu}\circ d_\e)\Theta^\sharp_{g} = L_{\al,\mu}[ d_\e \Theta^\flat_g] \qquad \text{where} \qquad  
\Theta^\flat_g = e^{z|D|_{\al,\mu}}g  
\quad 
\end{equation}
belongs to $ H^{s+\frac12,0}_{-a_0,a}$ since $ g \in H^s(\T)$ and  Lemma \ref{lem:und.varphi}.
We now use in a crucial way that 
$ L_{\al,\mu} \circ d_\e $ maps functions with an exponentially divergent zero mode as $ z \to - \infty $ into themselves. Indeed, 
by Lemma \ref{prop:mult.d} and the second estimate in \eqref{estlal}, the operator
$L_{\al,\mu} \circ d_\e
: H^{s+\frac12,0}_{-a_0,a} 
\to H^{s+\frac12,0}_{-a_0,a} $ is bounded uniformly in $(\al,\mu)\in \R\times (-\tfrac23,\tfrac23)$ with operatorial norm 
\begin{equation}\label{stimaLdamu}
 \norm{L_{\al,\mu} \circ d_\e}_{\cL(H^{s+\frac12,0}_{-a_0,a},H^{s+\frac12,0}_{-a_0,a})} \leq \norm{L_{\al,\mu}  }_{\cL(H^{s+\frac12,0}_{a},H^{s+\frac12,0}_{-a_0,a})} \norm{d_\e}_{\cL(H^{s+\frac12,0}_{-a_0,a},H^{s+\frac12,0}_{a})} \leq C_s |\e|   \, .
\end{equation}
Thus, provided $|\e|\leq \epsilon_0 (s) $ is small enough, we can invert by Neumann series  the operator 
$\uno -  L_{\al,\mu}\circ d_\e$ in 
$ \cL(H^{s+\frac12,0}_{-a_0,a})$, obtaining, by \eqref{invimpl},  
\begin{align}
    \Theta_g^\sharp 
    & = (\uno - L_{\al,\mu}\circ d_\e)^{-1}L_{\al,\mu}\circ d_\e \,  [ e^{z|D|_{\al,\mu}}g]
    \label{formulathetasharppp} \\
  &  = L_{\al,\mu}\circ d_\e \, (\uno - L_{\al,\mu}\circ d_\e)^{-1} \,  [ e^{z|D|_{\al,\mu}}g] \label{formulathetasharppp1}
\end{align}
because
the operators $(\uno - L_{\al,\mu}\circ d_\e)^{-1}$ and $L_{\al,\mu}\circ d_\e$ commute.
The second  estimate in \eqref{est:thetasharpnonzero}
follows by 
\eqref{formulathetasharppp},
\eqref{stimaLdamu} and
\eqref{mappaBa}. The first 
estimate in \eqref{est:thetasharpnonzero}
is a consequence of  \eqref{formulathetasharppp1},
\eqref{estlal}-1)
and \eqref{mappaBa}. 
The third estimate in \eqref{est:thetasharpnonzero} follows by
$$
\|\pa_z \Theta^\sharp_g\vert_{z=0}\|_{s+1} \stackrel{\eqref{est:trace}} {\lesssim_s}  
\|\pa_z \Theta^\sharp_g \|_{\subalign{ &\scriptscriptstyle s + \frac32,0\\  &\scriptscriptstyle-a_0, a}}
\stackrel{\Cref{lem:prop}} {\lesssim_s}  
\|\pa_z \Theta^\sharp_g \|_{\subalign{ &\scriptscriptstyle s + \frac52,2\\  &\scriptscriptstyle-a_0, a}} \lesssim_s \al^2|\e|\|g\|_s 
$$
by the first estimate in \eqref{est:thetasharpnonzero}.
The fourth estimate in  \eqref{est:thetasharpnonzero}
is a consequence of  
$$
\begin{aligned}
    \|\pa_z\Theta^\sharp_g \vert_{z=0}\|_{s} & \stackrel{\eqref{formulathetasharppp1}}= \Big\| \pa_z L_{\al,\mu} \big[ d_\e (\uno-L_{\al,\mu}\circ d_\e)^{-1} e^{z|D|_{\al,\mu}} g\big]\big\vert_{z=0}\Big\|_{s} \\
    &\stackrel{\eqref{estlal}-3)}{\lesssim} |\al|^\frac32\|d_\e (\uno-L_{\al,\mu}\circ d_\e)^{-1} e^{z|D|_{\al,\mu}} g\|_{s,0,a} \stackrel{\eqref{algebrassa}}{\lesssim_s} |\al|^\frac32 |\e| \|e^{z|D|_{\al,\mu}}g\|_{\subalign{ &\scriptscriptstyle s,0\\  &\scriptscriptstyle-a_0, a}} \\
    &\stackrel{\text{2nd of }\eqref{mappaBa}}{\lesssim_s} |\al||\e| \|g\|_s \, .  
\end{aligned}
$$
\noindent{\sc Proof that} 
\be\label{scopo1D25} 
g \mapsto (\alpha^2 \e)^{-1}\pa_z\Theta^\sharp_g\vert_{z=0} \in \cA(\R\times (-\tfrac23,\tfrac23),\e_0; H^s(\T), H^{s+1}(\T)) \, . 
\ee 
By \eqref{lemmad8ana}
the map 
  $(\al, \mu) \mapsto \al^{-2} L_{\al,\mu}\in \cA(\R\times (-\frac23,\frac23),+\infty;H^{s+\frac12,0}_a,H^{s + 
\frac52,2}_{-a_0,a})$, 
and by  \Cref{prop:mult.d}
the map $\e \mapsto [u\mapsto \e^{-1} d_\e(x) u]$ is analytic from  $B_{\e_0}(0) \to \cL(H^{s+\frac12,0}_{-a_0,a}, H^{s+\frac12,0}_a)$, 
hence \Cref{lem:prodintM}-$(i)$ implies that 
the composition
$(\al, \mu, \e) \mapsto (\al^2\e)^{-1} L_{\al,\mu} \circ d_\e$ belongs to 
$\cA(\R\times (-\tfrac23,\tfrac23),\e_0;H^{s+\frac12,0}_{-a_0,a},H^{s+\frac52,0}_{-a_0,a})$.
Moreover by Lemma \ref{lem:prodintM}-$(ii)$ the operator $(\uno - L_{\al,\mu}\circ d_\e)^{-1} \in \cA(\R\times (-\tfrac23,\tfrac23),\e_0;H^{s+\frac12,0}_{-a_0,a},H^{s+\frac12,0}_{-a_0,a})$ and by \Cref{lem:prodintM}-$(i),(ii)$ and \eqref{ezDA}. 
\be\label{quafa1}
[g\mapsto(\al^2\e)^{-1}\Theta^\sharp_g ]= \al^{-2}L_{\al,\mu} \circ \e^{-1}d_\e \circ (\uno - L_{\al,\mu}\circ d_\e)^{-1}e^{z|D|_{\al,\mu}} \in \cA(\R\times (-\tfrac23,\tfrac23),\e_0;H^{s}(\T),H^{s+\frac52,2}_{-a_0,a}) \, . 
\ee
Since  $\pa_z: H^{s,b}_{-a_0,a} \to H^{s-1,b-1}_{-a_0,a} $ and the trace operator $\Gamma :H^{s,1}_{-a_0,a} \to H^{s-\frac12}$ in \eqref{optraccia} do not depend on $\al,\mu$ and $\e$, and therefore trivially satisfy the axioms of Definition \ref{def:tM},
we deduce \eqref{scopo1D25}  by \eqref{quafa1} and \Cref{lem:prodintM}-$(i)$. 
\\[1mm]
\noindent{\sc Proof that}  
\be\label{scopo2D25}
g \mapsto \pa_z\Theta^\sharp_{g}\vert_{z=0} \in \mathtt{F} \, . 
\ee
For any $ z \leq 0 $ the operator 
$ g\mapsto \pa_z 
\Theta^{\sharp}_{g}(\cdot, z)  $  in 
\eqref{formulathetasharppp1} is the composition of 
the Fourier multipliers $ e^{z|D|_{\al,\mu}}$, $ \pa_z $, 
and $ f( L_{\al,\mu} \circ d_\epsilon )$ 
where $ f(\zeta) := \frac{\zeta}{1-\zeta}$. For any 
$ z \leq 0 $ the operator 
$ L_{\al,\mu} \circ d_\epsilon $ belongs to 
$ \mathtt{F} $ by \eqref{deintF}
and since  $L_{\al,\mu}$  in \eqref{sol} is a Fourier multiplier (in $ x $). Thus 
$ f( L_{\al,\mu} \circ d_\epsilon ) \in \mathtt{F} $ by \Cref{TFjprop}-($iii$) and  \eqref{scopo2D25}
follows. 
\end{proof}

We finally prove properties of the operators $T_\lambda$, $\wt T_\lambda$ defined in \eqref{opT} used in Lemma \ref{solution_BVP}.
    \begin{lemma}\label{Test}
      Let $a,a_0>0$.     
\\[1mm]   
      If $\Re \lambda>-a$, then
    \be\label{item1d} \| T_\lambda p\|_{L^{2,a}} \leq \frac{\| p\|_{L^{2,a}}}{\Re\lambda + a} \, , \qquad |T_\lambda p(z)|\leq  \frac{e^{az} \| p\|_{L^{2,a}}}{\sqrt{2(\Re\lambda + a)}} \quad \forall z\leq 0\, .
    \ee
Moreover, 
\be\label{item2d}
     \| \wt T_\lambda p\|_{L^{2,a}} \leq \frac{\| p\|_{L^{2,a}} }{\Re\lambda - a}   \ \ \text{if }\Re \lambda > a    \,
      , \quad |\wt T_\lambda p(z)|\leq
     \begin{cases}
     \left(\frac{e^{2az} - e^{2\Re\lambda\,  z}}{2(\Re\lambda - a)}\right)^\frac12\| p\|_{L^{2,a}}  & \mbox{ if }  \ \Re\lambda \neq a  \\
     e^{az} |z|^{\frac12} \| p\|_{L^{2,a}} &  \mbox{ if } \Re\lambda  =  a \, . 
     \end{cases}
\ee
If $\Re\lambda>-a_0$  one has
\be\label{item3d}
            \| \wt T_\lambda p\|_{L^{2,-a_0}} \leq \frac{\|p\|_{L^{2,a}}}{\Re\lambda+a_0}  \, . 
\ee
        The following maps are analytic:
        \begin{equation}\label{mapsT}
            \begin{aligned}
                    \{\Re \lambda>-a \}\ni \lambda&\to T_{\lambda}\in \cL(L^{2,a},L^{2,a})\, , \qquad  \{\Re \lambda>-a \}\ni \lambda\to [T_{\lambda}\cdot](0)\in \cL(L^{2,a},\C) \, , \\
                    \{\Re \lambda>a \}\ni \lambda&\to \wt T_{\lambda}\in \cL(L^{2,a},L^{2,a})\, , \qquad \{\Re \lambda>-a_0 \}\ni \lambda\to \wt T_{\lambda}\in \cL(L^{2,a},L^{2,-a_0})\, .
                \end{aligned}
            \end{equation}
            
    \end{lemma}
    
    \begin{proof} 
     Using the definition of $T_\lambda$ in \eqref{opT} we estimate the $\|\cdot\|_{L^{2,a}}$ norm in \eqref{L2ascalar} by
    \begin{equation*}
        \| T_\lambda p\|_{L^{2,a}}^2 = \int_{-\infty}^0 e^{-2az}\Big|\int_{-\infty}^z e^{\lambda(t-z)}p(t)\de t\Big|^2\de z \leq  \frac{1}{(\Re\lambda+a)^2}\int_{-\infty}^0 \Big(\int_{-\infty}^z e^{-at}|p(t)|(\Re\lambda+a)e^{(\Re\lambda+a) (t-z)}\de t\Big)^2\de z \, .      
    \end{equation*}
   If $\Re\lambda>-a$ the  measure $(\Re\lambda+a)e^{(\Re\lambda+a) (t-z)}\de t$ on $(-\infty,z)$ is normalized and by using the Cauchy-Schwartz inequality with this measure we obtain
    $$
    \| T_\lambda p\|_{L^{2,a}}^2\leq  \frac{1}{(\Re\lambda+a)^2}\int_{-\infty}^0 \int_{-\infty}^z e^{-2at}|p(t)|^2(\Re\lambda+a)e^{(\Re\lambda+a) (t-z)}\de t\de z \, .
    $$
    Thus, exchanging the order of integration 
    \begin{equation*}
        \begin{split}
            \| T_\lambda p\|_{L^{2,a}}^2&\leq\frac{1}{(\Re\lambda+a)^2} \int_{-\infty}^0 e^{-2at}|p(t)|^2\int_t^0 (\Re\lambda+a)e^{(\Re\lambda+a) (t-z)}\de z\de t \\
            &\leq \frac{1}{(\Re\lambda+a)^2} \int_{-\infty}^0 e^{-2at}|p(t)|^2 \de t= \frac{\|p\|_{L^{2,a}}^2}{(\Re\lambda+a)^2}   
        \end{split}
    \end{equation*}
 proving the first inequality in \eqref{item1d}.
    The second inequality in \eqref{item1d} follows by \eqref{opT}, \eqref{L2ascalar} and Cauchy-Schwartz inequality. The bounds \eqref{item2d} are  proven similarly to \cite[Lemma C.1, formula (C.6)]{BMV2}. Though special carefulness is required to prove the right of \eqref{item2d} in the case $\lambda = a$.

We now prove \eqref{item3d}, first in case $  \Re\lambda \neq a$.  
     Recalling the definition of $\wt T_\lambda$ in \eqref{opT}, we have
    \begin{align}\label{provvvv}
             \| \wt T_\lambda p\|_{L^{2,-a_0}}^2 &= \int_{-\infty}^0 e^{2a_0z}\left|\int_{z}^ 0 e^{\lambda(z-t)}p(t)\de t\right|^2\de z = \int_{-\infty}^0 e^{2(a_0+a)z}\left|\int_{z}^ 0 e^{-at}p(t)  e^{(\lambda-a)(z-t)}\de t\right|^2\de z  \notag \\
             &= \int_{-\infty}^0 e^{2(a_0+a)z}\Big|\tfrac{1-e^{(\Re\lambda-a)z}}{\Re\lambda-a}\Big|^2\left|\int_{z}^ 0 e^{-at}p(t)  \left|\tfrac{\Re\lambda-a}{1-e^{(\Re\lambda-a)z}} \right|e^{(\lambda-a)(z-t)}\de t\right|^2\de z \, . 
        \end{align}
    Note that the measure 
    $$
    \de\mu(t) := \left|\frac{\Re\lambda-a}{1-e^{(\Re\lambda-a)z}}\right|e^{(\Re\lambda-a)(z-t)}\de t
    = \frac{\Re\lambda-a}{1-e^{(\Re\lambda-a)z}} e^{(\Re\lambda-a)(z-t)}\de t
    $$ 
    is normalized on $(z,0)$.
    Using the Cauchy-Schwartz inequality with respect to the measure $ d \mu (t) $, and exchanging the order of integration, yield, if
      $\Re\lambda>-a_0 $,  
    \begin{align}
        \eqref{provvvv} &\leq  \int_{-\infty}^0e^{-2at}|p(t)|^2\int_{-\infty}^t  e^{2(a_0+a)z} \tfrac{1-e^{(\Re\lambda-a)z}}{\Re\lambda-a} e^{(\Re\lambda-a)(z-t)}\de z \de t \notag 
        \\ &= \int_{-\infty}^0e^{-2at}|p(t)|^2 \tfrac{1}{\Re\lambda-a}\Big[\tfrac{e^{2(a_0+a)t}}{\Re\lambda + 2a_0 + a}-\tfrac{e^{(\Re\lambda +2a_0 +a)t}}{2(\Re\lambda + a_0)} \Big]\de t \, . \label{passin1}
    \end{align}
Since
\begin{align}\label{ilmass}
    \max_{t \in (-\infty,0]} \tfrac{1}{\Re\lambda-a}\Big[\tfrac{e^{2(a_0+a)t}}{\Re\lambda + 2a_0 + a}-\tfrac{e^{(\Re\lambda +2a_0 +a)t}}{2(\Re\lambda + a_0)} \Big] 
&\leq \max \left\{ \frac{1}{2(\Re\lambda + a_0)(\Re\lambda+2a_0+a)} , \frac{1}{(\Re\lambda+2a_0+a)^2} \right\} \notag \\
&\leq \frac{1}{(\Re\lambda+a_0)^2} \, ,
\end{align}
where, on the left of the first line, the first value is the  maximum, attained at $t = 0$, if $\Re\lambda < a$, and the second is an upper bound of the maximum if $\Re\lambda>a$. In conclusion we deduce \eqref{item3d} by 
\eqref{provvvv}, \eqref{passin1},  and \eqref{ilmass}, if $  \Re\lambda \neq a$. 
If  $\Re\lambda = a$ formula \eqref{item3d} is proved   by
$$
\|\wt T_\lambda p\|^2_{L^{2,-a_0}} \stackrel{\eqref{provvvv}}{=} \int_{-\infty}^0 e^{2(a+a_0)z} \left|\int_{z}^0e^{-\lambda t}p(t)\de t\right|^2\de z \leq  \int_{-\infty}^0 e^{2(a+a_0)z}|z|\int_{-\infty}^0 e^{-2at}|p(t)|^2\de t\de z  = \frac{\|p\|_{L^{2,a}}^2}{4(a+a_0)^2}  .
$$
    To verify the analyticity of the maps in \eqref{mapsT} it is only left to check the weak analyticity, since the uniform boundedness on compact sets follow from the first inequalities of \eqref{item1d}-\eqref{item2d}. To do that, by the density of the step functions in $L^2$, it suffices to test the weak analyticity on indicator functions of intervals, see for instance the proof of \Cref{lem:und.varphi}.
    \end{proof}

\subsection{Proof of  Proposition \ref{DNonR3} and \Cref{DNProp1}}\label{sec:proofs.FDN}

\paragraph{Proof of \Cref{DNProp1}.} 
For any $g\in H^s(\T)$, $s \in \R$, 
$(\al,\mu)\in\R\times (-\tfrac23,\tfrac23)$ and  
$|\epsilon | \leq \epsilon_0(s) $ determined  in \Cref{lem:Thetasharp},  there is a unique solution of the elliptic problem \eqref{ellprobtransf} as in  \eqref{decompsolution}
with $ \Theta^\flat_g (x,z) =  
( e^{z|D|_{\al,\mu}}g )(x)$ in 
\eqref{propg}
and $\Theta_g^\sharp$ in \Cref{lem:Thetasharp}. Then 
the fiber-Dirichlet Neumann operator in \eqref{def:Galmu} is 
\be\label{ilGdecom}
 \cG(\al, \mu, \e)[g]= \pa_z(\Theta_g^\flat + \Theta_g^\sharp)\vert_{z=0}  =  |D|_{\al,\mu}g + \cG^\sharp (\al,\mu,\e)[g] \, ,
 \quad \cG^\sharp (\al,\mu,\e)[g] := \pa_z \Theta_g^\sharp \vert_{z=0} \, , 
\ee
proving  
\eqref{est:DNpert} for any $|\mu| < 2/3 $, 
in view of 
\eqref{est:thetasharpnonzero}.
Using the purely algebraic 
 covariance property \eqref{Gmu+k} proved below,
 formula \eqref{decoGsha} and  the bound  \eqref{est:DNpert}  hold on each vertical strip $ \R\times (k-\tfrac23,k+\tfrac23)$, $k \in \Z$, 
 with a constant depending on $ |k| $, thus on $\R^2$. 
Finally 
 $(\al, \mu,\e) \mapsto (\alpha^2 \e)^{-1} \cG^\sharp({\al,\mu,\e}) $ belongs to $ \cA(\R^2,\e_0;H^s,H^{s+1})$ and to  $\mathtt{F}$ by 
\eqref{pazThetainA} and  the covariance property. This proves
\eqref{regGs}. 

\paragraph{Proof of Proposition \ref{DNonR3}.}
The operator
$\cG(\al,\mu,\e) $ satisfies  \eqref{dominiGL} 
by \Cref{DNProp1}.

\noindent
{\sc Self-adjointness and Hamiltonianity.} 
We first show  that 
$ (\cG(\al,\mu,\e) g,f) = 
 (g, \cG(\al,\mu,\e) f) $ 
for any $f,g\in\cC^\infty(\T;\C)$.
 By the divergence Theorem 
 (the boundary  contribution 
 at infinity vanishes by \Cref{decaysolv} for any $ (\al,\mu) \neq (0,0) $) 
 we have 
 \begin{align}
  &    \int_\T  \cG(\al,\mu,\e)[g](x) 
     \overline{f(x)} \de x \stackrel{\eqref{def:Galmu}} = \int_\T \overline{f(x)} \pa_z \Theta_g (x,z)\vert_{z=0} \de x  = \int_{\T\times \R_-} \textup{div}\big(\overline{\Theta_f(x,z)}\grad_{x,z}\Theta_g(x,z)\big)\de x \de z \notag \\
     &= \int_{\T\times \R_-} \big(\overline{\grad_{x,z}\Theta_f(x,z)}\cdot \grad_{x,z}\Theta_g(x,z) + \overline{\Theta_f(x,z)}\Delta_{x,z}\Theta_g(x,z)\big)\de x \de z \notag \\
     &\stackrel{\eqref{ellprobtransf}}{=} \int_{\T\times \R_-} \big(\overline{\grad_{x,z}\Theta_f(x,z)}\cdot \grad_{x,z}\Theta_g(x,z)  
  + \overline{\Theta_f(x,z)} (-2\im\mu\pa_x +\al^2(1+d_\e(x,z)) + \mu^2)\Theta_g(x,z)\big)\de x \de z \label{simmetr}
 \end{align}
which is equal to 
 $ (g, \cG(\al,\mu,\e) f) $.
 Thus $\cG(\al,\mu,\e) $ is self-adjoint, being the sum of the selfadjoint operator
 $|D|_{\al,\mu}$ with domain $H^1(\T)$ and the operator 
 $\cG^\sharp(\al,\mu,\e) $ in \eqref{ilGdecom}, which is  bounded and symmetric on $L^2(\T)$ by \eqref{est:thetasharpnonzero}
 and \eqref{simmetr}.
 Thus, $\cB(\al,\mu,\e)$ in \eqref{operator Bmualep}  is self-adjoint on $L^2(\T;\C^2)$ and $\cL(\al,\mu,\e)$ is Hamiltonian.
 \\[1mm]
\noindent
{\sc Reversibility.}
Since $ d_\epsilon (x,z) $
is real and even in $ x $, 
if  $ \Theta_g (x, z) $ solves  the elliptic problem \eqref{ellprobtransf},  then $\overline{ \Theta_g(-x,z)} $ solves the same problem with Dirichlet  datum $\bar g (-x)$. Therefore by uniqueness $\Theta_{\bar g^\vee} (x,z) = \overline{ \Theta_g(-x,z)}$ and 
$$
\cG(\al,\mu,\e)[\bar g^\vee] (x) \stackrel{\eqref{def:Galmu}} = \pa_z  \Theta_{\bar g^\vee }  (x,z) \vert_{z=0}    = \overline{ \pa_z \Theta_g (-x,z)}\vert_{z=0} = \overline{\cG(\al,\mu,\e)[g] }(-x) \, ,  
$$
proving \eqref{DN-rev2}. 
\\[1mm]
\noindent
{\sc Gauge covariance.}
If  $\Theta_g(x,z)$ solves  \eqref{ellprobtransf}, then $e^{- \im k x} \Theta_g(x,z)$ solves the same problem with $\mu \leadsto \mu+k$ and datum $g(x)\leadsto e^{- \im k x} g(x)$.
Therefore, in view of \eqref{def:Galmu},
$$
\cG(\al,\mu+k,\e)[e^{- \im k x} g(x)] = \pa_z [e^{- \im k x} \Theta_g(x,z)]\vert_{z=0} = e^{- \im k x} \pa_z\Theta_g(x,z)\vert_{z=0} = e^{- \im k x} \cG(\al,\mu,\e)[g] \, , 
$$
proving \eqref{Gmu+k}.  
\\[1mm]
\noindent
{\sc Unperturbed operators.}
It follows  because 
  $\Theta^\sharp_g\vert_{\e = 0} = 0$ for any $g$, see \Cref{lem:Thetasharp}.
  \\[1mm]
\noindent
{\sc Symmetry.} If $\Theta_g(\al,\mu,\e;x,z)$ is a solution of 
\eqref{ellprobtransf}, then $\overline{\Theta_g(\al,\mu,\e;x,z)}$ solves the same problem with $\mu\leadsto -\mu$ and $g\leadsto \overline{g}$. Thus, by  uniqueness, $\Theta_{\overline{g}}(\al,-\mu,\e;x,z) = \overline{\Theta_g(\al,\mu,\e)}$ 
and the operator \eqref{def:Galmu}  satisfies
$ \overline{\cG (\alpha,\mu,\e) g } = \cG (\alpha, -\mu , \epsilon ) [\bar g ] $.
Thus $\cB(\al,\mu,\e)$ in \eqref{operator Bmualep} and 
$ \cL(\al,\mu,\e)=\cJ\cB(\al,\mu,\e)$  satisfy  the same property. The operator $\cG(\al,\mu,\e)$  is even in $\al$ since  the elliptic problem \eqref{ellprobtransf} only depends on $\al^2$. The other entries of
$\cB(\al,\mu,\e)$ and $\cL(\al,\mu,\e)$ in  \eqref{cLame1} and \eqref{operator Bmualep}  are  independent of  $\al$. This proves  
\eqref{cGcLmenomu}.

\section{Expansion of the 
basis $ \cF $ }\label{ProofExpansion}

We now prove Lemma \ref{expansion1}. 
We perform a  Taylor expansion of the operators in  \eqref{svisLep}, 
\begin{equation}
    \label{defijksL}
    \sL^{[i,j]}(\al^2,\mu^2,\e) = \sL^{[i,j,0]} + \e \sL^{[i,j,1]} + \cO(\rho^2,\e^2)\, , \quad i, j \in \{0,1\} \, , 
\end{equation}
where
$ \sL^{[i,j,0]} := \sL^{[i,j]}(0,0,0) $ and $ \sL^{[i,j,1]} := \pa_\e \sL^{[i,j]}(0,0,0) $.

\begin{lemma}
The operator $\sL(\al,\mu,\e)$ in
\eqref{sLresca}  expands as
\begin{equation}\label{expcL1}
    \sL(\al,\mu,\e) = \sL(0,0,0)  + \rho \sL^{[1,0,0]} + \im\mu \sL^{[0,1,0]} + \e \sL^{[0,0,1]}  
   +\rho\e \sL^{[1,0,1]} + \im\mu\e {\sL}^{[0,1,1]} + \cO(\rho^2,\e^2) \, , 
\end{equation}
with
\begin{equation}
\sL^{[1,0,0]} = \begin{bmatrix} 0 & \Pi_0 \\ 0 & 0 \end{bmatrix} \, ,
\quad \sL^{[0,1,0]} = \begin{bmatrix}0  &  -\im \, \sgn (D)\\ 0 & 0\end{bmatrix} \, ,\quad 
\sL^{[0,0,1]} = \begin{bmatrix} \pa_x \circ p_1(x) & 0 \\ -a_1(x) & p_1(x)\circ \pa_x \end{bmatrix} \, , \label{cLfirstorder}
\end{equation}
where $a_1(x)=p_1(x)=-2\cos(x)$.
The second order terms are
\begin{equation}\label{cLmisto}
    \sL^{[0,1,1]}=
    \begin{bmatrix}
         p_1(x) & 0 \\
        0 &  p_1(x)
    \end{bmatrix}\, , \quad \sL^{[1,0,1]}= 0 \, . 
\end{equation}
\end{lemma}

\begin{proof}
    Use \eqref{svisLep}, \eqref{defijksL}, \eqref{quattropuntosei} and \eqref{apexp}.
    \end{proof}

    Next we express the jets of the operators  $P_{\al,\mu,\e}$ and $U_{\al,\mu,\e}$
    in terms of $\sL(\al,\mu,\e)$.
\begin{lemma}\label{lem:U.expansion}
{\bf (Taylor expansion of projectors)}
 The projector $P_{\al,\mu,\e}$ in \eqref{Pproj} expands as 
\begin{equation}\label{expP1}
\begin{aligned}
    P_{\al,\mu,\e} = P_{0,0,0} &+\rho P^{[1,0,0]}+ \im\mu P^{[0,1,0]} + \e P^{[0,0,1]}   
    + \im \mu \e P^{[0,1,1]} 
    +\rho \e  P^{[1,0,1]}+ \cO(\rho^2,\e^2)
\end{aligned}
\end{equation}
 where the real-to-real operators  $P^{[i,j,k]} $ are 
\begin{align}\label{Pder}
 P^{[i,j,k]} := \frac{1}{2\pi\im} \oint_\Gamma R^{[i,j,k]}(\lambda) \de\lambda 
\end{align}
and 
\begin{align}\label{R000}
  &   R^{[0,0,0]}(\lambda) = (\lambda-\sL(0,0,0))^{-1} =: R_0(\lambda)\, , \quad R^{[0,0,1]}(\lambda) = R_0(\lambda) \sL^{[0,0,1]}R_0(\lambda) \, , \\
& \label{Rijk}
R^{[0,1,1]} (\lambda)=  R_0(\lambda) (\sL^{[0,1,0]} R_0(\lambda)  \sL^{[0,0,1]} +
\sL^{[0,0,1]} R_0(\lambda) \sL^{[0,1,0]} 
+\sL^{[0,1,1]} )R_0(\lambda) 
\end{align}
and any other $ R^{[i,j,k]} $ is obtained after permutation of the apex indices.

The operator $U_{\al,\mu,\e}P_{0,0,0}$ has the expansion
\begin{equation}\label{expU}
    U_{\al,\mu,\e}P_{0,0,0}= P_{0,0,0} + \e U^{[0,0,1]} + \im\mu U^{[0,1,0]} + \rho U^{[1,0,0]} 
    + \im\mu\e U^{[0,1,1]} + \rho\e U^{[1,0,1]} + \cO(\rho^2,\e^2)\, ,
\end{equation}
where the jets in \eqref{expU} are 
the real-to-real operators 
 \begin{align}
 &  U_{0,0,0}P_{0,0,0}=P_{0,0,0} \, , \quad U^{[0,0,1]}P_{0,0,0}=P^{[0,0,1]}P_{0,0,0} \, ,  \label{Ufirstorder}\\
& U^{[0,1,1]}P_{0,0,0} =
\big(P^{[0,1,1]}- \tfrac12 P_{0,0,0}P^{[0,1,1]} \big)P_{0,0,0} \, .  \label{Umix} 
 \end{align}
Identities \eqref{Ufirstorder} and \eqref{Umix} hold after any permutation of the apex indices.
\end{lemma}

\begin{proof}
We write 
\be\label{Pamu2}
P_{\al,\mu,\e} = \frac{1}{2\pi\im}\oint_\Gamma \left( \uno - R_0(\lambda) \cR_{\al,\mu,\e}\right)^{-1}R_0(\lambda) \de\lambda \, , \qquad \cR_{\al,\mu,\e} := \sL(\al,\mu,\e) - \sL(0,0,0)\, .
\ee
Inserting the expansion \eqref{expcL1} in \eqref{Pamu2}  and Neumann expanding  $\left( \uno - R_0(\lambda) \cR_{\al,\mu,\e}\right)^{-1}$, we get 
$$
\begin{aligned}
    P_{\al,\mu,\e} &= \sum_{n\geq 0} \frac1{2\pi\im}\oint_\Gamma \bigg[ R_0(\lambda)\left( \rho \sL^{[1,0,0]} + \im\mu \sL^{[0,1,0]} + \e \sL^{[0,0,1]} + \im\mu\e \sL^{[0,1,1]} + \rho\e \sL^{[1,0,1]} + \cO(\rho^2,\e^2) \right) \bigg]^n R_0(\lambda)  \de\lambda  \\&= \frac{1}{2\pi\im}\oint_\Gamma \bigg[ R_0(\lambda) + R_0(\lambda)\left( \rho \sL^{[1,0,0]} + \im\mu \sL^{[0,1,0]} + \e \sL^{[0,0,1]} + \im\mu\e \sL^{[0,1,1]} + \rho\e \sL^{[1,0,1]} \right) R_0(\lambda) \\
    &\qquad \qquad+ \im\mu \e R_0(\lambda) \left(  \sL^{[0,1,0]} R_0(\lambda)\sL^{[0,0,1]} + \sL^{[0,0,1]} R_0(\lambda)\sL^{[0,1,0]}\right)R_0(\lambda) \\
    &\qquad \qquad+ \rho \e R_0(\lambda)\left( \sL^{[1,0,0]} R_0(\lambda)\sL^{[0,0,1]} + \sL^{[0,0,1]} R_0(\lambda)\sL^{[1,0,0]}\right)R_0(\lambda) \,
    \bigg] \de \lambda  +  \cO(\rho^2,\e^2)
\end{aligned}
$$
which is \eqref{expP1} with $P^{[i,j,k]}$ as in \eqref{Pder}, c.f. \eqref{R000}, \eqref{Rijk}. 
 By \eqref{OperatorU} one has the Taylor expansion  in $\cL(H^1)$
$$
   U_{\al,\mu,\e}P_{0,0,0}  = P_{\al,\mu,\e}P_{0,0,0} + \frac{1}{2}(P_{\al,\mu,\e}-P_{0,0,0})^2P_{\al,\mu,\e}P_{0,0,0} +\cO(P_{\al,\mu,\e}-P_{0,0,0})^4   \, ,
  $$
  where  $\cO(P_{\al,\mu,\e}-P_{0,0,0})^4 = \cO(\text{ord. }4) := \cO(\rho^4,\rho^3\e,\dots,\e^4) \in \cL(H^1)$. Thus the first order jets of $U_{\al,\mu,\e}$ are given, using \eqref{expP1}, by \eqref{Ufirstorder}. Using again \eqref{expP1} the second order terms are 
  $$
  \begin{aligned}
      (P_{\al,\mu,\e}-P_{0,0,0})^2P_{\al,\mu,\e}P_{0,0,0} = (\rho P^{[1,0,0]} + \im\mu P^{[0,1,0]} + \e P^{[0,0,1]})^2 P_{0,0,0} + \cO(\textup{ord. }3)\\
      =\rho \e (P^{[1,0,0]}P^{[0,0,1]} + P^{[0,0,1]}P^{[1,0,0]}) + \im\mu\e (P^{[0,1,0]}P^{[0,0,1]} + P^{[0,0,1]}P^{[0,1,0]}) +  \cO(\rho^2,\e^2)\, ,
  \end{aligned}
  $$
and using  the identities
$P^{[0,1,0]} P^{[0,0,1]} P_{0,0,0} + P^{[0,0,1]} P^{[0,1,0]} P_{0,0,0} = - P_{0,0,0} P^{[0,1,1]} P_{0,0,0}$,
which holds also for any permutation of the indices, we derive \eqref{Umix}. In particular the jets in \eqref{Ufirstorder} and \eqref{Umix} are real-to-real operators. The previous identity is obtained inserting \eqref{expP1} in $P_{\al,\mu,\e}^2 = P_{\al,\mu,\e}$ and identifying the results of the left and right hand sides term by term, also multiplying by $P_{0,0,0}$ to the right.
\end{proof}

The  vectors 
$ f_k^\sigma(\al,\mu,\e) = 
U_{\al,\mu,\e} f_k^\sigma $ satisfy  
\eqref{decF} because of \eqref{Ureg} and \eqref{symmmmUP}. 
In view of \eqref{expU}, \eqref{Ufirstorder}, \eqref{Umix} 
the vectors 
$ f_k^\sigma(\al,\mu,\e) $ 
have the Taylor expansion
\begin{equation}\label{ordinibase}
\begin{aligned}
    f_k^\sigma(\al,\mu,\e) = f_k^\sigma &+ \e P^{[0,0,1]} f_k^\sigma+ \im\mu P^{[0,1,0]} f_k^\sigma +\rho P^{[1,0,0]} f_k^\sigma
    + \im\mu\e  U^{[0,1,1]}f_k^\sigma +  \rho\e  U^{[1,0,1]} f_k^\sigma  + \cO(\rho^2,\e^2) \, .
\end{aligned}
\end{equation}
The jets of the expansion \eqref{ordinibase}
are computed in \Cref{lem:jetsP,jetsordine2}. 
Preliminary, in view of 
 \eqref{Pder}, we need to know the action of  $R_0(\lambda)$ on the vectors 
\begin{equation}
\label{fksigma}
f_k^+:=\vet{\cos(kx)}{\sin(kx)},
\quad f_k^- :=\vet{-\sin(kx)}{\cos(kx)},
\quad f_{-k}^+ :=\vet{\cos(kx)}{-\sin(kx)},
\quad
f_{-k}^- :=\vet{\sin(kx)}{\cos(kx)} , \quad k \in \N \, .  
\end{equation}

\begin{lemma}\label{lem:VUW}
The space $ H^1$ decomposes as 
$
H^1 =  \cV_{0,0,0} \oplus \cU \oplus \cW_{H^1} $, with $\cW_{H^1}:= \overline{\bigoplus\limits_{k\geq 2} \cW_k}^{H^1}\!\!\!\!\!\!\!
$, where the subspaces $\cV_{0,0,0} $, $  \cU $, $ \cW_k $ defined below, are 
invariant  under   $\sL(0,0,0) $ and  the following properties hold:
\\[1mm]
($i$) $ \cV_{0,0,0} = \text{span} \{ f^+_1, f^-_1, f^+_0, f^-_0\}$  is the generalized kernel of $\sL(0,0,0)$. For any $ \lambda \neq 0 $ the operator 
$ \lambda-\sL(0,0,0) :  \cV_{0,0,0} \to \cV_{0,0,0} $ is invertible and  
 \begin{align} 
 \label{primainversione2}R_0(\lambda)f_1^+ = \frac1\lambda f_1^+ \, ,
\quad 
R_0(\lambda)f_1^- = \frac1\lambda f_1^-,
\quad  R_0(\lambda)f_0^- = \frac1\lambda f_0^- \, ,   \quad
R_0(\lambda)f_0^+ = \frac1\lambda f_0^+ - \frac{1}{\lambda^2} f_0^- \, .
\end{align} 
($ii$) $\cU := \text{span}\left\{ f_{-1}^+, f_{-1}^-  \right\}$.   For any 
$ \lambda \neq \pm 2 \im $ the operator 
$ \lambda-\sL(0,0,0) :  \cU \to \cU $ is invertible and
\begin{equation}
\label{primainversione3}
 R_0(\lambda) f_{-1}^+ = \frac{1}{\lambda^2+4}\left(\lambda f_{-1}^+ - 2 f_{-1}^-\right), \quad R_0(\lambda) f_{-1}^- = \frac{1}{\lambda^2+4}\left(2 f_{-1}^+ + \lambda f_{-1}^-\right) \, .
\end{equation}
($iii$) 
Each
subspace $\cW_k:= \text{span}\left\{f_k^+, \ f_k^-, f_{-k}^+, \ f_{-k}^- \right\}$ is  invariant under $ \sL(0,0,0) $.  Let $\cW_{L^2}:=\overline{\bigoplus\limits_{k\geq 2} \cW_k}^{L^2}\!\!\!\!\!\!$. For any
$|\lambda| < \frac12$, the operator 
$ \lambda - \sL(0,0,0) :  \cW_{H^1} \to \cW_{L^2} $ is invertible and, 
 for any $f \in \cW_{L^2} $, 
\begin{equation}
\label{primainversione4}
 R_0(\lambda) f  = - (\pa_x^2 + |D|)^{-1} \begin{bmatrix} \partial_x & - |D| \\ 1 & \partial_x\end{bmatrix} f + \lambda \varphi_f(\lambda, x) \, ,
\end{equation}
for some analytic  function  $\lambda \mapsto \varphi_f(\lambda, \cdot) \in H^1$.
\end{lemma}

\begin{proof}
See \cite[Lemma A.2]{BMV1}
\end{proof}
We shall also use the following formulas, obtained  by 
\eqref{cLfirstorder} and \eqref{base3e}:
\begin{equation}\label{derivoeps}
\begin{aligned}
&\sL^{[0,0,1]}f_1^+ = 2\vet{\sin(2x)}{0} \, , \quad  
\sL^{[0,0,1]}f_1^- = 2\vet{\cos(2x)}{0} \, ,  \quad
\sL^{[0,0,1]}f_0^+ = 2\vet{\sin(x)}{\cos(x)} \, , \quad
\sL^{[0,0,1]}f_0^- = 0 \, , \\
& \sL^{[1,0,0]}f_k^\sigma= 0\quad \forall \, (k,\sigma)\neq (0,-)\, , \qquad
 \sL^{[1,0,0]}f_0^- =  f_0^+ \, .
\end{aligned}
\end{equation}

\begin{lemma}\label{lem:jetsP}
{\bf (First order jets)}
The first order jets   $P^{[0,0,1]} f_k^\sigma$, $P^{[0,1,0]}f_k^\sigma$, $P^{[1,0,0]}f_k^\sigma$ of the perturbed basis $ \{ f_k^\sigma (\al,\mu,\e)\, ; \ k = 0,1 $, $ \sigma = \pm \} $, are
\begin{equation}\label{tuttederivate}
\begin{aligned}
 &P^{[0,0,1]}f^+_1 =\vet{2\cos(2x)}{\sin(2x)} \, ,\ \ \ 
  P^{[0,0,1]}f^-_1 =\vet{-2\sin(2x)}{\cos(2x)} \, , \ \ \ 
   P^{[0,0,1]}f^+_0 = f^+_{-1} \, , \ \ \  
   P^{[0,0,1]}f^-_0 =0 \, ,   \\
    &P^{[0,1,0]} f_1^+ = \frac{1}{4} f_{-1}^-\, , \quad P^{[0,1,0]} f_1^- = \frac{1}{4} f_{-1}^+\, , \quad  P^{[0,1,0]} f_0^\pm = 0\, ,     \quad P^{[1,0,0]} f_k^\sigma = 0 \quad \forall k=0,1\, , \sigma = \pm\, .
\end{aligned}
\end{equation}
\end{lemma}
\begin{proof}
The first line of \eqref{tuttederivate} is computed in \cite[Lemma A.3]{BMV1}.
Now, by \eqref{Pder}, \eqref{R000} and \eqref{cLfirstorder}, 
$$
\begin{aligned}
    P^{[0,1,0]}f_1^+ &= \frac{1}{2\pi\im}\oint_\Gamma R_0(\lambda) \begin{bmatrix}
        0 & -\im \sgn(D) \\
        0 & 0
    \end{bmatrix} R_0(\lambda) f_1^+ \de\lambda = \frac{1}{2\pi\im}\oint_\Gamma  \frac 1 \lambda R_0(\lambda) \begin{bmatrix}
        0 & -\im \sgn(D) \\
        0 & 0
    \end{bmatrix}  f_1^+ \de\lambda \\
    &=-\frac{1}{2\pi\im}\oint_\Gamma  \frac{1}{2\lambda} R_0(\lambda)  (f_1^+ +f_{-1}^+)\de\lambda \stackrel{\eqref{primainversione2},\eqref{primainversione3}}{=} \frac{1}{4}f_{-1}^- \, . 
\end{aligned}
$$
Similarly one computes $P^{[0,1,0]}f_1^- = \frac{1}{4}f_{-1}^+$, $P^{[0,1,0]}f_0^\pm = 0$. 
We now compute $P^{[1,0,0]} f^+_1$. 
First note that, by \eqref{cLfirstorder}, every zero-average function is in the kernel of $\sL^{[1,0,0]}$, and that, by Lemma \ref{lem:VUW}, $(\sL(0,0,0)-\lambda)^{-1}$ preserves zero average functions. Then, by \eqref{Pder},\eqref{R000} and  \eqref{fksigma}, for $k=1$ we have that $P^{[1,0,0]} f^\sigma_1 = 0$. It is then sufficient to test $P^{[1,0,0]}$ of $f_0^\pm$. By \eqref{primainversione2} and \eqref{Pder} we have
$$ P^{[1,0,0]}f_0^+   =\ \frac{1}{2\pi\im} \oint_\Gamma R_0(\lambda)\sL^{[1,0,0]}\left(\frac{1}{\lambda}f_0^+ - \frac{1}{\lambda^2}f_0^-\right)  \de\lambda \stackrel{\eqref{derivoeps}}{=} -\frac{1}{2\pi\im} \oint_\Gamma \frac{1}{\lambda^2}R_0(\lambda) f_0^+ \stackrel{\eqref{primainversione2}}{=} 0\, ,$$
using again the residue theorem. The computation of $P^{[1,0,0]}f^-_0$is analogous.
\end{proof}


We then consider  the second order terms in the expansion \eqref{ordinibase}.

\begin{lemma}\label{jetsordine2}
{\bf (Second order jets)}
The second order jets  
$ U^{[0,1,1]}f_k^\sigma $ and 
$  U^{[1,0,1]}f_k^\sigma $ 
of the perturbed basis $ \{ f_k^\sigma (\al,\mu,\e)\, ; \ k = 0,1 $, $ \sigma = \pm \} $, 
are 
 \begin{align}
& U^{[0,1,1]}f_1^+ =  \vet{odd(x)}{even(x)}\, , \ \  U^{[0,1,1]}f_1^- = \vet{even(x)}{odd(x)} \, , \label{struttura2.1}\\ &  U^{[0,1,1]}f_0^+ =  \vet{odd(x)}{even_0(x)}\, ,\qquad \quad U^{[0,1,1]}f_0^- = \vet{even_0(x)}{odd(x)} \label{struttura2.2}\\
&  U^{[1,0,1]}f_1^\pm = 0,\qquad 
U^{[1,0,1]}f_0^+ = \tfrac{1}{4}f_{-1}^+,
\qquad 
U^{[1,0,1]} f_0^- = \tfrac{1}{2}f_{-1}^-   \, . \label{struttura2.3}
\end{align}
\end{lemma}

\begin{proof}
First we consider \eqref{struttura2.1}, \eqref{struttura2.2}. By \eqref{ordinibase} the functions $\im\mu\e U^{[0,1,1]}f_k^\sigma$ are purely imaginary jets of  $f_k^\sigma(\al,\mu,\e)$, which form the symplectic and reversible basis $\cF$ in \eqref{basisF}. Therefore the parity properties in \eqref{struttura2.1},\eqref{struttura2.2} follow by \eqref{lem:revbasis}. Moreover, the functions in \eqref{struttura2.2} have zero, average as we now show: by \eqref{Pder}, the operator $P^{[0,1,1]}$ has the form 
in \eqref{Pder}, cf. \eqref{Rijk}. 
The vectors $f_0^\pm$ are both supported on the zero-mode, cf. \eqref{fksigma}. Moreover $R_0(\lambda)$ acts invariantly both on functions with zero average and functions supported on the zero mode. By \eqref{cLfirstorder}-\eqref{cLmisto}, $\sL^{[0,1,0]}$ preserves zero-average functions and annihilates functions supported on the zero mode, while $\sL^{[0,0,1]}$ and $\sL^{[0,1,1]}$ map functions supported on the zero mode to functions supported on the modes $\pm 1$. Therefore 
$
P^{[0,1,1]}f_0^\pm
$ have both zero average, and so does $U^{[0,1,1]}f_0^\pm$ by \eqref{Umix}. 

Next consider the operator $P^{[1,0,1]}$ in \eqref{Pder}, given by 
\begin{equation*}
    \begin{aligned}
        P^{[1,0,1]} =  &\frac{1}{2\pi\im} \oint_\Gamma R_0(\lambda) \sL^{[1,0,0]} R_0(\lambda)  \sL^{[0,0,1]} R_0(\lambda) \de\lambda
        + \frac{1}{2\pi\im} \oint_\Gamma R_0(\lambda) \underbrace{\sL^{[1,0,1]}}_{=0 \mbox{ by } \eqref{cLmisto}} R_0(\lambda)  \de\lambda
        \\
        &+\frac{1}{2\pi\im} \oint_\Gamma R_0(\lambda) \sL^{[0,0,1]} R_0(\lambda) \sL^{[1,0,0]} R_0(\lambda) \de\lambda = \mathrm{I} + \mathrm{II} .
    \end{aligned}
\end{equation*}
We claim that $\mathrm{I}f = 0$ for every $f \in \cV_{0,0,0}$.
Indeed, by Lemma \ref{lem:VUW}-item $(i)$, $R_0(\lambda)$ maps $\cV_{0,0,0}$ into itself. By the expressions in \eqref{derivoeps}, $\sL^{[0,0,1]}$ maps $\cV_{0,0,0}$ into zero-average functions, a property which is preserved by the Fourier multiplier  $R_0(\lambda)$.
Finally, zero-average functions belong to $\ker \sL^{[1,0,0]}$ by the first equation in \eqref{cLfirstorder}, proving that  $\mathrm{I} = 0$.

Let us compute the action of $\mathrm{II}$: arguing as above, if $f$ is zero average, $R_0(\lambda)f$ is zero average as well, and thus $\sL^{[1,0,0]}R_0(\lambda) f= 0$. As a consequence, $P^{[1,0,1]} f_{1}^\pm = 0$. 
We now compute $\mathrm{II}\, f_0^-$: 
\begin{equation*}
    \begin{aligned}
 \mathrm{II}\, f_0^-  & \stackrel{\eqref{primainversione2}}{=} \frac{1}{2\pi \im}\oint_\Gamma \frac1\lambda R_0(\lambda)\sL^{[0,0,1]} R_0(\lambda) \sL^{[1,0,0]} f_0^- \de \lambda \\
&\stackrel{\eqref{cLfirstorder}}{=} \frac{1}{2\pi \im}\oint_\Gamma \frac1\lambda R_0(\lambda)\sL^{[0,0,1]} R_0(\lambda) f_0^+ \de \lambda \stackrel{\eqref{primainversione2}}{=}  \frac{1}{2\pi \im}\oint_\Gamma \frac1\lambda R_0(\lambda)\sL^{[0,0,1]} \left(\frac1\lambda f_0^+ -\frac{1}{\lambda^2 }f_0^- \right)   \\
&\stackrel{\eqref{derivoeps}}{=} \frac{1}{2\pi \im}\oint_\Gamma \frac2{\lambda^2} R_0(\lambda) f_{-1}^-\de \lambda  \stackrel{\eqref{primainversione3}}{=} -\frac{1}{2\pi \im}\oint_\Gamma \frac2{\lambda^2}  \frac{1}{\lambda^2+4}\left(-2 f_{-1}^+ - \lambda f_{-1}^-\right)\de\lambda = \frac{1}{2}f_{-1}^- \, .
    \end{aligned}
\end{equation*}
Finally we are left to compute $\mathrm{II}\,f_0^+$:
\begin{equation*}
    \begin{aligned}
\mathrm{II}\, f_0^+&  \stackrel{\eqref{primainversione2}}{=} \frac{1}{2\pi \im}\oint_\Gamma R_0(\lambda) \sL^{[0,0,1]} R_0(\lambda) \sL^{[1,0,0]} \left( \frac1\lambda f_0^+ - \frac{1}{\lambda^2}f_0^- \right) \de \lambda \\
& \stackrel{\eqref{cLfirstorder}}{=} -\frac{1}{2\pi \im}\oint_\Gamma \frac1{\lambda^2} R_0(\lambda) \sL^{[0,0,1]} R_0(\lambda) f_0^+ \de \lambda \stackrel{\eqref{primainversione2}}{=}  \frac{1}{2\pi \im}\oint_\Gamma \frac1{\lambda^2} R_0(\lambda) \sL^{[0,0,1]} \left(-\frac1\lambda f_0^+ +\frac{1}{\lambda^2 }f_0^- \right)   \\
& \stackrel{\eqref{derivoeps}}{=} -\frac{1}{2\pi \im}\oint_\Gamma \frac2{\lambda^3} R_0(\lambda) f_{-1}^-\de \lambda  \stackrel{\eqref{primainversione3}}{=} \frac{1}{2\pi \im}\oint_\Gamma \frac2{\lambda^3}  \frac{1}{\lambda^2+4}\left(-2 f_{-1}^+ - \lambda f_{-1}^-\right)\de\lambda = \frac{1}{4}f_{-1}^+ \, .
    \end{aligned}
\end{equation*}
In both the previous computations, the last step made use of the residue theorem. 

The proof of \eqref{struttura2.3} is concluded using \eqref{Umix} and  noticing that $P^{[1,0,1]} \cV_{0,0,0} \subset \cU$ in Lemma \ref{lem:VUW}$(ii)$, and thus  $P_{0,0,0}P^{[1,0,1]}\cV_{0,0,0} = \{0\}$.
\end{proof}

We now provide further information about the vectors $f_k^\sigma(0,0,\e)$ at  $\al,\mu=0$.

 \begin{lemma}\label{expe0}
Property \eqref{nonzeroaverage} holds.
 \end{lemma}
 
 \begin{proof}
By \cite{NS}, for any $\e\neq 0$  small, the operator $\cL(0,0,\e) $ possesses the
eigenvalue $ 0 $  of algebraic multiplicity $4$, 
and the generalized kernel of $\cL (0,0,\e)$ is spanned by  four vectors $U_1$, $\wt U_2$, $U_3$, $U_4$ satisfying
$$
    \cL(0,0,\e) U_1 = 0\, , \quad \cL (0,0,\e) \wt U_2 = 0\, , \quad \cL (0,0,\e) U_3 = \al_\e \wt U_2\, , \quad \cL (0,0,\e) U_4 = -U_1\, , \quad U_1 = \vet{0}{1} = f_0^- \,.
$$
Therefore, 
 $(\sL_{0,0,\e}-\lambda)^{-1}f_0^-= -\frac{1}{\lambda}f_0^- $ and then $P_{0,0,\e}f_0^- =f_0^-$.
 Then  \eqref{OperatorU} yields   
$f_0^-(0,0,\e) = U_{0,0,\e} f_0^- = f_0^-$ for any  $ \e$ sufficiently small, proving the last of \eqref{nonzeroaverage}. 
Let us prove the others.
In view of \eqref{lem:revbasis} and the fact that $U_{0,0,\e}$ is real by  \eqref{symmmmUP}, 
$ \footnotesize  f_k^+ (0,0,\e) =\vet{even(x)}{odd(x)} $, $ \footnotesize  f_k^- (0,0,\e) =\vet{odd(x)}{even(x)} $,
 for any  $ k=0,1 $. 
  By  \eqref{Usymp} the basis $\{f_k^\sigma(0,0,\e)\}$ is symplectic  and, since 
  $\cJ f_0^-(0,0,\e)  = \cJ f_0^- = \footnotesize \vet{1}{0}$, for any $\epsilon $, we get  
$$
 0 = \molt{\cJ f_0^-(0,0,\e)}{f_1^+(0,0,\e)} = \Big( \vet{1}{0}, f_1^+(0,0,\e) \Big) \, , \quad 
 1 = \big( \cJ f_0^-(0,0,\e), f_0^+(0,0,\e) \big) = \Big( \vet{1}{0}, f_0^+(0,0,\e) \Big) \, . 
$$
Thus the first component of both $f_1^+(0,0,\e)$ and $f_0^+(0,0,\e)- \footnotesize \vet{1}{0}$ has zero average,
 proving 
 \eqref{nonzeroaverage}.
 \end{proof}

We now provide further information about the the vectors $f_k^\sigma(\al,\mu,0)$ at $\e=0$. 
\begin{lemma}\label{expam0}
We have $f_0^\sigma (\al,\mu,0) =  f_0^\sigma $  for any $ \sigma = \pm  $ and $ (\alpha, \mu)$ small. 
\end{lemma}

\begin{proof}
The operator $\sL(\al,\mu,0)  
$  leaves invariant  the subspace $\mathcal{Z}:=\text{span}\{f_0^+,\,f_0^-\}$ since
$$ 
\sL(\al,\mu,0) f_0^+ = -f_0^-  \, , \quad 
\sL(\al,\mu,0) f_0^- = 
\sqrt{\alpha^2 
+ \mu^2} f_0^+ \, .
$$
The operator $\restr{\sL(\al,\mu,0)}{\mathcal{Z}}$ has the two eigenvalues $\pm\im (\al^2+\mu^2)^\frac14$, which, for small $\al,\mu$, lie inside the loop $\Gamma$ around $0$ in \eqref{Pproj}.  
Then, by \eqref{projdec},  we have 
 $\mathcal{Z} \subset {\cV}_{\al,\mu,0} = \text{Rg}(P_{\al,\mu,0}) $ and 
$ P_{\al,\mu,0} f_0^\pm = f_0^\pm $, 
$ f_0^\pm(\al,\mu,0) = U_{\al,\mu,0} f_0^\pm = f_0^\pm $ 
for any $ \al,\mu $  small.
\end{proof}

\begin{lemma}\label{identityvtransholdstrue}
Identities \eqref{vtrans} hold true.
\end{lemma}
\begin{proof}
We claim that the vectors $v_1^\pm (\al, \mu)$ defined in \eqref{unperturbed.eigv} satisfy 
\be\label{Uv00}
v_1^\pm (\al, \mu)  = U_{\al, \mu,0} v_1^\pm(0,0)  \quad \text{with}
\quad U_{\al,\mu,\e} \text{ in }\eqref{OperatorU}. 
\ee
  Indeed $v_1^+(0,0) \in \cV_{0,0,0}$ (recall \eqref{unperturbed.eigv}, \eqref{base3e}) and  thus  $U_{\al, \mu,0} v_1^+(0,0) $ belongs to the subspace $  \cV_{\al, \mu, 0}$, which is spanned  by 
    the four eigenvectors $v_1^\pm(\al, \mu), v_0^{\pm}(\al, \mu)$ of $ \sL (\alpha, \mu,0) $ in \eqref{unperturbed.eigv} whose eigenvalues \eqref{eig.0} are close to $ 0 $. 
    Since $U_{\al, \mu,0}$ is a Fourier multiplier, cf. \Cref{lem:Kato1} ,  and $ v_1^+(0,0) $ is supported on the harmonic $e^{\im x}$, 
    while $v_0^\pm(\al, \mu)$, $v_1^-(\al,\mu)$ are supported on different harmonics, we deduce that
\begin{equation}\label{66:1843}
    U_{\al, \mu,0} v_1^+(0,0)  = \nu(\al, \mu) v_1^+(\al, \mu)
    \quad \text{for some}\quad 
    \nu(\al, \mu) \in \C \setminus \{0\} \, .  
\end{equation}
    We now show that $\nu(\al, \mu) = 1$. Indeed, by symplecticity,  
    \begin{align*}
 &   \cW_c \big(  U_{\al, \mu,0} v_1^+(0,0),  U_{\al, \mu,0} v_1^+(0,0) \big) \stackrel{\eqref{Usymp}} =  \cW_c \big(   v_1^+(0,0),   v_1^+(0,0) \big) \stackrel{ \eqref{symp.nonzero}} = - \im \\
&   \cW_c \big(  U_{\al, \mu,0} v_1^+(0,0),  U_{\al, \mu,0} v_1^+(0,0) \big) 
\stackrel{\eqref{66:1843}} 
= |\nu(\al, \mu)|^2 \cW_c\big( v_1^+(\al, \mu), v_1^+(\al, \mu)) 
\stackrel{ \eqref{symp.nonzero}} = - \im  |\nu(\al, \mu)|^2
    \end{align*}
implying  that $|\nu(\al, \mu)| = 1$.
    Furthermore, by reversibility,
    using \eqref{Usymp},  \eqref{rev.nonzero}, 
    \eqref{66:1843},   
\eqref{def:cinvolution}, we have  
    \begin{align*}
     - \nu(\al, \mu)\,  v_1^+(\al, \mu) = U_{\al, \mu, 0} [\varrho_c v_1^+(0,0)] =    \varrho_c \left( U_{\al, \mu,0} v_1^+(0,0)\right) = - \bar{\nu(\al, \mu)} v_1^+(\al, \mu) 
    \end{align*}
    implying  that $\nu(\al, \mu) $ is real. Hence $\nu(\al, \mu) \in \{\pm 1 \}$. 
    Finally, since the left hand side of \eqref{66:1843}  and $v_1^+(\al, \mu)$ tend to $v_1^+(0,0)$ as $(\al, \mu) \to (0,0)$, 
    we deduce that $\nu(\al, \mu) = 1 $.  This  proves \eqref{Uv00} for $v_1^+$. The case $v_1^-$ is analogous. 
    Then 
\eqref{vtrans} follows also recalling \eqref{base3e}-\eqref{v1pv1-}.
\end{proof}

The expansions  \eqref{exf41}-\eqref{exf44} follow by \eqref{ordinibase}, 
\Cref{lem:jetsP,jetsordine2} and the parity properties of the remainders $\cO(\rho^2)$ and $\cO(\e^2)$,  which follow by  \eqref{lem:revbasis} and  \Cref{expe0,expam0,identityvtransholdstrue}. Equations \eqref{nonzeroaverage} and \eqref{vtrans} are respectively proved in \Cref{expe0} and \Cref{identityvtransholdstrue}.
The  proof of  \Cref{expansion1} is complete.

\section{Computation of $\fa_\tp$, $\fc_\tp$ and $\fb_2$}\label{sec:5}
 
We now prove \Cref{prop:values,lem:coeff.b2}. 
We have to compute the second order jets  
\begin{equation}
    \label{generalapcp}
 \fa_\tp = \left(\mathfrak{B}^{(\tp)}_2 v^{(\tp)}_+, v^{(\tp)}_+ \right)\, , \quad \fb_2 = \left(\mathfrak{B}^{(2)}_2 v^{(2)}_-, v^{(2)}_+ \right)\, , \quad 
   \fc_\tp = \left(\mathfrak{B}^{(\tp)}_2 v^{(\tp)}_-, v^{(\tp)}_- \right)\, , \ \quad 
   \forall \tp \geq 3 \, , 
\end{equation}
of the functions (cf.   \eqref{tocomputematrixb},
\eqref{tocomputematrixB}) 
$$
\begin{aligned}
& \fa^{(\tp)}(\al,\mu,\e) = ( \mathfrak{B}^{(\tp)}
(\al, \mu,\e) v^{(\tp)}_+(\al,\mu), v^{(\tp)}_+(\al,\mu)) 
\, , \\
&  \fb^{(2)}(\al,\mu,\e) 
= (\mathfrak{B}^{(2)}(\al, \mu,\e) v^{(2)}_-(\al,\mu), v^{(2)}_+(\al,\mu)) \, ,   \\
& 
\fc^{(\tp)}(\al,\mu,\e) = (\mathfrak{B}^{(\tp)}(\al, \mu,\e) v^{(\tp)}_-(\al,\mu), v^{(\tp)}_-(\al,\mu)) \, .
\end{aligned}
$$
By \eqref{regBtp}, each $\fa_{\tp}(\alpha,\mu),\, \fb_{\tp}(\alpha,\mu),\, \fc_{\tp}(\alpha,\mu) $ belongs to $\cA(K^{(\tp)},+\infty;\R)$ for $\tp \geq 3$, or to $\cA_P(K^{(2)},+\infty;\R)$ if $\tp = 2 $.
\\[1mm]
{\bf Notation.} 
In this section, we denote $\cO(\e^n)$ a bounded operator $A(\al,\mu,\e)$ in  $\cA(K^{(\tp)},\e^{(\tp)};H^1(\T;\C^2),L^2(\T;\C^2))$ for some $\tp\geq 3$, or to $\cA_P(K^{(2)},\e^{(2)};H^1(\T;\C^2),L^2(\T;\C^2))$, satisfying 
$ \|A(\al,\mu,\e)\|_{\cL(H^1(\T;\C^2),L^2(\T;\C^2))} \leq C |\e|^n $
for all $ (\al,\mu,\e) \in K^{(\tp)}\times B_{\e^{(\tp)}}(0) $.
\\[1mm]
\noindent{\bf Expansion of ${\cal B} (\al,\mu,\e) $.} For  any $ (\alpha,\mu) \in \R^2 $,
the operator ${\cal B}  (\al,\mu, \e)  $ in 
\eqref{operator Bmualep} 
has the Taylor  expansion 
\begin{equation*}
{\cal B} ( \al,\mu,\e ) = {\cal B}_0(\al,\mu) + \e{\cal B}_1(\al,\mu) + \e^2 {\cal B}_2(\al,\mu)+\cO(\e^3)  
\end{equation*}
where, in view of  \Cref{DNProp1}, \eqref{cG12action} and \eqref{apexp}, 
\begin{equation}
\label{Bsani.appF} 
{\cal B}_0:=
{\cal B}_0(\al,\mu) 
:= \begin{bmatrix} 1 & -( \pa_x + \im \mu) \\  \pa_x  + \im \mu&  |D|_{\al,\mu}\end{bmatrix}
  \, , \   \ \ 
{\cal B}_\ell:={\cal B}_{\ell} (\al,\mu)
:=\begin{bmatrix} a_\ell (x) & -p_\ell (x)(\pa_x+\im  \mu) \\ (\pa_x+\im  \mu)\circ p_\ell (x) & \cG_\ell (\al,\mu) \end{bmatrix} \, ,
\end{equation}
for any $ \ell \in \N  $. 
For any $ |\kappa|\leq \ell $, $\kappa \equiv \ell$ mod $2$, the $\kappa$th-band of the operators in   \eqref{Bsani.appF} are 
\begin{equation}\label{paritymatters}
{\cal B}_{\ell}^{[\kappa]}(\al, \mu) = 
 \dfrac{1}{2}\begin{bmatrix} e^{\im\kappa x} a_\ell^{[\kappa]}  & - e^{\im\kappa x} p_\ell^{[\kappa]}(\pa_x+\im \mu) \\  
 e^{\im\kappa x} p_\ell^{[\kappa]}(\pa_x+\im \mu + \im \kappa ) &  2 \cG_\ell^{[\kappa]}(\al, \mu)  \end{bmatrix} \, ,
\end{equation}
where, in view of \eqref{apexp},
\begin{equation}
\label{pklakl}
 p_1^{[\pm 1]} :=a_1^{[\pm 1]} := -2  \, , \quad
p_2^{[0]} := 3 \, , \quad a_2^{[0]} := 4\, ,\quad 
p_2^{[\pm 2]} := a_2^{[\pm 2]} := -2 \, . 
\end{equation} 
{\bf Expansion of $P^{(\tp)}_{\al,\mu,\e}$.}
For any 
\be \label{fuoripal}
(\al,\mu) \in K^{(\tp)} \quad 
\text{if} \ \tp\geq 3 , \quad 
\text{resp.} \  (\al,\mu) \in K^{(2)}\setminus B_{\rho_2}(0,0) \ 
\text{for some} \ 
\rho_2 > 0 \, , 
\ee
 the  projector $P^{(\tp)}_{\al,\mu,\e}$ in  \eqref{Pproj:all} is given by the integral \eqref{Pproj:tp}, resp. 
\eqref{Pproj:2}, and
has the Taylor  expansion 
\begin{equation*}
P^{(\tp)}_{\al,\mu,\e} =P_0^{(\tp)}(\al,\mu) + \e P_1^{(\tp)}(\al,\mu)  + \cO(\e^2) \, 
\end{equation*}
where 
\begin{align}\label{Psani}
 &P_0^{(\tp)}:= P_0^{(\tp)}(\al,\mu) := P^{(\tp)} _{\al,\mu,0} \ , \quad P_1^{(\tp)}:= P_1^{(\tp)}(\al,\mu) := \mathcal{P}^{(\tp)} \big[\cB_1(\al,\mu)\big]  \, , \\
& \label{hP}
 \mathcal{P}^{(\tp)}  \big[A \big] := \frac{1 }{2\pi \im } \oint_{\Gamma^{(\tp)}(\al, \mu)} (\lambda-\cL(\al,\mu,0))^{-1}\,  \cJ A \, (\lambda-\cL(\al,\mu,0))^{-1} \de\lambda \, ,
 \end{align}
and $\Gamma^{(\tp)}(\al, \mu)$ are  circuits defined in  \Cref{KatonearMcLean}.
\\[1mm]
{ \bf Expansion of $\mathfrak{B}^{(\tp)} (\al,\mu,\e) $.}
For any  $ (\al,\mu) $ 
as in \eqref{fuoripal}
the operator $\mathfrak{B}^{(\tp)} (\al,\mu,\e) $  in \eqref{Bgotico} has  the Taylor expansion
 \begin{equation*}
\mathfrak{B}^{(\tp)} (\al,\mu,\e) =
\mathfrak{B}_0^{(\tp)}(\al,\mu) +
\e \mathfrak{B}_1^{(\tp)}(\al,\mu) +
\e^2 \mathfrak{B}_2^{(\tp)}(\al,\mu) +\cO(\e^3)
\end{equation*}
where, cf. \cite[Lemma 3.6]{BMV_ed},
 \begin{equation}\label{ordini012}
\mathfrak{B}_0^{(\tp)}  := [P_0^{(\tp)}]^*{\cal B}_0 P_0^{(\tp)} \, , 
 \qquad 
 \mathfrak{B}_1^{(\tp)} := [P_0^{(\tp)}]^*{\cal B}_1 P_0^{(\tp)} \, , 
 \qquad 
 \mathfrak{B}_2^{(\tp)} := [P_0^{(\tp)}]^*\mathbf{Sym}[{\cal B}_2+{\cal B}_1 \, P_1^{(\tp)}]  P_0^{(\tp)}\, , 
\end{equation}
with operators $ ({\cal B}_\ell)_{\ell =0,1,2} $  in \eqref{Bsani.appF}, 
$ (P_\ell^{(\tp)})_{\ell =0,1} $  in \eqref{Psani} and $\mathbf{Sym}[A]:= \frac12 A+ \frac12 A^*$.
\\[1mm]
{\bf Entanglement coefficients.}
We introduce  the {\it entanglement coefficients} 
\begin{equation}\label{entcoeff}
\ent{\ell }{\kappa }{j'}{j}{\sigma'}{\sigma}
:=\ent{\ell }{\kappa }{j'}{j}{\sigma'}{\sigma}(\al,\mu)
:= \big({\cal B}_\ell^{[\kappa]} v_j^\sigma, v_{j'}^{\sigma'}  \big)   \, ,\qquad \ell \in \N_0 \, ,\ j',j\in \Z\, , \ \sigma,\sigma'=\pm\, ,
\end{equation}
where  $\cB_\ell^{[\kappa]}$ is the $\kappa$th-band  operator  in \eqref{band.def} of $\cB_\ell$ in  \eqref{Bsani.appF}.
The entanglement coefficients are well defined only for $(\al,\mu) \not \in \{0\}\times \Z $, where the $v_j^\sigma(\al,\mu)$ in \eqref{unperturbed.eigv} are well defined,  and satisfy 
\begin{equation}\label{entprop}
\ent{\ell}{\kappa }{j'}{j}{\sigma'}{\sigma} \equiv 0 \ \text{ if } \ \sigma'j'\neq \sigma j+\kappa \, ,\quad \quad \bar{\ent{\ell}{\kappa }{j'}{j}{\sigma'}{\sigma}} =\ \ent{\ell}{-\kappa }{j}{j'}{\sigma}{\sigma'}\, . 
\end{equation} 
The next lemma gives effective formulas to compute the action of  $\cJ \cB_\ell$ and   $ \mathcal{P}^{(\tp)}[\cB_{\ell}]$ on the 
basis
$\{v_j^\sigma\}$
in \eqref{unperturbed.eigv}.

\begin{lemma}
\label{actionofL}
Let $(\al,\mu)\in K^{(\tp)} \setminus (\{0\}\times \Z)$, $\tp\geq3$, or $(\al,\mu)\in K^{(2)} \setminus B_{\rho_2}(0,0)$.  Then the following hold:\\
$(i)$ for any $\ell\in \N_0$ and $j,\kappa\in \Z$ and $\sigma = \pm$, one has
\begin{equation}\label{Lsacts}
\cJ {\cal B}_\ell^{[\kappa]} v_j^\sigma =
 \sum_{\sigma_1 = \pm \atop \sigma_1 j_1 = \sigma j + \kappa } -\im \sigma_1  
 \, 
  \ent{\ell}{\kappa}{ j_1}{j}{\sigma_1}{\sigma}\; v_{ j_1}^{\sigma_1}  \ . 
\end{equation}
$(ii)$ Let $(j,\sigma)$ such that the eigenvalue  $\im\omega_j^\sigma$ of $ {\cL } (\alpha, \mu,0)$ in \eqref{eig.0} lies inside the circuit $\Gamma^{(\tp)}$ (see \eqref{separationinproof} and \eqref{McL.d}), and let $\kappa\in\Z$, $\kappa\notin \{ 0,-\sigma\tp\}$. 
Then the operator $ \mathcal{P}^{(\tp)}[\cB_{1}^{[\kappa]}]  $ defined by \eqref{hP} satisfies 
\begin{align}
\label{hppar}
\mathcal{P}^{(\tp)}[\cB_{1}^{[\kappa]}]  v_j^\sigma =
\sum_{\sigma_1= \pm \atop \sigma_1 j_1 = \sigma j + \kappa  } \,  \sigma_1 \frac{\ent{1}{\kappa}{j_1}{j}{\sigma_1}{\sigma}}{\omega_{j_1}^{\sigma_1}-\omega_j^\sigma}\; v_{j_1}^{\sigma_1} \, .
\end{align} 

\smallskip
\noindent$(iii)$ Let $(j,\sigma)$ as in item $(ii)$ and $(j',\sigma') = (-(j-\tp),-\sigma)$ (so that $\im\omega_{j'}^{\sigma'}$ lies inside $\Gamma^{(\tp)}$). Then for any $\kappa_1, \kappa_2 \in \Z \setminus \{  0, -\sigma \tp\}$, one has 
\begin{align}
&\big( \cB_{1}^{[\kappa_1]}\mathcal{P}^{(\tp)}[\cB_{1}^{[\kappa_2]}]  v_j^\sigma , v_{j'}^{\sigma'}\big)  =  \sum_{\substack{\sigma_1= \pm \\ \sigma_1 j_1  = \sigma j + \kappa_2 \\
\ \sigma'j' = \sigma_1 j_1 + \kappa_1}}\!\!\!\!\!\!\!  \sigma_1
\frac{
\ent{1}{\kappa_2}{j_1}{j}{\sigma_1}{\sigma} \, 
\ent{1}{\kappa_{1}}{j'}{j_1}{\sigma'}{\sigma_1} 
}{ \omega_{j_1}^{\sigma_1}-\omega_j^\sigma} \label{mirabilis} \\
&\big( \cB_{1}^{[\kappa_2]}  v_j^\sigma , \mathcal{P}^{(\tp)}[\cB_{1}^{[-\kappa_1]}]v_{j'}^{\sigma'}\big) =\sum_{\substack{\sigma_1= \pm \\ \sigma_1 j_1  = \sigma j + \kappa_2 \\
\ \sigma'j' = \sigma_1 j_1 + \kappa_1}}\!\!\!\!\!\!\!  \sigma_1
\frac{
\ent{1}{\kappa_2}{j_1}{j}{\sigma_1}{\sigma} \, 
\ent{1}{\kappa_{1}}{j'}{j_1}{\sigma'}{\sigma_1} 
}{ \omega_{j_1}^{\sigma_1}-\omega_{j'}^{\sigma'}} \, . \label{mirabilis2}
\end{align}
\end{lemma}
\begin{proof}
$(i)$ 
Since the subspace of vectors supported on the harmonic $j$ is spanned by $\{v_{j}^+, v_{-j}^-\}$ 
(cf. \eqref{unperturbed.eigv}), and the action of 
operator $\cJ {\cal B}_\ell^{[\kappa]}$ 
shifts the harmonics of $ \kappa$, we write 
$ \cJ {\cal B}_\ell^{[\kappa]} v_j^\sigma = \alpha^- v_{-(\sigma j+\kappa)}^- + \alpha^+  v_{\sigma j+\kappa}^+ $. 
Then
\begin{equation}\label{troppocaldo}
\big( \cJ {\cal B}_\ell^{[\kappa]} v_j^\sigma, \cJ v_{-(\sigma j+\kappa)}^- \big) = \alpha^- \big(  v_{-(\sigma j+\kappa)}^-, \cJ v_{-(\sigma j+\kappa)}^- \big) + \alpha^+ \big(  v_{\sigma j+\kappa}^+, \cJ v_{-(\sigma j+\kappa)}^- \big) \stackrel{\eqref{symp.nonzero}}{=} -\im \alpha^- \, . 
\end{equation}
 On the other hand the left-hand side of \eqref{troppocaldo} writes as
$$
\big( \cJ {\cal B}_\ell^{[\kappa]} v_j^\sigma, \cJ v_{-(\sigma j+\kappa)}^- \big) = \big({\cal B}_\ell^{[\kappa]} v_j^\sigma, v_{-(\sigma j+\kappa)}^- \big) \stackrel{\eqref{entcoeff}}{=} \ent{\ell}{\kappa}{-(\sigma j+\kappa)}{j}{-}{\sigma} \stackrel{\eqref{troppocaldo}} = - \im \al^- \, .
$$
Similarly $ \ent{\ell}{\kappa}{\sigma j+\kappa}{j}{+}{\sigma} =  \big( \cJ {\cal B}_\ell^{[\kappa]} v_j^\sigma, \cJ v_{\sigma j+\kappa}^+ \big) =  \im \alpha^+ $
and  \eqref{Lsacts} follows. 
\\[1mm]
$(ii)$ 
By \eqref{eig.0}-\eqref{unperturbed.eigv} we have $(\lambda-\cL(\al,\mu,0))^{-1} v_j^\sigma = \frac{1}{\lambda-\im\omega_j^\sigma} v_j^\sigma$
and, 
using \eqref{Lsacts}, we obtain
\be\label{intENT}
(\lambda-\cL(\al,\mu,0))^{-1}\cJ {\cal B}_{1}^{[\kappa]} (\lambda-\cL(\al,\mu,0))^{-1} v_{j}^{\sigma}  = \sum_{\substack{\sigma_1 = \pm\\ \sigma_1 j_1 = \sigma j + \kappa}} -\im \sigma_1  \frac{ \ent{1}{\kappa}{j_1}{j}{\sigma_1}{\sigma} }{(\lambda-\im \omega_{j_1}^{\sigma_1})(\lambda-\im \omega_{j}^{\sigma})} v_{j_1}^{\sigma_1} \, .
\ee
Integrating \eqref{intENT} as in \eqref{hP}, and noting that $\im \omega_j^\sigma$ lies inside $\Gamma^{(\tp)}$, while $\im \omega_{j_1}^{\sigma_1}$ lies outside  $\Gamma^{(\tp)}$ (since $\kappa\neq -\sigma\tp$ and $\Gamma^{(\tp)}$ only contains the pair of eigenvalues 
$ \{ \im\omega_{j}^{\sigma}, \im\omega_{-(j-\tp)}^{-\sigma} \}$, see \eqref{separationinproof}) we  obtain \eqref{hppar} by the residue Theorem.
\\[1mm]
$(iii)$ Identity  \eqref{mirabilis} 
follows by  \eqref{hppar}, \eqref{entcoeff}, \eqref{entprop}. 
To prove \eqref{mirabilis2}  note that \eqref{hppar} holds also with $(j,\sigma)\leadsto(j',\sigma') = (-(j-\tp),-\sigma)$, since $\im\omega_{j'}^{\sigma'}$ also lies inside $\Gamma^{(\tp)}$, and $ -\kappa_1 \neq - \sigma' \tp $. Then use \eqref{entcoeff}
and \eqref{entprop}. 
\end{proof}

We now give explicit formulas for the entanglement  coefficients  \eqref{entcoeff}.

\begin{lemma}\label{lem:entexp}
Let $(\al,\mu) \not \in \{0\}\times \Z $.
For any $\ell \in \{ 1,2\}$, for any $|\kappa| \leq \ell$ with the same parity of $\ell$, for any $\sigma,\sigma' \in \{\pm\} $, $j,j' \in \Z$ with $\sigma' j' = \sigma j + \kappa$, the  entanglement coefficients  \eqref{entcoeff} is 
\begin{equation} \label{entexp}
\begin{aligned}
     \ent{\ell}{\kappa}{j'}{j}{\sigma'}{\sigma} &=\frac{
c_{\ell,\sigma j}^{\sgn(\kappa)}(\al, \mu) }{2\sqrt{\Omega_\al( \sigma j +\mu)\, \Omega_\al( \sigma'j'+\mu)}}  \\
&\qquad + \frac{\sigma \sigma'\sqrt{\Omega_\al(\sigma j+ \mu)\, \Omega_\al(\sigma'j' + \mu)}}{4}\left( 
a_\ell^{[\kappa]}  - p_\ell^{[\kappa]} \left( 
\frac{\sigma(\sigma j + \mu)}{\Omega_\al(\sigma j + \mu)} + 
\frac{\sigma' (\sigma'j' + \mu)}{\Omega_\al(\sigma' j' + \mu)}
\right)
\right) \
\end{aligned}
\end{equation}
where $c_{\ell, n}^{\pm}(\al, \mu)$ and $c^0_{2,n}(\al, \mu)$ are given in \eqref{actioncG12}, and we denote $\sgn (0) = 0$.     The entanglement coefficients  \eqref{entexp} extend continuously at any $(\al,\mu) \in \{0\}\times \Z$. 
\end{lemma}

\begin{proof}
In view of \eqref{paritymatters},   \eqref{unperturbed.eigv}, \eqref{cG12action} we obtain
\begin{equation}\label{cBlkappavjsigma}
    {\cal B}_{\ell}^{[\kappa]} v_j^\sigma = \frac{e^{\im (\sigma j+\kappa )x}}{2\sqrt{2\Omega_\al (\sigma j + \mu)}}\vet{\im \sigma a_\ell^{[\kappa]}\Omega_\al(\sigma j + \mu) - \im p_\ell^{[\kappa]}(\sigma j + \mu)}{-\sigma p_\ell^{[\kappa]}(\sigma j + \kappa + \mu)\Omega_\al(\sigma j + \mu) + 2 c_{\ell,\sigma j}^{\sgn (\kappa)} } .
\end{equation}
Taking the scalar product of \eqref{cBlkappavjsigma} with $v_{j'}^{\sigma'}$ in \eqref{unperturbed.eigv}, $\sigma'j'=\sigma j + \kappa$, we obtain \eqref{entexp}. 
    The formula in \eqref{entexp} extend continuously at any $(\al,\mu) \in \{0\}\times \Z$: the ratio $\tfrac{\sigma j + \mu}{\Omega_\al(\sigma j + \mu)}\to 0$ as  $(\al,\mu) \to (0, -\sigma j)$, and $c_{\ell,\sigma j}^{\sgn (\kappa)} (\al,\mu) = \cO(\al^2) $ by \eqref{actioncG12}, so that $\tfrac{
c_{\ell,\sigma j}^{\sgn(\kappa)}(\al, \mu) }{2\sqrt{\Omega_\al( \sigma j +\mu)\, \Omega_\al( \sigma'j'+\mu)}} \to 0$ as $(\al,\mu) \to (0,-\sigma j)$ or $(\al,\mu) \to (0,-\sigma' j')$. 
\end{proof}

 We now compute
$\fa_\tp(\al,\mu), \fc_\tp(\al,\mu)$, $ \fb_2 (\al,\mu) $
in terms of the entanglement coefficients  \eqref{entexp}
for any  $ (\al,\mu) $ 
satisfying  \eqref{fuoripal}. 
\\[1mm]
\noindent{\bf Computation of $\fa_\tp$}.
In view of \eqref{ivMC}, we 
have 
$v^{(\tp)}_+ = v_\tm^+$ 
where $\tm = \frac{\tp}{2}$ if $\tp$ is even, resp. $\tm = \frac{\tp-1}{2}$ if $\tp$ is odd. 
Then using  \eqref{ordini012},  
$P_{1}^{[\pm1]} := {\mathcal P}^{(\tp)}[\cB_{1}^{[\pm 1]}]$ and  \eqref{mirabilis}, 
\be
\label{fapappendix}\begin{aligned}
 \fa_\tp \stackrel{\eqref{generalapcp}}
 =\left( \kB^{(\tp)}_{2} v_\tm^+, v_\tm^+ \right)  &= 
 \left( \cB_{2}  v_\tm^+,  v_\tm^+ \right) 
 +
 \tfrac12 \left( \cB_{1} P_{1}^{(\tp)}   v_\tm^+,  v_\tm^+ \right) 
  +
\tfrac12   \left( \cB_{1}  v_\tm^+,  P_{1}^{(\tp)}  v_\tm^+ \right) \\
   & = 
   	\underbrace{ \big( {\cal B}_{2}^{[0]}  v_\tm^+,  v_\tm^+\big)}_{ =:\mathrm{Aa}  }
	 + 
   \tfrac12
  \underbrace{ \big( {\cal B}_{1}^{[+1]} P_{1}^{[-1]} v_\tm^+,  v_\tm^+\big)}_{=:\mathrm{Ab} } 
    +
    \tfrac12 
\underbrace{    \big( {\cal B}_{1}^{[-1]} P_{1}^{[+1]} v_\tm^+,  v_\tm^+\big)}_{=:\mathrm{Ac} }
\\ &\quad +
 \tfrac12
 \underbrace{\big( {\cal B}_{1}^{[+1]}  v_\tm^+, P_{1}^{[+1]} v_\tm^+\big) }_{=\mathrm{Ac}}
 +
 \tfrac12 
\underbrace{ \big( {\cal B}_{1}^{[-1]} v_\tm^+, P_{1}^{[-1]} v_\tm^+\big) }_{=\mathrm{Ab}}
= \mathrm{Aa} + \mathrm{Ab} + \mathrm{Ac}\, ,
   \end{aligned}
   \ee
   where the second and fifth terms are equal by comparing \eqref{mirabilis} and \eqref{mirabilis2} with $(j,\sigma)=(j',\sigma')=(\tm,+)$, analogously for the third and fourth. Then,  using also \eqref{entcoeff}, \eqref{entprop} and \eqref{mirabilis}, we have 
\begin{align} \label{A}
   & \mathrm{Aa} \stackrel{\eqref{entcoeff}}{=}  {\tB_{2}^{[0]}}_{\tm,\tm}^{+,+}  \, \\ 
   & \mathrm{Ab} = \big( {\cal B}_{1}^{[+1]} P_{1}^{[-1]} v_\tm^+,  v_\tm^+\big) = 
  \big( {\cal B}_{1}^{[-1]}  v_\tm^+, P_{1}^{[-1]} v_\tm^+\big)   =
  \sum_{\sigma =\pm\atop \sigma j = \tm - 1}  \sigma \frac{  |\ent{1}{-1}{j}{\tm}{\sigma}{+}|^{2}}{\omega_j^{\sigma}- \omega_{\tm}^+}
   = 
    \frac{ |\ent{1}{-1}{\tm-1}{\tm}{+}{+}|^{2}}{\omega_{\tm-1}^{+}- \omega_{\tm}^+}
    -
     \frac{  |\ent{1}{-1}{1-\tm}{\tm}{-}{+}|^{2}}{\omega_{1-\tm}^{-}- \omega_{\tm}^+}\, ,\notag \\
   & \mathrm{A c} 
   = \big( {\cal B}_{1}^{[-1]} P_{1}^{[+1]}v_\tm^+, v_\tm^+\big)
   =  \big( {\cal B}_{1}^{[+1]} v_\tm^+, P_{1}^{[+1]} v_\tm^+\big) 
   =  \sum_{\sigma =\pm\atop \sigma j = \tm + 1}  \sigma \frac{ |\ent{1}{+1}{j}{\tm}{\sigma}{+}|^{2}}{\omega_j^{\sigma}- \omega_{\tm}^+} = 
    \frac{  |\ent{1}{+1}{\tm+1}{\tm}{+}{+}|^{2}}{\omega_{\tm+1}^{+}- \omega_{\tm}^+} 
    -
     \frac{ |\ent{1}{+1}{-\tm-1}{\tm}{-}{+}|^{2}}{\omega_{-\tm-1}^{-}- \omega_{\tm}^+} 
   \, .  \notag 
   \end{align} 
\noindent{\bf Computation of $\fc_\tp$}. 
In  view of \eqref{ivMC}, we have 
$v^{(\tp)}_- = v_\tm^-$ where 
 $\tm = \frac{\tp}{2}$ if $\tp$ is even, resp. $\tm = \frac{\tp+1}{2}$ if $\tp$ is odd. 
Then, arguing as above,   
\be\label{fcpappendix}\begin{aligned}
\fc_\tp \stackrel{\eqref{generalapcp}} =  \left( \kB^{(\tp)}_{2} v_{\tm}^-, v_{\tm}^- \right) 
 &= 
\underbrace{\big( {\cal B}_{2}^{[0]} v_{\tm}^-, v_{\tm}^-\big) }_{= \mathrm{Ca}}
+
 \frac12
 \underbrace{\big( {\cal B}_{1}^{[+1]} P_{1}^{[-1]}v_{\tm}^-, v_{\tm}^-\big)}_{= \mathrm{Cb}}
  +
  \frac12
  \underbrace{ \big( {\cal B}_{1}^{[-1]} P_{1}^{[+1]}v_{\tm}^-, v_{\tm}^-\big)}_{= \mathrm{Cc}}
  \\ &\quad + 
  \frac12
  \underbrace{\big( {\cal B}_{1}^{[+1]} v_{\tm}^-, P_{1}^{[+1]} v_{\tm}^-\big) }_{= \mathrm{Cc}}
  +
  \frac12 
  \underbrace{\big( {\cal B}_{1}^{[-1]} v_{\tm}^-, P_{1}^{[-1]} v_{\tm}^-\big)}_{= \mathrm{Cb}} = 
   \mathrm{Ca} +  \mathrm{Cb} +  \mathrm{Cc}\, ,
  \end{aligned}\ee
where, using also \eqref{entcoeff}, \eqref{entprop},  \eqref{mirabilis},    \eqref{mirabilis2}, 
\begin{equation}\label{C}
    \begin{aligned}
& \mathrm{Ca} \stackrel{\eqref{entcoeff}}{=} {\tB_{2}^{[0]}}_{\tm,\tm}^{- , -}  \\
& \mathrm{Cb} = \big( {\cal B}_{1}^{[+1]} P_{1}^{[-1]}v_{\tm}^-, v_{\tm}^-\big)
= \big( {\cal B}_{1}^{[-1]} v_{\tm}^-, P_{1}^{[-1]} v_{\tm}^-\big)
 \stackrel{\eqref{mirabilis}}{=}
\dfrac{\big|{\tB_{1}^{[-1]}}_{-\tm-1,\tm}^{+,-}\big|^2}{\omega_{-\tm-1}^+ - \omega_{\tm}^-}- \dfrac{\big|{\tB_{1}^{[-1]}}_{\tm+1,\tm}^{-,-}\big|^2}{\omega_{\tm+1}^- - \omega_{\tm}^-} 
 \\
& \mathrm{Cc} = 
\big( {\cal B}_{1}^{[-1]} P_{1}^{[+1]}v_{\tm}^-, v_{\tm}^-\big) = \big( {\cal B}_{1}^{[+1]} v_{\tm}^-, P_{1}^{[+1]} v_{\tm}^-\big)
\stackrel{\eqref{mirabilis}}{=} \frac{|\ent{1}{1}{-\tm+1}{\tm}{+}{-}|^2}{\omega_{-\tm+1}^+-\omega_{\tm}^-}- \frac{|\ent{1}{1}{\tm-1}{\tm}{-}{-}|^2}{\omega_{\tm-1}^--\omega_{\tm}^-} \, . 
\end{aligned}
\end{equation}

\noindent{\bf Computation of $\fb_2$}. In view of \eqref{ivMC}, \eqref{ordini012}, $P_{1}^{[\pm1]}={\mathcal P}^{(2)}[\cB_{1}^{[\pm 1]}]$ we get
\begin{equation}\label{fb2appendix}
\fb_2  \stackrel{\eqref{generalapcp}} =   \left( \kB^{(2)}_{2} v_1^-, v_1^+ \right)
= \underbrace{ \big({\cal B}_{2}^{[+2]} v_1^-\,,\,v_{1}^+\big)}_{=:\mathrm{Ba}} + 
\frac12
\underbrace{ \big({\cal B}_{1}^{[+1]}P_{1}^{[+1]} v_1^-\, , \,v_{1}^+\big)}_{=:\mathrm{Bb}} + 
\frac12 
\underbrace{
\big({\cal B}_{1}^{[+1]} v_1^-\, , \, P_{1}^{[-1]}v_1^+\big)}_{=:\mathrm{Bc}}\, .
\end{equation}
Then
\begin{align}\label{B}
\mathrm{Ba} & \stackrel{\eqref{entcoeff} }{=}  {\tB_2^{[+2]}}_{1,1}^{+,-}  \ , \\ \notag
\mathrm{Bb} &
\stackrel{\eqref{mirabilis}}{=} \sum_{\sigma=\pm}\sigma \frac{\ent{1}{+1}{0}{1}{\sigma}{-}\, \ent{1}{+1}{1}{0}{+}{\sigma}}{\omega_0^{\sigma}-\omega_1^-}=\frac{\ent{1}{+1}{0}{1}{+}{-}\, \ent{1}{+1}{1}{0}{+}{+}}{\omega_0^{+}-\omega_1^-} - \frac{\ent{1}{+1}{0}{1}{-}{-}\, \ent{1}{+1}{1}{0}{+}{-}}{\omega_0^{-}-\omega_1^-} \ , \\
\mathrm{Bc} & \stackrel{\eqref{mirabilis2}}{=} \sum_{\sigma=\pm}\sigma \frac{\ent{1}{+1}{0}{1}{\sigma}{-}\, \ent{1}{+1}{1}{0}{+}{\sigma}}{\omega_0^{\sigma}-\omega_1^+}=\frac{\ent{1}{+1}{0}{1}{+}{-}\, \ent{1}{+1}{1}{0}{+}{+}}{\omega_0^{+}-\omega_1^+} - \frac{\ent{1}{+1}{0}{1}{-}{-}\, \ent{1}{+1}{1}{0}{+}{-}}{\omega_0^{-}-\omega_1^+} \, . \notag
\end{align}
    Note that, in  the formulas
    \eqref{A}, \eqref{C}, resp. \eqref{B},   the entanglement coefficients
are defined for all $(\alpha, \mu) \in \R^2 $
(also at 
$ \{0\}\times \Z $ according to   \Cref{lem:entexp}) and, by \eqref{0905:1850}, the denominators never vanish 
for $\tp\geq 3$
inside $ K^{(\tp)}$, resp. 
$ K^{(2)}\setminus \{(0,0)\}$.

The next lemmata use the Mathematica 
notebook at 
\url{https://git.sissa.it/amaspero/benjamin-feir-3d}.
We first prove \Cref{prop:values}. 

\begin{lemma}\label{expbT}
For any $\tp \geq 3 $ odd, \eqref{apcpatextremes} holds. 
\end{lemma}

\begin{proof}
     Evaluate   $ \fa_\tp (\alpha,\mu)$
     in \eqref{fapappendix}-\eqref{A} with $\tm = \frac{\tp-1}{2}$, inserting the  entanglement coefficients \eqref{entexp} with 
     $ \ell = 1 $ and $a_1^{[\kappa]},p_1^{[\kappa]}$ in \eqref{pklakl}, $\omega_k^\sigma$ in \eqref{eig.0}, at $\al=0$, $\mu=\mu_*^+ (\tp) = 1+\tm+\tm^2 $ and $\mu =  \mu_*^- (\tp) = -\tm-\tm^2 $ (at $\al = 0$ the coefficients $c_{\ell, n}^{\pm}(\al, \mu)$  in \eqref{actioncG12} vanish).  
      Similarly  we evaluate   $ \fc_\tp (\alpha,\mu)$   in \eqref{fcpappendix}-\eqref{C} with $\tm = \frac{\tp+1}{2} $. 
\end{proof}

We now prove \Cref{lem:coeff.b2}.

\begin{lemma}\label{lem:b2.alongMcLean}
   \eqref{b2.alongMcLean} holds.
\end{lemma}

\begin{proof}
The terms $  \fb_2(0,\tfrac54) = 0 $ and $ \pa_\mu \fb_2(0,\tfrac54) = \tfrac{1}{2\sqrt 3} $ are directly computed  by \eqref{fb2appendix}-\eqref{B}, 
\eqref{entexp} and \eqref{pklakl}.

We now prove that 
$\fb_2(\al,\mu)<0$ for all $(\al, \mu) \in \cM^{(2)}\setminus \{(0, \pm \tfrac54)\} $. By the symmetries \eqref{symmetryapcpbp}, it is sufficient to check this on the first quadrant $\{\al>0,\mu>0\}.$ To do this, introduce variables $(r,\theta) \in (0, +\infty) \times \T$ by 
 $\mu = 1 + r \cos \theta,\
        \al = r\sin \theta$. 
        Then, using the first of \eqref{McLean1}, the  portion  in the first quadrant of the unperturbed McLean curve $\cM^{(2)}$  is described by the parametrization   
        
    \begin{minipage}{0.45\linewidth}
        $$
        \theta(r) = \cos ^{-1}\left(-\frac{2 (r+4)}{\sqrt{r}}+\frac{3}{r}+6\right)\, , \ \ r\in [1/4,1] \,  
    $$
    \end{minipage}
    \hfill
    \begin{minipage}{0.45\linewidth}
    \includegraphics[width=\linewidth]{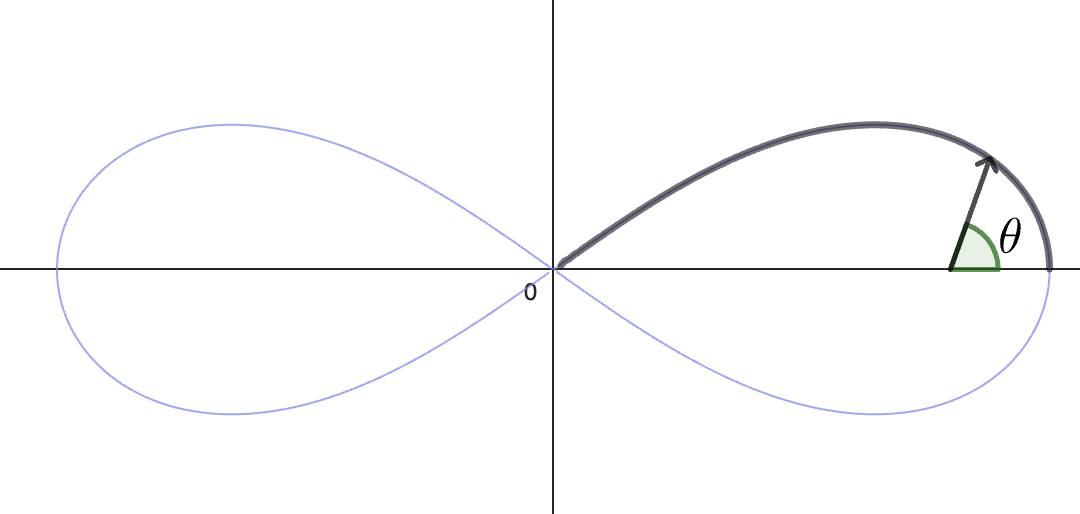}
    \end{minipage}
    
   \noindent The function  $\fb_2$  in \eqref{fb2appendix}-\eqref{B} with 
   the entanglement coefficients \eqref{entexp} 
   $ \ell = 1 $ and $a_1^{[\kappa]},p_1^{[\kappa]}$ in \eqref{pklakl}, $\omega_k^\sigma$ in \eqref{eig.0},
   evaluated at the points $(\al,\mu) = (r\sin (\theta(r)),1+r\cos(\theta(r))) \in \cM^{(2) }$ has the form
   \small
    \begin{equation*}
        \begin{aligned}
            \fb_2 &= \frac{\left(4 r-8 \sqrt{r}+3\right) \left(76 r^{3/2}+24 r^{5/2}-4 r^3-59 r^2-50 r+12 \sqrt{r}+4 \sqrt{r-2 \sqrt{r}+7}-4 \sqrt{r \left(r-2 \sqrt{r}+7\right)}+5\right)}{8 \sqrt[4]{r}\sqrt{2 -\sqrt{r}} \left(r-2 \sqrt{r}+\sqrt{r-2 \sqrt{r}+7}-\sqrt{r \left(r-2 \sqrt{r}+7\right)}+2\right)}\\
            &+\frac{\left(\sqrt{r}-1\right) }{4 \sqrt[4]{r}\sqrt{2-\sqrt{r}} \left(r-2 \sqrt{r}+3\right) \left(-5 \sqrt{r}+\sqrt{r-2 \sqrt{r}+7}+5\right)} \bigg(-6596 r^{3/2}-7558 r^{5/2}-2368 r^{7/2}-160 r^{9/2}\\&\qquad -32 \sqrt{r^7 \left(r-2 \sqrt{r}+7\right)}+16 r^5-752 \sqrt{r^5 \left(r-2 \sqrt{r}+7\right)}+776 r^4+\left(224 \sqrt{r-2 \sqrt{r}+7}+4993\right) r^3\\&\qquad-1946 \sqrt{r^3 \left(r-2 \sqrt{r}+7\right)}+5 \left(304 \sqrt{r-2 \sqrt{r}+7}+1663\right) r^2+\left(1550 \sqrt{r-2 \sqrt{r}+7}+3667\right) r\\&\qquad-1326 \sqrt{r}+126 \sqrt{r-2 \sqrt{r}+7}-690 \sqrt{r \left(r-2 \sqrt{r}+7\right)}+225\bigg) \, . 
        \end{aligned}
    \end{equation*}
    \normalsize
Apart from the common factor $\sqrt[4]{r}\sqrt{2-\sqrt{r}}$ in the denominator, such expression involves only powers of $ x := \sqrt{r}$ and 
$ y := \sqrt{r-2\sqrt{r}+7}$. In terms of $(x,y)$ any solution of  $\fb_2 = 0$ solves the algebraic system 
\begin{equation}\label{systemapF}
           (2 x-3)^2 (2 x-1)^2 P(x,y)   = 0 \, , \quad
            x^2 - 2 x + 7 - y^2 = 0 \, , 
    \end{equation}
    where 
$ P(x,y) := 2 x^7-2 x^6 y-14 x^6+12 x^5 y+53 x^5-39 x^4 y-125 x^4+4 x^3 y^2 +76 x^3 y+188 x^3-12 x^2 y^2
        -96 x^2 y-176 x^2+16 x y^2+72 x y+97 x-8 y^2-27 y-25 $.
System  \eqref{systemapF}  has solutions with 
$x=\frac12$ and $x = \frac32 $.
We look for the 
other possible solutions of
$ P(x,y) = 0 $, $ Q(x,y) := x^2-2x+7-y^2 = 0 
$, by searching for their resultant polynomial $ R(x) $, which is a polynomial 
in $x$ whose zeros are the $x$-coordinates of the common zeros of $P $ and $ Q $.
     Mathematica finds out that the resultant polynomial
    between $ P(x,y) $ and $ Q (x,y ) $  is 
    $
    R(x) = 
    6 (-3 + 2 x) (-1 + 2 x) (3 - 2 x + x^2)^4 $.  
   Therefore {\it all} real solutions of  \eqref{systemapF} correspond to $x = \frac12$ or $x = \frac32$.
    Since  $r\in [1/4,1]$ then  $x = \sqrt{r} \in [1/2,1]$, and thus the only solution of 
    $ \fb_2 (\al,\mu)  = 0 $ has $ x = \frac12 $,  that corresponds to  $(r, \theta) = (\frac14, 0)$ and $(\al,\mu)=(0, \frac54)$. 
 Since, 
 by  \Cref{b00},  $\fb_2(0,0) = -\frac12$,
 the function $\fb_2$, restricted to the McLean curve $\cM^{(2)}$ in the first quadrant is strictly negative except at  $(\al,\mu)=(0,\frac54)$ where it vanishes.
\end{proof}

\begin{lemma}\label{lem:a2b2c2}
 \eqref{b2onmclean} holds.
\end{lemma}

\begin{proof} 
The coefficients $\beta_2(0,\tfrac54)$, $T_1$ and $T_2$ have been computed in \cite{BMV4}. 
We first report 
the notation in \cite{BMV4}.

\noindent{\sc Notation in \cite{BMV4}.} 
The unperturbed eigenvectors
\be\label{basisBMV4}
f_j^\sigma(\mu):=\frac{1}{\sqrt{2\Omega_0(j+\mu)}}\vet{-\sqrt{\sigma}\Omega_0(j+\mu)}{\sqrt{-\sigma}}e^{\im j x}\, , \qquad \text{cf. \cite[equations (1.17), (1.18)]{BMV4}} \, , 
\ee
are equal, 
by comparison with the  $v_j^\sigma(\al,\mu)$ in \eqref{unperturbed.eigv}, to 
\be\label{fjvj}
f_j^\sigma(\mu)=\sqrt{-\sigma} \, v_{\sigma j}^\sigma (0,\mu) \, . 
\ee
Next 
\be\label{bgotBMV}
\kB(\tfrac14,\e) := P^*(\tfrac14,0)\,  U^*(\tfrac14,\e)\, \cB(0,\tfrac14,\e) \, U(\tfrac14,\e) \, P(\tfrac14,0)\, , 
\quad \text{cf.  \cite[equation (2.5)]{BMV4}} \, , 
\ee
where 
\be\label{defPBMV}
\begin{aligned}
& P(\tfrac14,\e) := \oint_\Gamma (\lambda-\cL(0,\tfrac14,\e))^{-1}\frac{\de\lambda}{2\pi\im} \, , \\
&  U(\tfrac14,\e) :=(\uno - (P(\tfrac14,\e)-P(\tfrac14,0))^2)^{-1/2}\Big(P(\tfrac14,\e) P(\tfrac14,0)+(\uno-P(\tfrac14,\e))(\uno-P(\tfrac14,0))\Big)
\end{aligned}
\ee
and 
$\Gamma\subset \C$ is a path, counterclockwise oriented, separating the double eigenvalue $ \lambda_2^+(0,\tfrac14)
= \lambda_0^-(0,\tfrac14) = \im \tfrac34 $  from the rest of the spectrum (cf. \cite[Lemma 5]{BMV4}), for instance $\Gamma = \pa B_{r_2}(\im \tfrac34 ) = \Gamma^{(2)}(0,\tfrac54)$, according to  \eqref{Gpincl}.

It is proved in  
 \cite[equation (5.2b)]{BMV4}  that
\be\label{beta3BMV4}
\im\beta_3 :=\left( \kB_{0,4}\, f_0^-(\tfrac14),f_2^+(\tfrac14)\right) = - \im\frac{39\sqrt{3}}{512} 
\ee
where $\kB_{0,4}$ the fourth-order jet in $\e$ of the operator \eqref{bgotBMV}, i.e.  
\be\label{gotB04BMV}
\kB_{0,4} := \frac{1}{4!}\pa_\e^4\kB(\tfrac14,\e)\vert_{\e=0}  \ , \qquad 
\text{cf. \cite[equations (3.8c), (3.1a)]{BMV4}} \, .
\ee
{\sc Computation of $ \beta_2(0,\tfrac54) $.}
In view of   \eqref{abceta} and \eqref{ivMC}
we have to compute 
\be\label{beta2054}
\beta_2(0,\tfrac54) = \left(\kB_4^{(2)}(0,\tfrac54) \, v_1^-(0,\tfrac54), \, v_1^+(0,\tfrac54)\right) \quad 
\text{where} \quad 
\kB_4^{(2)}(0,\tfrac54) 
\stackrel{\eqref{Bgotpex}} = \frac1{4!}\pa_\e^4\kB^{(2)}(0,\tfrac54,\e)\vert_{\e=0}\, . 
\ee
We claim  that 
\be\label{2002:1240}
 \beta_2(0,\tfrac54)   = - \beta_3  \stackrel{\eqref{beta3BMV4}}{=} \frac{39\sqrt{3}}{512} \, . 
\ee
By the  covariance property  \eqref{cBLmu+k} with $ k = 1 $ we have 
\begin{equation}\label{cLsymmetry6}
        \cL(\al, \mu + 1 , \e)  = e^{-\im x} \, \cL(\al, \mu, \e) \, e^{ \im x} 
    \quad 
    \Rightarrow 
    \quad 
    \big(\lambda - \cL(\al,\mu+1, \e) \big)^{-1}= e^{-\im x} \, \big(\lambda -  \cL(\al,\mu, \e) \big)^{-1} \, e^{\im x}\, .
    \end{equation}   
    Integrating \eqref{cLsymmetry6} along
    the circuit $ \Gamma $ 
    around the unperturbed double eigenvalue 
$\lambda_2^+(0,\tfrac14) = \lambda_0^-(0,\tfrac14) = \lambda_1^+(0,\tfrac54)=\lambda_1^-(0,\tfrac54)$
(use the  covariance property \eqref{labeling.eig}),     
    we get  
    \begin{equation}\label{transfPp}
        e^{-\im x}P(\tfrac14,\e)e^{\im x}  \stackrel{\eqref{defPBMV}} = \oint_{\Gamma}e^{-\im x}(\lambda-\cL(0,\tfrac14,\e))^{-1}e^{\im x} \frac{\de\lambda}{2\pi \im}  
        \stackrel{\eqref{cLsymmetry6}} =  \oint_{\Gamma^{(2)}(0,\frac{5}{4})}(\lambda-\cL(0,\tfrac54,\e))^{-1} \frac{\de\lambda}{2\pi \im} \stackrel{\eqref{P2}}= P^{(2)}_{0,\tfrac54,\e} \, .
    \end{equation}
Thus,  recalling \eqref{Umclean}, \eqref{Bgotico}, \eqref{Bgotpex}, 
\eqref{cBLmu+k} with $ k = 1 $, 
\eqref{defPBMV}, \eqref{bgotBMV},
\eqref{gotB04BMV}, \eqref{beta2054} we have 
\be\label{2002:1238}
 U^{(2)}_{0,\tfrac54,\e} = e^{-\im x} U(\tfrac14,\e)e^{\im x} \, \qquad  \quad \kB_4^{(2)}(0,\tfrac54)= e^{-\im x}\kB_{0,4} e^{\im x}\ . 
\ee
Moreover,    by direct inspection of \eqref{unperturbed.eigv}, 
\be\label{2002:12:36}
 v_j^\sigma(0,\tfrac14) = v_{j - \sigma}^\sigma (0,\tfrac54) e^{\im x} \ . 
 \ee
Therefore we get the chain of identities
\begin{align*}
     \im \beta_3 \stackrel{\eqref{beta3BMV4}, \eqref{fjvj}}{=}  -\im \left(\kB_{0,4}\,  v_0^-(0,\tfrac14), v_2^+(0,\tfrac14) \right) 
     \stackrel{\eqref{2002:12:36}, \eqref{2002:1238}}{=}
      - \im \left(\kB_4^{(2)}(0,\tfrac54)   v_1^-(0,\tfrac54), v_1^+(0,\tfrac54)\right) \stackrel{\eqref{beta2054}}{=} -\im \beta_2(0,\tfrac54)
     \, ,
\end{align*}
 proving \eqref{2002:1240}.
\\[1mm]
{\sc Computation of $ T_1,  T_2 $.}
Finally, in view of  \eqref{tracciaBep}
and \eqref{matrixentries}, we get 
$ T_1 = \pa_\mu ( \omega_-^{(2)}(\al, \mu) - \omega^{(2)}_+(\al, \mu) )\vert_{(0,\frac54)} = \frac43 $.   
  The term $T_2 = \fa_2(0,\tfrac54) + \fc_2(0,\tfrac54) = 19/16 $ as  in \cite{BMV4} is computed by \eqref{fapappendix} and \eqref{fcpappendix}, with entanglement coefficients in \eqref{entexp}, both evaluated at $\tm = 1$ and $(\al,\mu)=(0,\tfrac54)$.
\end{proof}

\footnotesize{
\noindent{\bf Acknowledgments:}
We are grateful to Bernard Deconinck, Ryan Creedon and Walter Strauss for helpful and inspiring discussions on this problem, and Levent Batakci for showing us beautiful numerical computations of the spectrum.
 A. Maspero and A. Radakovic are supported by the European Union  ERC CONSOLIDATOR GRANT 2023 GUnDHam, Project Number: 101124921.
 Views and opinions expressed are however those of the authors only and do not necessarily reflect those of the European Union or the European Research Council. Neither the European Union nor the granting authority can be held responsible for them.
}

\end{document}